\documentclass[reqno,10pt,centertags]{amsart} 
\usepackage{amsmath,amsthm,amscd,amssymb,latexsym,esint,upref,stmaryrd,
enumerate,verbatim,yfonts,bm,mathtools}
\usepackage{color}
\usepackage{hyperref} 
\newcommand*{\mailto}[1]{\href{mailto:#1}{\nolinkurl{#1}}}
\newcommand{\arxiv}[1]{\href{http://arxiv.org/abs/#1}{arXiv:#1}}

\newcommand{\bbC}{{\mathbb{C}}}

\newcommand{\bbN}{{\mathbb{N}}}

\newcommand{\bbQ}{{\mathbb{Q}}}
\newcommand{\bbR}{{\mathbb{R}}}
\newcommand{\bbS}{{\mathbb{S}}}

\newcommand{\bbZ}{{\mathbb{Z}}}

\newcommand{\bsA}{{\boldsymbol{A}}}
\newcommand{\bsB}{{\boldsymbol{B}}}

\newcommand{\bsD}{{\boldsymbol{D}}}

\newcommand{\bsH}{{\boldsymbol{H}}}
\newcommand{\bsI}{{\boldsymbol{I}}}

\newcommand{\bsT}{{\boldsymbol{T}}}

\newcommand{\cA}{{\mathcal A}}
\newcommand{\cB}{{\mathcal B}}
\newcommand{\cC}{{\mathcal C}}

\newcommand{\cF}{{\mathcal F}}

\newcommand{\cH}{{\mathcal H}}
\newcommand{\cI}{{\mathcal I}}

\newcommand{\cK}{{\mathcal K}}
\newcommand{\cL}{{\mathcal L}}
\newcommand{\cM}{{\mathcal M}}
\newcommand{\cN}{{\mathcal N}}

\newcommand{\cP}{{\mathcal P}}

\newcommand{\cR}{{\mathcal R}}
\newcommand{\cS}{{\mathcal S}}
\newcommand{\cT}{{\mathcal T}}

\newcommand{\gB}{{\mathfrak{B}}}

\newcommand{\gF}{{\mathfrak{F}}}

\newcommand{\beq}{\begin{equation}}
\newcommand{\enq}{\end{equation}}

\newcommand{\z}{\zeta}

\DeclareMathOperator{\esssup}{ess.sup} 

\DeclareMathOperator{\ran}{ran}
\DeclareMathOperator{\dom}{dom}

\DeclareMathOperator{\tr}{tr}

\DeclareMathOperator*{\nlim}{n-lim}
\DeclareMathOperator*{\slim}{s-lim}

\DeclareMathOperator{\Det}{det}

\DeclareMathOperator*{\sgn}{sgn}

\renewcommand{\Re}{\text{\rm Re}}
\renewcommand{\Im}{\text{\rm Im}}
\renewcommand{\ln}{\text{\rm ln}}

\newcommand{\loc}{\operatorname{loc}}

\newcommand{\dott}{\,\cdot\,}

\newcommand{\no}{\notag}
\newcommand{\lb}{\label}
\newcommand{\f}{\frac}
\newcommand{\ul}{\underline}
\newcommand{\ol}{\overline}

\newcommand{\wti}{\widetilde}
\newcommand{\Oh}{O}
\newcommand{\oh}{o}
\newcommand{\hatt}{\widehat} 
\newcommand{\bi}{\bibitem}

\newcommand{\LN}{[L^2(\bbR^n)]^N} 
\newcommand{\LrhoN}{[L_{\rho}^2(\bbR^n)]^N} 
\newcommand{\LtauN}{[L_{\tau}^2(\bbR^n)]^N} 
\newcommand{\LmwtitauN}{[L_{-\wti \tau}^2(\bbR^n)]^N} 
\newcommand{\LwtitauN}{[L_{\wti \tau}^2(\bbR^n)]^N} 
\newcommand{\LnuN}{[L_{\nu}^2(\bbR^n)]^N} 
\newcommand{\WoneN}{[W^{1,2}(\bbR^n)]^N} 
\newcommand{\WtwoN}{[W^{2,2}(\bbR^n)]^N} 
\newcommand{\SN}{[S(\bbR^n)]^N} 

\renewcommand{\ge}{\geqslant}

\let\geq\geqslant
\let\leq\leqslant

\makeatletter
\def\theequation{\@arabic\c@equation}

\allowdisplaybreaks 
\numberwithin{equation}{section}

\newtheorem{theorem}{Theorem}[section]
\newtheorem{proposition}[theorem]{Proposition}
\newtheorem{lemma}[theorem]{Lemma}
\newtheorem{corollary}[theorem]{Corollary}
\newtheorem{definition}[theorem]{Definition}
\newtheorem{hypothesis}[theorem]{Hypothesis}

\theoremstyle{remark}
\newtheorem{remark}[theorem]{Remark}

\begin{document}

\title[Limiting Absorption Principle and Spectral Shift Functions]{The Limiting Absorption Principle  
for Massless Dirac Operators, Properties of Spectral Shift Functions, and an Application to the 
Witten Index of Non-Fredholm Operators}

\author[A.\ Carey et al.]{Alan Carey}  
\address{Mathematical Sciences Institute, Australian National University, 
Kingsley St., Canberra, ACT 0200, Australia 
and School of Mathematics and Applied Statistics, University of Wollongong, NSW, Australia,  2522}  
\email{\mailto{acarey@maths.anu.edu.au}}
\urladdr{\url{http://maths.anu.edu.au/~acarey/}}
  
\author[]{Fritz Gesztesy}
\address{Department of Mathematics, 
Baylor University, Sid Richardson Bldg., 1410 S.\,4th Street,
Waco, TX 76706, USA}
\email{\mailto{Fritz\_Gesztesy@baylor.edu}}
\urladdr{\url{http://www.baylor.edu/math/index.php?id=935340}}

\author[]{Galina Levitina} 
\address{Mathematical Sciences Institute, Australian National University, 
Kingsley St., Canberra, ACT 0200, Australia } 
\email{\mailto{Galina.Levitina@anu.edu.au}}
\urladdr{\url{https://maths.anu.edu.au/people/academics/galina-levitina}}

\author[]{Roger Nichols}
\address{Department of Mathematics (Dept.~6956), The University of Tennessee at Chattanooga, 615 McCallie Ave, Chattanooga, TN 37403, USA}
\email{\mailto{Roger-Nichols@utc.edu}}
\urladdr{\url{https://sites.google.com/mocs.utc.edu/rogernicholshomepage/home}}

\author[]{Fedor Sukochev}
\address{School of Mathematics and Statistics, UNSW, Kensington, NSW 2052,
Australia} 
\email{\mailto{f.sukochev@unsw.edu.au}}
\urladdr{\url{https://research.unsw.edu.au/people/scientia-professor-fedor-sukochev}}

\author[]{Dmitriy Zanin} 
\address{School of Mathematics and Statistics, UNSW, Kensington, NSW 2052,
Australia} 
\email{\mailto{d.zanin@unsw.edu.au}}
\urladdr{\url{https://research.unsw.edu.au/people/dr-dmitriy-zanin}}

\date{\today}
\thanks{A.C., G.L., and F.S. gratefully acknowledge financial support from the Australian
Research Council.}

\@namedef{subjclassname@2020}{\textup{2020} Mathematics Subject Classification}

\begin{abstract} 
Applying the theory of strongly smooth operators, we derive a limiting absorption principle (LAP) on any compact interval in $\bbR \backslash \{0\}$ for the free massless Dirac operator, 
\begin{equation*} 
H_0 = \alpha \cdot (-i \nabla) 
\end{equation*} 
in $[L^2(\bbR^n)]^N$, $n \in \bbN$, $n \geq 2$, $N=2^{\lfloor(n+1)/2\rfloor}$. We then use this to demonstrate the absence of singular continuous spectrum of interacting massless Dirac operators $H = H_0 +V$, where the entries of the (essentially bounded) 
matrix-valued potential $V$ decay like $\Oh\big(|x|^{-1 - \varepsilon}\big)$ as $|x| \to \infty$ for 
some $\varepsilon > 0$. This includes the special case of electromagnetic potentials decaying 
at the same rate. In addition, we derive a one-to-one correspondence between embedded eigenvalues of $H$ 
in $\bbR \backslash \{0\}$ and the eigenvalue $- 1$ of the (normal boundary values of the) 
Birman--Schwinger-type operator 
\[
\ol{V_2 (H_0 - (\lambda_0 \pm i 0)I_{[L^2(\bbR^n)]^N})^{-1} V_1^*}. 
\]
Upon expressing $\xi(\dott; H,H_0)$ as normal boundary values of regularized Fredholm determinants to the real axis, we then prove that in the concrete case $(H,H_0)$, under appropriate hypotheses on $V$ (implying the decay of $V$ like $\Oh\big(|x|^{-n -1 - \varepsilon}\big)$ as $|x| \to \infty$), the associated spectral shift function satisfies $\xi(\dott;H,H_0) \in C((-\infty,0) \cup (0,\infty))$,
and that the left and right limits at zero, $\xi(0_{\pm}; H,H_0) = \lim_{\varepsilon \downarrow 0} \xi(\pm \varepsilon; H,H_0)$, exist.

This fact is then used to express the resolvent regularized Witten index of the non-Fredholm operator 
$\bsD_\bsA^{}$ in $L^2\big(\bbR;[L^2(\bbR^n)]^N\big)$ given by 
\[
\bsD_\bsA^{} = \f{d}{dt} + \bsA,
\quad \dom(\bsD_\bsA^{})= W^{1,2}\big(\bbR; [L^2(\bbR^n)]^N\big) \cap \dom(\bsA_-),    
\]
where  
\[
\bsA = \bsA_- + \bsB, \quad \dom(\bsA) = \dom(\bsA_-). 
\]
Here $\bsA$, $\bsA_-$, $\bsA_+$, $\bsB$, and $\bsB_+$ in $L^2\big(\bbR;[L^2(\bbR^n)]^N\big)$ are generated with the help of  the Dirac-type operators $H, H_0$ and potential matrices $V$, via 
\begin{align*} 
& A(t) = A_- + B(t), \; t \in \bbR, \quad A_- = H_0, \quad A_+ = A_- + B_+ = H, \\ 
& B(t)=b(t) B_+, \; t \in \bbR, \quad B_+ = V, 
\end{align*}
in $[L^2(\bbR^n)]^N$, assuming 
\begin{align*} 
& b^{(k)} \in C^{\infty}(\bbR) \cap L^{\infty}(\bbR; dt), \; k \in \bbN_0, \quad b' \in L^1(\bbR; dt),  \\ 
& \lim_{t \to \infty} b(t) = 1, \quad \lim_{t \to - \infty} b(t) = 0.    
\end{align*}
In particular, $A_{\pm}$ are the asymptotes of the family $A(t)$, $t \in \bbR$, as $t \to \pm \infty$ 
in the norm resolvent sense. \big(Here $L^2(\bbR;\cH) = \int_{\bbR}^{\oplus} dt \, \cH$ and 
$\bsT = \int_{\bbR}^{\oplus} dt \, T(t)$ represent direct integrals of Hilbert spaces and operators.\big) 

Introducing the nonnegative, self-adjoint operators 
\[
\bsH_1 = \bsD_{\bsA}^{*} \bsD_{\bsA}^{}, \quad 
\bsH_2 = \bsD_{\bsA}^{} \bsD_{\bsA}^{*}.
\]
in $L^2\big(\bbR;[L^2(\bbR^n)]^N\big)$, one of the principal results proved in this manuscript expresses the resolvent regularized Witten index 
$W_{k,r}(\bsD_\bsA^{})$ of $\bsD_\bsA^{}$ in terms of spectral shift functions via
\begin{align*}
W_{k,r}(\bsD_\bsA^{}) = \xi_L(0_+; \bsH_2, \bsH_1) 
&= [\xi(0_+;H,H_0) + \xi(0_-;H,H_0)]/2, \\ 
& \hspace*{2.3cm} k \in \bbN, \; k \geq \lceil n/2 \rceil.
\end{align*} 

Here the notation $\xi_L(0_+; \bsH_2, \bsH_1)$ indicates that $0$ is a right Lebesgue point for 
$\xi(\dott; \bsH_2, \bsH_1)$, and $W_{k,r}(T)$ represents the $k$th resolvent regularized Witten index of the densely defined, closed operator $T$ in the complex, separable Hilbert space $\cK$, defined by
\[
W_{k,r}(T) = \lim_{\lambda \uparrow 0} (- \lambda)^k 
\tr_{\cK} \big((T^* T - \lambda I_{\cK})^{-k} - (T T^* - \lambda I_{\cK})^{-k}\big),
\]
whenever the limit exists for some $k \in \bbN$. 
\end{abstract}

\maketitle


{\scriptsize{\tableofcontents}}
\normalsize

\section{Introduction} \lb{s1}

The primary motivation writing this manuscript was to express the Witten index of a certain class 
of non-Fredholm operators, generated from multi-dimensional, massless Dirac operators, in terms of appropriate underlying spectral shift functions. This goal necessitated a detailed control over continuity properties (more precisely, the existence of Lebesgue points) for the spectral shift functions involved, and hence the bulk of this manuscript is devoted to an exhaustive investigation of the spectral properties of multi-dimensional, massless Dirac operators. 

To set the stage, let $n \in \bbN$, $n \geq 2$, $N=2^{\lfloor(n+1)/2\rfloor}$, and denote by $\alpha_j$, $1\leq j\leq n$, 
$\alpha_{n+1} := \beta$, $n+1$ anti-commuting Hermitian $N\times N$ matrices with squares equal to $I_N$, that is, 
\begin{equation} \lb{1.1}
\alpha_j^*=\alpha_j,\quad \alpha_j\alpha_k + \alpha_k\alpha_j = 2\delta_{j,k}I_N, 
\quad 1\leq j,k\leq n+1, 
\end{equation}
and introduce in $\LN$ the free massless Dirac operator
\begin{equation}
H_0 = \alpha  \cdot (-i \nabla) = \sum_{j=1}^n \alpha_j (-i \partial_j),\quad \dom(H_0) = \WoneN,  \lb{1.2}
\end{equation}
where $\partial_j = \partial / \partial x_j$, $1 \leq j \leq n$. Introducing the self-adjoint matrix-valued potential 
$V = \{V_{\ell,m}\}_{1 \leq \ell,m \leq N}$ satisfying for some fixed $\rho \in (1,\infty)$, $C \in (0,\infty)$, 
\begin{equation}
V \in [L^{\infty} (\bbR^n)]^{N \times N}, \quad 
|V_{\ell,m}(x)| \leq C \langle x \rangle^{- \rho} \, \text{ for a.e.~$x \in \bbR^n$, $1 \leq \ell,m \leq N$,}    \lb{1.3}
\end{equation}
then permits one to introduce the massless Dirac operator $H$ in $\LN$ via 
\begin{equation}
H = H_0 + V, \quad \dom(H) = \dom(H_0) = \WoneN.  \lb{1.4}
\end{equation}

In this context we recall our convention 
\begin{equation}
[L^2(\bbR^n)]^N = L^2(\bbR^n; \bbC^N), \quad [W^{1,2}(\bbR^n)]^N = W^{1,2}(\bbR^n; \bbC^N), \, 
\text{ etc.,}
\end{equation}
to be used throughout.

Then $H_0$ and $H$ are self-adjoint in $\LN$, 
with essential spectrum covering the 
entire real line,
\begin{equation}
\sigma_{ess} (H) = \sigma_{ess} (H_0) = \sigma (H_0) = \bbR,
\end{equation}
In addition,
\begin{equation}
\sigma_{ac}(H_0) = \bbR, \quad \sigma_p(H_0) = \sigma_{sc}(H_0) = \emptyset.  
\end{equation}

Relying on the theory of (local) Kato-smoothness (see, e.g., 
\cite[Sect.~XIII.7]{RS78}, \cite[Ch.~4]{Ya92}, and \cite[Chs.~0--2]{Ya10}) and its variant, strong (local) Kato-smoothness (see, e.g., \cite[Ch.~4]{Ya92}, \cite[Chs.~0--2]{Ya10}), then yields under the stated hypotheses on $V$ that actually, 
\begin{align}
& \sigma_{ess}(H) = \sigma_{ac}(H) = \bbR,    \lb{1.8} \\
& \sigma_{sc}(H) = \emptyset,      \lb{1.9} \\
& \sigma_s(H) \cap (\bbR \backslash \{0\}) = \sigma_p(H) \cap (\bbR \backslash \{0\}),    \lb{1.10}
\end{align}
with the only possible accumulation points of $\sigma_p(H)$ being $0$ and $\pm \infty$. Relations 
\eqref{1.8}--\eqref{1.10} describe only the tip of the proverbial iceberg in connection with Sections \ref{s2} and \ref{s3}. In fact, leading up to \eqref{1.10} we establish a limiting absorption principle (LAP) on any compact interval in $\bbR \backslash \{0\}$ for the free (i.e., non-interacting) massless Dirac operator 
$H_0$, prove the absence of singular continuous spectrum of $H = H_0 + V$ for matrix elements 
$V_{\ell,m}$, $1 \leq \ell, m \leq N$, of $V$ decaying like $\Oh\big(|x|^{-1 - \varepsilon}\big)$ as $|x| \to \infty$ for some $\varepsilon > 0$, derive H\"older continuity of the boundary values 
$(H_0 - (\lambda \pm i 0)I_{[L^2(\bbR^n)]^N})^{-1}$ in appropriate weighted $L^2$-spaces for $\lambda$ varying in compact subintervals of $\bbR \backslash \{0\}$, and derive H\"older continuity of the boundary values $(H - (\lambda \pm i 0)I_{[L^2(\bbR^n)]^N})^{-1}$ in appropriate weighted $L^2$-spaces for 
$\lambda$ varying in compact subintervals of $\bbR \backslash \{0\}$ away from the possibly embedded eigenvalues of $H$. In particular, factoring $V$ into $V = V_1^* V_2$, we derive a $1-1$-correspondence between embedded eigenvalues of $H$ in $\bbR \backslash \{0\}$ and the eigenvalue $- 1$ of the (normal boundary values of the) Birman--Schwinger-type operator $\ol{V_2 (H_0 - (\lambda_0 \pm i 0)I_{[L^2(\bbR^n)]^N})^{-1} V_1^*}$. 

This leaves open the existence of eigenvalues embedded in the essential spectrum, and particularly, the existence of an eigenvalue $0$. To deal with these situations one follows \cite[Theorems~2.3]{KOY15} and assumes in addition that 
\begin{align}
& \text{$V: \bbR^n \to \bbC^{n \times n}$ is Lebesgue measurable and self-adjoint 
a.e.\ on $\bbR^n$, and that}  \no \\
& \quad \text{for some $R > 0$, $V \in \big[C^1(E_R)\big]^{N \times N}$, where 
$E_R=\{x \in \bbR^n\,|\, |x| > R\}$,}    \lb{1.11} \\
& \text{and}     \no \\
\begin{split} 
& |x|^{1/2} V_{\ell,m}(x) \underset{|x| \to \infty}{=} \oh(1), \quad 
(x \cdot \nabla V_{\ell,m})(x) \underset{|x| \to \infty}{=} \oh(1), \quad 1 \leq \ell,m \leq N,    \lb{1.11a} \\
& \quad \text{uniformly with respect to directions.} 
\end{split} 
\end{align} 
Under all these conditions on $V$ one then obtains 
\begin{equation}
\sigma_p(H) \subseteq \{0\}.        \lb{1.12}
\end{equation}
This still leaves open the possibility of an eigenvalue $0$. To exclude that as well 
\cite[Theorems~2.1]{KOY15} assume in addition that 
\begin{equation}
\esssup_{x \in \bbR^n} |x| \|V(x)\|_{\cB(\bbC^N)} \leq C \, \text{ for some $C \in (0, (n-1)/2)$,} \lb{1.13}
\end{equation}
with $\|\,\cdot\,\|_{\cB(\bbC^N)}$ denoting the operator norm of an $N \times N$ matrix in $\bbC^N$. Then finally, 
\begin{equation}
\sigma_p(H) = \emptyset,        \lb{1.14}
\end{equation}
and hence $H$ and $H_0$ are unitarily equivalent under these conditions on $V$. The facts \eqref{1.12} and \eqref{1.14} are discussed in detail Section \ref{s4}.

We emphasize, however, that in the bulk of this manuscript we will not assume \eqref{1.13} as we explicitly intend to include situations with $0$ an eigenvalue (and/or a threshold resonance) of $H$. 

Section \ref{s5} provides a detailed study of the Green's function (matrix) of the free Dirac operator $H_0$, that is, the integral kernel of the resolvent $(H -z I_{[L^2(\bbR^n)]^N})^{-1}$, in terms of the Hankel function of order $1$ and half integer index $(n-2)/2$ and $n/2$, 
\begin{align}
& G_0(z;x,y) := (H_0 - z I)^{-1}(x,y)     \no \\ 
& \quad = i 4^{-1} (2 \pi)^{(2-n)/2} |x - y|^{2-n} z \, [z |x - y|]^{(n-2)/2} 
H_{(n-2)/2}^{(1)} (z |x - y|) I_N   \lb{1.15} \\
& \qquad - 4^{-1} (2 \pi)^{(2-n)/2} |x - y|^{1-n} [z |x - y|]^{n/2} H_{n/2}^{(1)} (z |x - y|) \, 
\alpha \cdot \f{(x - y)}{|x - y|}.    \no
\end{align}
The Green's function $G_0(z;\, \cdot \,,\, \cdot \,)$ of $H_0$ continuously extends to $z \in \ol{\bbC_+}$, in particular, the limit $z \to 0$ exists, 
\begin{align}
\begin{split} 
& \lim_{\substack{z \to 0, \\ z \in \ol{\bbC_+} \backslash\{0\}}} G_0(z;x,y) := G_0(0+i\,0;x,y)   \\
& \quad = i 2^{-1} \pi^{-n/2} \Gamma(n/2) \, \alpha \cdot \f{(x - y)}{|x - y|^n}, \quad x, y \in \bbR^n, \; x \neq y, \; n \in \bbN, \; n \geq 2,    \lb{1.16} 
\end{split} 
\end{align} 
and no blow up occurs for all $n \in \bbN$, $n \geq 2$. 

This section ends with various boundedness properties of integral operators $R_{0,\delta}$ and 
$R_{0,\delta}(z)$ in $[L^2(\bbR^n)]^N$, $n \geq 2$, associated with integral kernels that are bounded entrywise by 
\begin{equation}
|R_{0,\delta}( \, \cdot \,, \, \cdot \,)_{j,k}| \leq C \langle \, \cdot \, \rangle^{-\delta} 
|G_0(0; \, \cdot \,, \, \cdot \,)_{j,k}| \langle \, \cdot \, \rangle^{-\delta}, \quad \delta \geq 1/2, \; 1 \leq j,k \leq N,  
\lb{1.17} 
\end{equation}
and 
\begin{align}
\begin{split} 
|R_{0,\delta}(z; \, \cdot \,, \, \cdot \,)_{j,k}| \leq C \langle \, \cdot \, \rangle^{-\delta} 
|G_0(z;\, \cdot \,, \, \cdot \,)_{j,k}| \langle \, \cdot \, \rangle^{-\delta},& \\
\delta \geq (n + 1)/4, \; z \in \ol{\bbC_+}, \; 1 \leq j,k \leq N,      \lb{1.18} 
\end{split}
\end{align}  
for some $C \in (0,\infty)$. In particular, we prove that 
\begin{align}
& R_{0,\delta} \in \cB\big([L^2(\bbR^n)]^N\big), \quad \delta \geq 1/2,    \lb{1.19} \\
& R_{0,\delta}(z) \in \cB\big([L^2(\bbR^n)]^N\big), \quad \delta > (n+1)/4, \; z \in \ol{\bbC_+}.    \lb{1.20}
\end{align}

Section \ref{s6} takes the boundedness property of $R_{0,\delta}$ and $R_{0,\delta}(z)$ a step further by proving trace ideal properties. In fact, employing interpolation techniques for trace ideals, we prove, among a variety of related results, that for $n \geq 2$ and $\delta > (n+1)/4$, 
\begin{equation}
R_{0,\delta}(z) \in \cB_p\big([L^2(\bbR^n)]^N\big), \quad p > n, \; z \in \ol{\bbC_+}.    \lb{1.21}
\end{equation}

Since we are dealing with $n$-dimensional Dirac operators, $n \geq 2$, the study of resolvents alone is insufficient and certain $n$-dependent powers of the resolvent of $H_0$ and $H$ naturally enter the analysis. As a result, in Section \ref{s7} we prove that for all $k \in \bbN$, $k \geq n$,
\begin{equation}
\big[(H - z I_{[L^2(\bbR^n]^N)})^{-k} - (H_0 - z I_{[L^2(\bbR^n]^N)})^{-k}\big] 
\in \cB_1\big([L^2(\bbR^n)]^N\big), \quad z \in \bbC \backslash \bbR,
\end{equation}
as well as for all $\varepsilon > 0$,
\begin{equation}
V \big(H_0^2 + I_{[L^2[\bbR^n]^N)}\big)^{-(n/2) - \varepsilon}  \in \cB_1\big([L^2(\bbR^n)]^N\big), 
\end{equation}
now assuming additional decay of $V$ of the type, for some $\varepsilon > 0$, 
\begin{equation}
V \in [L^{\infty} (\bbR^n)]^{N \times N}, \quad 
|V_{\ell,m}(x)| \leq C \langle x \rangle^{- n - \varepsilon} \, \text{ for a.e.~$x \in \bbR^n$, $1 \leq \ell,m \leq N$.}    \lb{1.24}
\end{equation}

The next two sections, Sections \ref{s8} and \ref{s8a}, are devoted to the notion of the spectral shift function for a pair of self-adjoint operators $(S,S_0)$ in $\cH$, particularly building on work of Yafaev \cite{Ya05}: We start by introducing the class of functions $\gF_m(\bbR)$, $m \in \bbN$, by
\begin{align}
& \gF_m(\bbR) := \big\{f \in C^2(\bbR) \, \big| \, 
f^{(\ell)} \in L^{\infty}(\bbR); \text{ there exists } 
\varepsilon >0 \text{ and } f_0 = f_0(f) \in \bbC    \no \\
& \quad  \text{ such that } 
\big(d^{\ell}/d \lambda^{\ell}\big)\big[f(\lambda) - f_0 \lambda^{-m}\big] \underset{|\lambda|\to \infty}{=} 
\Oh\big(|\lambda|^{- \ell - m - \varepsilon}\big), \; \ell = 0,1,2 \big\}     \lb{1.25}
\end{align} 
(it is implied that $f_0 = f_0(f)$ is the same as $\lambda \to \pm \infty$); one observes that 
$C_0^{\infty}(\mathbb{R}) \subset \mathfrak F_m(\mathbb{R})$, $m \in \bbN$. Assuming that 
\begin{align}
\begin{split} 
& \dom(S) = \dom(S_0), \quad (S-S_0) \in \cB(\cH), \\
&\text{for some $0<\varepsilon<1/2$, } \,  
(S-S_0)(S_0^2+I_\cH)^{- (m/2) - \varepsilon}\in \cB_1(\cH),    \lb{1.26}
\end{split} 
\end{align}
the following are then the principal results of Section \ref{s8}: Let $m \in \bbN$, then 
\begin{equation}
[f(S) - f(S_0)] \in \cB_1(\cH), \quad \in \gF_m(\bbR), 
\end{equation}
and there exists a function
\begin{equation} \lb{1.28}
\xi(\,\cdot\,;S,S_0) \in L^1\big(\bbR; (1+|\lambda|)^{-m-1}\, d\lambda\big)
\end{equation}
such that the following trace formula holds, 
\begin{equation}\lb{1.29}
\tr_{\cH}(f(S) - f(S_0)) = \int_{\bbR}d\lambda\, \xi(\lambda;S,S_0)f'(\lambda),\quad f\in \gF_m(\bbR).
\end{equation}
In particular, one has 
\begin{equation} 
\big[(S-zI_{\cH})^{-m} - (S_0-zI_{\cH})^{-m}\big] \in\cB_1(\cH), \quad z \in \bbC \backslash \bbR,  \lb{1.30}
\end{equation} 
and 
\begin{equation} \lb{1.31}
\tr_{\cH}\big((S-zI_{\cH})^{-m} - (S_0-zI_{\cH})^{-m}\big) = -m \int_{\bbR} \frac{\xi(\lambda;S,S_0) d\lambda}{(\lambda - z)^{m+1}},   \quad z\in \bbC\backslash \bbR.
\end{equation}

The following Section \ref{s8a} then derives an explicit representation for the spectral shift function 
$\xi(\,\cdot\,;S,S_0)$ in terms of normal boundary values to the real axis of regularized Fredholm determinants as follows: Slightly extending our set of hypotheses and now assuming that 
$S_0$ and $S$ are self-adjoint operators in $\cH$ with $(S-S_0) \in \cB(\cH)$, we suppose in addition the following two conditions: \\[1mm] 
$(i)$ If $m \in \bbN$ is odd, assume that 
\begin{equation} \lb{1.32}
\big[(S - zI_{\cH})^{-m} - (S_0 - zI_{\cH})^{-m}\big] \in \cB_1(\cH),\quad z\in \bbC\backslash \bbR,
\end{equation} 
and 
\begin{equation} \lb{1.33}
(S-S_0)(S_0 - zI_{\cH})^{-j}\in \cB_{(m+1)/j}(\cH),\quad j\in \bbN,\; 1\leq j\leq m+1. 
\end{equation}
$(ii)$ If $m \in \bbN$ is even, assume that for some $0<\varepsilon<1/2$, 
\begin{equation} \lb{1.34}
(S-S_0) \big(S_0^2+I_\cH\big)^{- (m/2) - \varepsilon}\in \cB_1(\cH).
\end{equation}
(In this case one can show that \eqref{1.33} holds as well). 

Introducing 
\begin{equation} \lb{1.35}
F_{S,S_0}(z) := \ln\big(\Det_{\cH,m+1}\big((S-zI_{\cH})(S_0-zI_{\cH})^{-1}\big)\big), 
\quad z\in \bbC \backslash \bbR,
\end{equation}
where $(\Det_{\cH,m+1}(\dott)$ denotes the $(m+1)$st regularized Fredholm determinant, and 
introducing the analytic function $G_{S,S_0}(\dott)$ in $\bbC \backslash \bbR$ such that
\begin{align} \lb{1.36}
&\frac{d^m}{d z^m}G_{S,S_0}(z)    \no \\
& \quad =\tr_{\cH}\Bigg(\frac{d^{m-1}}{d z^{m-1}}
\sum_{j=0}^{m-1}(-1)^{m-j}(S_0-zI_{\cH})^{-1}\big[(S-S_0)(S_0 - z I_{\cH})^{-1}\big]^{m-j}\Bigg),  \no \\
& \hspace*{9.5cm} z\in \bbC \backslash \bbR,    
\end{align}
the main result of Section \ref{s8a} then reads as follows: If $F_{S,S_0}$ and $G_{S,S_0}$ have normal (or nontangential) boundary values on $\bbR$, then for a.e.~$\lambda\in\bbR$, 
\begin{equation} \lb{1.37}
\xi(\lambda;S,S_0)= \pi^{-1}\Im (F_{S,S_0}(\lambda+i0)) - \pi^{-1}\Im (G_{S,S_0}(\lambda+i0)) 
+ P_{m-1}(\lambda)\,\text{ for a.e.~$\lambda\in \bbR$},
\end{equation}
where $P_{m-1}$ is a polynomial of degree less than or equal to $m - 1$. 

The subsequent two sections then analyze \eqref{1.35} and \eqref{1.36} and their normal boundary 
values to the real axis in the concrete case where $S=H$ and $S_0 = H_0$. 

More precisely, Section \ref{s9} 
establishes continuity properties of $\Im(F_{H,H_0}(\lambda + i0))$, $\lambda \in \bbR$, by invoking a lengthy study of threshold spectral properties of $H$, following an approach by Jensen and Nenciu \cite{JN01}, and,   especially, by Erdo{\u g}an, Goldberg, and Green \cite{EGG19}, \cite{EGG20}, \cite{EG13}, \cite{EG17}. In particular, we recall an exhaustive study of eigenvalues $0$ and /or resonances at $0$ and finally prove that 
under assumptions \eqref{1.11}, \eqref{1.11a}, and \eqref{1.24}, 
$F_{H,H_0}(\dott)$, has normal boundary values on $\bbR\backslash \{0\}$. In addition, the boundary values to $\bbR$ 
of the function $\Im(F_{H,H_0}(z))$, $z \in \bbC_+$, are continuous on $(-\infty, 0) \cup (0,\infty)$, 
\begin{equation}
\Im(F_{H,H_0}(\lambda + i0)) \in C((-\infty,0) \cup (0,\infty)),      \lb{1.38} 
\end{equation}
and the left and right limits at zero, 
\begin{equation}
\Im(F_{H,H_0}(0_{\pm} + i 0)) = \lim_{\varepsilon \downarrow 0} \Im(F_{H,H_0}(\pm \varepsilon + i 0)), 
\lb{1.39}
\end{equation}   
exist. In particular, if $0$ is a regular point for $H$ (i.e., in the absence of any zero energy eigenvalue and resonance of $H$), then 
\begin{equation}
\Im(F_{H,H_0}(\lambda + i0)) \in C(\bbR).     \lb{1.40} 
\end{equation}

Under the following strengthened decay assumption on $V$, for some $\varepsilon > 0$, 
\begin{equation}
V \in [L^{\infty} (\bbR^n)]^{N \times N}, \quad 
|V_{\ell,m}(x)| \leq C \langle x \rangle^{- n - 1 - \varepsilon} \, \text{ for a.e.~$x \in \bbR^n$, $1 \leq \ell,m \leq N$,}    \lb{1.42}
\end{equation}
an unrelenting barrage of estimates finally proves in Section \ref{s10} that if 
$n \in \bbN$ is odd, $n \geq 3$, then $\frac{d^n}{d z^n}G_{H,H_0}(\dott)$ is analytic in $\bbC_+$ and 
continuous in $\ol{\bbC_+}$. If $n \in \bbN$ is even, then $\frac{d^n}{d z^n}G_{H,H_0}(\dott)$ is analytic in 
$\bbC_+$, continuous in $\ol{\bbC_+}\backslash\{0\}$. Moreover, if $n\geq 4$, then
\begin{align}
\bigg\|\frac{d^n}{dz^n}G_{H,H_0}(\dott)\bigg\|_{\cB(\bbC^N)}\underset{\substack{z \to 0, \\ z \in \ol{\bbC_+} \backslash\{0\}}}{=}O\big(|z|^{-[n - (n/(n-1))]}\big),    \lb{1.43}
\end{align}
and if $n=2$, then for any $\delta\in (0,1)$,
\begin{align}
\bigg\|\frac{d^2}{dz^2}G_{H,H_0}(\dott)\bigg\|_{\cB(\bbC^2)}\underset{\substack{z \to 0, \\ z \in \ol{\bbC_+} \backslash\{0\}}}{=} O\big(|z|^{-(1+\delta)}\big).    \lb{1.44} 
\end{align}

Thus, combining \eqref{1.37}--\eqref{1.40}, \eqref{1.43}, \eqref{1.44} finally yields the first principal 
result of Section \ref{s11} in the following form: 
Under the hypotheses \eqref{1.11}, \eqref{1.11a}, and \eqref{1.42}, 
\begin{equation}
\xi(\dott;H,H_0) \in C((-\infty,0) \cup (0,\infty)),      \lb{1.45} 
\end{equation}
and the left and right limits at zero, 
\begin{equation}
\xi(0_{\pm}; H,H_0) = \lim_{\varepsilon \downarrow 0} \xi(\pm \varepsilon; H,H_0), 
\lb{1.46}
\end{equation}   
exist. In particular, if $0$ is a regular point for $H$, then 
\begin{equation}
\xi(\dott;H,H_0) \in C(\bbR).     \lb{1.47} 
\end{equation}

This represents the main spectral theoretic result derived in this manuscript. The remainder of Section \ref{s11} then describes our application to the (resolvent regularized) Witten index of a particular class of 
non-Fredholm operators in the space $L^2\big(\bbR;[L^2(\bbR^n)]^N\big)$ in connection with multi-dimensional, massless Dirac operators.  

This requires some preparations to which we turn next. We recall a bit of notation: Linear operators in the Hilbert space $L^2(\bbR; dt; \cH)$, in short, $L^2(\bbR; \cH)$, will be denoted by calligraphic boldface symbols of the type $\bsT$, to distinguish them from operators $T$ in $\cH$. In particular, operators denoted by $\bsT$ in the Hilbert space $L^2(\bbR;\cH)$ represent operators associated with a 
family of operators $\{T(t)\}_{t\in\bbR}$ in $\cH$, defined by
\begin{align}
&(\bsT f)(t) = T(t) f(t) \, \text{ for a.e.\ $t\in\bbR$,}    \no \\
& f \in \dom(\bsT) = \bigg\{g \in L^2(\bbR;\cH) \,\bigg|\,
g(t)\in \dom(T(t)) \text{ for a.e.\ } t\in\bbR;    \lb{1.48}  \\
& \quad t \mapsto T(t)g(t) \text{ is (weakly) measurable;} \, 
\int_{\bbR} dt \, \|T(t) g(t)\|_{\cH}^2 <  \infty\bigg\}.   \no
\end{align}
In the special case, where $\{T(t)\}$ is a family of bounded operators on $\cH$ with 
$\sup_{t\in\bbR}\|T(t)\|_{\cB(\cH)}<\infty$, the associated operator $\bsT$ is a bounded operator on $L^2(\bbR;\cH)$ with $\|\bsT\|_{\cB(L^2(\bbR;\cH))} = \sup_{t\in\bbR}\|T(t)\|_{\cB(\cH)}$.

For brevity we will abbreviate $\bsI := I_{L^2(\bbR; \cH)}$ in the following and note that in the concrete situation of $n$-dimensional, massless Dirac operators at hand, $\cH = [L^2(\bbR^n)]^N$. 

Denoting 
\begin{equation}
A_- = H_0, \quad B_+ = V, \quad A_+ = A_- + B_+ = H, 
\end{equation}
we introduce two families of operators in 
$[L^2(\bbR^n)]^N$ by 
\begin{align}
\begin{split} 
& B(t) = b(t) B_+, \quad t \in \bbR,     \\
& b^{(k)} \in C^{\infty}(\bbR) \cap L^{\infty}(\bbR; dt), \; k \in \bbN_0, \quad b' \in L^1(\bbR; dt),  \\ 
& \lim_{t \to \infty} b(t) = 1, \quad \lim_{t \to - \infty} b(t) = 0,    \lb{1.50} \\
& A(t) = A_- + B(t), \quad t \in \bbR.      
\end{split} 
\end{align}

Next, following the general setups described in \cite{CGGLPSZ16b}, \cite{CGLPSZ16}--\cite{CLPS20}, 
\cite{GLMST11}, \cite{Pu08}, the operators  
$\bsA$, $\bsB, \bsA '=\bsB'$ are now given in terms of the families $A(t)$, $B(t)$, and $B'(t)$, $t\in\bbR$, as in \eqref{1.48}. In addition, $\bsA_-$  in $L^2\big(\bbR;[L^2(\bbR^n)]^N\big)$ represents 
the self-adjoint (constant fiber) operator defined by 
\begin{align}
&(\bsA_- f)(t) = A_- f(t) \, \text{ for a.e.\ $t\in\bbR$,}   \no \\
& f \in \dom(\bsA_-) = \bigg\{g \in L^2\big(\bbR;[L^2(\bbR^n)]^N\big) \,\bigg|\,
g(t)\in \dom(A_-) \text{ for a.e.\ } t\in\bbR,    \no \\
& \quad t \mapsto A_- g(t) \text{ is (weakly) measurable,} \,  
\int_{\bbR} dt \, \|A_- g(t)\|_{[L^2(\bbR^n)]^N}^2 < \infty\bigg\}.    \lb{1.51}
\end{align} 
At this point one can introduce the fundamental operator $\bsD_\bsA^{}$ in 
$L^2\big(\bbR;[L^2(\bbR^n)]^N\big)$ by 
\begin{equation}
\bsD_\bsA^{} = \f{d}{dt} + \bsA,
\quad \dom(\bsD_\bsA^{})= W^{1,2}\big(\bbR; [L^2(\bbR^n)]^N\big) \cap \dom(\bsA_-),   \lb{1.52}
\end{equation} 
where  
\begin{align}
& \bsA = \bsA_- + \bsB, \quad \dom(\bsA) = \dom(\bsA_-),     \\
& \bsB \in \cB\big(L^2\big(\bbR;[L^2(\bbR^n)]^N\big)\big). 
\end{align}
Here the operator $d/dt$ in $L^2\big(\bbR;[L^2(\bbR^n)]^N\big)$  is defined by 
\begin{align}
& \bigg(\f{d}{dt}f\bigg)(t) = f'(t) \, \text{ for a.e.\ $t\in\bbR$,} \quad 
f \in \dom(d/dt) = W^{1,2} \big(\bbR; [L^2(\bbR^n)]^N\big).     \lb{1.54} 
\end{align} 
Since $\bsD_\bsA^{}$ is densely defined and closed in $L^2\big(\bbR;[L^2(\bbR^n)]^N\big)$, one can introduce the nonnegative, self-adjoint operators 
$\bsH_j$, $j=1,2$, in $L^2\big(\bbR;[L^2(\bbR^n)]^N\big)$ by
\begin{equation}
\bsH_1 = \bsD_{\bsA}^{*} \bsD_{\bsA}^{}, \quad 
\bsH_2 = \bsD_{\bsA}^{} \bsD_{\bsA}^{*}.
\end{equation} 

Introducing the operator $\bsH_0$ in $L^2\big(\bbR; [L^2(\bbR^n)]^N\big)$ by  
\begin{equation}
\bsH_0 = - \f{d^2}{dt^2} + \bsA_-^2, \quad \dom(\bsH_0) 
= W^{2,2}\big(\bbR; [L^2(\bbR^n)]^N\big) \cap \dom\big(\bsA_-^2\big),     \lb{1.56}
\end{equation}
then $\bsH_0$ is self-adjoint and one obtains the following decomposition of the operators $\bsH_j$, 
$j=1,2$, 
\begin{align}
& \bsH_j = - \f{d^2}{dt^2} + \bsA^2 + (-1)^j \bsA^{\prime}    \no \\
& \hspace*{5.5mm} = \bsH_0 + \bsB \bsA_- + \bsA_- \bsB + \bsB^2 + (-1)^j \bsB^{\prime},  
\lb{1.57} \\
& \, \dom(\bsH_j) = \dom(\bsH_0), \quad j =1,2.      \no 
\end{align} 

Next, we turn to a canonical approximation procedure: Consider the characteristic function for the interval $[- \ell,\ell] \subset \bbR$,  
\begin{equation}
\chi_{\ell}(\nu) = \chi_{[- \ell,\ell]}(\nu), \quad \nu \in \bbR, \; \ell \in \bbN,   \lb{1.58}
\end{equation} 
and hence $\slim_{\ell \to \infty} \chi_{\ell}(A_-) = I_{[L^2(\bbR^n)]^N}$.     
Introducing  
\begin{align}
\begin{split} 
& A_{\ell}(t) = A_- + \chi_{\ell}(A_-) B(t) \chi_{\ell}(A_-) = A_- + B_{\ell}(t), \\
& \dom(A_{\ell}(t)) = \dom(A_-), \quad \ell \in \bbN, \; t \in \bbR,    \lb{1.59} \\
& B_{\ell}(t) = \chi_{\ell}(A_-) B(t) \chi_{\ell}(A_-), \quad 
\dom(B_{\ell}(t)) = [L^2(\bbR^n)]^N, \quad \ell \in \bbN, \; t \in \bbR,   \\
& A_{+,\ell} = A_- + \chi_{\ell}(A_-) B_+ \chi_{\ell}(A_-), \quad \dom(A_{+,\ell}) = \dom(A_-), 
\quad \ell \in \bbN,   \\
\end{split} 
\end{align}
one concludes that 
\begin{align} 
& A_{+,\ell} - A_- = \chi_{\ell}(A_-) B_+ \chi_{\ell}(A_-) \in \cB_1\big([L^2(\bbR^n)]^N\big), 
\quad \ell \in \bbN,    \lb{1.60} \\ 
& A_{\ell}'(t) = B_{\ell}'(t) = \chi_{\ell}(A_-) B'(t) \chi_{\ell}(A_-) \in \cB_1\big([L^2(\bbR^n)]^N\big), 
\quad \ell \in \bbN, \; t \in \bbR.    \lb{1.61} 
\end{align}

As a consequence of \eqref{1.60}, the spectral shift functions 
$\xi(\, \cdot \, ; A_{+,\ell}, A_-)$, $\ell \in \bbN$, exist and are uniquely determined by 
\begin{equation}
\xi(\, \cdot \, ; A_{+,\ell}, A_-) \in L^1(\bbR; d\nu), \quad \ell \in \bbN.     \lb{1.63} 
\end{equation}

We also note the analogous decompositions,
\begin{align}
& \bsH_{j,\ell} = - \f{d^2}{dt^2} + \bsA_{\ell}^2 + (-1)^j \bsA_{\ell}^{\prime}    \no \\
& \hspace*{6.7mm} = \bsH_0 + \bsB_{\ell} \bsA_- + \bsA_- \bsB_{\ell} 
+ \bsB_{\ell}^2 + (-1)^j \bsB_{\ell}^{\prime}, \\
& \, \dom(\bsH_{j,\ell}) = \dom(\bsH_0) = W^{2,2}\big(\bbR; [L^2(\bbR^n)]^N\big),  
\quad \ell \in \bbN, \;  j =1,2,  \no 
\end{align}
with
\begin{equation}
\bsB_{\ell} = \chi_{\ell}(\bsA_-) \bsB \chi_{\ell}(\bsA_-), \quad 
\bsB_{\ell}^{\prime} = \chi_{\ell}(\bsA_-) \bsB^{\prime} \chi_{\ell}(\bsA_-), \quad \ell \in \bbN.   
\end{equation} 

Under hypotheses \eqref{1.11}, \eqref{1.11a}, \eqref{1.42}, and 
\begin{equation}
V_{\ell,m} \in W^{4n,\infty}(\bbR^n), \quad 1 \leq \ell, m \leq N,     \lb{1.66} 
\end{equation} 
it is proven in \cite{CLPS20} that  
\begin{align}
\begin{split} 
& [(\bsH_2 - z \, \bsI)^{-q} - (\bsH_1 - z \, \bsI)^{-q}\big], \, 
\big[(\bsH_{2,\ell} - z \, \bsI)^{-q} - (\bsH_{1,\ell} - z \, \bsI)^{-q}\big]    \lb{1.67} \\
& \quad \in \cB_1\big(L^2\big(\bbR; [L^2(\bbR^n)]^N\big)\big), \quad \ell \in \bbN, \; 
 r \in \bbN, \; r \geq  \lceil n/2 \rceil , 
\end{split}
\end{align} 
and 
\begin{align}
& \lim_{\ell\to\infty} \big\|\big[(\bsH_{2,\ell} - z \, \bsI)^{-q} - (\bsH_{1,\ell} - z \, \bsI)^{-q}\big]   \no \\
& \hspace*{1cm} - [(\bsH_2 - z \, \bsI)^{-q} - (\bsH_1 - z \, \bsI)^{-q}\big]
\big\|_{\cB_1(L^2(\bbR;[L^2(\bbR^n)]^N))} = 0,     \lb{1.68} \\ 
& \hspace*{8.25cm} z \in \bbC \backslash [0,\infty).   \no
\end{align} 

Relations \eqref{1.67} together with the fact that $\bsH_j \geq 0$, $\bsH_{j,\ell} \geq 0$, $\ell \in \bbN$, $j=1,2$, implies the existence and uniqueness of spectral shift functions 
$\xi(\, \cdot \,; \bsH_2, \bsH_1)$ and 
$\xi(\, \cdot \,; \bsH_{2,\ell}, \bsH_{1,\ell})$ for the pair of operators $(\bsH_2, \bsH_1)$ and 
$(\bsH_{2,\ell}, \bsH_{1,\ell})$, $\ell \in \bbN$, respectively, employing the normalization 
\begin{equation}
\xi(\lambda; \bsH_2, \bsH_1) = 0, \, 
\xi(\lambda; \bsH_{2,\ell}, \bsH_{1,\ell}) = 0, \quad  \lambda < 0, \; \ell \in \bbN     \lb{1.69} 
\end{equation}
(cf.\ \cite[Sect.~8.9]{Ya92}). Moreover, 
\begin{equation}
\xi(\dott; \bsH_2, \bsH_1) \in L^1\big(\bbR; (1+ |\lambda|)^{-q-1} d\lambda\big).    \lb{1.70}
\end{equation}
Since 
\begin{equation} 
\int_{\bbR} dt \, \|A_{\ell}'(t)\|_{\cB_1([L^2(\bbR^n)]^N)} < \infty, \quad \ell \in \bbN,   \lb{1.71} 
\end{equation} 
employing $b'(\dott) \in L^1(\bbR;dt)$, one obtains (cf.\ \cite{GLMST11}, \cite{Pu08}) 
\begin{equation}
[(\bsH_{2,\ell} - z \, \bsI)^{-1} - (\bsH_{1,\ell} - z \, \bsI)^{-1}\big] 
\in \cB_1\big(L^2\big(\bbR;[L^2(\bbR^n)]^N\big)\big), \quad \ell \in \bbN,  
\end{equation}
and hence 
\begin{equation}
\xi(\dott; \bsH_{2,\ell}, \bsH_{1,\ell}) \in L^1(\bbR; d\lambda), \quad \ell \in \bbN.  
\end{equation}
In addition, \eqref{12.28}, \eqref{1.69}, and \eqref{1.71} imply the approximate trace formula,   
\begin{equation}
\int_{[0,\infty)} \f{\xi(\lambda; \bsH_{2,\ell}, \bsH_{1,\ell}) \, d\lambda}{(\lambda - z)^2} 
= \f{1}{2} \int_{\bbR} \f{\xi(\nu; A_{+,\ell}, A_-) \, d\nu}{(\nu^2 -z)^{3/2}}, \quad 
\ell \in \bbN, \; z \in \bbC \backslash [0,\infty),      \lb{1.74}
\end{equation}
which in turn implies 
\begin{equation}
\xi(\lambda; \bsH_{2,\ell}, \bsH_{1,\ell}) = \begin{cases} \f{1}{\pi} \int_{- \lambda^{1/2}}^{\lambda^{1/2}} 
\f{\xi(\nu; A_{+,\ell}, A_-) \, d \nu}{(\lambda - \nu^2)^{1/2}}, & 
\text{for a.e.~$\lambda > 0$,} \\
0, & \lambda < 0, 
\end{cases} 
\quad \ell \in\bbN,    \lb{1.75} 
\end{equation}
via a Stieltjes inversion argument. 

Given hypothesis \eqref{1.24}, we will prove in Theorem \ref{t7.4} that  
\begin{align} 
\begin{split} 
\big[(A_+ - z I_{[L^2(\bbR^n)]^N})^{-r_0} - (A_- - z I_{[L^2(\bbR^n)]^N})^{-r_0} \big] \in \cB_1\big([L^2(\bbR^n)]^N\big),&    \lb{1.76} \\ 
r_0 \in \bbN, \; r_0 \geq 2 \lfloor n/2 \rfloor + 1, \; z \in \bbC\backslash\bbR.&
\end{split} 
\end{align} 
Since $2 \lfloor n/2 \rfloor + 1$ is always odd, \cite[Theorem~2.2]{Ya05} yields the existence 
of a spectral shift function $\xi(\dott; A_+,A_-)$ for the pair $(A_+, A_-)$ satisfying  
\begin{equation}
\xi(\dott; A_+,A_-) \in L^1\big(\bbR; (1 + |\nu|)^{-q_0-1} d \nu\big).      \lb{1.77} 
\end{equation}
While  \eqref{1.77} does not determine $\xi(\dott; A_+,A_-)$ uniquely, one can show (following \cite[Theorem~4.7]{CGLNPS16}) that there exists a unique spectral shift function $\xi(\dott; A_+,A_-)$ 
given by the limiting relation 
\begin{equation}
\lim_{\ell \to \infty} \xi(\dott; A_{+,\ell},A_-)  = \xi(\dott; A_+,A_-) \, \text{ in $L^1\big(\bbR; (1 + |\nu|)^{-q_0-1} d\nu\big)$,}     \lb{1.78} 
\end{equation}
and hence we will always choose this particular spectral shift function in \eqref{1.78} for the pair $(A_+,A_-)$ in the following.

At this point one can entertain the limit $\ell \to \infty$ in \eqref{1.75}: Indeed, \eqref{1.75} yields 
\begin{align}
& \int_{[0,\infty)} d \lambda \, \xi(\lambda; \bsH_{2,\ell}, \bsH_{1,\ell}) f'(\lambda)  
= \f{1}{\pi} \int_{[0,\infty)} d \lambda \, f'(\lambda) \int_{- \lambda^{1/2}}^{\lambda^{1/2}} 
\f{\xi(\nu; A_{+,\ell}, A_-) \, d \nu}{(\lambda - \nu^2)^{1/2}}    \no \\
& \quad = \f{1}{\pi} \int_{\bbR} d\nu \, \xi(\nu; A_{+,\ell},A_-) F'(\nu), \quad \ell \in \bbN,   \lb{1.79} 
\end{align}
where $F'$ is defined by 
\begin{equation}
F'(\nu) = \int_{\nu^2}^{\infty} d\lambda \, f'(\lambda) (\lambda - \nu^2)^{-1/2}, \quad \nu \in \bbR.
\lb{1.80} 
\end{equation}
The limit $\ell \to \infty$ on left-hand side of \eqref{1.79} is controlled via \eqref{1.68}, and, since $F' \in C^{\infty}_0(\bbR)$,  the right-hand side of \eqref{1.79} is controlled via \eqref{1.78}, implying 
\begin{align}
& \int_{[0,\infty)} d \lambda \, \xi(\lambda; \bsH_2, \bsH_1) f'(\lambda)   
= \f{1}{\pi} \int_{\bbR} d \nu \, \xi(\nu; A_+, A_-) F'(\nu)   \no \\
& \quad = \f{1}{\pi} \int_{\bbR} d \lambda \, f'(\lambda) \int_{- \lambda^{1/2}}^{\lambda^{1/2}} 
\f{\xi(\nu; A_+, A_-) \, d \nu}{(\lambda - \nu^2)^{1/2}} \chi_{[0,\infty)}(\lambda), \quad 
f \in C^{\infty}_0(\bbR),    
\end{align}
and hence 
\begin{equation}
\xi(\lambda; \bsH_2, \bsH_1) = \f{1}{\pi} \int_{- \lambda^{1/2}}^{\lambda^{1/2}} 
\f{\xi(\nu; A_+, A_-) \, d \nu}{(\lambda - \nu^2)^{1/2}} \, \text{ for a.e.~$\lambda > 0$,} 
\lb{1.82} 
\end{equation}
due to our normalization in \eqref{1.69}.  This establishes the limiting relation $\ell \to \infty$ of \eqref{1.75}.

Having established \eqref{1.82}, we turn to the resolvent regularized Witten index of the operator 
$\bsD_\bsA^{}$. Since $\sigma(A_{\pm}) = \bbR$, in particular, $0 \notin \rho(A_+) \cap \rho(A_-)$, 
$\bsD_\bsA^{}$ is a non-Fredholm operator. Even though $\bsD_\bsA^{}$ is a non-Fredholm operator, its Witten index is well-defined and expressible in terms of the spectral shift functions for the pair of operators 
$(\bsH_2, \bsH_1)$ and $(A_+, A_-)$ as will be shown below.

To introduce an appropriately (resolvent regularized) Witten index of $\bsD_\bsA^{}$, we consider a densely defined, closed operator $T$ in the complex, separable Hilbert space $\cK$ and assume that for some 
$k \in \bbN$, and all $\lambda < 0$
\begin{equation}
\big[(T^* T - \lambda I_{\cK})^{-k} - (T T^* - \lambda I_{\cK})^{-k}\big] \in \cB_1(\cK).   
\end{equation}
Then the $k$th resolvent regularized Witten index of $T$ is defined by
\begin{equation}
W_{k,r}(T) = \lim_{\lambda \uparrow 0} (- \lambda)^k 
\tr_{\cK} \big((T^* T - \lambda I_{\cK})^{-k} - (T T^* - \lambda I_{\cK})^{-k}\big),   \lb{1.84}
\end{equation}
whenever the limit exists. The analogous semigroup regularized definition reads, 
\begin{equation}
W_s(T) = \lim_{t \uparrow \infty} \tr_{\cK}
\big(e^{- t T^*T} - e^{-t TT^*}\big),
\end{equation}
but in this manuscript it suffices to employ \eqref{1.84}.

The second main result of Section \ref{s11}, and at the same time the main result of this manuscript, the characterization of the Witten index of $\bsD_\bsA^{}$ in terms of spectral shift functions, can thus be summarized as follows: 

\begin{theorem} \lb{t1.1} 
Assume hypotheses \eqref{1.11}, \eqref{1.11a}, \eqref{1.42}, and \eqref{1.66}. 
Then $0$ is a right Lebesgue point of $\xi(\dott; \bsH_2, \bsH_1)$, denoted by 
$\xi_L(0_+; \bsH_2, \bsH_1)$, and 
\begin{equation}
\xi_L(0_+; \bsH_2, \bsH_1) = [\xi(0_+;A_+,A_-) + \xi(0_-;A_+,A_-)]/2.
\end{equation}
In addition, the resolvent regularized Witten index $W_{k,r}(\bsD_\bsA^{})$ of $\bsD_\bsA^{}$ 
exists for all $k \in \bbN$, $k \geq \lceil n/2\rceil$, and 
equals
\begin{align}
\begin{split} 
W_{k,r}(\bsD_\bsA^{}) = \xi_L(0_+; \bsH_2, \bsH_1) &= [\xi(0_+;A_+,A_-) + \xi(0_-;A_+,A_-)]/2    \\
&= [\xi(0_+;H,H_0) + \xi(0_-;H,H_0)]/2.    \lb{1.86}
\end{split} 
\end{align}
\end{theorem}

This is the first result of this kind applicable to non-Fredholm operators in a partial differential operator setting involving multi-dimensional massless Dirac operators.
In a sense, a project that started with Pushnitski in 2008, was considerably extended in scope in 
\cite{GLMST11}, and further developed with the help of \cite{CGGLPSZ16b}--\cite{CLPS20}, \cite{GN20}, \cite{GN20a}, finally comes full circle.

Appendix \ref{sA} collects some useful results on block matrix operators, Appendix \ref{sB} is devoted to asymptotic results for Hankel functions, Appendic \ref{sC} presents low-energy expansions and estimates for the free Dirac Green's function in the massless case, Appendix \ref{sD} recalls a product formula for modified (regularized) Fredholm determinants, and finally, Appendix \ref{sE} collects some of the notational conventions used throughout this manuscript.

\section{Some Background on (Locally) Smooth Operators} 
\lb{s2}

In this section we first recall a few basic facts on the notion of (local) Kato-smoothness (see, e.g., 
\cite[Sect.~XIII.7]{RS78}, \cite[Ch.~4]{Ya92}, and \cite[Chs.~0--2]{Ya10}) and then recall a variant, strong (local) Kato-smoothness (see, e.g., \cite[Ch.~4]{Ya92}, \cite[Chs.~0--2]{Ya10}), as these 
concepts will be useful in subsequent sections.

\begin{definition} \lb{d2.1} 
Let $S$ be self-adjoint in $\cH$ and $T \in \cC(\cH,\cK)$ with $\dom(S) \subseteq dom(T)$ 
and fix $\varepsilon_0 > 0$. Then $T$ is called $S$-Kato-smooth $($in short, $S$-smooth in the following\,$)$ if for each $f \in \cH$, 
\begin{align}
\begin{split}
& \|T\|_{S}^2 = \sup_{\varepsilon \in(0, \varepsilon_ 0), \; \|f\|_{\cH}=1} \f{1}{4 \pi^2} \int_{\bbR} d \lambda \, 
\Big[\big\|T(S - (\lambda + i \varepsilon) I_{\cH})^{-1}f\big\|^2_{\cH}     \lb{3.1} \\
& \hspace*{5.2cm} + \big\|T(S - (\lambda - i \varepsilon) I_{\cH})^{-1}f\big\|^2_{\cH}\Big] < \infty.
\end{split} 
\end{align}
\end{definition}

It suffices to require \eqref{3.1} for a dense set of $f \in \cH$ as $T$ is closed.  If $T$ is $S$-smooth, then $T$ is infinitesimally bounded with respect to $S$. 

In terms of unitary groups, $T$ 
is $S$-smooth if and only if for all $f \in \cH$, $e^{- i t S}f \in \dom(T)$ for a.e.~$t \in \bbR$ and 
\begin{equation}
\f{1}{2 \pi} \int_{\bbR} dt \, \big\|T e^{- i t S} f\big\|^2_{\cH} \leq C_0 \|f\|^2_{\cH} 
\end{equation}
for some constant $C_0 \in (0,\infty)$ ($C_0$ can be chosen to be $\|T\|^2_{S}$, but not  
smaller).

An immediate consequence regarding the absence of singular spectrum derives from the fact 
that if $T$ is $S$-smooth then
\begin{equation}
\ol{\ran(T^*)} \subseteq \cH_{ac}(S). 
\end{equation}  
In particular, 
\begin{equation}
\text{if, in addition, $\ker (T) = \{0\}$, then $\cH_{ac}(S) = \cH$,}
\end{equation}
and hence the spectrum of $S$ is purely absolutely continuous, 
\begin{equation}
\sigma(S) = \sigma_{ac}(S), \quad \sigma_{p}(S) = \sigma_{sc}(S) = \emptyset. 
\end{equation}
Here $\cH_{ac}(S)$ denotes the absolutely continuous subspace associated with $S$. 

Moreover, as long as $B \in \cB(\cK, \cL)$ (with $\cL$ another complex, separable Hilbert space), $BT$ is $S$-smooth whenever $T$ is $S$-smooth. 

Finally, if $T \in \cC(\cH,\cK)$ and for all $z \in \bbC \backslash \bbR$, 
$T (S -z I_{\cH})^{-1} T^*$ has a bounded closure in $\cH$ satisfying for some fixed 
$\varepsilon_0 > 0$, 
\begin{equation}
C_1 = \sup_{\lambda \in \bbR, \; \varepsilon \in(0, \varepsilon_0)} 
\big\|\ol{T (S -(\lambda + i \varepsilon ) I_{\cH})^{-1} T^*}\big\|_{\cB(\cH)} < \infty, 
\lb{3.4}
\end{equation}
then $T$ is $S$-smooth with $\|T\|_{S} \leq C_1/\pi$. 

While Definition \ref{d2.1} describes a global condition, a local version can be introduced as follows:

\begin{definition} \lb{d2.2} 
Let $S$ be self-adjoint in $\cH$ and $T \in \cC(\cH,\cK)$ with $\dom(S) \subseteq dom(T)$. 
$T$ is called $S$-Kato-smooth on a Borel set $\Lambda \subseteq \bbR$ 
$($in short, $S$-smooth on $\Lambda$ in the following\,$)$ if $T E_S(\Lambda)$ is $S$-smooth.
\end{definition}

Again, if $B \in \cB(\cK, \cL)$, then $BT$ is $S$-smooth on $\Lambda_0$ whenever $T$ is. 

For $T E_S(\Lambda)$ to be well-defined it suffices that 
$E_S(\Lambda) \cH \cap \dom(S) \subseteq \dom(T)$. 

If $T$ is $S$-smooth on $\Lambda$ then
\begin{equation}
\ol{\ran\big((T E_S(\Lambda))^*\big)} \subseteq \cH_{ac}(S), 
\end{equation}  
in particular, 
\begin{align}
\begin{split} 
& \text{ if, in addition, $\ker (T) = \{0\}$, then} \\
& \quad \sigma(S) \cap \Lambda = \sigma_{ac}(S) \cap \Lambda, \quad 
\sigma_{p}(S) \cap \Lambda = \sigma_{sc}(S) \cap \Lambda = \emptyset. 
\end{split} 
\end{align}

If $T \in \cC(\cH,\cK)$ and for all $z \in \bbC \backslash \bbR$, 
$T (S -z I_{\cH})^{-1} T^*$ has a bounded closure in $\cH$ satisfying for some fixed 
$\varepsilon_0 > 0$, 
\begin{equation}
\sup_{\lambda \in \Lambda, \; \varepsilon \in (0,\varepsilon_0)} 
\big\|\ol{T (S -(\lambda + i \varepsilon ) I_{\cH})^{-1} T^*}\big\|_{\cB(\cH)} < \infty, 
\lb{3.9}
\end{equation}
or
\begin{equation}
\sup_{\lambda \in \Lambda, \; \varepsilon \in (0,\varepsilon_0)} \varepsilon 
\big\|\ol{T (S -(\lambda \pm i \varepsilon ) I_{\cH})^{-1}}\big\|_{\cB(\cH)} < \infty, 
\lb{3.10}
\end{equation}
then $T$ is $S$-smooth on $\ol \Lambda$. 

Next, following \cite[Sect.~4.4]{Ya10}, we turn to the concept of strongly smooth operators on 
a compact interval $\Lambda_0 = [\lambda_1, \lambda_2]$, $\lambda_j \in \bbR$, $j=1,2$, 
$\lambda_1 < \lambda_2$ (tailored toward certain applications to differential operators). 
This requires some preparations: Given a separable complex 
Hilbert space $\cH_0$, one considers the (nonseparable) Banach space of $\cH_0$-valued 
H\"older continuous functions of order $\tau \in (0,1]$, denoted by 
$C^{\tau}(\Lambda_0; \cH_0)$, with norm 
\begin{equation}   
\|f \|_{C^{\tau}(\Lambda_0; \cH_0)} = \sup_{\lambda, \lambda' \in \Lambda_0} \bigg(
\|f(\lambda)\|_{\cH_0} + \f{\|f(\lambda) - f(\lambda')\|_{\cH_0}}{|\lambda - \lambda'|^{\tau}}\bigg), 
\quad f \in C^{\tau}(\Lambda_0; \cH_0).  
\end{equation}
Suppose the self-adjoint operator $S$ in $\cH$ has purely absolutely continuous spectrum 
on $\Lambda_0$, that is, 
\begin{equation}
\sigma(S) \cap \Lambda_0 = \sigma_{ac}(S) \cap \Lambda_0, \quad 
\sigma_{p}(S) \cap \Lambda_0 = \sigma_{sc}(S) \cap \Lambda_0 = \emptyset,
\end{equation}
of constant multiplicity $m_0 \in \bbN \cup \{\infty\}$ on $\Lambda_0$, with $\dim(\cH_0) = m_0$. 
In addition, let 
\begin{equation}
\cF_0: \begin{cases} E_S(\Lambda_0) \cH \to L^2(\Lambda_0; d \lambda; \cH_0), \\
f \mapsto \cF_0 f := \wti f,  \end{cases} \text{be unitary,}
\end{equation}
and ``diagonalizing'' $S$, that is, turning $S E_S(\Lambda_0)$ into a multiplication operator. 
More precisely, $\cF_0$ generates a spectral representation of $S$ via, 
\begin{equation}
(\cF_0 E_S(\Omega) f)(\lambda) = \chi_{\Omega \cap \Lambda_0}(\lambda) \wti f(\lambda), 
\quad f \in E_S(\Lambda_0) \cH. 
\end{equation}

With these preparations in place, we are now in position to define the notion of strongly smooth operators (cf.\ \cite[Sect.~4.4]{Ya92}, where a more general concept is introduced):

\begin{definition} \lb{d2.3} 
Let $S$ be self-adjoint in $\cH$ with purely absolutely continuous spectrum 
of constant $($possibly, infinite\,$)$ multiplicity on $\Lambda_0$ and suppose that $T \in \cC(\cH,\cK)$ with $\dom(S) \subseteq dom(T)$.~Then $T$ is called strongly $S$-Kato-smooth on $\Lambda_0$ $($in short, strongly $S$-smooth 
on $\Lambda_0$ in the following\,$)$, with exponent $\tau \in (0,1]$, 
if $\cF_0 (T E_S(\Lambda))^*: \cK \to C^{\tau}(\Lambda_0; \cH_0)$ is continuous, that is, 
for $f = (T E_S(\Lambda_0))^* \xi$, $\xi \in \cK$, 
\begin{align}
\begin{split}
& \big\|\wti f(\lambda)\big\|_{\cH_0} = \|(\cF_0(T E_S(\Lambda_0))^* \xi)(\lambda)\|_{\cH_0}
\leq C \|\xi\|_{\cK},    \\ 
& \big\|\wti f(\lambda) - \wti f(\lambda')\big\|_{\cH_0} \leq C |\lambda - \lambda'|^{\tau} \|\xi\|_{\cK},
\end{split} 
\end{align}
with $C \in (0,\infty)$ independent of $\lambda, \lambda' \in \Lambda_0$ and $\xi \in \cK$. 
\end{definition}

Not surprisingly, the terminology chosen is consistent with the fact that 
\begin{equation}
\text{if $T$ is strongly $S$-smooth on $\Lambda_0$, then it is $S$-smooth on $\Lambda_0$.}  
\end{equation}
Moreover, as long as $B \in \cB(\cK, \cL)$ (with $\cL$ another complex, separable Hilbert space) 
and $T$ is strongly $S$-smooth with exponent $\tau \in (0,1]$ on $\Lambda_0$, then $BT$ is 
strongly $S$-smooth on $\Lambda_0$ with the same exponent $\tau \in (0,1]$. 

Next, we recall a perturbation approach in which $S$ corresponds to the ``sum'' of an unperturbed 
self-adjoint operator $S_0$ in $\cH$ and a perturbation $V$ in $\cH$ that can be factorized into a product $V_1^* V_2$ as follows: 
Suppose $V_j \in \cC(\cH,\cK)$, $j=1,2$, with 
\begin{equation}
V_j(|S_0| + I_{\cH})^{-1/2} \in \cB(\cH,\cK), \quad j =1,2,     \lb{3.17}
\end{equation} 
and the symmetry condition, 
\begin{equation}
(V_1 f, V_2 g)_{\cK} = (V_2 f, V_1 g)_{\cK}, \quad f, g \in \dom\big(|S_0|^{1/2}\big). 
\end{equation} 
In addition, suppose that for some (and hence for all) $z \in \rho(S_0)$, 
$V_2 (S_0 - zI_{\cH})^{-1} V_1^*$ has a bounded extension in $\cK$, which is then given by 
its closure
\begin{equation}
\ol{V_2 (S_0 - zI_{\cH})^{-1} V_1^*} 
= V_2 (S_0 - z I_{\cH})^{-1/2} \big[V_1 (S_0 - \ol{z} I_{\cH})^{-1/2}\big]^*.      \lb{3.19} 
\end{equation}
Here the operator $\ol{V_2 (S_0 - zI_{\cH})^{-1} V_1^*} $ represents an abstract 
Birman--Schwinger-type operator. 

Finally, we assume that 
\begin{equation}
\big[I_{\cK} + \ol{V_2 (S_0 - z_0 I_{\cH})^{-1} V_1^*}\big]^{-1} \in \cB(\cK) \, 
\text{ for some $z_0 \in \rho(S_0)$.} 
\end{equation}
Then the equation 
\begin{align} 
R(z) &= (S_0 -z I_{\cH})^{-1}   \no \\
& \quad - \big[V_1(S_0 - \ol{z} I_{\cH})^{-1}\big]^* \big[I_{\cH} + \ol{V_2 (S_0 - z I_{\cH})^{-1} V_1^*}\big]^{-1} V_2(S_0 - z I_{\cH})^{-1},    \no \\
& \hspace*{8.5cm} z \in \bbC \backslash \bbR,    \lb{3.21} 
\end{align}
defines the resolvent of a self-adjoint operator $S$ in $\cH$, that is, 
\begin{equation} 
R(z) = (S - z I_{\cH})^{-1}, \quad z \in \bbC \backslash \bbR,    \lb{3.22}
\end{equation}  
with $S \supseteq S_0 + V_1^* V_2$ (the latter defined on $\dom(S_0) \cap \dom (V_1^* V_2)$, which may consist of $\{0\}$ only); for details we refer to \cite{Ka66} (see also \cite{GLMZ05}, 
\cite[Sect.~1.9]{Ya92}).

One also has 
\begin{align}
(S - z I_{\cH})^{-1} - (S_0 - z I_{\cH})^{-1} 
& = - \big[V_1 (S - \ol{z} I_{\cH})^{-1}\big]^* V_2 (S_0 - z I_{\cH})^{-1}    \no \\
&= - \big[V_1 (S_0 - \ol{z} I_{\cH})^{-1}\big]^* V_2 (S - z I_{\cH})^{-1},    \no \\
& \hspace*{4.35cm} z \in \bbC \backslash \bbR,  
\end{align}
and 
\begin{align} 
(S_0 - z I_{\cH})^{-1} &= (S -z I_{\cH})^{-1}   \no \\
& \quad - \big[V_1(S - \ol{z} I_{\cH})^{-1}\big]^* \big[I_{\cH} - \ol{V_2 (S - z I_{\cH})^{-1} V_1^*}\big]^{-1} V_2(S - z I_{\cH})^{-1},    \no \\
& \hspace*{7.5cm} z \in \bbC \backslash \bbR,   
\end{align}
as well as, 
\begin{align} 
\ol{V_1 (S - z I_{\cH})^{-1} V_1^*} &= \ol{V_1 (S_0 - z I_{\cH})^{-1} V_1^*}      \\ 
& \quad - \big[\ol{V_1 (S - z I_{\cH})^{-1} V_1^*}\big] 
\big[\ol{V_2 (S_0 - z I_{\cH})^{-1} V_1^*}\big], 
\quad z \in \bbC \backslash \bbR,   \no 
\end{align} 
implying
\begin{align}
\begin{split} 
\ol{V_1 (S - z I_{\cH})^{-1} V_1^*} = \ol{V_1 (S_0 - z I_{\cH})^{-1} V_1^*} 
\big[I_{\cK} + \ol{V_2 (S_0 - z I_{\cH})^{-1} V_1^*}\big]^{-1},      \lb{3.26} \\ 
z \in \bbC \backslash \bbR.
\end{split} 
\end{align}
Similarly,
\begin{align}
\begin{split} 
\ol{V_2 (S - z I_{\cH})^{-1} V_1^*} =  I_{\cH} - 
\big[I_{\cK} + \ol{V_2 (S_0 - z I_{\cH})^{-1} V_1^*}\big]^{-1}, \quad z \in \bbC \backslash \bbR.    \lb{3.26a} 
\end{split} 
\end{align}

The remaining results in this section are all taken from \cite[Sects.~4.4--4.7]{Ya10}.

\begin{theorem} \lb{t2.4}
Assuming the hypotheses on $S_0$, $V_j$, $j=1,2$, employed in \eqref{3.17}--\eqref{3.26}, 
suppose the following additional conditions hold: \\[1mm]
$(i)$ $V_2$ is $S_0$-smooth on $\Lambda_0 = [\lambda_1, \lambda_2]$, $\lambda_j \in \bbR$, 
$j=1,2$, $\lambda_1 < \lambda_2$. \\[1mm] 
$(ii)$ The analytic, operator-valued functions $\ol{V_2(S_0 - z I_{\cH})^{-1} V_1^*}$, 
$\ol{V_1 (S_0 - z I_{\cH})^{-1} V_1^*}$ on $\Re(z) \in (\lambda_1, \lambda_2)$, $\Im(z) \neq 0$, are continuous in $\cB(\cK)$-norm up to and including the two rims of the ``cut'' along $(\lambda_1, \lambda_2)$.  
\\[1mm] 
$(iii)$ For some $k \in \bbN$, $\big[\ol{V_2 (S_0 - z I_{\cH})^{-1} V_1^*}\big]^k \in \cB_{\infty}(\cK)$, 
$\Im(z) \neq 0$. \\[1mm] 
Define 
\begin{align}
& \cN_{\pm} = \big\{\lambda \in \Lambda_0 \, \big| \, \text{there exists $0 \neq f \in \cK$ s.t.} \, 
- f = \ol{V_2 (S_0 - (\lambda \pm i 0) I_{\cH})^{-1} V_1^*} f\big\},   \no \\ 
& \cN = \cN_- \cup \cN_+.    \lb{3.27} 
\end{align} 
Then $\cN_{\pm}, \cN$ are closed and of Lebesgue measure zero. Moreover, the analytic, operator-valued function $\ol{V_1 (S - z I_{\cH})^{-1} V_1^*}$ on $\Re(z) \in (\lambda_1, \lambda_2)$, 
$\Im(z) \neq 0$, is continuous in $\cB(\cK)$-norm up to and including the two rims of the ``cut'' along 
$(\lambda_1, \lambda_2) \backslash \cN$. If, in addition, $\ker(V_1) = \{0\}$, then 
$S$ has purely absolutely continuous spectrum on $\Lambda_0 \backslash \cN$, that is, 
\begin{equation}
\sigma(S) \cap (\Lambda_0 \backslash \cN) = \sigma_{ac}(S) \cap (\Lambda_0 \backslash \cN), 
\quad \sigma_{p}(S) \cap (\Lambda_0 \backslash \cN) 
= \sigma_{sc}(S) \cap (\Lambda_0 \backslash \cN) = \emptyset. 
\end{equation} 
\end{theorem}

As detailed in \cite[Remark~4.6.3]{Ya92}, condition $(iii)$ in Theorem \ref{t2.4} can be 
replaced by the following one: \\[1mm]
$(iii')$ Suppose there exist $z_{\pm} \in \rho(S_0)$, $\pm \Im(z_{\pm}) > 0$, such that 
\begin{equation}
\big[I_{\cK} + \ol{V_2 (S_0 - z_{\pm} I_{\cH})^{-1} V_1^*}\big]^{-1} \in \cB(\cK), 
\end{equation}
and 
\begin{align}
\begin{split} 
& \ol{V_2 (S_0 - z I_{\cH})^{-1} V_1^*} - \ol{V_2 (S_0 - z' I_{\cH})^{-1} V_1^*}   \\ 
& \quad = (z - z') \ol{V_2 (S_0 - z I_{\cH})^{-1} (S_0 - z' I_{\cH})^{-1} V_1^*} \in \cB_{\infty}(\cK), 
\quad z, z' \in \rho(S_0). 
\end{split} 
\end{align}

Next we strengthen the hypotheses in Theorem \ref{t2.4} by invoking the notion of strong 
$S_0$-smoothness:

\begin{theorem} \lb{t2.5}
Assuming the hypotheses on $S_0$, $V_j$, $j=1,2$, employed in \eqref{3.17}--\eqref{3.26}, 
suppose in addition the following conditions hold: \\[1mm]
$(i)$ $S_0$ has purely absolutely continuous spectrum of constant multiplicity 
$m_0 \in \bbN \cup \{\infty\}$ on $\Lambda_0 = [\lambda_1, \lambda_2]$, $\lambda_j \in \bbR$, 
$j=1,2$, $\lambda_1 < \lambda_2$. \\[1mm] 
$(ii)$ $V_j$ are strongly $S_0$-smooth on any compact subinterval of $\Lambda_0$ with exponents $\tau_j > 0$, $j=1,2$.  \\[1mm] 
$(iii)$ For some $k \in \bbN$, $\big[\ol{V_2 (S_0 - z I_{\cH})^{-1} V_1^*}\big]^k 
\in \cB_{\infty}(\cK)$, $\Im(z) \neq 0$. \\[1mm] 
Then the analytic, operator-valued functions 
\begin{equation} 
\ol{V_1 (S_0 - z I_{\cH})^{-1} V_1^*}, \,
\ol{V_2 (S - z I_{\cH})^{-1} V_1^*}, \; \text{ $\big($resp., $\ol{V_1 (S - z I_{\cH})^{-1} V_1^*}$$\big)$} 
\end{equation} 
on $\Re(z) \in (\lambda_1, \lambda_2)$, $\Im(z) \neq 0$, are H\"older continuous in $\cB(\cK)$-norm with 
exponent $\min\{\tau_1, \tau_2\}$ up to and including the two rims of the ``cut'' along 
$(\lambda_1, \lambda_2)$ $($resp., $(\lambda_1, \lambda_2) \backslash \cN$$)$. \\[1mm] 
Moreover, the local wave operators
\begin{equation}
W_{\pm} (S,S_0; \Lambda_0) = \slim_{t \to \pm \infty} e^{i t S} e^{- i t S_0} P_{S_0, ac}(\Lambda_0), 
\end{equation}
with $P_{S_0, ac}(\Lambda_0) = E_{S_0} (\Lambda_0) E_{S_0, ac}$, and $ E_{S_0, ac}$ the projection onto the absolutely continuous subspace of $S_0$, exist and are complete, that is, 
\begin{equation}
\ker(W_{\pm} (S,S_0; \Lambda_0)) = \cH \ominus E_{S_0, ac} (\Lambda_0) \cH, \quad 
\ran(W_{\pm} (S,S_0; \Lambda_0)) = P_{S, ac} (\Lambda_0) \cH, 
\end{equation}
with $P_{S, ac}(\Lambda_0) = E_{S} (\Lambda_0) E_{S, ac}$, and $ E_{S, ac}$ the projection onto the absolutely continuous subspace of $S$. \\[1mm] 
For the remainder of this theorem suppose in addition that $\tau_1 > 1/2$. Then 
\begin{equation}
\cN_{\pm} = \cN = \sigma_p(S) \cap \Lambda_0 
\end{equation}
and the $($geometric\,$)$ multiplicities of the eigenvalue $\lambda_0 \in \Lambda_0$ of $S$ and the 
eigenvalue $- 1$ of $\ol{V_2 (S_0 - (\lambda_0 \pm i 0) I_{\cH})^{-1} V_1^*}$ coincide.
If in addition, $\ker(V_1) = \{0\}$, then 
$S$ has no singularly continuous spectrum on $\Lambda_0$, that is, 
\begin{equation}
\sigma(S) \cap \Lambda_0 = \sigma_{ac}(S) \cap \Lambda_0, \quad 
\sigma_{sc}(S) \cap \Lambda_0 = \emptyset, 
\end{equation} 
and the singular spectrum of $S$ on the interior, $(\lambda_1, \lambda_2)$, of $\Lambda_0$ 
consists only of eigenvalues of finite multiplicity with no accumulation point in 
$(\lambda_1, \lambda_2)$, in particular, 
\begin{equation}
\sigma_s(S) \cap (\lambda_1, \lambda_2) = \sigma_p(S) \cap (\lambda_1, \lambda_2).  
\end{equation} 
\end{theorem}

Again, condition $(iii)$ in Theorem \ref{t2.5} can be replaced by condition $(iii')$ above. 

To make the transition from local to global wave operators we also recall the following result.

\begin{theorem} \lb{t2.6}
Assuming the hypotheses on $S_0$, $V_j$, $j=1,2$, employed in \eqref{3.17}--\eqref{3.26}, 
suppose in addition the following conditions hold: \\[1mm]
$(i)$ $S_0$ has purely absolutely continuous spectrum of constant multiplicity 
$m_{0,k} \in \bbN \cup \{\infty\}$ on a system of intervals 
$\Lambda_{0,\ell} = [\lambda_{1,\ell}, \lambda_{2,\ell}]$, $\lambda_{j,\ell} \in \bbR$, 
$j=1,2$, $\lambda_{1,\ell} < \lambda_{2,\ell}$, $\ell \in \cI$, $\cI \subseteq \bbN$ an appropriate index 
set, such that 
\begin{equation}
\sigma (S_0) \bigg \backslash \bigcup_{\ell \in \cI} \Lambda_{0,\ell} \, \text{ has Lebesgue 
measure zero.}
\end{equation}
$(ii)$ $V_j$ are strongly $S_0$-smooth on any compact subinterval of $\Lambda_{0,\ell}$ with 
exponents $\tau_{j,\ell} > 0$, $j=1,2$, $\ell \in \cI$.  \\[1mm] 
$(iii)$ For some $k \in \bbN$, $\big[\ol{V_2 (S_0 - z I_{\cH})^{-1} V_1^*}\big]^k 
\in \cB_{\infty}(\cK)$, $\Im(z) \neq 0$. \\[1mm] 
Then the $($global\,$)$ wave operators
\begin{equation}
W_{\pm} (S,S_0) = \slim_{t \to \pm \infty} e^{i t S} e^{- i t S_0} , 
\end{equation}
exist and are complete, that is, 
\begin{equation}
\ker(W_{\pm} (S,S_0)) = \{0\}, \quad 
\ran(W_{\pm} (S,S_0)) = E_{S, ac} \cH, 
\end{equation}
with $ E_{S, ac}$ the projection onto the absolutely continuous subspace of $S$.   
\end{theorem}

\smallskip

For additional references in the context of smooth operator theory, limiting absorption principles,   
and completeness of wave operators, see, for instance, \cite{Ag75}, \cite{ABG96}, 
\cite{BM97}, \cite[Ch.~17]{BW83}, \cite{BD87}, \cite{Ge08}, 
\cite{GJ07}, \cite{Ka66}, \cite{Ku78}, \cite{MP96}, 
\cite[Sect.~XIII.7]{RS78}, \cite{Ri06}, \cite[Ch.~4]{Ya92}, \cite[Chs.~0--2]{Ya10}.

\section{A Limiting Absorption Principle for Interacting, Massless Dirac Operators} 
\lb{s3}

In this section, following \cite[Sects.~1.11, 2.1, 2.2]{Ya10}, we apply the abstract framework of strongly smooth operators of the preceding section to the concrete case of massless Dirac operators with electromagnetic potentials.  

To rigorously define the free massless $n$-dimensional Dirac operators to be studied in the sequel, we now introduce the following set of basic hypotheses assumed for the remainder of this manuscript (these hypotheses will have to be strengthened later on).

\begin{hypothesis} \lb{h3.1}  Let $ n \in \bbN$, $n\geq 2$. \\[1mm] 
$(i)$ Set $N=2^{\lfloor(n+1)/2\rfloor}$ and let $\alpha_j$, $1\leq j\leq n$, 
$\alpha_{n+1} := \beta$, denote $n+1$ anti-commuting Hermitian $N\times N$ matrices with squares equal to $I_N$, that is, 
\begin{equation} \lb{2.1}
\alpha_j^*=\alpha_j,\quad \alpha_j\alpha_k + \alpha_k\alpha_j = 2\delta_{j,k}I_N, 
\quad 1\leq j,k\leq n+1.
\end{equation}
$(ii)$  Introduce 
in $\LN$ the free massless Dirac operator
\begin{equation}
H_0 = \alpha  \cdot (-i \nabla) = \sum_{j=1}^n \alpha_j (-i \partial_j),\quad \dom(H_0) = \WoneN,  \lb{2.2}
\end{equation}
where $\partial_j = \partial / \partial x_j$, $1 \leq j \leq n$. \\[1mm] 
$(iii)$ Next, consider the self-adjoint matrix-valued potential 
$V = \{V_{\ell,m}\}_{1 \leq \ell,m \leq N}$ satisfying for some fixed $\rho \in (1,\infty)$, $C \in (0,\infty)$, 
\begin{equation}
V \in [L^{\infty} (\bbR^n)]^{N \times N}, \quad 
|V_{\ell,m}(x)| \leq C \langle x \rangle^{- \rho} \, \text{ for a.e.~$x \in \bbR^n$, $1 \leq \ell,m \leq N$.}    \lb{4.1}
\end{equation}
Under these assumptions on $V$, the massless Dirac operator $H$ in $\LN$ is defined via
\begin{equation}
H = H_0 + V, \quad \dom(H) = \dom(H_0) = \WoneN.  \lb{4.2}
\end{equation}
\end{hypothesis}

Then $H_0$ and $H$ are self-adjoint in $\LN$, 
with essential spectrum covering the 
entire real line,
\begin{equation}
\sigma_{ess} (H) = \sigma_{ess} (H_0) = \sigma (H_0) = \bbR,
\end{equation}
a consequence of relative compactness of $V$ with respect to $H_0$. In addition,
\begin{equation}
\sigma_{ac}(H_0) = \bbR, \quad \sigma_p(H_0) = \sigma_{sc}(H_0) = \emptyset.  
\end{equation}

On occasion (cf.\ Section \ref{s7}) we will drop the self-adjointness hypothesis on the $N \times N$ matrix $V$ and 
still define a closed operator $H$ in $\LN$ as in \eqref{4.2}.  

For completeness we also recall that the massive free Dirac operator in $\LN$ associated 
with the mass parameter $m > 0$ then would be of the form
\begin{equation}
H_0(m) = H_0 + m \beta, \quad \dom(H_0(m)) = \WoneN, \; m > 0, \; \beta = \alpha_{n+1}, 
\end{equation}
but we will primarily study the massless case $m=0$ in this manuscript.

In the special one-dimensional case $n=1$, one can choose for $\alpha_1$ either a real constant or one of the three Pauli matrices. Similarly, in the massive case, $\beta$ would typically be a second Pauli matrix (different from $\alpha_1$). For simplicity we confine ourselves to $n \in \bbN$, 
$n \geq 2$, in the following. 

Let $S(\bbR^n)$ denote the Schwartz space of rapidly decreasing functions on $\bbR^n$ and 
$S'(\bbR^n)$ the space of tempered distributions. In addition, for any $n \in \bbN$, we also introduce the scale of weighted $L^2$-spaces, 
\begin{equation}
L^2_s(\bbR^n) = \{f \in S'(\bbR^n) \, | \, \|f\|_{L^2_s(\bbR^n)} 
:= \|\langle x \rangle^s  f\|_{L^2(\bbR^n)} < \infty\}, \quad s \in \bbR.
\end{equation}
Defining $Q_j$ as the operator of multiplication by $x_j$, $1 \leq j \leq n$, in $L^2(\bbR^n)$, and 
introducing $Q = (Q_1,\dots,Q_n)$, one notes that 
\begin{equation} 
\dom(\langle Q \rangle^s) = L^2_s(\bbR^n), \quad s \in \bbR.
\end{equation} 

Employing the relations \eqref{2.1}, one observes that 
\begin{equation} 
H_0(m)^2 = I_N \big[-\Delta + m^2 I_{[L^2(\bbR^n)]^N}\big], \quad \dom\big( H_0(m)^2\big) = \WtwoN, \; m \geq 0.   \lb{2.7} 
\end{equation}

\begin{remark} \lb{r3.2}
Since we permit a (sufficiently decaying) matrix-valued potential $V$ in $H$, this includes, in particular, the case of electromagnetic interactions introduced via minimal coupling, that is, $V$ 
describes also special cases of the form,
\begin{equation}
H(q,A) := \alpha \cdot (-i \nabla - A) + q I_N = H_0 + [q I_N - \alpha \cdot A], \quad 
\dom(H(q,A)) = \WoneN,
\end{equation}
where $(q,A)$ represent the electromagnetic potentials on $\bbR^n$, with 
$q: \bbR^n \to \bbR$, $q \in L^{\infty}(\bbR^n)$, $A = (A_1,\dots,A_n)$, $A_j: \bbR^n \to \bbR$, 
$A_j \in L^{\infty}(\bbR^n)$, $ 1 \leq j \leq n$, and for some fixed $\rho > 1$, $C \in (0,\infty)$, 
\begin{equation}
|q(x)| + |A_j(x)| \leq C \langle x \rangle^{- \rho}, \quad x \in \bbR^n, \; 1 \leq j \leq n.  
 \lb{4.14}
\end{equation}
${}$ \hfill $\diamond$
\end{remark}

To analyze the spectral properties of $H$ we first turn to the spectral representation of 
$H_0 = \alpha \cdot (-i \nabla)$ (see also Thaller \cite[Sect.~5.6]{Th92} and 
Yafaev \cite[Sect.~2.4]{Ya10}). 
Introducing the unitary Fourier transform in $L^2(\bbR^n)$ via 
\begin{equation}
\cF : \begin{cases} 
L^2(\bbR^n; d^nx) \to L^2(\bbR^n; d^np) \\ 
f \mapsto (\cF f)(p) := f^{\wedge}(p) = \slim_{R \to \infty} (2 \pi)^{- n/2}
\int_{|x| \leq R} d^n x \, e^{- i p \cdot x} f(x),      \lb{3.14} 
\end{cases} 
\end{equation}
with $\slim$ abbreviating the limit in the topology of $L^2(\bbR^n)$, one obtains
\begin{equation} \lb{3.16y}
H_0 = \cF^{-1} |p| (\alpha \cdot \omega) \cF, 
\end{equation}
employing polar coordinates in Fourier space, $p = |p| \, \omega$, $\omega \in S^{n-1}$. Since 
by \eqref{2.1} (see also \eqref{2.7}) 
\begin{equation}
(\alpha \cdot \omega)^2 = I_N, \quad \omega \in S^{n-1}, 
\end{equation}
the self-adjoint matrix $\alpha \cdot \omega$ has eigenvalues 
$\pm 1$ of multiplicity $N/2$ with associated spectral projection matrices of rank $N/2$  
denoted by $\Pi_{\pm}(\omega)$, 
\begin{equation}
\alpha \cdot \omega = \Pi_+(\omega) - \Pi_-(\omega), \quad \omega \in S^{n-1}.  
\end{equation}

Introducing
\begin{equation} \lb{3.19y}
T(\omega) = 2^{-1/2}(\alpha_{n+1} + \alpha \cdot \omega),\quad \omega\in S^{n-1}, 
\end{equation}
one infers that $T(\omega)\in \bbC^{N\times N}$ is Hermitian symmetric for each $\omega\in S^{n-1}$.  In addition, the anti-commutation property in \eqref{2.1} implies
\begin{align}
T(\omega)T(\omega)^* &= 2^{-1}\big[\alpha_{n+1}^2 + (\alpha\cdot \omega)\alpha_{n+1} + \alpha_{n+1}(\alpha\cdot\omega) + (\alpha\cdot\omega)^2\big]\\
&= I_N,\quad \omega\in S^{n-1},\no
\end{align}
so that $T(\omega)$ is actually unitary for each $\omega\in S^{n-1}$.  The reason for introducing the unitary matrix $T(\omega)$, $\omega\in S^{n-1}$, is that it can be used to diagonalize the matrix $\alpha\cdot p$.  Indeed, writing $p\in\bbR^n$ in polar coordinates as $p=|p|\omega$ with $\omega\in S^{n-1}$, one obtains
\begin{align}
T(\omega)|p|\alpha_{n+1}T(\omega)^*&= 2^{-1}|p|\big[\alpha_{n+1} + 2\alpha\cdot\omega - (\alpha\cdot\omega)^2\alpha_{n+1} \big]\no\\
&=|p|(\alpha\cdot\omega)\no\\
&=\alpha\cdot p,\lb{3.21y}
\end{align}
so that $\alpha\cdot p$ is unitarily equivalent to $|p|\alpha_{n+1}$ in $\bbC^N$.  Of course, 
$\alpha_{n+1}$ is Hermitian symmetric, so it may be diagonalized by conjugating with a fixed 
(i.e., $p$-independent) unitary matrix $U\in \bbC^{N\times N}$.  We may assume without loss 
that the columns of $U$ are arranged so that
\begin{equation} \lb{3.22y}
\alpha_{n+1} =
U\begin{pmatrix}
-I_{N/2} & 0\\
0 & I_{N/2}
\end{pmatrix}U^*,
\end{equation}
where $0$ denotes the zero matrix in $\bbC^{(N/2)\times(N/2)}$. The facts \eqref{3.21y} and \eqref{3.22y} combine to yield
\begin{equation} \lb{3.23y}
\alpha\cdot p = \wti T(\omega) |p| \begin{pmatrix}
-I_{N/2} & 0\\
0 & I_{N/2}
\end{pmatrix} \wti T(\omega)^*,\quad p=|p|\omega\in \bbR^n,
\end{equation}
where
\begin{equation}
\wti T(\omega) := T(\omega)U,\quad \omega\in S^{n-1},
\end{equation}
and then \eqref{3.16y} implies
\begin{equation} \lb{3.25y}
H_0 = \cF^{-1}\wti T(\omega) |p| \begin{pmatrix}
-I_{N/2} & 0\\
0 & I_{N/2}
\end{pmatrix} \wti T(\omega)^*\cF.
\end{equation}
A simple manipulation in \eqref{3.25y} yields
\begin{equation} \lb{3.26y}
\wti T(\omega)^*\cF H_0 = |p| \begin{pmatrix}
-I_{N/2} & 0\\
0 & I_{N/2}
\end{pmatrix} \wti T(\omega)^*\cF.
\end{equation}
To ``diagonalize'' $H_0$, we introduce the notation
\begin{equation}
P_- :=\begin{pmatrix}
I_{N/2} & 0 \\
0 & 0
\end{pmatrix},
\quad
P_+:=\begin{pmatrix}
0 & 0 \\
0 & I_{N/2}
\end{pmatrix},
\end{equation}
and define the transformation
\begin{equation}
\cF_{H_0}:\LN\to L^2(\bbR;d\lambda;[L^2(S^{n-1})]^N)
\end{equation}
according to
\begin{align}
\begin{split} 
 (\cF_{H_0}f)(\lambda,\omega)=
\begin{cases}
|\lambda|^{(n-1)/2} P_- \wti T(\omega)^* f^{\wedge}(|\lambda|\omega), & \lambda<0,\\
|\lambda|^{(n-1)/2} P_+ \wti T(\omega)^* f^{\wedge}(|\lambda|\omega), & \lambda\geq 0,
\end{cases}&     \\
 \omega\in S^{n-1},\; f\in \LN.&
\end{split} 
\end{align}
The transformation $\cF_{H_0}$ is unitary.  In fact,
\begin{align}
\|\cF_{H_0}f\|_{L^2(\bbR;d\lambda;[L^2(S^{n-1})]^N)}^2&=\int_0^{\infty}d\lambda\, |\lambda|^{n-1}\int_{S^{n-1}}d^{n-1}\omega\,\Big\{ \big\|P_-\wti T(\omega)^* f^{\wedge}(|\lambda|\omega)\big\|_{\bbC^N}^2\no\\
&\qquad  + \big\|P_+\wti T(\omega)^* f^{\wedge}(|\lambda|\omega)\big\|_{\bbC^N}^2\Big\},\quad f\in \LN.\lb{3.30y}
\end{align}
Since $(P_-\xi, P_+\eta)_{\bbC^N}=0$ for all $\xi,\eta\in \bbC^N$, an application of the Pythagorean theorem in \eqref{3.30y} yields
\begin{align}
\|\cF_{H_0}f\|_{L^2(\bbR;d\lambda;[L^2(S^{n-1})]^N)}^2& = \int_0^{\infty}d\lambda\, |\lambda|^{n-1} \int_{S^{n-1}}d^{n-1}\omega\, \big\| \wti T(\omega)^* f^{\wedge}(|\lambda|\omega)\big\|_{\bbC^N}^2\no\\
& = \big\| f^{\wedge}\big\|_{[L^2(\bbR^n)]^N}^2\no\\
& = \big\|f\big\|_{[L^2(\bbR^n)]^N}^2,\quad f\in \LN.
\end{align}
To check that $\cF_{H_0}$ correctly diagonalizes $H_0$ in the sense that
\begin{equation} \lb{3.32y}
(\cF_{H_0}H_0f)(\lambda,\,\cdot\,) = \lambda (\cF_{H_0}f)(\lambda,\,\cdot\,)\, \text{ for a.e.~$\lambda\in \bbR$},\; f\in \WoneN,
\end{equation}
one considers separately the cases $\lambda<0$ and $\lambda\geq 0$.  Indeed, for a fixed $f\in \WoneN$, one applies \eqref{3.26y} to obtain
\begin{align}
(\cF_{H_0}H_0f)(\lambda,\omega)&= |\lambda|^{(n-1)/2}P_- \wti T(\omega)^* (H_0f)^{\wedge}(|\lambda|\omega)\no\\
&= -|\lambda|P_-\wti T(\omega)^* f^{\wedge}(|\lambda|\omega)\no\\
&=\lambda (\cF_{H_0}f)(\lambda,\omega),\quad \lambda<0,\; \omega\in S^{n-1},\lb{3.33y}
\end{align}
and, similarly,
\begin{align}
(\cF_{H_0}H_0f)(\lambda,\omega)&= |\lambda|^{(n-1)/2}P_+ \wti T(\omega)^* (H_0f)^{\wedge}(|\lambda|\omega)\no\\
&= |\lambda|P_+\wti T(\omega)^* f^{\wedge}(|\lambda|\omega)\no\\
&=\lambda (\cF_{H_0}f)(\lambda,\omega),\quad \lambda\geq 0,\; \omega\in S^{n-1}.\lb{3.34y}
\end{align}
Equations \eqref{3.33y} and \eqref{3.34y} combine to yield \eqref{3.32y}.  Of course, \eqref{3.32y} generalizes to
\begin{equation} \lb{335y}
(\cF_{H_0}\psi(H_0)f)(\lambda,\,\cdot\,) = \psi(\lambda)(\cF_{H_0}f)(\lambda,\,\cdot\,)\, 
\text{ for a.e.~$\lambda\in \bbR$},\; f\in \dom(\psi(H_0)),
\end{equation}
for any measurable function $\psi$ on $\bbR$.

Consequently, \cite[Proposition~2.4.1]{Ya10} applies to $H_0$, resulting in the following facts:

\begin{proposition} \lb{p3.3}  
Suppose Hypothesis \ref{h3.1}\,$(i)$,\,$(ii)$ and let $\gamma > 1/2$. Then $\langle Q \rangle^{-\gamma}$ is strongly $H_0$-smooth on compact subintervals of 
$\bbR \backslash \{0\}$ with exponent $\tau > 0$ given by
\begin{equation}
\tau = \begin{cases}
\gamma - (1/2), & \gamma \in ((1/2), (3/2)), \\
1 - \varepsilon, & \gamma = 3/2, \; \varepsilon \in (0,1), \\
1, & \gamma \geq 3/2. 
\end{cases}
\end{equation}
\end{proposition}

We note that for $\tau > 1/2$, $z \in \bbC \backslash \bbR$, 
\begin{equation}
\langle \, \cdot \, \rangle^{-\tau} (H_0 -zI_{[L^2(\bbR^n)]^N})^{-1} \langle \, \cdot \, \rangle^{-\tau} 
\in \cB_{\infty}\big([L^2(\bbR^n)]^N\big), 
\end{equation}
a special case of the well-known general fact (cf., e.g., \cite[p.~41]{Ya92}), 
\begin{align} 
\begin{split}
& f(Q) g(-i \nabla) \in \cB_{\infty}\big(L^2(\bbR^n)\big) \, \text{ for any 
$f, g \in L^{\infty}(\bbR^n)$} \\
& \quad \text{with $\lim_{|x|\to\infty} f(x) = 0 = \lim_{|p|\to\infty} g(p)$.}
\end{split}
\end{align}  
To make the connection with the results collected in Section \ref{s2}, we identify $S_0$ and $H_0$ and $S$ with 
$H = H_0 + V$, and we factorize $V$ according to 
\begin{equation} \lb{3.25aa}
V = V_1^* V_2, \quad V_1 = V_1^* = \langle \, \cdot \, \rangle^{-\tau} I_N, \quad 
V_2 = \langle \, \cdot \, \rangle^{\tau} V, 
\end{equation}
with $V$ satisfying the conditions in \eqref{4.1} for some fixed $\rho > 1$, and hence, with 
$\tau \in (1/2, \rho)$, 
\begin{equation} \lb{formula2}
\|V_2(\, \cdot \,)\|_{\cB(\bbC^N)} \leq C \langle \, \cdot \, \rangle^{-(\rho - \tau)}.
\end{equation} 
In addition, 
\begin{align} 
\begin{split} 
& \langle \, \cdot \, \rangle^{-\tau} \big(H -zI_{[L^2(\bbR^n)]^N} \big)^{-1} \langle \, \cdot \, \rangle^{-\tau}  
= \langle \, \cdot \, \rangle^{-\tau} \big(H_0 -zI_{[L^2(\bbR^n)]^N} \big)^{-1} \langle \, \cdot \, \rangle^{-\tau}   \\
& \quad \times \big[I_{[L^2(\bbR^n)]^N} + \langle \, \cdot \, \rangle^{\tau} V \, 
\big(H_0 -zI_{[L^2(\bbR^n)]^N} \big)^{-1} \langle \, \cdot \, \rangle^{-\tau} \big]^{-1},  
\quad z \in \bbC \backslash \bbR, 
\end{split} 
\end{align}
to mention just a few analogs of the abstract facts collected in \eqref{3.17}--\eqref{3.26}, which all 
apply in this concrete setting of massless Dirac operators. 

Thus, temporarily assuming $\rho > 3/2$ in \eqref{4.1}, Theorem \ref{t2.5} applies to 
$S = H_0 +V$ with 
\begin{equation}
\tau_1 = \tau - (1/2) > 1/2, \, \text{ necessitating } \, \tau > 1, 
\lb{4.25} 
\end{equation} 
and 
\begin{equation}
\tau_2 = (\rho - \tau) - (1/2) > 0, \, \text{ requiring } \, \rho > 3/2. 
\lb{4.26} 
\end{equation}

Actually, as shown in \cite[p.~98--99]{Ya10} in the context of the Laplacian $h_0$ in $L^2(\bbR^n)$, 
\begin{equation} 
h_0 = - \Delta, \quad \dom(h_0) = H^2(\bbR^n) 
\end{equation} 
it suffices to assume 
just $\rho > 1$ in \eqref{4.1} (even though this cannot be inferred directly from abstract results, 
the latter require $\rho > 3/2$ as outlined in \eqref{4.25}, \eqref{4.26}) and 
$\tau \in (1/2, \rho - 1/2)$. A closer examination of \cite[p.~98--99]{Ya10} (see also 
\cite[p.~118]{Ya10}) reveals that 
there is nothing special about $h_0$ and precisely the same results apply to 
$H_0 = \alpha \cdot (-i \nabla)$ as we discuss next. 

Applying Theorems \ref{t2.4}--\ref{t2.6}, to the pair $(H, H_0)$ and to a union of 
compact intervals exhausting $(-\infty, 0) \cup (0,\infty)$, combined with the approach in 
\cite[p.~98--99, 118]{Ya10},  thus yield the following result: 

\begin{theorem} \lb{t3.4}
Assume Hypothesis \ref{h3.1} and consider $H$ as defined in \eqref{4.2}. Then 
\begin{align}
& \sigma_{ess}(H) = \sigma_{ac}(H) = \bbR,    \lb{3.44} \\
& \sigma_{sc}(H) = \emptyset,      \lb{3.45} \\
& \sigma_s(H) \cap (\bbR \backslash \{0\}) = \sigma_p(H) \cap (\bbR \backslash \{0\}),    \lb{3.46}
\end{align}
with the only possible accumulation points of $\sigma_p(H)$ 
being $0$ and $\pm \infty$.  If
\begin{equation}
\cN_0 := \sigma_p(H) \cap (\bbR \backslash \{0\}) = \sigma_d(H) \cap (\bbR \backslash \{0\}),
\end{equation}
then the operators 
\begin{align} 
\begin{split} 
& \ol{V_2 (H_0 - (\lambda \pm i 0)I_{[L^2(\bbR^n)]^N})^{-1} V_1^*}, \; 
\ol{V_1 (H_0 - (\lambda \pm i 0)I_{[L^2(\bbR^n)]^N})^{-1} V_1^*} \lb{formula3}\\ 
& \quad \text{$\big($resp.,  
$\ol{V_1 (H - (\lambda \pm i 0)I_{[L^2(\bbR^n)]^N})^{-1} V_1^*}$$\big)$} 
\end{split} 
\end{align} 
are H\"older continuous in $\cB([L^2(\bbR^n)]^N)$-norm with respect to $\lambda$ varying in compact subintervals of 
$\bbR \backslash \{0\}$ $($resp., $\bbR \backslash (\{0\} \cup \cN_0)$$)$. In particular, with $\cN_{\pm}$ defined in analogy to 
\eqref{3.27} by 
\begin{align}
\begin{split} 
\cN_{\pm} = \big\{\lambda \in \bbR \backslash \{0\} \, \big| \, & \, \text{there exists 
$0 \neq f \in \cK$ s.t.} \\
& - f = \ol{V_2 (H_0 - (\lambda \pm i 0)I_{[L^2(\bbR^n)]^N})^{-1} V_1^*} f\big\},    \lb{4.33} 
\end{split} 
\end{align} 
one obtains 
\begin{equation}
\cN_+ = \cN_- := \cN_0,
\end{equation}
and the $($geometric\,$)$ multiplicities of the eigenvalue $\lambda_0 \in \bbR \backslash \{0\}$ 
of $H$ and the eigenvalue $- 1$ of 
$\ol{V_2 (H_0 - (\lambda_0 \pm i 0)I_{[L^2(\bbR^n)]^N})^{-1} V_1^*}$ coincide and are finite. 
Finally, the global wave operators
\begin{equation}
W_{\pm} (H,H_0) = \slim_{t \to \pm \infty} e^{i t H} e^{- i t H_0} ,    \lb{3.50} 
\end{equation}
exist and are complete, that is, 
\begin{equation}
\ker(W_{\pm} (H,H_0)) = \{0\}, \quad 
\ran(W_{\pm} (H,H_0)) = E_{H, ac} \cH,    \lb{3.51a}
\end{equation}
with $ E_{H, ac}$ the projection onto the absolutely continuous subspace of $H$.   
\end{theorem}
\begin{proof}

As discussed above, Theorems \ref{t2.5} and \ref{t2.6} apply to $S_0=H_0$ and $S=H$ and a union of closed intervals $\Lambda_0$ exhausting $(-\infty, 0) \cup (0,\infty)$  under the additional assumption that $\rho>3/2$ (and $\tau \in (1, \rho - 1/2)$). Hence, Theorem \ref{t3.4} is proved subject to $\rho > 3/2$.

To improve this to $\rho > 1$ (and $\tau \in (1/2, \rho - 1/2)$) we now follow  
\cite[p.~98--99, 118]{Ya10}.  First, one notes that if $\lambda\in \bbR\backslash\{0\}$ is an eigenvalue of $H$ with a corresponding eigenvector $\psi\in \LN$, then
\begin{equation} \lb{3.39aa}
\|(\cF_{H_0}\psi)(\mu,\,\cdot\,)\|_{[L^2(S^{n-1})]^N}\leq C_1 |\mu - \lambda|^{\rho-(3/2)},\quad \mu\in \bbR,
\end{equation}
for some $C_1\in (0,\infty)$.
In fact, by \eqref{4.1}, $g:=-V\psi\in \LrhoN$, and since $\langle Q\rangle^{-\rho}$ is strongly $H_0$-continuous with exponent $\rho - (1/2) > (1/2)$ by Proposition \ref{p3.3}, the function $\wti g := \cF_{H_0}g$ is H\"older continuous:
\begin{equation} \lb{3.40aa}
\|\wti g(\mu,\,\cdot\,) - \wti g(\lambda,\,\cdot\,)\|_{[L^2(S^{n-1})]^N}\leq C |\mu - \lambda|^{\rho - (1/2)}, \quad \mu \in \bbR,
\end{equation}
for some constant $C\in (0,\infty)$, which is independent of $\lambda$ and $\mu\in \bbR$.  In addition, since $g = H_0\psi - \lambda \psi$, the spectral representation in \eqref{335y} yields
\begin{equation} \lb{3.41aa}
\wti g(\mu,\,\cdot\,) = (\mu - \lambda)(\cF_{H_0}\psi)(\mu,\,\cdot\,),\quad \mu\in \bbR,
\end{equation}
which implies $\wti g(\lambda,\,\cdot\,)=0$.  Therefore, \eqref{3.40aa} reduces to
\begin{equation}
\|\wti g(\mu,\,\cdot\,)\|_{[L^2(S^{n-1})]^N} \leq C |\mu - \lambda|^{\rho-(1/2)},\quad \mu\in \bbR,
\end{equation}
and then \eqref{3.41aa} yields \eqref{3.39aa} with $C_1=C$.  In addition, one also notes that the equation
$(H_0 - \lambda I_{[L^2(\bbR^n)]^N})\psi = -V\psi$ implies
\begin{equation} \lb{3.43aa}
g=-V(H_0 - (\lambda \pm i0)I_{[L^2(\bbR^n)]^N})^{-1}g,
\end{equation}
since $\psi = (H_0 - (\lambda \pm i0)I_{[L^2(\bbR^n)]^N})^{-1}g$ and $\psi = 0$ if $g = 0$.

To prove that non-zero eigenvalues of $H$ have finite multiplicity and may only accumulate at $0$ and $\pm \infty$, one may follow the proof of \cite[Proposition 1.9.2]{Ya10} essentially verbatim; one only needs to replace $\bbR_+$ by $\bbR\backslash\{0\}$.


Next, one proves that for any $\tau \in (0,1/2]$ and $p<2(1-2\tau)^{-1}$, and for any compact set 
$X\subset \bbR$, 
\begin{equation} \lb{3.44aa}
\int_{X} d\lambda\, \|(\cF_{H_0}f)(\lambda,\,\cdot\,)\|_{[L^2(S^{n-1})]^N}^2 \leq C_2\|f\|_{\LtauN}^p,\quad f\in \LtauN,
\end{equation}
for some $C_2=C_2(\alpha,p,X)\in (0,\infty)$.  To prove \eqref{3.44aa}, one can follow, with minor modifications, the proof of \cite[Proposition 1.9.3]{Ya10}.  Indeed, for an arbitrary compact set 
$X\subset \bbR$, one introduces the family of spaces
\begin{equation}
L^p\big(X; d\lambda; [L^2(S^{n-1})]^N\big),\quad p\in [1,\infty)\cup\{\infty\},
\end{equation}
and observes that by \cite[Theorem 1.1.4]{Ya10},
\begin{equation}
\cF_{H_0}f\in  L^{\infty}\big(X; d\lambda; [L^2(S^{n-1})]^N\big),\quad f\in \LtauN,\; \tau>1/2.
\end{equation}
The formula
\begin{align}
\begin{split} 
(\cT(f_1,f_2))(\lambda) = ((\cF_{H_0}f_1)(\lambda,\,\cdot\,),(\cF_{H_0}f_2)(\lambda,\,\cdot\,))_{[L^2(S^{n-1})]^N}&\\
\text{ for a.e.~$\lambda\in X$ and $(f_1,f_2)\in \LtauN\times \LtauN$},&
\end{split} 
\end{align}
defines a bilinear map for each $q\geq 1$ and $\tau\geq 0$:
\begin{equation}
\cT:\LtauN\times \LtauN\to L^q(X,d\lambda).
\end{equation}
The map $\cT$ is continuous for $\tau_0=0$, $q_0=1$ and $\tau_1>1/2$, $q_1=\infty$, so by Calder\'on's complex bilinear interpolation theorem (c.f., e.g., \cite[\S 1.19.5]{Tr95}, \cite{Ca64}), for any $s\in [0,1]$ the map $\cT$ is continuous for
\begin{equation}
\tau = \tau(s) = s\tau_0 + (1-s)\tau_1,\quad q^{-1}=q(s)^{-1} = sq_0^{-1}+ (1-s)q_1^{-1},
\end{equation}
and
\begin{equation} \lb{3.51}
\|\cT(f_1,f_2)\|_{L^q(X,d\lambda)} \leq C(\tau,X)\|f_1\|_{\LtauN}\|f_2\|_{\LtauN},\quad f_1,f_2\in \LtauN.
\end{equation}
Taking $f_1=f_2=f\in \LtauN$ and $q=p/2$ in \eqref{3.51} yields \eqref{3.44aa}.

In analogy with \cite[Lemma 1.9.4]{Ya10}, if $h\in \LtauN$ for some $\tau\in (1/2,1]$ and $(\cF_{H_0}h)(\lambda,\,\cdot\,)=0$ for some $\lambda\in \bbR\backslash \{0\}$, then
\begin{equation} \lb{3.52aa}
(H_0-(\lambda\pm i0)I_{[L^2(\bbR^n)]^N})^{-1}h\in \LmwtitauN,\quad \wti \tau > 1-\tau.
\end{equation}
To prove \eqref{3.52aa}, it suffices to show
\begin{align}
\begin{split} 
\big|\big((H_0-(\lambda\pm i0)I_{[L^2(\bbR^n)N]})^{-1}h,g\big)_{[L^2(\bbR^n)]^N}\big|\leq C_3\|h\|_{\LtauN}\|g\|_{\LwtitauN},&\\
g\in \SN,& 
\end{split} 
\end{align}
for some $C_3 = C_3(\lambda) \in (0,\infty)$.  Using the spectral representation for $H_0$, one infers that
\begin{align} 
\begin{split} 
&\big((H_0-(\lambda\pm i0)I_{[L^2(\bbR^n)]^N})^{-1}h,g\big)_{[L^2(\bbR^n)]^N}    \\
&\quad = \int_{\bbR}\, d\mu\, (\mu - \lambda \mp i0)^{-1}\big( \wti h(\mu,\,\cdot\,),\wti g(\mu,\,\cdot\,)\big)_{[L^2(S^{n-1})]^N},\quad 
g\in \SN.\lb{3.53aa} 
\end{split} 
\end{align}
Since
\begin{align}
\wti h, \wti g\in L^2\big([0,\infty); d\lambda; L^2(S^{n-1})^{N/2}\big) \oplus \, L^2\big((-\infty,0]); d\lambda; L^2(S^{n-1})^{N/2}\big),
\end{align}
it suffices to estimate the integral in \eqref{3.53aa} over a compact neighborhood, say 
$X_{\lambda}$, of the point $\lambda$.  By Proposition \ref{p3.3},
\begin{align}
\begin{split} 
\big\|\wti h(\mu,\,\cdot\,)\big\|_{[L^2(S^{n-1})]^N} &= \big\|\wti h(\lambda,\,\cdot\,) - \wti h(\mu,\,\cdot\,)\big\|_{[L^2(S^{n-1})]^N}     \\
&\leq C_0|\lambda - \mu|^{\tau-(1/2)}\|h\|_{\LtauN}, 
\end{split} 
\end{align}
for some $C_0\in (0,\infty)$, and consequently,
\begin{align}
&\bigg| \int_{X_{\lambda}}\, d\mu\, (\mu - \lambda \mp i0)^{-1}\big( \wti h(\mu,\,\cdot\,),\wti g(\mu,\,\cdot\,)\big)_{[L^2(S^{n-1})]^N}\bigg|    \lb{3.56aa} \\
&\quad  \leq C_0 \|h\|_{\LtauN}\int_{X_{\lambda}} \, d\mu\, |\lambda - \mu|^{\tau - (3/2)}\big\|\wti g(\mu,\,\cdot\,)\big\|_{[L^2(S^{n-1})]^N},\quad g\in \SN.   \no 
\end{align}
An application of H\"older's inequality yields for any conjugate pair $p^{-1}+q^{-1}=1$ the estimate
\begin{align}
\begin{split} 
& \int_{X_{\lambda}} \, d\mu\, |\lambda - \mu|^{\tau - (3/2)}\big\|\wti g(\mu,\,\cdot\,)\big\|_{[L^2(S^{n-1})]^N}    \\ 
& \quad \leq \bigg(\int_{X_{\lambda}}\, d\mu\, |\lambda-\mu|^{-q((3/2)-\tau)}\bigg)^{q^{-1}} 
 \bigg(\int_{X_{\lambda}}\, d\mu\, \big|\wti g(\mu,\,\cdot\,)\big\|_{[L^2(S^{n-1})]^N}^p\bigg)^{p^{-1}}.\lb{3.57aa}
 \end{split} 
\end{align}
By \eqref{3.44aa} with $C_2=C_2(\alpha,p,X_{\lambda})$, 
\begin{equation} \lb{3.58aa}
\int_{X_{\lambda}}\,d\mu\, \big\|\wti g(\mu,\,\cdot\,)\big\|_{[L^2(S^{n-1})]^N}^p 
\leq C_2 \|g\|_{\LwtitauN}^p,
\end{equation}
where $p<2(1-2\wti \tau)^{-1}$.  Therefore, the conjugate exponent satisfies $q>2(1+2\wti \tau)^{-1}$, and consequently $q((3/2)-\tau)>(3-2\tau)(1+2\wti \tau)^{-1}$. Thus, if $\wti \tau > 1 - \tau$, that is, 
if $\tau + \wti \tau >1$, then $(3-2\tau)(1+2\wti \tau)^{-1}<1$, so $q((3/2)-\tau)$ may be chosen to be smaller than $1$, rendering the first integral on the right-hand side in \eqref{3.57aa} finite. In conclusion, \eqref{3.56aa}, \eqref{3.57aa}, and \eqref{3.58aa} combine to yield the desired estimate.

Finally, we turn to the issue of absence of singular continuous spectrum for $H$. Introducing the set 
$\cN:=\cN_+\cup \cN_-$, so that
\begin{equation}
\sigma_s(H)\backslash\{0\} \subseteq \cN, 
\end{equation}
to prove that $\sigma_{sc}(H) = \emptyset$, it suffices to show that any $\lambda \in \cN$ must be an eigenvalue of $H$.  To this end, let $\lambda\in \cN$, so that there exists an $f\in \LN\backslash \{0\}$ such that 
\begin{equation} \lb{3.60aa}
- f = \ol{V_2 (H_0 - (\lambda \pm i 0)I_{[L^2(\bbR^n)]^N})^{-1} V_1^*} f,
\end{equation}
with $V_1$ and $V_2$ taken as in \eqref{3.25aa} with $\tau\in (1/2,\rho-(1/2))$.  Introducing
\begin{equation} \lb{3.61aa}
g=\langle \,\cdot\,\rangle^{-\tau}f\in \LtauN,
\end{equation}
the equations for $f$ in \eqref{3.60aa} may be recast as \eqref{3.43aa}.  In view of \eqref{3.61aa} and the fact that $\wti g(\lambda)=0$, the estimate in \eqref{3.52aa} applies to $h=g$:
\begin{equation}
(H_0-(\lambda\pm i0)I_{[L^2(\bbR^n)]^N})^{-1}g\in \LmwtitauN,\quad \wti \tau > 1-\tau.
\end{equation}
Then the condition in \eqref{4.1} and the identity in \eqref{3.43aa} combine to yield
\begin{equation}
g=-V(H_0 - (\lambda \pm i0)I_{[L^2(\bbR^n)]^N})^{-1}g \in \LnuN,\quad \nu < \tau + \rho - 1.
\end{equation}
Iterating the same argument $\ell\in \bbN$ times yields
\begin{equation} \lb{3.78yy}
g\in \LnuN,\quad \nu< \tau + \ell(\rho - 1).
\end{equation}
If $\ell$ is chosen so that $\ell > (1-\tau)/(\rho-1)$, then $\tau + \ell(\rho - 1)>1$.  Therefore, \eqref{3.78yy} implies, in particular, that $g\in \LwtitauN$ for some $\wti \tau\in (1,3/2)$.  Consequently, by Proposition \ref{p3.3}, $\wti g$ is H\"older continuous of order $\wti \tau - (1/2)$.  Next, the 
function $\psi:=(H_0-(\lambda \pm i0)I_{[L^2(\bbR^n)]^N})^{-1}g$ belongs to $\dom(H)=\dom(H_0)=\WoneN$ 
since $\wti \psi(\mu) = \wti g(\mu)(\mu-\lambda)^{-1}$, $\wti g(\lambda)=0$, and $\wti \tau - (1/2) > 1/2$.  By \eqref{3.43aa}, $\psi$ satisfies the Dirac equation $H\psi = \lambda \psi$.  Moreover, $\psi$ is a nontrivial solution.  Indeed, if $\psi=0$, then $g = -V\psi=0$.  Of course, one then obtains $f = \langle \,\cdot\,\rangle^{\tau}g=0$, which contradicts the assumption $f\in \LN\backslash\{0\}$.  Therefore, $\psi$ is an eigenfunction and $\lambda$ is a corresponding eigenvalue.  As a result, $\cN\subseteq \sigma_p(H)$ and $\sigma_{sc}(H)=\emptyset$.
\end{proof}

\begin{remark} \lb{r3.5}
The fact that $\big\|\ol{V_2 (H_0 - (\lambda \pm i 0)I_{[L^2(\bbR^n)]^N})^{-1} V_1^*}\big\|_{\cB(L^2(\bbR^3))}$ does not 
decay as $\lambda \to \pm \infty$ shows that in principle one cannot rule out eigenvalues of $H$ running off to $\infty$ 
and/or $-\infty$. In fact, it has been shown in \cite{KOY15a} that for all $\tau > 1/2$, there exists a constant $C_{\tau} \in (0, \infty)$ such that 
\begin{equation}
\sup_{z \in \bbC \backslash \bbR} \big\| \langle \, \cdot \, \rangle^{-\tau} (H_0 - z I_{[L^2(\bbR^3)]^4})^{-1} \langle \, \cdot \, \rangle^{-\tau} \big\|_{\cB([L^2(\bbR^3)]^4)} \leq C_{\tau} 
\quad   \lb{3.78A}
\end{equation} 
and that 
\begin{align}
\begin{split}
\big\| \langle \, \cdot \, \rangle^{-\tau} (H_0 - (\lambda \pm i 0) I_{[L^2(\bbR^3)]^4})^{-1} 
\langle \, \cdot \, \rangle^{-\tau} \big\|_{\cB([L^2(\bbR^3)]^4)}&     \\ 
\text{ does not decay as $|\lambda| \to \infty$ for any $\tau > 1/2$.}&   \lb{3.78B}
\end{split}
\end{align} 
In the case of massive Dirac operators (i.e., with $H_0$ replaced by $H_0(m)$), the condition $\tau > 1/2$ needs to 
be replaced by $\tau \geq 1$. For results in this direction we also refer to \cite{Pl01}, \cite{Pl02}, \cite{PSU95}, \cite{Ya93}. 
This contrasts sharply with the case of Schr\"odinger operators where a Riemann--Lebesgue-type argument yields 
decay of the underlying Birman--Schwinger operator (see, e.g., \cite[Theorem~III.13]{Si71}). \hfill $\diamond$
\end{remark}

\begin{remark} \lb{r3.6}
The transformation in \eqref{3.19y} employed to diagonalize $\alpha\cdot p$ is similar to the celebrated Foldy--Wouthuysen transformation (see, e.g., \cite{Co04}, \cite{Th88}, \cite[Sect.~5.6]{Th92}). The latter is well known to diagonalize $H_0$. In fact, introducing the unitary $N \times N$ block operator matrix $U_N$ in $ \LN$, $n \in \bbN$, $n \geq 2$, 
via 
\begin{equation}
U_N = 2^{-1/2} \big[I_N + \beta (\alpha \cdot (-i \nabla)) |-i \nabla|^{-1} \big], \quad 
U_N^{-1} = 2^{-1/2} \big[I_N - \beta (\alpha \cdot (-i \nabla)) |-i \nabla|^{-1} \big], 
\end{equation}
one infers that 
\begin{align}
\begin{split} 
& \wti H_0 = U_N H_0 U_N^{-1} 
= \begin{pmatrix} I_{N/2} (- \Delta)^{1/2} & 0 \\ 0 & - I_{N/2} (-\Delta)^{1/2}\end{pmatrix}, \\
& \dom\big(\wti H_0\big) = \WoneN.       \lb{4.41}
\end{split} 
\end{align}
It is worth pointing out that every result in this section has a verbatim analog for operators of the 
type 
\begin{equation}
\wti H_0 + V, \, \text{ and } \,  I_{N} \big(- \Delta + m^2 I_{L^2(\bbR^n)}\big)^{1/2} + V, \; m \geq 0, 
\end{equation} 
in $ \LN$, with $V$ satisfying \eqref{4.1}. More generally, $(- \Delta)^{1/2}$ can be replaced 
by general fractional powers  $(- \Delta)^{\gamma}$, $\gamma > 0$, and even by more general 
functions $h(-\Delta)$ (cf.\ \cite{BN97}). This comment is of some significance as a large body of 
work went into studying $ I_{N} (- \Delta)^{1/2} + V$ (especially, in the scalar case $N=1$) over the past two decades. We refer, for instance, to \cite{BN97}, \cite{BCG20}, \cite{Ic88}, \cite{Ki10}, \cite{Ki11}, 
\cite{Ma06}, \cite{RU16}, \cite[p.~124]{Si79}, \cite{Um94}, \cite{Um95}, \cite{Um00}, \cite{Um06}, \cite{UN93}, \cite{UW08}, \cite{We07}. \hfill $\diamond$
\end{remark}

In the following section we will recall conditions on $V$ that yield the absence of eigenvalues 
of $H$ (implying unitary equivalence of $H$ and $H_0$ via the wave operators $W_{\pm}(H,H_0)$ 
in Theorem \ref{t3.4}, see Remark \ref{r4.3}). 

We conclude this section with some hints at additional literature (beyond 
\cite[Sects.~1.11, 2.1, 2.2]{Ya10}) concerning the absence of singular continuous spectrum 
and proofs of limiting absorption principles for operators of the form $H_0 + V$. 

In the the case of three-dimensional massless Dirac operators, the absence of singular continuous spectrum of $H$ with scalar potentials 
(i.e., $V = v \, I_N$), including the case of long-range interactions $v$, was proved in \cite{Da05}. The limiting absorption 
principle for $H_0$ in three dimensions  was derived in \cite{SU08}. For the proof of existence of absolutely continuous spectrum of massless Dirac  operators  for $n=3$, where $V = v \, \beta$, see \cite{De04}. To the best of our knowledge, these references in the special case $n=3$ comprise all explicit statements about the absence of the singular continuous spectrum of $H$ and/or the limiting absorption principle for $H_0$. So Theorem \ref{t3.4} is new for $n \in \bbN \backslash \{3\}$, $n \geq 2$, which is particularly interesting in the case $n=2$ as the latter is related to applications involving graphene. On the other hand, we emphasize that Theorem \ref{t3.4} is a direct consequence of the material presented by Yafaev in \cite[Sect.~2.4]{Ya10}. In the context of massless Dirac operators in dimension $n=2$ we also refer to \cite{EGG20} (see also \cite{EGG19}).We also note that a global limiting absorption principle for $H_0$ on $\bbR$ for all $n \in \bbN$, $n \geq 2$, was proved in \cite{CGKLNPS18}, \cite{KOY15a}, \cite{BU18}.

For the case of massive Dirac operators $H(m) = H + m \, \beta$, $m > 0$, we also refer to \cite{BH92}, \cite{BG10}, \cite{BMP93}, \cite{EG17}, \cite{EGT18}, \cite{EGT19}, \cite{GM01}, \cite{IM99}, \cite{It95}, \cite{Mo64}, \cite{Ok08}, \cite{PSU95}, 
\cite{PSU98}, \cite{Sa17}, \cite{Th72}, \cite{Vo88}, \cite[Sect.~1.12]{Ya10}, \cite{Ya72}, \cite{Ya76}, \cite{Ya93}. 

Finally, for scattering theory for Dirac operators we refer, for instance, to \cite{Da05}, \cite{Ec74}, \cite{GM01}, 
\cite{GS74}, \cite{It95}, \cite{LT88}, \cite{Na76}, \cite{PSU95}, \cite{PSU98}, \cite{Pr63}, \cite{Sa01}, \cite{Si79}, \cite{Th81}, \cite{Th91}, \cite{Th91a}, \cite[Ch.~8]{Th92},  \cite{TE86}, \cite{Th72}, \cite{VW73}, \cite[Sect.~1.12]{Ya10}.

\section{On the Absence of Eigenvalues for Interacting, Massless Dirac Operators} 
\lb{s4}

In this section we briefly comment on results concerning the absence of eigenvalues of 
massless Dirac operators. Since we are particularly interested in potentials vanishing 
at $\pm \infty$, implying the absence of spectral gaps of $H$,  
\begin{equation}
\sigma_{ess}(H) = \bbR, 
\end{equation} 
the absence of eigenvalues is equivalent to the absence of eigenvalues embedded into the 
essential spectrum of $H$ (a somewhat unusual situation from a quantum mechanical point 
of view).

In the context of massive Dirac operators $H(m) = H + m \, \beta$, with mass parameter 
$m > 0$ and vanishing potentials at $\pm \infty$, there exists a fair number of papers 
describing the absence of embedded eigenvalues in the essential spectrum of $H(m)$,
\begin{equation}
\sigma_{ess}(H(m)) = (-\infty, -m] \cup [m,\infty) 
\end{equation}
(or in certain regions of the essential spectrum), especially in the three-dimensional case, $n=3$. Relevant references in this context are, 
for instance, \cite{Al72}, \cite{BG87}, \cite{Ka76}, \cite{Ka81}, \cite{KOY15}, \cite{Ok03}, 
\cite{Ro70} (however, this reference is imprecise w.r.t. implicit smoothness assumptions on the electromagnetic potential coefficients), \cite{Th76}, \cite{Vo87}.

The existence of threshold eigenvalues (and/or resonances) at $\pm m$ are discussed, for 
instance, in \cite{El01}, \cite{SU11}. 

In the massless case, $m=0$, zero eigenvalues and/or zero-energy resonances (as well as the absence zero-energy resonances) are treated in 
\cite{AMN99}, \cite{AMN00}, \cite{AMN00a}, \cite{AMN00b}, \cite{AC79}, \cite{Ai16}, 
\cite{BE01}, \cite{BE02}, \cite{BE03}, \cite[Ch.~4]{BE11}, \cite{BES08}, \cite{BV15}, 
\cite{DP16}, \cite{El00}, \cite{El02}, \cite{ES01}, \cite{FLL86}, \cite{KOY15}, \cite{LY86}, 
\cite{Pe08}, \cite{Ro69}, \cite{RS06}, \cite{Sc10}, \cite{SU08}, \cite{SU08a}, \cite{SU11}, 
\cite{SU12}, \cite{SU14}, \cite{SU15}, \cite{ZG13}. A fair number of these references consider 
the case of Pauli operators in three dimensions, $[\sigma \cdot (-i \nabla - A)]^2$, with 
$\sigma =(\sigma_1, \sigma_2, \sigma_3)$ the standard Pauli matrices. 


The subject of absence of zero modes, especially, zero-energy eigenvalues, for massless Dirac operators has hardly been studied. Exceptions are \cite{DO99} (see also \cite{KY99}), 
\cite{KOY03}, \cite{KOY15},  and properties 
of the corresponding (generalized) eigenfunctions are discussed in 
\cite{Ai16}, \cite[Sect.~4.4]{BE11}, \cite{BES08}, \cite{SU08}, \cite{SU08a}, \cite{ZG13}.

Here we recall the following special cases of results in \cite[Theorems~2.1, 2.3]{KOY15}:

\begin{theorem} \lb{t4.1} 
Let $n \in \bbN$, $n \geq 2$. \\[1mm] 
$(i)$ Assume that $V: \bbR^n \to \bbC^{N \times N}$ is Lebesgue measurable and self-adjoint 
a.e.\ on $\bbR^n$, satisfying 
\begin{equation}
\esssup_{x \in \bbR^n} |x| \|V(x)\|_{\cB(\bbC^N)} \leq C \, \text{ for some $C \in (0, (n-1)/2)$,} \lb{4.3}
\end{equation}
with $\|\,\cdot\,\|_{\cB(\bbC^N)}$ denoting the operator norm of an $N \times N$ matrix in $\bbC^N$. Then any distributional solution 
$u \in [H^1_{loc}(\bbR^n)]^N \cap [L^2(\bbR^n; |x|^{-1} d^n x)]^N$ of 
$(H_0 + V)u = 0$ vanishes identically.  \\[1mm]
$(ii)$ Suppose that 
\begin{align}
& \text{$V: \bbR^n \to \bbC^{n \times n}$ is Lebesgue measurable and self-adjoint 
a.e.\ on $\bbR^n$, and that}  \no \\
& \quad \text{for some $R > 0$, $V \in \big[C^1(E_R)\big]^{N \times N}$, where 
$E_R=\{x \in \bbR^n\,|\, |x| > R\}$,}    \lb{4.3A} \\
& \text{and}     \no \\
\begin{split} 
& |x|^{1/2} V_{\ell,m}(x) \underset{|x| \to \infty}{=} \oh(1), \quad 
(x \cdot \nabla V_{\ell,m})(x) \underset{|x| \to \infty}{=} \oh(1), \quad 1 \leq \ell,m \leq N,    \lb{4.4} \\
& \quad \text{uniformly with respect to directions.} 
\end{split} 
\end{align}

Then if for some $\lambda \in \bbR \backslash \{0\}$, $u \in [L^2(E_R)]^N$ satisfies 
$(H_0 + V)u = \lambda u$ on $E_R$ in the distributional sense, then $u$ vanishes identically on 
$\bbR^n$. \\[1mm]
$(iii)$ The self-adjoint realization $H$ of $H_0 + V$ satisfying 
\begin{equation}
\int_{\bbR^n} d^n x \, |x|^{-1} \|f(x)\|_{\bbC^N}^2 < \infty, \quad f \in \dom(H),   \lb{4.5} 
\end{equation}
has no eigenvalue zero in case $(i)$ and no eigenvalue $\lambda \in \bbR \backslash \{0\}$ 
in case $(ii)$. 
\end{theorem}

\begin{remark} \lb{r4.2} 
We note that Theorem \ref{t4.1} due to \cite[Theorems~2.1, 2.3]{KOY15} appears to have been the first, and up to now, the only result available proving absence of eigenvalues of $H$ in the massless case. The results on global $H_0$-smoothness of $\langle Q \rangle^{-1}$ proven in \cite{CGKLNPS18} now yields a second such result (and unitary equivalence of $H$ and $H_0$) for $\|V\|_{\cB(L^2(\bbR^n))}$ sufficiently small.  
\hfill $\diamond$
\end{remark}

\begin{remark} \lb{r4.3} 
In the context of Theorem \ref{t4.1}\,$(iii)$ we note that if $V$ satisfies \eqref{4.1} and hence 
$\dom(H) = \WoneN$, then Kato's inequality (cf., e.g., \cite[p.~19--20]{BE11}, 
\cite{He77}),
\begin{equation}
\int_{\bbR^n} d^n x \, |x|^{-1} |f(x)|^2 \leq C_n \int_{\bbR^n} d^n p \, 
|p| \big|\hatt f(p)\big|^2, 
\quad f \in S(\bbR^n), \; n \in \bbN, \; n \geq 2,
\end{equation} 
for appropriate constants $C_n \in (0,\infty)$, $n \geq 2$ (Kato's inequality extends to the homogeneous Sobolev space $D^{1/2}(\bbR^n)$ of order $1/2$), yields, in particular,
\begin{align}
& \int_{\bbR^n} d^n x \, |x|^{-1} |f(x)|^2 \leq C_n \int_{\bbR^n} d^n p \, 
|p| \big|\hatt f(p)\big|^2 
\leq C_n \int_{\bbR^n} d^n p \, \big[1 + |p|^2\big] \big|\hatt f(p)\big|^2,   \no \\ 
& \hspace*{7cm} f \in H^1(\bbR^n), \; n \in \bbN, \; n \geq 2. 
\end{align} 
Thus, under assumption \eqref{4.1} on $V$, condition \eqref{4.5} holds automatically. \hfill $\diamond$
\end{remark}

We summarize the discussion on absence of eigenvalues in this section as follows:

\begin{corollary} \lb{c4.4}
$(i)$ In addition to Hypothesis \ref{h3.1} assume that $V$ satisfies conditions \eqref{4.3A}, \eqref{4.4}. Then
\begin{equation}
\sigma_p(H) \subseteq \{0\}.        \lb{4.8}
\end{equation}
$(ii)$ In addition to Hypothesis \ref{h3.1} assume that $V$ satisfies conditions \eqref{4.3}--\eqref{4.4}. Then 
\begin{equation}
\sigma_p(H) = \emptyset.        \lb{4.9}
\end{equation}
Moreover, $H$ and $H_0$ are unitarily equivalent. 
\end{corollary}
\begin{proof}
The inclusion \eqref{4.8} follows from Theorem \ref{t4.1}\,$(ii)$,\,$(iii)$ and Remark \ref{r4.3}. 
The fact \eqref{4.9} follows from Theorem \ref{t4.1}\,$(i)$, Remark \ref{r4.3}, and \eqref{4.8}. Unitary equivalence of $H$ and $H_0$ 
is a consequence of \eqref{3.45}, \eqref{3.50}, \eqref{3.51a}, and \eqref{4.9}. 
\end{proof}

\section{The Green's Functions of $H_0(m)$ and $H_0$}     \lb{s5}

In this section we study the Green's function for $H_0$, that is, the integral kernel of the resolvent 
of $H_0$. 

We start, however, with the Green's function of the Laplacian in $L^2(\bbR^n)$,
\begin{equation} 
h_0 = - \Delta, \quad \dom(h_0) = H^2(\bbR^n).
\end{equation}
The Green's function of $h_0$, denoted by $g_0(z; \, \cdot \,, \, \cdot \,)$, is of the form,
\begin{align}
& g_{0}(z;x,y) := (h_0 - z I_{L^2(\bbR^n)})^{-1}(x,y)   \no \\ 
& \quad = \begin{cases} (i/2) z^{- 1/2} e^{i z^{1/2} |x-y|}, & n =1, \; z\in\bbC\backslash\{0\}, \\[2mm]
(i/4) \big(2\pi z^{-1/2} |x - y|\big)^{(2-n)/2} 
H^{(1)}_{(n-2)/2}\big(z^{1/2}|x - y|\big), & n\ge 2, \; z\in\bbC\backslash\{0\}, 
\end{cases}    \lb{5.2}  \no \\
& \hspace*{6.3cm}   \Im\big(z^{1/2}\big) > 0, 
\; x, y \in\bbR^n, \; x \neq y, 
\end{align}
and for $z=0$, $n \geq 3$, 
\begin{equation}
g_0(0;x,y) = \displaystyle{\f{1}{(n-2) \omega_{n-1}} |x - y|^{2-n}}, \quad n \ge 3, \; x, y \in\bbR^n, \; x \neq y. 
\end{equation}
Here $H^{(1)}_{\nu}(\, \cdot \,)$ denotes the Hankel function of the first kind 
with index $\nu\geq 0$ (cf.\ \cite[Sect.\ 9.1]{AS72}) and 
$\omega_{n-1}=2\pi^{n/2}/\Gamma(n/2)$ ($\Gamma(\, \cdot \,)$ the Gamma function, 
cf.\ \cite[Sect.\ 6.1]{AS72}) represents the area of the unit sphere $S^{n-1}$ in $\bbR^n$. 

As $z\to 0$, $g_0(z; \, \cdot \,, \, \cdot \,)$ is continuous on the off-diagonal for $n\geq 3$, 
\begin{align}
\begin{split}
\lim_{\substack{z \to 0\\z \in \bbC \backslash \{0\}}} g_0(z;x,y) = g_0(0;x,y) 
= \f{1}{(n-2) \omega_{n-1}} |x - y|^{2-n},& \\ 
 x, y \in\bbR^n, \; x \neq y, \; n \in \bbN, \; n \ge 3,&    \lb{5.3a} 
 \end{split} 
\end{align} 
but blows up for $n=1$ as 
\begin{align}
& g_0(z;x,y) \underset{\substack{z\to 0 \\ z\in \bbC \backslash\{0\}}}{=} 
(i/2) z^{-1/2} - 2^{-1} |x-y| + \Oh\big(z^{1/2} |x-y|^2\big), \quad x, y \in \bbR, 
\end{align}
and for $n=2$ as
\begin{align} 
\begin{split} 
& g_0(z;x,y) \underset{\substack{z\to 0 \\ z\in \bbC \backslash\{0\}}}{=} 
-\f{1}{2\pi} \ln\big(z^{1/2}|x - y|/2\big)
\big[1+\Oh\big(z|x - y|^2\big)\big] 
+ \f{1}{2\pi} \psi(1)     \\
& \hspace*{2.7cm}  + \Oh\big(|z||x - y|^2\big),   \quad 
 x, y \in\bbR^2, \; x \neq y.    \lb{5.4} 
 \end{split} 
\end{align} 
Here $\psi(w)=\Gamma'(w)/\Gamma(w)$ denotes the digamma function 
(cf. \cite[Sect.\,6.3]{AS72}). 

For reasons of subsequent comparisons with the case of the free massive Dirac operator 
$H_0(m) = H_0 + m \, \beta$, $m > 0$, we now start with the latter and compute,
\begin{align}
& (H_0(m) - z I_{[L^2(\bbR^n)]^N})^{-1} = (H_0(m) + z I_{[L^2(\bbR^n)]^N}) 
\big(H_0(m)^2 - z^2 I_{[L^2(\bbR^n)]^N}\big)^{-1}    \no \\ 
& \quad = (- i \alpha \cdot \nabla + m \, \beta + z I_{[L^2(\bbR^n)]^N})
\big(h_0 - \big(z^2 - m^2\big) I_{[L^2(\bbR^n)]^N}\big)^{-1} I_N,    \lb{5.5a} 
\end{align}
employing 
\begin{equation}
H_0(m)^2 = (h_0 + m^2 I_{[L^2(\bbR^n)]^N}) I_N. 
\end{equation}

Assuming 
\begin{equation}
m > 0, \; z \in \bbC \backslash (\bbR \backslash [-m,m]), \; \Im\big(z^2 - m^2\big)^{1/2} > 0, \; 
x, y \in \bbR^n, \; x \neq y, \; n \in \bbN, \; n \geq 2,    \lb{5.7}
\end{equation}
and exploiting \eqref{5.5a}, one thus obtains for the Green's function $G_0(m,z;\, \cdot \,,\, \cdot \,)$ 
of $H_0(m)$, 
\begin{align}
& G_0(m,z;x,y) := (H_0(m) - z I_{[L^2(\bbR^n)]^N})^{-1}(x,y)    \no \\
& \quad = i 4^{-1} (2 \pi)^{(2-n)/2} |x - y|^{2-n}(m \, \beta + z I_N)  \no \\
& \qquad \times  \big[\big(z^2 - m^2\big)^{1/2} |x - y|\big]^{(n-2)/2} 
H_{(n-2)/2}^{(1)} \big(\big(z^2 - m^2\big)^{1/2} |x - y|\big)     \no \\
& \qquad - 4^{-1} (2 \pi)^{(2-n)/2} |x - y|^{1 - n} \, \alpha \cdot \f{(x - y)}{|x - y|}    \no \\
& \qquad \quad \times \big[\big(z^2 - m^2\big)^{1/2} |x - y|\big]^{n/2} 
H_{n/2}^{(1)} \big(\big(z^2 - m^2\big)^{1/2} |x - y|\big).    
\end{align}
Here we employed the identity (\cite[p.~361]{AS72}),
\begin{equation}
\big[H_{\nu}^{(1)} (\zeta)\big]' = - H_{\nu + 1}^{(1)}(\zeta) + \nu \, \zeta^{-1} H_{\nu}^{(1)}(\zeta), \quad 
\nu, \zeta \in \bbC.  
\end{equation}

Equations \eqref{5.9}, \eqref{5.10} reveal the facts (still assuming \eqref{5.7}),
\begin{align}
& \lim_{\substack{z \to \pm m \\ z \in \bbC \backslash \{\pm m\}}} 
G_0(m,z; x,y) = 4^{-1} \pi^{- n/2} \Gamma((n-2)/2) |x - y|^{2-n} (m \, \beta \pm m I_N)     \lb{5.14} \\
& \quad + i 2^{-1} \pi^{-n/2} \Gamma(n/2) \, \alpha \cdot \f{(x - y)}{|x - y|^n}, 
\quad m > 0, \; x, y \in \bbR^n, \; x \neq y, \; n \in \bbN, \; n \geq 3,    \no \\
& G_0(m,z; x,y) \underset{\substack{z \to \pm m \\ z \in \bbC \backslash \{\pm m}}{=}  
- (4 \pi)^{-1} \ln\big(z^2 - m^2\big) (m \, \beta \pm m I_2)        \no \\
&  \quad  - (2 \pi)^{-1} \ln(|x - y|) (m \, \beta \pm m I_2) 
+ i (2 \pi)^{-1} \, \alpha \cdot \f{(x - y)}{|x - y|^2}    \lb{5.15} \\
&  \quad  
+ \Oh\big(\big(z^2 - m^2\big) \ln\big(z^2 - m^2\big)\big), 
\quad m > 0, \; x, y \in \bbR^2, \; x \neq y.  \no
\end{align}
(Here the remainder term $\Oh\big(\big(z^2 - m^2\big) \ln\big(z^2 - m^2\big)\big)$ depends 
on $x, y \in \bbR^2$, but this is of no concern at this point.) 
In particular, $G_0(m,z; \, \cdot \, , \, \cdot \,)$ blows up logarithmically as $z \to \pm m$ 
in two dimensions, $n=2$, just as $g_0(z, \, \cdot \, , \, \cdot \,)$ does as $z \to 0$. 

By contrast, the massless case is quite different and assuming 
\begin{equation}
z \in \bbC_+, \; x, y \in \bbR^n, \; x \neq y, \; n \in \bbN, \; n \geq 2,    \lb{5.16}
\end{equation}
one computes in the case $m=0$ for the Green's function $G_0(z;\, \cdot \,,\, \cdot \,)$ of $H_0$, 
\begin{align}
& G_0(z;x,y) := (H_0 - z I_{[L^2(\bbR^n)]^N})^{-1}(x,y)     \no \\ 
& \quad = i 4^{-1} (2 \pi)^{(2-n)/2} |x - y|^{2-n} z \, [z |x - y|]^{(n-2)/2} 
H_{(n-2)/2}^{(1)} (z |x - y|) I_N   \lb{5.17} \\
& \qquad - 4^{-1} (2 \pi)^{(2-n)/2} |x - y|^{1-n} [z |x - y|]^{n/2} H_{n/2}^{(1)} (z |x - y|) \, 
\alpha \cdot \f{(x - y)}{|x - y|}.    \no
\end{align}
The Green's function $G_0(z;\, \cdot \,,\, \cdot \,)$ of $H_0$ continuously extends to $z \in \ol{\bbC_+}$. In addition, in the massless case $m=0$, the limit $z \to 0$ exists\footnote{Our choice of notation $0+i\,0$ in 
$G_0(0+i\,0;x,y)$ indicates that the limit $\lim_{z \to 0}$ is performed in the closed upper half-plane 
$\ol{\bbC_+}$.}, 
\begin{align}
\begin{split} 
& \lim_{\substack{z \to 0, \\ z \in \ol{\bbC_+} \backslash\{0\}}} G_0(z;x,y) := G_0(0+i\,0;x,y)   \\
& \quad = i 2^{-1} \pi^{-n/2} \Gamma(n/2) \, \alpha \cdot \f{(x - y)}{|x - y|^n}, \quad x, y \in \bbR^n, \; x \neq y, \; n \in \bbN, \; n \geq 2,    \lb{5.18} 
\end{split} 
\end{align} 
and no blow up occurs for all $n \in \bbN$, $n \geq 2$. 

\begin{remark} \lb{r5.1} 
$(i)$ The observation of an absence of blow up in $G_0(z;\, \cdot \,,\, \cdot \,)$ as $z \to 0$ is consistent with the sufficient condition for the Dirac operator $H = H_0 + V$ (in dimensions $n \in \bbN$, $n \geq 2$), with $V$ an appropriate self-adjoint $N \times N$ matrix-valued potential, having no eigenvalues, as derived in \cite[Theorems~2.1, 2.3]{KOY15}. \\[1mm]
$(ii)$ The asymptotic behavior, for some $d_n \in (0,\infty)$, 
\begin{align}
\|G_0(0+i\,0;x,y)\|_{\bbC^N} \underset{\substack{z \to 0, \\ z \in \ol{\bbC_+} \backslash\{0\}}}{=}  
d_n |x-y|^{1-n}, \quad x, y \in \bbR^n, \; x \neq y, \; n \in \bbN, \; n \geq 2, 
\end{align} 
implies the absence of zero-energy resonances (cf.\ Section \ref{s9} for a detailed discussion) of $H$ for $n\in \bbN$, $n\geq 3$, for sufficiently fast decaying short-range potentials $V$ at infinity, as $|\, \cdot \,|^{1-n}$ lies in 
$L^2(\bbR^n)$ near infinity if and only if $n \geq 3$. This is consistent with observations in \cite{Ai16}, \cite[Sect.~4.4]{BE11}, \cite{BES08}, \cite{BGW95}, \cite{SU08}, \cite{SU08a}, \cite{ZG13} for $n=3$ (see also Remark \ref{r9.8}\,$(ii)$). 
This should be contrasted with the behavior of Schr\"odinger operators where 
\begin{align}
\begin{split}
\lim_{\substack{z \to 0 \\ z \in \bbC \backslash \{0\}}} g_0(z;x,y) = g_0(0;x,y) 
= \f{1}{(n-2) \omega_{n-1}} |x - y|^{2-n},& \\ 
 x, y \in\bbR^n, \; x \neq y, \; n \in \bbN, \; n \ge 3,&    
 \end{split} 
\end{align} 
implies the absence of zero-energy resonances of $h= h_0 + w$ for $n\in \bbN$, $n\geq 5$, again for sufficiently fast decaying short-range potentials $w$ at infinity, as $|\, \cdot \,|^{2-n}$ lies in $L^2(\bbR^n)$ near infinity if and only if $n \geq 5$, as observed in \cite{Je80}.
\hfill $\diamond$
\end{remark}

\begin{remark} \lb{r5.2} 
In the special case $n=3$, the identities
\begin{align}
H_{1/2}^{(1)}(\zeta) &= -i\bigg(\frac{2}{\pi}\bigg)^{1/2} \frac{e^{i\zeta}}{\zeta^{1/2}},\quad \\
H_{3/2}^{(1)}(\zeta) &= -\bigg(\frac{2}{\pi}\bigg)^{1/2}\frac{e^{i\zeta}(\zeta+i)}{\zeta^{3/2}},\quad \zeta\in \bbC\backslash\{0\},
\end{align}
combine in \eqref{5.17} to yield
\begin{align}
&G_0(z;x,y)\\
&\quad= \frac{e^{iz|x-y|}}{4\pi |x-y|}\bigg[z \, I_N + z \, \alpha\cdot \frac{(x-y)}{|x-y|} + i \, \alpha\cdot \frac{(x-y)}{|x-y|^2}\bigg],\quad x,y\in \bbR^3,\; x\neq y,\; z\in \bbC_+.\no
\end{align}
\hfill $\diamond$
\end{remark}

\begin{remark}
It is possible to expand the massless Dirac Green's function $G_0(z;\,\cdot\,,\,\cdot\,)$ in powers of $z$ in such a way that several coefficients in the expansion vanish (the precise number of vanishing coefficients depending on the dimension $n$) for odd dimensions $n\geq 5$.  This observation relies on the following connection between the modified Bessel and spherical Bessel functions (cf., e.g., \cite[10.1.1]{AS72}):
\begin{align}
H_{j+(1/2)}^{(1)}(\zeta) = (2\pi^{-1}\zeta)^{1/2} h_j^{(1)}(\zeta),\quad \zeta\in \bbC\backslash\{0\},\; j\in \bbN.\lb{B.33}
\end{align}
Moreover, by \cite[10.1.16]{AS72},
\begin{align}
h_j^{(1)}(\zeta) = i^{-(j+1)}\zeta^{-1}e^{i\zeta}\sum_{k=0}^j\frac{(j+k)!}{k!(j-k)!}(-2i\zeta)^{-k},\quad \zeta\in \bbC\backslash\{0\},\; j\in \bbN.\lb{B.34}
\end{align}
Upon combining \eqref{B.33} and \eqref{B.34}, one obtains for odd dimensions $n\geq 3$,
\begin{align}
&H_{(n/2)-1}^{(1)}(\zeta) = H_{[(n-3)/2]+(1/2)}^{(1)}(\zeta)\lb{B.35}\\
&\quad= 2^{1/2}\pi^{-1/2}i^{(1-n)/2}\zeta^{-1/2}e^{i\zeta}\sum_{k=0}^{(n-3)/2}\frac{([(n-3)/2]+k)!}{k!([(n-3)/2]-k)!}(-2i\zeta)^{-k},\quad \zeta \in \bbC\backslash\{0\},\no
\end{align}
and
\begin{align}
&H_{(n/2)}^{(1)}(\zeta) = H_{[(n-1)/2]+(1/2)}^{(1)}(\zeta)    \no \\
&\quad= 2^{1/2}\pi^{-1/2}i^{-(n+1)/2}\zeta^{-1/2}e^{i\zeta}\sum_{k=0}^{(n-1)/2}\frac{([(n-1)/2]+k)!}{k!([(n-1)/2]-k)!}(-2i\zeta)^{-k},     \lb{B.36} \\
& \hspace*{9cm} \zeta\in \bbC\backslash\{0\}.   \no
\end{align}
Thus, using the expansions \eqref{B.35} and \eqref{B.36} in \eqref{5.17}, one obtains the following expansion for the massless Dirac Green's function in odd dimensions $n\geq 3$:
\begin{align} \lb{B.39}
&G_0(z;x,y) = i (-1)^{(1-n)/2}2^{-(n+1)/2}\pi^{(1-n)/2}e^{iz|x-y|}I_N   \no\\
&\qquad \times \sum_{k=0}^{(n-3)/2}\frac{([(n-3)/2]+k)!}{k!([(n-3)/2]-k)!}(-2)^{-k}(iz)^{-k+[(n-1)/2]}|x-y|^{-k-[(n-1)/2]}\no\\
&\quad+ i(-1)^{(1-n)/2}2^{-(n+1)/2}\pi^{(1-n)/2}e^{iz|x-y|}\alpha\cdot\frac{(x-y)}{|x-y|}    \\
&\qquad \times \sum_{k=0}^{(n-1)/2}\frac{([(n-1)/2]+k)!}{k!([(n-1)/2]-k)!}(-2)^{-k}(iz)^{-k+[(n-1)/2]}|x-y|^{-k-[(n-1)/2]},\no\\
&\hspace*{7.55cm} x,y\in \bbR^n,\; x\neq y,\; z\in \bbC\backslash\bbR.\no
\end{align}
Introducing the power series for the exponential in \eqref{B.39} and reordering the series to combine like powers of $iz$, one obtains
\begin{align} \lb{B.40}
&G_0(z;x,y) = (-1)^{(3-n)/2}2^{-(n+1)/2}\pi^{(1-n)/2}z|x-y|\sum_{j=0}^{\infty}d_j(iz)^j|x-y|^{-(n-1-j)}I_N\no\\
&\quad + i(-1)^{(1-n)/2}2^{-(n+1)/2}\pi^{(1-n)/2}\sum_{j=0}^{\infty}d_j'(iz)^j|x-y|^{-(n-1-j)}\alpha\cdot\frac{(x-y)}{|x-y|},\no\\
&\hspace*{7cm} x,y\in \bbR^n,\; x\neq y,\; z\in \bbC\backslash\bbR, 
\end{align}
where for each $j\in \bbN_0$, the numerical coefficients $d_j$ and $d_j'$ are given by
\begin{align}
d_j &= \sum_{\substack{k=0 \\ k\geq [(n-3)/2]-j}}^{(n-3)/2}\frac{([(n-3)/2]+k)!}{k!([(n-3)/2]-k)!}(-2)^{-k}
\frac{1}{(j+k-[(n-3)/2])!},\lb{B.41}\\
d_j'&= \sum_{\substack{k=0 \\ k\geq [(n-1)/2]-j}}^{(n-1)/2}\frac{([(n-1)/2]+k)!}{k!([(n-1)/2]-k)!}(-2)^{-k}
\frac{1}{(j+k-[(n-1)/2])!},\quad j\in \bbN_0.\lb{B.42}
\end{align}

In odd dimensions $n\geq 5$, certain of the coefficients $d_j$ and $d_j'$ in \eqref{B.41} and \eqref{B.42} vanish based on the following combinatorial identity:

\smallskip

\noindent 
{\bf Proposition} ({\cite[Lemma 3.3]{Je80}}).  
If $m\in \bbN$ and
\begin{equation}
c_j:= \sum_{\substack{k=0 \\ k\geq m-j}}^m \frac{(m+k)!}{k!(m-k)!}(-2)^{-k}\frac{1}{(k+j-m)!},\quad j\in \bbN_0,
\end{equation}
then $c_j=0$ for $j=1,3,\ldots,2m-1$.

Applying this proposition with $m=(n-3)/2$ and $m=(n-1)/2$, one infers that for $n\geq5$ is odd, the free massless Green's function $G_0(z,\,\cdot\,,\,\cdot\,)$ is given by \eqref{B.40}--\eqref{B.42} and
\begin{align}
d_j&=0\quad\text{for all odd $j\in \bbN$ satisfying $1\leq j\leq n-4$},\\
d_j'&=0\quad\text{for all odd $j\in\bbN$ satisfying $1\leq j\leq n-2$}.
\end{align}
\hfill $\diamond$
\end{remark}

Since $H_0$ has no spectral gap, $\sigma(H_0) = \bbR$, but $h_0$ has the half-line 
$(-\infty, 0)$ in its resolvent set, a comparison of $h_0$ with the massive free 
Dirac operator $H_0(m) = H_0 + m \, \beta$, $m > 0$, with spectral gap $(- m, m)$, replacing the energy $z=0$ by $z=\pm m$, is quite natural and then exhibits a similar logarithmic blowup behavior as $z \to 0$
in dimensions $n=2$.

Returning to our analysis of the resolvent of $H_0$, the asymptotic behavior \eqref{5.9}--\eqref{5.11} 
implies for some $c_n \in (0,\infty)$,
\begin{align} 
\|G_0(0+i\,0;x,y)\|_{\cB(\bbC^N)} \leq c_n |x - y|^{1-n}, \quad 
x, y \in \bbR^n, \; x \neq y, \; n \in \bbN, \; n \geq 2,  \lb{5.19}
\end{align}
and for given $R \geq 1$,
\begin{align} 
& \|G_0(z;x,y)\|_{\cB(\bbC^N)} \leq c_{n,R}(z) e^{- \Im(z) |x - y|} \begin{cases} 
 |x - y|^{1-n}, & |x - y| \leq 1, \; x \neq y, \\
1, & 1 \leq |x - y | \leq R, \\
 |x - y|^{(1 - n)/2}, & |x - y| \geq R,
\end{cases}     \no \\ 
& \hspace*{4.65cm} 
z \in \ol{\bbC_+}, \; x, y \in \bbR^n, \; x \neq y, \; n \in \bbN, \; n \geq 2,  
\lb{5.20}
\end{align}
for some $c_{n,R}(\, \cdot \,)\in (0,\infty)$ continuous and locally bounded on $\ol{\bbC_+}$. 

For future purposes we now rewrite $G_0(z; \, \cdot \,,\, \cdot\,)$ as follows:
\begin{align}
& G_0(z;x,y) = i 4^{-1} (2 \pi)^{(2-n)/2} |x - y|^{2-n} z \, [z |x - y|]^{(n-2)/2} 
H_{(n-2)/2}^{(1)} (z |x - y|) I_N   \no \\
& \qquad - 4^{-1} (2 \pi)^{(2-n)/2} |x - y|^{1-n} [z |x - y|]^{n/2} H_{n/2}^{(1)} (z |x - y|) \, 
\alpha \cdot \f{(x - y)}{|x - y|}    \no \\
& \quad = |x - y|^{1-n} f_n(z,x-y),     \lb{5.35} \\
& \hspace*{9mm}
z \in \ol{\bbC_+}, \; x, y \in \bbR^n, \; x \neq y, \; n \in \bbN, \; n \geq 2,   \no 
\end{align}
where $f_n$ is continuous and locally bounded on $\ol{\bbC_+} \times \bbR^n$, in addition,
\begin{align}
\begin{split} 
\|f_n(z,x)\|_{\cB(\bbC^{N})} \leq c_n(z) e^{- \Im(z) |x|} \begin{cases} 
1, & 0 \leq |x| \leq 1, \\
|x|^{(n -1)/2}, & |x| \geq 1, 
\end{cases}&     \lb{5.35a} \\ 
z \in \ol{\bbC_+}, \; x, y \in \bbR^n,&
\end{split} 
\end{align}
for some constant $c_n(\, \cdot \,)\in (0,\infty)$ continuous and locally bounded on 
$\ol{\bbC_+}$. In particular, decomposing $G_0(z; \, \cdot \,,\, \cdot\,)$ into
\begin{align} 
G_0(z;x,y) &= G_0(z;x,y) \chi_{[0,1]}(|x-y|)
 + G_0(z;x,y) \chi_{[1,\infty)}(|x-y|)    \no \\
& = G_{0, <}(z;x-y) + G_{0, >}(z;x-y),     \lb{5.36} \\
& \hspace*{6mm} 
z \in \ol{\bbC_+}, \; x, y \in \bbR^n, \; x \neq y, \; n \in \bbN, \; n \geq 2,   \no
\end{align}
where
\begin{align} 
G_{0, <}(z;x-y) &:= G_0(z;x,y) \chi_{[0,1]}(|x-y|),\lb{5.39w}\\
G_{0, >}(z;x-y) &:= G_0(z;x,y) \chi_{[1,\infty)}(|x-y|),\\
& \hspace*{6mm}z \in \ol{\bbC_+}, \; x, y \in \bbR^n, \; x \neq y, \; n \in \bbN, \; n \geq 2,   \no
\end{align}
one verifies that 
\begin{align}
\begin{split}
|G_{0, >}(z;x-y)_{j,k}| \leq \begin{cases} C_n |x-y|^{-(n-1)}, & z=0, \\
C_n(z) |x-y|^{-(n-1)/2}, & z \in \ol{\bbC_+},  \\
\end{cases}&     \\
x,y \in \bbR^n, \; |x-y| \geq 1, \; 1 \leq j,k \leq N,&      \lb{5.36a} 
\end{split}
\end{align} 
for some constants $C_n, C_n(\, \cdot \,) \in (0, \infty)$, in particular, 
\begin{equation}
G_{0,>}(z;\, \cdot \,) \in [L^{\infty}(\bbR^n)]^{N \times N}, 
\quad z \in \ol{\bbC_+},    \lb{5.37}
\end{equation}
and that
\begin{equation}\lb{5.42c}
\text{$G_{0,>}(\, \cdot \,;\, \cdot \,)$ is continuous on 
$\ol{\bbC_+} \times \bbR^n$.}
\end{equation}

In the next section, we will use the decomposition \eqref{5.36} to derive trace ideal properties of operators of the type 
$F_1(\, \cdot \,) (H_0 - z I_{[L^2(\bbR^n)]^N})^{-1} F_2(\, \cdot \,)$, employing results of 
\cite[Subsection~5.4]{BKS91} 
in the case $n \geq 3$.  We also derive trace ideal properties of 
$\langle \, \cdot \, \rangle^{-\delta} (H_0 - z I_{[L^2(\bbR^n)]^N})^{-1} \langle \, \cdot \, \rangle^{-\delta}$ 
in the case $n \geq 2$ using a different approach based on a combination of Sobolev's inequality and complex interpolation. 

\section{Trace Ideal Properties of $F_1(\, \cdot \,) (H_0 - z I_{[L^2(\bbR^n)]^N})^{-1} F_2(\, \cdot \,)$ and 
$\langle \, \cdot \, \rangle^{-\delta} (H_0 - z I_{[L^2(\bbR^n)]^N})^{-1}
\langle \, \cdot \, \rangle^{-\delta} $}     \lb{s6}

In the first part of this section we derive trace ideal properties of operators of the type 
$F_1(\, \cdot \,) (H_0 - z I_{[L^2(\bbR^n)]^N})^{-1} F_2(\, \cdot \,)$, employing results of 
\cite[Subsection~5.4]{BKS91} 
in the case $n \geq 3$. In the second part of this section we derive trace ideal properties of 
$\langle \, \cdot \, \rangle^{-\delta} (H_0 - z I_{[L^2(\bbR^n)]^N})^{-1} \langle \, \cdot \, \rangle^{-\delta}$ 
in the case $n \geq 2$ by a different approach based on a combination of Sobolev's inequality and complex interpolation. These two approaches are independent and complement each other. 

The considerations \eqref{5.35}--\eqref{5.37} readily imply the following facts:

\begin{lemma} \lb{l5.3}
Let $n \in \bbN$, $n \geq 2$, and $F,H \in [L^2(\bbR^n)]^{N \times N}$. Introducing 
\begin{equation}
R_{0,>,F,H} (z;x,y) = F(x) G_{0, >}(z;x-y) H(y), \quad z \in \ol{\bbC_+}, 
\; x,y \in \bbR^n,   \lb{5.38}
\end{equation} 
the integral operator $R_{0,>,F,H} (z)$ in $L^2(\bbR^n)$ with integral kernel $R_{0,>,F,H} (z;\, \cdot \,,\, \cdot \,)$ satisfies
\begin{equation}
R_{0,>,F,H} (z) \in \cB_2\big([L^2(\bbR^n)]^N\big), \quad z \in \ol{\bbC_+},  
\lb{5.39} 
\end{equation}
and $R_{0,>,F,H} (\, \cdot \,)$ is continuous on $\ol{\bbC_+}$ with respect to 
the $\|\, \cdot \,\|_{\cB_2([L^2(\bbR^n)]^N)}$-norm. 

In particular, this applies to $F, H$ satisfying for some constant $C \in (0,\infty)$, 
\begin{equation}
|F_{j,k}|, |H_{j,k}|  \leq C \langle \, \cdot \, \rangle^{- \delta}, \quad \delta > n/2, \; 1 \leq j, k \leq N. 
\end{equation} 
\end{lemma}
\begin{proof}
We apply Theorem \ref{tA.2}\,$(iii)$ and Lemma \ref{lA.3}. 

Let $F,H\in [L^2(\bbR^n)]^{N \times N}$ and $z\in \overline{\bbC_+}$ be fixed.  To prove \eqref{5.39}, it suffices to show
\begin{equation} \lb{5.41}
\big\|R_{0,>,F,H} (z;\,\cdot\,,\,\cdot\,)\big\|_{\cB_2(\bbC^N)} 
\in L^2(\bbR^{2n};d^n x \, d^n y)
\end{equation}
and apply \cite[Thm.~11.6]{BS87} (in the special case 
$L^2(\bbR^n \times \bbR^n; d^nx \,d^n y)$).  To prove \eqref{5.41} we recall
\begin{equation} \lb{5.42}
\|D\|_{\cB_2(\bbC^n)} \leq N^{1/2} \|D\|_{\cB(\bbC^N)},\quad D \in \bbC^{N\times N}.
\end{equation}
Then by \eqref{5.37} and \eqref{5.41},
\begin{align}
\|R_{0,>,F,H} (z;x,y)\|_{\cB_2(\bbC^N)} &\leq N^{1/2} 
\|F(x)\|_{\cB(\bbC^N)} \|G_{0,>}(z;x-y)\|_{\cB(\bbC^N)} \|H(y)\|_{\cB(\bbC^N)}\no\\
&\leq C(z) \|F(x)\|_{\cB(\bbC^N)} \|H(y)\|_{\cB(\bbC^N)},\quad x,y\in \bbR^n,\lb{5.43}
\end{align}
for an appropriate constant $C(z)>0$.  Since by hypothesis $F,H\in [L^2(\bbR^n)]^{N \times N}$, and hence,
\begin{equation}
\|F(\, \cdot \,)\|_{\cB(\bbC^N)}^2 \leq \|F(\, \cdot \,)\|_{\cB_2(\bbC^N)}^2 
= \sum_{j,k=1}^N |F_{j,k}(\, \cdot \,)|^2 \in L^1(\bbR^n),
\end{equation}
and analogously for $H$, the estimate in \eqref{5.43} implies \eqref{5.41}.

To prove the continuity claim, let $z,z'\in \overline{\bbC_+}$.  One computes (cf.~\cite[Thm.~11.6]{BS87})
\begin{align}
&\big\|R_{0,>,F,H} (z) - R_{0,>,F,H} (z') \big\|_{\cB_2([L^2(\bbR^n)]^N)}^2\no\\
&\quad = \int_{\bbR^n\times\bbR^n} d^n x \, d^n y \, 
\big\|R_{0,>,F,H} (z;x,y) - R_{0,>,F,H} (z';x,y) \big\|_{\cB_2(\bbC^N)}^2    \no\\
&\quad \leq N \int_{\bbR^n\times\bbR^n} d^n x \, d^n y \, 
\|F(x)\|_{\cB(\bbC^N)}^2 \big\|G_{0,>}(z;x-y) - G_{0,>}(z';x-y)\big\|_{\cB(\bbC^N)}^2   \no \\
& \hspace*{3.6cm} \times \|H(y)\|_{\cB(\bbC^N)}^2.    \lb{5.44}
\end{align}
An application of Lebesgue's dominated convergence theorem, making use of \eqref{5.37}, 
$F,H \in [L^2(\bbR^n)]^{N \times N}$, and the continuity of $G_{0,>}(z;x-y)$ with respect to $(z, x-y)$
in $\cB(\bbC^N)$ (see \eqref{5.42c}), then yields
\begin{equation}
\lim_{\substack{z\to z' \\ z, z' \in \ol{\bbC_+}}}\big\|R_{0,>,F,H} (z) - R_{0,>,F,H} (z') \big\|_{\cB_2([L^2(\bbR^n)]^N)}=0.
\end{equation}
\end{proof}

To improve upon Lemma \ref{l5.3}, we now recall the following version of Sobolev's inequality (see, e.g., \cite[Corollary~I.14]{Si71}).

\begin{theorem} \lb{t5.4}
Let $n \in \bbN$, $\lambda \in (0, n)$, $r, s \in (1,\infty)$, $r^{-1} + s^{-1} + 
\lambda n^{-1} = 2$, $f \in L^r(\bbR^n)$, $h \in L^s(\bbR^n)$. Then, there 
exists $C_{r,s,\lambda,n} \in (0,\infty)$ such that
\begin{equation}
\int_{\bbR^n \times \bbR^n} d^n x \, d^n y \, \f{|f(x)||h(y)|}{|x-y|^{\lambda}}
\leq C_{r,s,\lambda,n} \|f\|_{L^r(\bbR^n)} \|h\|_{L^s(\bbR^n)}.    \lb{6.18} 
\end{equation}
\end{theorem}

For subsequent purposes, we also recall some basic facts on $L^p$-properties of Riesz potentials (see, e.g., 
\cite[Sect.~V.1]{St70}):

\begin{theorem} \lb{t5.4A}
Let $n \in \bbN$, $\alpha \in (0,n)$, and introduce the Riesz potential operator $\cR_{\alpha,n}$ as follows: 
\begin{align}
\begin{split} 
& (\cR_{\alpha,n} f)(x) = \big((- \Delta)^{- \alpha/2} f\big)(x) 
= \gamma(\alpha,n)^{-1} \int_{\bbR^n} d^n y \, |x - y|^{\alpha - n} f(y),     \lb{5.50} \\
& \gamma(\alpha,n) = \pi^{n/2} 2^{\alpha} \Gamma(\alpha/2)/\Gamma((n-\alpha)/2), 
\end{split}
\end{align}
for appropriate functions $f$ $($see below\,$)$. \\[1mm] 
$(i)$ Let $p \in [1,\infty)$ and $f \in L^p(\bbR^n)$. Then the integral $(\cR_{\alpha,n} f)(x)$ converges for $($Lebesgue\,$)$ a.e.~$x \in \bbR^n$.   \\[1mm] 
$(ii)$ Let $1 < p < q < \infty$, $q^{-1} = p^{-1} - \alpha n^{-1}$, and $f \in L^p(\bbR^n)$. Then there 
exists $C_{p,q,\alpha,n} \in (0,\infty)$ such that
\begin{equation}
\|\cR_{\alpha,n} f\|_{L^q(\bbR^n)} 
\leq C_{p,q,\alpha,n} \|f\|_{L^p(\bbR^n)}.    \lb{5.51} 
\end{equation}
\end{theorem}

We also note the $\beta$ function-type integral (cf.\ \cite[p.~118]{St70}),
\begin{align}
& \int_{\bbR^n} d^n y \, |e_k - y|^{\alpha - n} |y|^{\beta - n} 
= \gamma(\alpha,n) \gamma(\beta,n)/\gamma(\alpha + \beta,n),     \no \\
& 0 < \alpha < n, \; 0 < \beta < n, \; \alpha + \beta < n,   \lb{5.52} \\
& e_k =(0,\dots,\underbrace{1}_{k},\dots,0), \; 1 \leq k \leq n.    \no 
\end{align}
and the Riesz composition formula (see \cite[Sects.~3.1, 3.2]{Du70}),
\begin{align}
& \int_{\bbR^n} d^n y \, |x_1 - y|^{\alpha - n} |y - x_2|^{\beta - n} 
= [\gamma(\alpha,n) \gamma(\beta,n)/\gamma(\alpha + \beta,n)] |x_1 - x_2|^{\alpha + \beta -n},   \no \\
& \hspace*{3.8cm} 0 < \alpha < n, \; 0 < \beta < n, \; \alpha + \beta < n, \; x_1, x_2 \in \bbR^n.     \lb{5.52A} 
\end{align}

For later use in Section \ref{s9}, we recall the following estimate taken from \cite[Lemma~6.3]{EG10}.
\begin{lemma} \lb{l3.12}
Let $n\in \bbN$ and $x_1,x_2\in \bbR^n$.  If $\alpha,\beta \in (0,n]$, $\varepsilon, \gamma \in (0,\infty)$, 
with $n + \gamma \geq \alpha + \beta$, and $\alpha + \beta \neq n$, then
\begin{align}
\begin{split}
&\int_{\bbR^n} d^ny \, |x_1 - y|^{\alpha - n} \langle y\rangle^{-\gamma-\varepsilon} |y - x_2|^{\beta - n}     \\
&\quad \leq C_{n,\alpha,\beta,\gamma,\varepsilon} 
\begin{cases}
|x_1 - x_2|^{-\max\{0,n - \alpha - \beta\}},& |x_1 - x_2|\leq 1,\\
|x_1 - x_2|^{-\min\{n - \alpha,n - \beta,n + \gamma - \alpha -\beta\}},& |x_1 - x_2| \geq 1,
\end{cases}
\end{split} 
\end{align}
where $C_{n,\alpha,\beta,\gamma,\varepsilon}\in (0,\infty)$ is an $x_1, x_2$-independent constant. \end{lemma}

Returning to $G_{0,>}(z;\, \cdot \,)$, we next combine the estimate \eqref{5.36a} with Theorem \ref{t5.4}, 
rather than just using the $L^{\infty}$-bound \eqref{5.37} on $G_{0,>}(z;\, \cdot \,)$ in Lemma \ref{l5.3}, 
yielding a considerable improvement of Lemma \ref{l5.3}.

\begin{theorem} \lb{t5.5}
Let $n \in \bbN$, $n \geq 2$. \\[1mm]
$(i)$ Let  $z = 0$ and $F,H \in [L^{4n/(n+\varepsilon)}(\bbR^n)]^{N \times N}$ for some 
$\varepsilon > 0$. Introducing the integral operator $R_{0,>,F,H} (0)$ in 
$L^2(\bbR^n)$ with integral kernel $R_{0,>,F,H} (0;\, \cdot \,,\, \cdot \,)$ as in \eqref{5.38}, then 
\begin{equation}
R_{0,>,F,H} (0) \in \cB_2\big([L^2(\bbR^n)]^N\big).   
\lb{5.37A} 
\end{equation}
In particular, this applies to $F, H$ satisfying for some constant $C \in (0,\infty)$, 
\begin{equation} 
|F_{j,k}|, |H_{j,k}|  \leq C \langle \, \cdot \, \rangle^{- \delta}, \quad \delta > n/4, \; 1 \leq j, k \leq N. 
\end{equation} 
$(ii)$ Let $z \in \ol{\bbC_+}$ and $F,H \in [L^{4n/(n+1)}(\bbR^n)]^{N \times N}$. 
Introducing the integral operator $R_{0,>,F,H} (z)$ in $L^2(\bbR^n)$ with 
integral kernel $R_{0,>,F,H} (z;\, \cdot \,,\, \cdot \,)$ as in \eqref{5.38}, then 
\begin{equation}
R_{0,>,F,H} (z) \in \cB_2\big([L^2(\bbR^n)]^N\big), \quad z \in \ol{\bbC_+},  
\lb{5.40A} 
\end{equation}
and $R_{0,>,F,H} (\, \cdot \,)$ is continuous on $\ol{\bbC_+}$ with respect to 
the $\|\, \cdot \,\|_{\cB_2([L^2(\bbR^n)]^N)}$-norm. 

In particular, this applies to $F, H$ satisfying for some constant $C \in (0,\infty)$, 
\begin{equation}
|F_{j,k}|, |H_{j,k}|  \leq C \langle \, \cdot \, \rangle^{- \delta}, \quad \delta > (n+1)/4, \; 1 \leq j, k \leq N.
\end{equation}   
\end{theorem}
\begin{proof}
Again, we apply Theorem \ref{tA.2}\,$(iii)$ and Lemma \ref{lA.3}. 

If $z=0$, then $R_{0,>,F,H} (0;\, \cdot \,,\, \cdot \,)$ generates a Hilbert--Schmidt operator in 
$L^2(\bbR^n)$ upon applying the following modified $z=0$ part in estimate \eqref{5.36a},  
\begin{equation}
|G_{0, >}(z;x-y)_{j,k}| \leq c_{n, \varepsilon}  |x-y|^{-(n-\varepsilon)/2}, \quad 
x,y \in \bbR^n, \; |x-y| \geq 1, \; 1 \leq j,k \leq N, 
\end{equation} 
for some constants $c_{n, \varepsilon} \in (0, \infty)$, combined with Sobolev's inequality 
in the form 
\begin{align}
\begin{split} 
& \int_{\bbR^n \times \bbR^n} d^nx d^n y \, \f{|f(x)|^2 |h(y)|^2}{|x-y|^{n-\varepsilon}}\chi_{[1,\infty)}(|x-y|)   \\
& \quad \leq C_{n, \varepsilon} \|f^2\|_{L^{2n/(n+\varepsilon)}(\bbR^n)} \|h^2\|_{L^{2n/(n+\varepsilon)}(\bbR^n)}, 
\end{split} 
\end{align}
identifying $r=s= 2n/(n+\varepsilon)$, $\lambda = n-\varepsilon$ in \eqref{6.18}. One verifies that 
$\langle \, \cdot \, \rangle^{-\delta} \in L^{4n/(n+\varepsilon)}(\bbR^n)$ if $\delta > (n + \varepsilon)/4$, 
and, since $\varepsilon > 0$ can be chosen arbitrarily small, if $\delta > n/4$.  

The general case $z \in \ol{\bbC_+}$ follows along the same lines using the modified estimate 
\eqref{5.36a}, 
\begin{equation}
|G_{0, >}(z;x-y)_{j,k}| \leq c_{n}  |x-y|^{-(n-1)/2}, \quad 
x,y \in \bbR^n, \; |x-y| \geq 1, \; 1 \leq j,k \leq N, 
\end{equation} 
for some constant $c_n \in (0, \infty)$, again combined with Sobolev's inequality in the form 
\begin{align}
\begin{split} 
& \int_{\bbR^n \times \bbR^n} d^nx d^n y \, \f{|f(x)|^2 |h(y)|^2}{|x-y|^{n-1}}\chi_{[1,\infty)}(|x-y|)   \\
& \quad \leq C_{n} \|f^2\|_{L^{2n/(n+1)}(\bbR^n)} \|h^2\|_{L^{2n/(n+1)}(\bbR^n)},     \lb{6.23A} 
\end{split} 
\end{align}
identifying $r=s= 2n/(n+1)$, $\lambda = n-1$ in \eqref{6.18}. One verifies that 
$\langle \, \cdot \, \rangle^{-\delta} \in L^{4n/(n+1)}(\bbR^n)$ if $\delta > (n + 1)/4$.  

Finally, continuity of $R_{0,>,F,H} (\, \cdot \,)$ on $\ol{\bbC_+}$ with respect to the 
$\|\, \cdot \,\|_{\cB_2([L^2(\bbR^n)]^N)}$-norm follows again by applying Lebesgue's dominated convergence theorem as in the proof of Lemma \ref{l5.3}. 
\end{proof}

We recall the following interesting results of McOwen \cite{Mc79} and 
Nirenberg--Walker \cite{NW73}, which provide necessary and sufficient conditions for the boundedness of 
certain classes of integral operators in $L^p(\bbR^n)$:

\begin{theorem} \lb{t5.6}
Let $n \in \bbN$, $c, d \in \bbR$, $c + d > 0$, $p \in (1,\infty)$, and $p'=p/(p-1)$. Then the following items 
$(i)$ and $(ii)$ hold. \\[1mm] 
$(i)$ Consider 
\begin{equation}
K_{c,d}(x,y) = |x|^{-c} |x - y|^{(c + d) - n} |y|^{-d}, 
\quad x, y \in \bbR^n, \; x \neq x',    \lb{5.21} 
\end{equation}
then the integral operator $K_{c,d}$ in $L^p(\bbR^n)$ with integral kernel 
$K_{c,d}(\, \cdot \,, \, \cdot \,)$ 
in \eqref{5.21} is bounded if and only if $c < n/p$ and $d < n/p'$. \\[1mm] 
$(ii)$ Consider 
\begin{equation}
\wti K_{c,d}(x,y) = (1+|x|)^{-c} |x - y|^{(c + d) - n} (1+|y|)^{-d}, 
\quad x, y \in \bbR^n, \; x \neq x',    \lb{5.21mc}
\end{equation}
then the integral operator $\wti K_{c,d}$ in $L^p(\bbR^n)$ with integral kernel 
$\wti K_{c,d}(\, \cdot \,, \, \cdot \,)$ 
in \eqref{5.21mc} is bounded if and only if $c < n/p$ and $d < n/p'$.
\end{theorem}

This result implies the following fact.

\begin{theorem} \lb{t5.7} 
Let $n \in \bbN$, $n \geq 2$. \\[1mm] 
$(i)$ Then the integral operator $R_{0,\delta}$ in $[L^2(\bbR^n)]^N$ with associated integral kernel $R_{0,\delta}( \, \cdot \,, \, \cdot \,)$ bounded entrywise by 
\begin{equation}
|R_{0,\delta}( \, \cdot \,, \, \cdot \,)_{j,k}| \leq C \langle \, \cdot \, \rangle^{-\delta} 
|G_0(0; \, \cdot \,, \, \cdot \,)_{j,k}| \langle \, \cdot \, \rangle^{-\delta}, \quad \delta \geq 1/2, \; 1 \leq j,k \leq N,  
\end{equation}
for some $C \in (0,\infty)$, is bounded, 
\begin{equation} 
R_{0,\delta} \in \cB\big([L^2(\bbR^n)]^N\big).   \lb{5.23}
\end{equation} 
$(ii)$ The integral operator $R_{0,\delta}(z)$ in $[L^2(\bbR^n)]^N$, with associated integral 
kernel $R_{0,\delta}(z; \, \cdot \,, \, \cdot \,)$ bounded entrywise by 
\begin{align}
\begin{split} 
|R_{0,\delta}(z; \, \cdot \,, \, \cdot \,)_{j,k}| \leq C \langle \, \cdot \, \rangle^{-\delta} 
|G_0(z;\, \cdot \,, \, \cdot \,)_{j,k}| \langle \, \cdot \, \rangle^{-\delta},& \\
\delta \geq (n + 1)/4, \; z \in \ol{\bbC_+}, \; 1 \leq j,k \leq N,  
\end{split}
\end{align}  
for some $C \in (0,\infty)$, is bounded,  
\begin{equation} 
R_{0,\delta}(z) \in \cB\big([L^2(\bbR^n)]^N\big), \quad z \in \ol{\bbC_+}.    \lb{5.22}
\end{equation} 
$(iii)$ The integral operator $R_{0,\alpha,\beta}(z)$ in $[L^2(\bbR^n)]^N$, with associated integral 
kernel $R_{0, \alpha, \beta}(z; \, \cdot \,, \, \cdot \,)$ bounded entrywise by 
\begin{align}
\begin{split} 
|R_{0, \alpha,\beta}(z; \, \cdot \,, \, \cdot \,)_{j,k}| \leq C \langle \, \cdot \, \rangle^{-\alpha} 
|G_0(z;\, \cdot \,, \, \cdot \,)_{j,k}| \langle \, \cdot \, \rangle^{-\beta},& \\
\alpha \geq (n - 1)/2, \; \beta \geq 1, \; z \in \ol{\bbC_+}, \; 1 \leq j,k \leq N,  
\end{split}
\end{align}  
for some $C \in (0,\infty)$, is bounded,  
\begin{equation} 
R_{0,\alpha,\beta}(z) \in \cB\big([L^2(\bbR^n)]^N\big), \quad z \in \ol{\bbC_+}.    \lb{5.24}
\end{equation} 
\end{theorem}
\begin{proof}
$(i)$ The inclusion \eqref{5.23} is then an immediate consequence of \eqref{5.18} and hence the estimate 
$|G_0(0;x,y)_{j,k}| \leq C |x - y|^{1-n}$, $x, y \in \bbR^n$, $x \neq y$, $1 \leq j,k \leq N$, 
Theorem \ref{t5.6}, choosing $c = d = 1/2$ in \eqref{5.21}, and an application of 
Theorem \ref{tA.2}\,$(i)$ and Lemma \ref{lA.3}. \\[1mm] 
$(ii)$ To prove the inclusion \eqref{5.22} we employ the estimates \eqref{5.9}--\eqref{5.11} 
(cf.\ also \eqref{5.20}) to obtain
\begin{align}
|G_0(z;x,y)_{j,k}| & \leq C(z) |x - y|^{1-n} \chi_{[0,1]}(|x - y|)    \no \\
& \quad + D(z) |x - y|^{(1-n)/2} \chi_{[1,\infty)}(|x - y|),      \lb{5.37a} \\
& \hspace*{-2mm} z \in \ol{\bbC_+}, 
\; x, y \in \bbR^n, \; x \neq y, \; 1 \leq j,k \leq N,      \no 
\end{align} 
for some $C, D(z) \in (0,\infty)$, and apply Theorems \ref{t5.6} (parts $(i)$ or $(ii)$) and \ref{tA.2}\,$(i)$ (cf.\ also 
Lemma \ref{lA.3}) to both terms on the right-hand sides of 
\eqref{5.37a}. The part $0 \leq |x-y| \leq 1$ in \eqref{5.37a} leads to $\delta \geq 1/2$, whereas the 
part $|x-y| \geq 1$ in \eqref{5.37a} yields $\delta \geq (n+1)/4$, implying \eqref{5.22}. \\[1mm]
$(iii)$ Again we employ the estimate \eqref{5.37a} and argue as in item $(ii)$ for the part where $|x-y| \leq 1$. 
For the part $|x-y| \geq 1$ in \eqref{5.37a} one employs Theorem \ref{t5.6} with $c = \alpha \geq (n-1)/2$ and 
$d = \beta \geq 1$. 
\end{proof}

Given the fact \eqref{5.40A}, we will now focus on $G_{0,<}(z;\, \cdot \,)$, $z \in \ol{\bbC_+}$.  We begin by recalling that $a(- i\nabla)$ is a convolution-type operator of the form,
\begin{equation}
(a(-i \nabla) \varphi)(x) := \big(\big(\cF^{-1}a\big) * \varphi\big)(x) = 
(2 \pi)^{-n/2} \int_{\bbR^n} d^n y \, a^{\vee}(x-y) \varphi(y), \quad 
\varphi \in \cS(\bbR^n),    \lb{6.1} 
\end{equation}
given $a \in \cS'(\bbR^n)$, $n \in \bbN$. We are particularly interested in operators of the type 
\begin{equation} 
b(Q) a(-i \nabla), 
\end{equation} 
with $Q$ abbreviating the operator of multiplication by the independent variable $x$, such that $b(Q) a(-i \nabla)$ extends to a bounded, actually, compact operator in $L^2(\bbR^n)$, in fact, we will focus on its membership in certain Schatten--von Neumann classes. The prime result we will employ from \cite[Subsection~5.4]{BKS91} in this context can be formulated as follows: 

\begin{theorem} [{\cite[Subsection~5.4, p.~103]{BKS91}}] \lb{t6.1}
Let $2 < r < s$, and suppose that $a, \psi \in L^r_{weak} (\bbR^n)$, 
$\psi >0$, $\|\psi\|_{L^r_{weak}(\bbR^n)} \leq 1$, and let $b$ be a measurable function such that $b / \psi \in 
L^s_{weak}\big(\bbR^n; \psi^r d^nx\big)$. Then 
\begin{equation}
b(Q) a(-i \nabla) \in \cB_s\big(L^2(\bbR^n)\big),
\end{equation}
and for some constant $C(r,s) \in (0,\infty)$, 
\begin{equation}
\|b(Q) a(-i \nabla)\|_{\cB_s(L^2(\bbR^n))} \leq C(r,s) \|a\|_{L^r_{weak}(\bbR^n)} 
\bigg(\int_{\bbR^n} d^n x \, b(x)^s \psi(x)^{r-s}\bigg)^{1/s}.
\end{equation}
\end{theorem}

Next, we recall (with $n \in \bbN$, $n \geq 2$) that
\begin{align}
& |\, \cdot \,|^{\gamma-n} \in L^{n/(n-\gamma)}_{weak} (\bbR^n), 
\quad 0 < \gamma < n,  \\
& \big(\big(|\, \cdot \,|^{\gamma-n}\big)^{\wedge}\big)(\xi) = c_n |\xi|^{-\gamma}, 
\quad \big(|\, \cdot \,|^{\gamma-n}\big)^{\wedge} \in L^{n/\gamma}_{weak} (\bbR^n), 
\quad 0 < \gamma < n,   \\
\begin{split} 
& G_{0,<}(z;\, \cdot \,)_{j,k} = |\, \cdot \,|^{1-n} f_n(z,\, \cdot \,)_{j,k} 
\chi_{[0,1]}(|\, \cdot \,|) \in L^{n/(n-1)}_{weak} (\bbR^n) \subset L^p(\bbR^n), \\ 
& \hspace*{5.85cm} 1 \leq j,k \leq N, \; p \in (0, n/(n-1)),  
\end{split} 
\end{align}
where $f_n(z,\dott)$ and $G_{0,<}(z;\, \cdot \,)$ are defined by \eqref{5.35a} and \eqref{5.39w}, respectively.  In addition, we recall the Hausdorff--Young inequality and its weak 
analog (cf., e.g., \cite[p.~32]{RS75}),
\begin{align}
& \big\|f^{\wedge}\big\|_{L^{p/(p-1)}(\bbR^n)} \leq D_{p,n} 
\|f\|_{L^p(\bbR^n)}, \quad p \in [1,2], \\
& \big\|f^{\wedge}\big\|_{L^{p/(p-1)}_{weak}(\bbR^n)} \leq C_{p,n} 
\|f\|_{L^p_{weak}(\bbR^n)}, \quad p \in (1,2),   \lb{wHYi}
\end{align} 
noting that $p/(p-1) \in (2, \infty)$ if $p \in (1,2)$. In particular, since 
$p = n/(n-1) \in (1,2)$ for $n \geq 3$, 
\begin{align} 
\begin{split}
[G_{0,<}(z;\, \cdot \,)_{j,k}]^{\wedge} = 
[|\, \cdot \,|^{1-n} f_n(z,\, \cdot \,)_{j,k} \chi_{[0,1]}(|\, \cdot \,|)]^{\wedge} 
 \in L^n_{weak} (\bbR^n),&       \\
1 \leq j,k \leq N, \; n \in \bbN, \; n \geq 3.&    \lb{6.10} 
\end{split} 
\end{align}

Thus, an application of Theorem \ref{t6.1} and yields the following result.

\begin{theorem} \lb{t6.2}
Let $n \in \bbN$, $n \geq 3$, $\varepsilon > 0$, and assume that for some $q \in (n, \infty)$, $F \in [L^q(\bbR^n; (1+|x|)^{(q-n)(1 + \varepsilon)} d^nx)]^{N \times N}$. Then,  
\begin{equation}
F(Q) G_{0,<}(z;-i \nabla)^{\wedge} \in \cB_q\big([L^2(\bbR^n)]^N\big), \quad z \in \ol{\bbC_+},  
\lb{6.11} 
\end{equation}
and
\begin{align}
& \big\|F(Q)_{j,\ell} [G_{0,<}(z;-i \nabla)_{\ell,k}]^{\wedge}\big\|_{\cB_q(L^2(\bbR^n))} \leq C_{\ell,k}(n, q) 
\big\|[G_{0,<}(z,\, \cdot \,)_{\ell,k}]^{\wedge}\big\|_{L^n_{weak} (\bbR^n)}    \no \\[1mm] 
& \quad \times \bigg(\int_{\bbR^n} (1 + |x|)^{(q-n)(1+\varepsilon)} d^n x \, |F(x)_{j,\ell}|^q\bigg)^{1/q},       \quad 1 \leq j,k,\ell \leq N.      \lb{6.12}
\end{align} 
In addition, the operator $F(Q) G_{0,<}(z;-i \nabla)^{\wedge}$ is continuous with respect to 
$z \in \ol{\bbC_+}$ in the $\cB_q\big([L^2(\bbR^n)]^N\big)$-norm. 
\end{theorem}
\begin{proof} Pick $k, \ell \in \{1, \dots, N\}$.  
Introducing $\psi(x) = c (1 + |x|)^{-1 - \varepsilon}$, $x \in \bbR^n$, 
$c > 0$, one infers that $\psi \in L^n(\bbR^n)$ and choosing $c = \big\|(1 + |\,\cdot\,|)^{-1-\varepsilon}\big\|_{L^n_{weak}(\bbR^n)}^{-1}$ yields $\|\psi\|_{L^n_{weak}(\bbR^n)} \leq 1$. Identifying 
$a^{\vee}$ in Theorem \ref{t6.1} and 
\begin{equation} 
G_{0,<}(z;\, \cdot \,)_{\ell,k} \in L^{n/(n-1)}_{weak} (\bbR^n) \subset L^p(\bbR^n), \quad  
p \in (0, n/(n-1)), 
\end{equation} 
the inclusion \eqref{6.10} yields 
\begin{equation}
a = \big(a^{\vee}\big)^{\wedge} = [G_{0,<}(z;\, \cdot \,)_{\ell,k}]^{\wedge} 
\in L^n_{weak} (\bbR^n).
\end{equation}
Identifying $b$ in Theorem \ref{t6.1} with $F_{j,\ell} \in L^q(\bbR^n; (1+|x|)^{(q-n)(1 + \varepsilon)} d^nx)$, $r$ with $n \geq 3$, and $s$ with $q > n$, one verifies that $F_{j,\ell}/\psi \in L^n\big(\bbR^n; \psi^n d^n x\big)$ and all hypotheses of Theorem \ref{t6.1} are satisfied. Hence, the inclusion \eqref{6.11} and the estimate \eqref{6.12} hold. 

Continuity of $F(Q) G_{0,<}(z;-i \nabla)^{\wedge}$ with respect to $z \in \ol{\bbC_+}$ in the 
$\cB_q\big([L^2(\bbR^n)]^N\big)$-norm follows from the estimate \eqref{6.12} (replacing 
$G_{0,<}(z;-i \nabla)$ by $G_{0,<}(z;-i \nabla) - G_{0,<}(z';-i \nabla)$), 
the explicit structure of $G_{0,<}(z;\, \cdot \,)$ in \eqref{5.35}, \eqref{5.36},  
and the continuity of $f_n(\, \cdot \,, \, \cdot \,)$, combined with the weak Hausdorff--Young 
inequality \eqref{wHYi} and the fact that $L^q(\bbR^n;d\rho) \subset L^q_{weak}(\bbR^n; d\rho)$ with 
$\|g\|_{L^q_{weak}(\bbR^n; d\rho)} \leq \|g\|_{L^p(\bbR^n; d\rho)}$, $g \in L^q(\bbR^n; d\rho)$, 
$q \in (0,\infty)$. Indeed, with $C_n \in (0,\infty)$ some universal constant, 
\begin{align}
& \big\|[G_{0,<}(z;\, \cdot \,)_{\ell,k}]^{\wedge} 
- G_{0,<}(z';\, \cdot \,)_{\ell,k}]^{\wedge}\big\|_{L^n_{weak} (\bbR^n)}     \no \\
& \quad \leq C_n \|G_{0,<}(z;\, \cdot \,)_{\ell,k} 
- G_{0,<}(z';\, \cdot \,)_{\ell,k}\|_{L^{n/(n-1)}_{weak} (\bbR^n)}    \no \\
& \quad \leq C_n \|G_{0,<}(z;\, \cdot \,)_{\ell,k} - G_{0,<}(z';\, \cdot \,)_{\ell,k}\|_{L^{n/(n-1)}(\bbR^n)} 
\underset{\substack{z \to z' \\[.5mm] z, z' \in \ol{\bbC_+}}}{\longrightarrow} 0,  
\end{align}
applying the dominated convergence theorem.  
\end{proof}

A combination of Theorems \ref{t5.5} and \ref{t6.2} then yields the first principal result of this section, which strengthens a part of Theorem \ref{t3.4} (see \eqref{formula3}) and shows that the Birman--Schwinger operators $\ol{V_2 (H_0 - zI_{[L^2(\bbR^n)]^N})^{-1} V_1^*}$  are continuous in the closed upper half-plane in an appropriate Schatten norm, provided that $V_j$, $j=1,2$ satisfy appropriate boundedness and decay hypotheses.

\begin{theorem} \lb{t6.3}
Let $n \in \bbN$, $n \geq 3$, $\varepsilon > 0$, and suppose that 
\begin{equation}
F_1 \in [L^q(\bbR^n; (1+|x|)^{(q-n)(1 + \varepsilon)} d^nx)]^{N \times N} \, 
\text{ for some $q \in (n, \infty)$,}
\end{equation}
and 
\begin{equation}
F_{\ell} \in [L^{4n/(n+1)}(\bbR^n)]^{N \times N} \cap [L^{\infty}(\bbR^n)]^{N \times N}, \quad \ell =1,2. 
\end{equation}
Introducing 
\begin{equation}
R_{0,F_1,F_2} (z,x,y) = F_1(x) G_0(z;x,y) F_2(y), \quad z \in \ol{\bbC_+}, 
\; x,y \in \bbR^n, \; x \neq y,  \lb{6.15}
\end{equation} 
the integral operator $R_{0,F_1,F_2} (z)$ in $[L^2(\bbR^n)]^N$ with integral kernel 
$R_{0,F_1,F_2} (z,\, \cdot \,,\, \cdot \,)$ satisfies
\begin{equation}
R_{0,F_1,F_2} (z) \in \cB_q\big([L^2(\bbR^n)]^N\big), \quad z \in \ol{\bbC_+},  
\lb{6.16} 
\end{equation}
and $R_{0,F_1,F_2} (\, \cdot \,)$ is continuous on $\ol{\bbC_+}$ with respect to 
the $\|\, \cdot \,\|_{\cB_q([L^2(\bbR^n)]^N)}$-norm. 

In particular, this applies to $F_{\ell}$, $\ell = 1,2$, satisfying for some constant $C \in (0,\infty)$, 
\begin{equation} 
|F_{\ell,j,k}|  \leq C \langle \, \cdot \, \rangle^{- \delta}, \quad \delta > (n+1)/4, \; 1 \leq j, k \leq N, 
\; \ell =1,2. 
\end{equation} 
\end{theorem}
\begin{proof}
Recalling the decomposition \eqref{5.36}, 
\begin{equation}
G_0(z;x,y) = G_{0, <}(z;x-y) + G_{0, >}(z;x-y), \quad x, y \in \bbR^n, \; x \neq y,
\lb{6.17} 
\end{equation}
(now employed for $n \in \bbN$, $n \geq 3$), one applies Theorem \ref{t5.5} to $G_{0,>}(z;\, \cdot \,)$ and Theorem \ref{t6.2} to $G_{0,<}(z;\, \cdot \,)$. 

One readily verifies that if $\delta > (n+1)/4$, then $\langle \, \cdot \, \rangle^{-\delta} I_N$ satisfies the conditions assumed on $F_{\ell}$, $\ell =1,2$, 
\end{proof}

This handles the case $n \geq 3$. Due to the condition $s > r > 2$ (in the underlying concrete case, $r=n$) in Theorem \ref{t6.1}, the special case $n=2$ in connection with $G_{0, <}(z;\, \cdot \,)$ does not subordinate to these techniques and hence will be treated using an alternative approach next (which actually applies to all dimensions $n \geq 2$). While Theorem \ref{t6.3} only handles the case 
$n \geq 3$, it has the advantage that it yields continuity of $R_{0,F_1,F_2} (\, \cdot \,)$ on $\ol{\bbC_+}$ (and hence, particularly along the real axis) in a straightforward manner. 

To describe an alternative approach to this circle of ideas, we start with some preparatory material 
on the following trace ideal interpolation result, see, for instance, \cite[Theorem~III.13.1]{GK69}, \cite[Theorem~0.2.6]{Ya10} (see also \cite{GLST15}, \cite[Theorem~III.5.1]{GK70}).  

\begin{theorem} \lb{t6.4} 
Let $p_j \in [1,\infty) \cup \{\infty\}$, $\Sigma= \{\zeta\in\bbC\,|\, \Re(\zeta) \in (\xi_1,\xi_2)\}$, 
$\xi_j \in \bbR$, $\xi_1 < \xi_2$, $j=1,2$. Suppose that $A(\zeta)\in\cB(\cH)$, $\zeta\in \ol \Sigma$ 
and that $A(\,\cdot\,)$ is analytic on $\Sigma$, continuous up to $\partial \Sigma$, and that 
$\|A(\, \cdot \,)\|_{\cB(\cH)}$ is bounded on $\ol \Sigma$. Assume that for some $C_j \in (0,\infty)$, 
\begin{equation}
\sup_{\eta \in \bbR}\|A(\xi_j + i \eta)\|_{\cB_{p_j}(\cH)} \leq C_j, \quad j=1,2.      \lb{6.19}
\end{equation}
Then
\begin{equation}
A(\zeta) \in \cB_{p(\Re(\zeta))}(\cH), \quad \f{1}{p(\Re(\zeta))} = \f{1}{p_1}
+ \f{\Re(\zeta) - \xi_1}{\xi_2 - \xi_1} \bigg[\f{1}{p_2} - \f{1}{p_1}\bigg], 
\quad \zeta \in \ol \Sigma,     \lb{6.20}
\end{equation}
and
\begin{equation}
\|A(\zeta)\|_{\cB_{p(\Re(\zeta))}(\cH)}  \leq C_1^{(\xi_2 - \Re(\zeta))/(\xi_2 - \xi_1)} 
C_2^{(\Re(\zeta) - \xi_1)/(\xi_2 - \xi_1)},
\quad \zeta\in \ol \Sigma.     \lb{6.21}
\end{equation}
In case $p_j = \infty$, $\cB_{\infty}(\cH)$ can be replaced by $\cB(\cH)$. 
\end{theorem}

A combination of Theorems \ref{t5.6}, \ref{t5.4} and \ref{t6.4} then yields the following fact 
(cf.\ \cite{GN20}).

\begin{theorem} \lb{t6.5} 
Let $n \in \bbN$, $n \geq 2$, $0 < 2 \gamma < n$, $\delta > \gamma$, and suppose  that 
$T_{\gamma,\delta}$ is an integral operator in $L^2(\bbR^n)$ whose integral kernel 
$T_{\gamma,\delta}(\, \cdot \, , \, \cdot \,)$ satisfies the estimate 
\begin{equation}
|T_{\gamma, \delta}(x,y)| \leq C \langle x \rangle^{-\delta} |x - y|^{2 \gamma - n} 
\langle y \rangle^{-\delta}, \quad x, y \in \bbR^n, \; x \neq y 
\end{equation}
for some $C \in (0,\infty)$. Then, 
\begin{equation} \lb{6.23}
T_{\gamma, \delta} \in \cB_p\big(L^2(\bbR^n)\big), \quad p > n/(2\gamma), \; p \geq 2,
\end{equation} 
and 
\begin{align}
\|T_{\gamma, \delta}\|_{\cB_{n/(2 \gamma - \varepsilon)}(L^2(\bbR^n))} 
& \leq \sup_{\eta \in \bbR} \big[\|T_{\gamma, \delta}(- 2 \gamma + \varepsilon + i \eta)
\|_{\cB(L^2(\bbR^n))}\big]^{2[- 2 \gamma + (n/2) + \varepsilon]/n}     \no \\
& \quad \times \sup_{\eta \in \bbR} \big[\|T_{\gamma, \delta}(- 2 \gamma + (n/2) + \varepsilon + i \eta)
\|_{\cB_2(L^2(\bbR^n))}\big]^{2(2 \gamma - \varepsilon)/n}     \lb{6.25}
\end{align}
for $0 < \varepsilon$ sufficiently small. 
\end{theorem} 
\begin{proof}
Following the idea behind Yafaev's proof of \cite[Lemma~0.13.4]{Ya10}, we 
introduce the analytic family of integral operators 
$T_{\gamma, \delta}(\, \cdot \,)$ in $L^2(\bbR^n)$ generated by the integral kernel 
\begin{equation}
T_{\gamma, \delta}(\zeta;x,y) = T_{\gamma, \delta}(x,y) \, \langle x \rangle^{- (\zeta/2)} 
|x - y|^{\zeta} 
\langle y \rangle^{- (\zeta/2)}, \quad x, y \in \bbR^n, \; x \neq y,  
\end{equation}
noting $T_{\gamma, \delta}(0) = T_{\gamma, \delta}$. 

By Theorems \ref{t5.6}\,$(ii)$ and \ref{tA.2}\,$(i)$ (for $N=1$), 
\begin{equation}
T_{\gamma, \delta}(\zeta) \in \cB\big(L^2(\bbR^n)\big), \quad 0 < \Re(\zeta) + 2 \gamma < n, 
\; \delta \geq \gamma. 
\end{equation}
To check the Hilbert--Schmidt property of $T_{\gamma, \delta}(\, \cdot \,)$ one 
estimates for the square of $|T_{\gamma, \delta}(\, \cdot \, ; \, \cdot \,,\, \cdot \,)|$,
\begin{align}
\begin{split} 
|T_{\gamma, \delta}(\zeta;x,y)|^2 \leq 
\langle x \rangle^{- 2\delta - \Re(\zeta)} 
|x - y|^{2 \Re(\zeta) + 4 \gamma - 2n}  \langle y \rangle^{- 2\delta - \Re(\zeta)},&  \\
x, y \in \bbR^n, \; x \neq y,&  
\end{split} 
\end{align}
and hence one can apply Theorem \ref{t5.4} upon identifying 
$\lambda = 2n - 4 \gamma- 2 \Re(z)$, $r=s=n/[\Re(\zeta) + 2 \gamma]$, and 
$f = h = \langle \, \cdot \, \rangle^{- [2 \delta + \Re(\zeta)]}$, to verify that 
$0 < \lambda < n$ translates into $n/2 < \Re(\zeta) + 2 \gamma < n$, and 
$f \in L^r(\bbR^n)$ holds with $r \in(1,2)$ if $\delta > \gamma$. Hence, 
\begin{equation}
T_{\gamma, \delta}(\zeta) \in \cB_2\big(L^2(\bbR^n)\big), 
\quad n/2 < \Re(\zeta) + 2 \gamma < n, \; \delta > \gamma.
\end{equation}  
It remains to interpolate between the $\cB\big(L^2(\bbR^n)\big)$ and 
$\cB_2\big(L^2(\bbR^n)\big)$ property, employing Theorem \ref{t6.4} as follows. 
Choosing $0 < \varepsilon$ sufficiently small, one identifies 
$\xi_1 = - 2 \gamma + \varepsilon$, $\xi_2 = - 2 \gamma + (n/2) + \varepsilon$, 
$p_1 = \infty$, $p_2 = 2$, and hence obtains
\begin{equation} \lb{6.28}
p(\Re(\zeta)) = n/[\Re(\zeta) + 2 \gamma - \varepsilon], 
\end{equation} 
in particular, $p(0) > n/(2 \gamma)$ (and of course, $p(0) \geq 2$).  Since $\varepsilon$ may be taken arbitrarily small, \eqref{6.23} follows from \eqref{6.28} and \eqref{6.25} is a direct consequence of 
\eqref{6.21}. 
\end{proof}

One notes that while subordination in general only applies to $\cB_p$-ideals with $p$ even (see the 
discussion in \cite[p.~24 and Addendum~E]{Si05}), the use of complex interpolation in Theorem \ref{t6.5} (and the focus on bounded and Hilbert--Schmidt operators) permits one to avoid this restriction.  

Combining Theorems \ref{t5.4}, \ref{t5.6}\,$(ii)$, \ref{t6.4}, and \ref{t6.5} then yields the second principal result of this section.

\begin{theorem} \lb{t6.6} 
Let $n \in \bbN$, $n \geq 2$. Then the integral operator $R_{0,\delta}$ in $[L^2(\bbR^n)]^N$ with integral kernel $R_{0,\delta}( \, \cdot \,, \, \cdot \,)$ permitting the entrywise bound  
\begin{equation}
|R_{0,\delta}( \, \cdot \,, \, \cdot \,)_{j,k}| \leq C \langle \, \cdot \, \rangle^{-\delta} 
|G_0(0+i\,0; \, \cdot \,, \, \cdot \,)_{j,k}| \langle \, \cdot \, \rangle^{-\delta}, \quad \delta > 1/2, \; 
1 \leq j,k \leq N, 
\end{equation}
for some $C \in (0,\infty)$, satisfies
\begin{equation} 
R_{0,\delta} \in \cB_p\big([L^2(\bbR^n)]^N\big), \quad p > n.   \lb{6.30}
\end{equation} 
In a similar fashion, the integral operator $R_{0,\delta}(z)$ in $[L^2(\bbR^n)]^N$ with integral 
kernel $R_{0,\delta}(z; \, \cdot \,, \, \cdot \,)$ permitting the entrywise bound  
\begin{align} 
\begin{split}
|R_{0,\delta}(z; \, \cdot \,, \, \cdot \,)_{j,k}| \leq C \langle \, \cdot \, \rangle^{-\delta} 
|G_0(z;\, \cdot \,, \, \cdot \,)_{j,k}| \langle \, \cdot \, \rangle^{-\delta},& \\ 
z \in \ol{\bbC_+}, \; \delta > (n+1)/4, \; 1 \leq j,k \leq N,& 
\end{split}
\end{align}  
for some $C \in (0,\infty)$, satisfies 
\begin{equation} 
R_{0,\delta}(z) \in \cB_p\big([L^2(\bbR^n)]^N\big), \quad p > n, 
\; z \in \ol{\bbC_+}.     \lb{6.32}
\end{equation} 
\end{theorem}
\begin{proof}
We will apply the fact \eqref{A.5}. 

The inclusion \eqref{6.30} is immediate from \eqref{5.18} (employing the elementary estimate 
$|G_0(0;x,y)_{j,k}| \leq C |x - y|^{1-n}$, $x, y \in \bbR^n$, $x \neq y$, $1 \leq j,k \leq N$) 
and Theorem \ref{t6.5} (with $\gamma = 1/2$). 

To prove the inclusion \eqref{6.32} we again employ the estimate \eqref{5.37a}. 
An application of Theorem \ref{t6.5} to both terms in \eqref{5.37a}, then yields for the part 
where $0 \leq |x-y| \leq 1$ that $\gamma = 1/2$ and hence $\delta > 1/2$ and $p > n$. Similarly, 
for the part where $|x-y| \geq 1$ one infers $\gamma = (n+1)/4$ and hence $\delta > (n+1)/4$ and 
$p > 2n/(n+1)$, $p \geq 2$, and thus one concludes $\delta > (n+1)/4$ and $p > n$.   
\end{proof}

\begin{remark} \lb{r6.7}
Continuity of $R_{0,\delta} (\, \cdot \,)$ on $\ol{\bbC_+}$ with respect to 
the $\|\, \cdot \,\|_{\cB_p([L^2(\bbR^n)]^N)}$-norm, $p > n$, appears to be more difficult to 
prove within this complex interpolation approach. In this context the first approach described in this 
section is by far simpler to apply, but in turn it is restricted to the case $n \geq 3$. In fact, as recorded 
in Theorem \ref{t6.3}, if $\delta > (n+1)/4$, then $\langle \, \cdot \, \rangle^{-\delta} I_N$ satisfies the conditions assumed on $F_{\ell}$, $\ell =1,2$, in Theorem \ref{t6.3}, implying the fact,  
\begin{align} 
\begin{split}
& \text{For $n \geq 3$, $\delta > (n+1)/4$, $R_{0,\delta} (\, \cdot \,)$ is continuous on $\ol{\bbC_+}$} \\ 
& \quad \text{with respect to 
the $\|\, \cdot \,\|_{\cB_p([L^2(\bbR^n)]^N)}$-norm, $p > n$.} 
\end{split}
\end{align} 
Fortunately, the remaining case $n=2$ can easily be handled directly as we demonstrate next. 
\hfill $\diamond$
\end{remark} 

\begin{corollary} \lb{c6.8}
Let $n = 2$, $\delta > 3/4$, and $z_0 \in \ol{\bbC_+}$. Then $R_{0,\delta} (\, \cdot \,)$, as introduced in Theorem \ref{t6.6}, satisfies 
\begin{equation} 
[R_{0,\delta}(z_1) - R_{0,\delta}(z_2)] \in \cB_2\big([L^2(\bbR^2)]^N\big), \quad 
z_j \in \ol{\bbC_+},  \; j=1,2,   \lb{6.36}
\end{equation} 
and 
\begin{equation}
\lim_{\substack{z \to z_0 \\ z \in \ol{\bbC_+}\backslash\{z_0\}}} \|R_{0,\delta}(z) 
- R_{0,\delta}(z_0)\|_{\cB_2([L^2(\bbR^2)]^N)} = 0.
\end{equation}
\end{corollary}
\begin{proof}
Once more we will apply the fact \eqref{A.5} (for $p=2$). 

By Theorem \ref{t5.5}\,$(ii)$ it suffices to focus on $G_{0,<} (z; \, \cdot \,)$. The explicit formula 
(see \eqref{B.13z}, \eqref{B.23z}),
\begin{align}
G_0(z;x,y) &= i 4^{-1} z \, H_0^{(1)} (z |x - y|) I_N  \no \\
& \quad - 4^{-1} |x - y|^{-1} [z |x - y|] H_1^{(1)} (z |x - y|) \, 
\alpha \cdot \f{(x - y)}{|x - y|},     \lb{6.37}  \\
& \hspace*{4.15cm} z \in \ol{\bbC_+}, \; x, y \in \bbR^2, \; x \neq y,    \no 
\end{align}
 together with the $z \to 0, \, z \in \ol{\bbC_+}\backslash \{0\}$ limit \eqref{B.28z}, then permit the following conclusions: Only if $z \to 0$ ($z \in \bbC_+ \backslash \{0\}$) and/or if $|x-y| \to 0$, can $G_0(z;x,y)$ develop a singularity which then is of the form $\ln(z|x-y|)$ and $|x-y|^{-1}$. (In all other circumstances $G_0$ is continuous 
 on $\bbC_+ \times \bbR^2 \times \bbR^2$.) However, the $|x-y|^{-1}$-singularity is $z$-independent and hence drops out in differences of the form $R_{0,\delta}(z_1) - R_{0,\delta}(z_2)$, $z_j \in \ol{\bbC_+}$, $j=1,2$. Thus one can safely ignore the $|x-y|^{-1}$-singularity. Thus, this only leaves the $\ln(z_j |x-y|)$-singularity, $j=1,2$, when considering $G_{0,<}(z_1;x,y) - G_{0,<}(z_2;x,y)$. This then yields the estimate (see also \eqref{5.35} and \eqref{5.35a}),
\begin{align}
& |G_{0,<}(z;x,y)_{j,k} - G_{0,<}(z_0;x,y)_{j,k}|    \no \\
& \quad \leq \begin{cases}  
C [|z| + |z_0|] |\ln(|x-y|)| + D(z,z_0), & z, z_0 \in \ol{\bbC_+}\big\backslash\{0\}, \\
C |z| |\ln(|x-y|)| + D(z), & z \in \ol{\bbC_+}\big\backslash\{0\}, \; z_0 = 0,   
\end{cases}      \lb{6.38} \\
& \hspace*{3.7cm}  x, y \in \bbR^2, \; 0 <|x - y| \leq 1, \; 1 \leq j, k \leq N,     \no 
\end{align}
with $C \in (0,\infty)$, and $D(\, \cdot \,, z_0), D(\, \cdot \,) \in (0,\infty)$ continuous and locally 
bounded on $\ol{\bbC_+}$. The logarithmic-type integral kernel in \eqref{6.37} can now be handled 
as in \cite[Proposition~7.1.17]{Ya10} (upon multiplying $R_0(z)$ by a factor of $z$ if $z_0=0$, and choosing $|z|=1$ in equation (7.1.25) in \cite[p.~272]{Ya10}), implying the asserted Hilbert--Schmidt property. Alternatively, one can use the very rough estimate (for some $c_0 \in (0,\infty)$) 
\begin{equation}
|\ln(|x-y|)|^2 \leq c_0 |x-y|^{-1}, \quad 0 < |x-y| \leq 1, 
\end{equation}
and apply the Sobolev inequality in the form of \eqref{6.18} with $n=2$, $\lambda = 1$, $r=s=4/3$, recalling that 
$\langle \dott \rangle^{- 2 \delta} \in L^{4/3}(\bbR^2)$ if $\delta > 3/4$.
\end{proof}

Combining Theorems \ref{t6.3}, \ref{t6.6}, Remark \ref{r6.7}, and Corollary \ref{c6.8}, we finally 
summarize  the principal results of this section as follows:

\begin{theorem} \lb{t6.9} 
Let $n \in \bbN$, $n \geq 2$ and consider the integral operator 
$R_{0, \delta}(z)$ in $[L^2(\bbR^n)]^N$ with integral 
kernel $R_{0,\delta}(z; \, \cdot \,, \, \cdot \,)$ permitting the entrywise bound  
\begin{align}
\begin{split}
|R_{0,\delta}(z; \, \cdot \,, \, \cdot \,)_{j,k}| \leq C \langle \, \cdot \, \rangle^{-\delta} 
|G_0(z;\, \cdot \,, \, \cdot \,)_{j,k}| \langle \, \cdot \, \rangle^{-\delta},& \\ 
z \in \ol{\bbC_+}, \; \delta > (n+1)/4, \; 1 \leq j, k \leq N,& 
\end{split}
\end{align}  
for some $C \in (0,\infty)$. Then $R_{0,\delta}(z)$ satisfies 
\begin{equation} 
R_{0,\delta}(z) \in \cB_p\big([L^2(\bbR^n)]^N\big), \quad 
p > n, \; z \in \ol{\bbC_+}.     \lb{6.40}
\end{equation} 
Moreover, if $n\geq 3$, $\delta > (n+1)/4$, then $R_{0,\delta} (\, \cdot \,)$ is continuous on 
$\ol{\bbC_+}$ with respect to the 
$\|\, \cdot \,\|_{\cB_p([L^2(\bbR^n)]^N)}$-norm for $p > n$. Finally, if $n=2$, $\delta > 3/4$, then
\begin{equation} 
[R_{0,\delta}(z_1) - R_{0,\delta}(z_2)] \in \cB_2\big([L^2(\bbR^2)]^N\big), \quad 
z_j \in \ol{\bbC_+},  \; j=1,2,   \lb{6.41}
\end{equation} 
and 
\begin{equation}
\lim_{\substack{z \to z_0 \\ z \in \ol{\bbC_+}\backslash\{z_0\}}} \|R_{0,\delta}(z) 
- R_{0,\delta}(z_0)\|_{\cB_2([L^2(\bbR^2)]^N)} = 0.     \lb{6.42}
\end{equation} 
\end{theorem}

\section{Powers of Resolvents and Trace Ideals} \lb{s7}

We now introduce the following considerably strengthened set of assumptions on the short-range potential $V$:

\begin{hypothesis} \lb{h7.1}
Let $n \in \bbN$ and suppose that $V$ satisfies for some constant $C \in (0, \infty)$ and $\rho\in (n,\infty)$, 
\begin{equation}
V \in [L^{\infty} (\bbR^n)]^{N \times N}, \quad 
|V_{\ell,m}(x)| \leq C \langle x \rangle^{- \rho} \, \text{ for a.e.~$x \in \bbR^n$, $1 \leq \ell,m \leq N$.}    
\end{equation}
\end{hypothesis}

Given Hypothesis \ref{h7.1}, the principal purpose of this section is to prove that for $k \geq n$,
\begin{equation} \lb{7.1}
\big[(H - zI_{[L^2(\bbR^n)]^N})^{- k} - (H_0 - zI_{[L^2(\bbR^n)]^N})^{- k} \big] \in \cB_1\big([L^2(\bbR^n) ]^N\big),\quad z\in \bbC\backslash\bbR.
\end{equation}

Here $H = H_0 + V$ is defined according to \eqref{4.2}, but we do not assume self-adjointness of the $N \times N$ matrix $V$ in this section.

The following arguments are straightforward generalizations of the arguments in \cite{Ya05} in the 
three-dimensional context $n=3$. We start with a study of $H_0$:

\begin{lemma} \lb{l7.2}
Let $r\in (0,\infty)$, $k\in \bbN$, and define $p(r,k):=n/\min\{r,k\}$.  If $p>p(r,k)$, $p\geq 1$, then
\begin{equation}
\langle \,\cdot\,\rangle^{-r}(H_0 - zI_{[L^2(\bbR^n)]^N})^{-k}\in \cB_p\big([L^2(\bbR^n)]^N\big),\quad z\in \bbC\backslash\bbR.
\end{equation}
In particular, choosing $r = n + \varepsilon$ for some $\varepsilon > 0$, $k = n + 1$, then 
$p(n + \varepsilon,n+1) < 1$, and hence 
\begin{equation}
\langle \,\cdot\,\rangle^{-n - \varepsilon}(H_0 - zI_{[L^2(\bbR^n)]^N})^{-n - 1}\in \cB_1\big([L^2(\bbR^n)]^N\big), 
\quad z\in \bbC\backslash\bbR.    \lb{7.4} 
\end{equation}
\end{lemma}
\begin{proof}
Since 
\begin{align}
\big(H_0-zI_{[L_2(\bbR^n)]^N}\big)^{-k}&=\big(H_0+zI_{[L_2(\bbR^n)]^N}\big)^{k}\big(H_0^2-z^2I_{[L_2(\bbR^n)]^N}\big)^{-k}\no\\
&=\big(h_0-z^2I_{L^2(\bbR^n)}\big)^{-k/2}I_N \no\\
&\quad \times \big(H_0+zI_{[L_2(\bbR^n)]^N}\big)^{k}\big(H_0^2-z^2I_{[L_2(\bbR^n)]^N}\big)^{-k/2},\\
&\hspace*{6.15cm}z\in \bbC\backslash\bbR,\no
\end{align}
and the operator $\big(H_0+zI_{[L_2(\bbR^n)]^N}\big)^{k}\big(H_0^2-z^2I_{[L_2(\bbR^n)]^N}\big)^{-k/2}$ is bounded, it is sufficient to prove the assertion for the operator $\langle \dott\rangle^{-r}\big(h_0-z^2I_{L^2(\bbR^n)}\big)^{-k/2}$. The latter follows from  \cite[Lemma 4.3,  p.~145]{Ya10}.
\end{proof}

Turning from $H_0$ to $H = H_0 +V$ then yields the following result. 

\begin{lemma}\lb{l7.3}
Assume that $V \in [L^{\infty} (\bbR^n)]^{N \times N}$, let $r\in (0,\infty)$, $k\in \bbN$, and define $p(r,k):=n/\min\{r,k\}$.  If $p>p(r,k)$, $p\geq 1$, then
\begin{equation} \lb{7.13}
\langle \,\cdot\,\rangle^{-r}(H - zI_{[L^2(\bbR^n)]^N})^{-k}\in \cB_p\big([L^2(\bbR^n)]^N\big),\quad z\in \bbC\backslash\bbR.
\end{equation}
In particular, choosing $r = n + \varepsilon$ for some $\varepsilon > 0$, $k = n + 1$, then 
$p(n + \varepsilon,n+1) < 1$, and hence 
\begin{equation}
\langle \,\cdot\,\rangle^{-n - \varepsilon}(H - zI_{[L^2(\bbR^n)]^N})^{-n - 1}\in \cB_1\big([L^2(\bbR^n)]^N\big), 
\quad z\in \bbC\backslash\bbR.
\end{equation}
\end{lemma}
\begin{proof}
Let $z\in \bbC\backslash\bbR$ and $r\in (0,\infty)$.  The proof employs induction on $k\in \bbN$.  In the base case, $k=1$, one writes
\begin{align}  
\begin{split} \lb{7.14}
& \langle \,\cdot\,\rangle^{-r}(H - zI_{[L^2(\bbR^n)]^N})^{-1} \\
& \quad = \big[\langle \,\cdot\,\rangle^{-r}(H_0 - zI_{[L^2(\bbR^n)]^N})^{-1} 
\big]\big[(H_0 - zI_{[L^2(\bbR^n)]^N})(H - zI_{[L^2(\bbR^n)]^N})^{-1} \big].
\end{split}
\end{align}
The first factor on the right-hand side in \eqref{7.14} belongs to $\cB_p\big([L^2(\bbR^n)]^N\big)$ for $p>p(r,1)$, $p\geq 1$ by Lemma \ref{l7.2}.  Since the second factor on the right-hand side in \eqref{7.14} is a bounded operator, \eqref{7.13} holds with $k=1$.

Suppose that \eqref{7.13} holds for $k\in \bbN$.  Multiplying throughout the commutator identity
\begin{equation}
(H - zI_{[L^2(\bbR^n)]^N})\langle\,\cdot\,\rangle^{-r} - \langle\,\cdot\,\rangle^{-r}(H - zI_{[L^2(\bbR^n)]^N}) 
= [H_0,\langle\,\cdot\,\rangle^{-r}]
\end{equation}
from the left by $(H - zI_{[L^2(\bbR^n)]^N})^{-1}$ and right by $(H - zI_{[L^2(\bbR^n)]^N})^{-k-1}$, one obtains
\begin{align}
& \langle \,\cdot\,\rangle^{-r}(H - zI_{[L^2(\bbR^n)]^N})^{-k-1} = (H - zI_{[L^2(\bbR^n)]^N})^{-1} 
\langle\,\cdot\,\rangle^{-r} (H - zI_{[L^2(\bbR^n)]^N})^{-k}       \no \\
&\quad + (H - zI_{[L^2(\bbR^n)]^N})^{-1}[H_0,\langle\,\cdot\,\rangle^{-r}](H - zI_{[L^2(\bbR^n)]^N})^{-k-1}.\lb{7.10c}
\end{align}
One has 
\begin{align}
& (H - zI_{[L^2(\bbR^n)]^N})^{-1}\langle\,\cdot\,\rangle^{-r} (H - zI_{[L^2(\bbR^n)]^N})^{-k}     \lb{7.17} \\
& \quad = \big[(H - zI_{[L^2(\bbR^n)]^N})^{-1}\langle\,\cdot\,\rangle^{-r(k+1)^{-1}}\big]  
\big[ \langle\,\cdot\,\rangle^{-kr(k+1)^{-1}}(H - zI_{[L^2(\bbR^n)]^N})^{-k}\big].  \no
\end{align}
Now, by the base case,
\begin{equation}
(H - zI_{[L^2(\bbR^n)]^N})^{-1}\langle\,\cdot\,\rangle^{-r(k+1)^{-1}}\in \cB_p\big([L^2(\bbR^n)]^N\big),\quad p>p\big(r(k+1)^{-1},1\big),
\end{equation}
and by the induction step
\begin{equation}
\langle\,\cdot\,\rangle^{-kr(k+1)^{-1}}(H - zI_{[L^2(\bbR^n)]^N})^{-k}\in \cB_p\big([L^2(\bbR^n)]^N\big),\quad p>p\big(kr(k+1)^{-1},k\big).
\end{equation}
Therefore, the product on the right-hand side in \eqref{7.17} belongs to the trace ideal $\cB_p\big([L^2(\bbR^n)]^N\big)$ for $p\geq 1$, with
\begin{equation} \lb{7.20}
p^{-1} < p\big(r(k+1)^{-1},1\big)^{-1} + p\big(kr(k+1)^{-1},k\big)^{-1}.
\end{equation}
To compute the right-hand side of \eqref{7.20}, one distinguishes the two possible cases: 
$(i)$ $r(k+1)^{-1}<1$ or $(ii)$ $r(k+1)^{-1}\geq 1$.

In case $(i)$, $r<k+1$, and
\begin{align}
p\big(kr(k+1)^{-1},k\big) &= n(k+1)r^{-1},\\
p\big(kr(k+1)^{-1},k\big)^{-1} &= n(k+1)r^{-1}k^{-1}.
\end{align}
Hence, the right-hand side of \eqref{7.20} equals
\begin{equation}
n^{-1}(k+1)^{-1}r + n^{-1}(k+1)^{-1}kr = n^{-1}r = n^{-1}\min\{r,k+1\} = p(r,k+1)^{-1},
\end{equation}
so the right-hand side in \eqref{7.17} belongs to $\cB_p\big([L^2(\bbR^n)]^N\big)$ for all indices $p>p(r,k+1)$.

In case $(ii)$, $r \geq k+1$, and
\begin{align}
p\big(kr(k+1)^{-1},k\big) &= n,\\
p\big(kr(k+1)^{-1},k\big)^{-1} &= nk^{-1}.
\end{align}
Hence the right-hand side of \eqref{7.20} equals
\begin{equation}
n^{-1}(k+1) = p(r,k+1)^{-1},
\end{equation}
so the right-hand side of \eqref{7.17}, and hence the first term on the right-hand side in \eqref{7.10c}, belongs to $\cB_p\big([L^2(\bbR^n)]^N\big)$ for all indices $p>p(r,k+1)$.  To treat the second term in \eqref{7.10c}, one uses
\begin{equation}\lb{7.21c}
[H_0,\langle\,\cdot\,\rangle^{-r}] = V_0,
\end{equation}
where
\begin{equation}
V_0(x) = -r \langle x\rangle^{-(r+2)}(\alpha \cdot x),\quad x\in \bbR^n,
\end{equation}
so that
\begin{equation}\lb{7.23c}
\|V_0(x)\|_{\cB(\bbC^N)} \leq C\langle x\rangle^{-(r+1)},\quad x\in \bbR^n,
\end{equation}
for an $x$-independent constant $C>0$.  Thus, the second term on the right-hand side in \eqref{7.10c} belongs to $\cB_p\big([L^2(\bbR^n)]^N\big)$ for all indices $p>p(r,k+1)$ by the same argument used to treat the first term.
\end{proof}

Given these preparations, the principal result of this section reads as follows. 

\begin{theorem} \lb{t7.4} 
Let $k \in \bbN$ with $k \geq n$ and suppose that $V$ satisfies Hypothesis \ref{h7.1}. Then
\begin{equation} \lb{7.25}
\big[(H - zI_{[L^2(\bbR^n)]^N})^{- k} - (H_0 - zI_{[L^2(\bbR^n)]^N})^{- k} \big] \in \cB_1\big([L^2(\bbR^n) ]^N\big), 
\quad z\in \bbC\backslash\bbR.
\end{equation}
\end{theorem}
\begin{proof}
Let $k \geq n$.  By the first resolvent equation,
\begin{align} 
\begin{split} \lb{7.26}
&(H-zI_{[L^2(\bbR^n)]^N})^{-1} - (H_0 - zI_{[L^2(\bbR^n)]^N})^{-1}    \\
& \quad = - (H - zI_{[L^2(\bbR^n)]^N})^{-1}V(H_0 - zI_{[L^2(\bbR^n)]^N})^{-1},\quad z\in \bbC\backslash \bbR.
\end{split} 
\end{align}
Differentiation of \eqref{7.26} with respect to $z$ yields
\begin{align}
\begin{split} 
& (H - zI_{[L^2(\bbR^n)]^N})^{- k} - (H_0 - zI_{[L^2(\bbR^n)]^N})^{- k}   \\
& \quad = - \sum_{j=1}^k (H - zI_{[L^2(\bbR^n)]^N})^{-j} 
V(H_0 - zI_{[L^2(\bbR^n)]^N})^{j-k-1},\quad z\in \bbC\backslash \bbR.\lb{7.27}
\end{split} 
\end{align}
From this point on, let $z\in \bbC\backslash \bbR$ be fixed and write
\begin{align}
&(H - zI_{[L^2(\bbR^n)]^N})^{-j}V(H_0 - zI_{[L^2(\bbR^n)]^N})^{j-k-1}\no\\
&\quad = \big[(H - zI_{[L^2(\bbR^n)]^N})^{-j}\langle x\rangle^{-j\rho (k+1)^{-1}}\big] 
\big[\langle x\rangle^{\rho}V\big]    \lb{7.28} \\
&\qquad \times \big[\langle x\rangle^{-(k+1-j)\rho(k+1)^{-1}}(H_0 - zI_{[L^2(\bbR^n)]^N})^{j-k-1}\big],  
\quad j\in \bbN,\; 1\leq j\leq k.   \no 
\end{align}
By Lemma \ref{l7.3}, for a fixed $j\in \bbN$, $1\leq j\leq k$,
\begin{equation}
(H - zI_{[L^2(\bbR^n)]^N})^{-j}\langle x\rangle^{-j\rho (k+1)^{-1}} \in \cB_p\big([L^2(\bbR^n)]^N\big),\quad p>p\big(j\rho(k+1)^{-1},j\big),
\end{equation}
and by Lemma \ref{l7.2}, 
\begin{align} 
\begin{split} 
\langle x\rangle^{-(k+1-j)\rho(k+1)^{-1}}(H_0 - zI_{[L^2(\bbR^n)]^N})^{j-k-1} \in \cB_p\big([L^2(\bbR^n)]^N\big),& \\
p>p\big((k+1-j)\rho(k+1)^{-1},k+1-j\big).&
\end{split} 
\end{align}
One distinguishes the two possible cases: $(i)$ $\rho(k+1)^{-1}<1$, or $(ii)$ $\rho(k+1)^{-1}\geq 1$.

In case $(i)$ with $\rho(k+1)^{-1}<1$, one computes
\begin{align}
p\big((k+1-j)\rho(k+1)^{-1},k+1-j\big) &= \frac{n}{(k+1-j)\rho(k+1)^{-1}},\\
p\big(j\rho(k+1)^{-1},j\big) &= \frac{n}{j\rho(k+1)^{-1}},
\end{align}
so that
\begin{align}
p\big((k+1-j)\rho(k+1)^{-1},k+1-j\big)^{-1} + p\big(j\rho(k+1)^{-1},j\big)^{-1}= \frac{\rho}{n} > 1.
\end{align}
Hence, the right-hand side of \eqref{7.28} belongs to $\cB_1\big([L^2(\bbR^n)]^N\big)$.

In case $(ii)$ with $\rho(k+1)^{-1}\geq 1$, one computes
\begin{align}
p\big((k+1-j)\rho(k+1)^{-1},k+1-j\big) &= \frac{n}{k+1-j} \\
p\big(j\rho(k+1)^{-1},j\big) &= \frac{n}{j},
\end{align}
so that
\begin{align}
\begin{split} 
&p\big((k+1-j)\rho(k+1)^{-1},k+1-j\big)^{-1} + p\big(j\rho(k+1)^{-1},j\big)^{-1}   \\
&\quad= \frac{k+1}{n} = \frac{k}{n} +\frac{1}{n} \geq 1 + \frac{1}{n} > 1.
\end{split} 
\end{align}
Hence, the right-hand side of \eqref{7.28} belongs to $\cB_1\big([L^2(\bbR^n)]^N\big)$.

In either case, the right-hand side of \eqref{7.28} belongs to $\cB_1\big([L^2(\bbR^n)]^N\big)$.  
Since $j\in \bbN$, $1\leq j\leq k$, was arbitrary, it follows that every term in the summation on the right-hand side in \eqref{7.27} belongs to $\cB_1\big([L^2(\bbR^n)]^N\big)$, and then \eqref{7.25} follows from the vector space properties of the trace class.
\end{proof}

We conclude this section by recalling a well-known result:

\begin{lemma} \lb {l7.5}
Suppose $p > n/\min(\tau, 2 \kappa)$, $p \geq 1$, with $\tau > 0$, $\kappa > 0$. Then
\begin{equation}
\langle \, \cdot \, \rangle^{- \tau} (h_0 + I_{L^2(\bbR^n)})^{- \kappa} \in \cB_p\big(L^2(\bbR^n)\big).  \lb{7.40}
\end{equation}
In particular, if $V$ satisfies Hypothesis \ref{h7.1} and $\kappa > n/2$,
\begin{equation}
V \big(H_0^2 + I_{[L^2(\bbR^n)]^N}\big)^{- \kappa} \in \cB_1\big([L^2(\bbR^n)]^N\big).    \lb{7.41}
\end{equation} 
\end{lemma}
\begin{proof}
While \eqref{7.40} is a special case of \cite[Proposition~3.1.5, Lemma~3.4.3]{Ya10} (see also 
\cite{GN20}, \cite[Ch.~4]{Si05}), \eqref{7.41} follows from combining \eqref{2.7} and \eqref{7.40}.
\end{proof}

\section{The Spectral Shift Function: Abstract Facts} \lb{s8}

The significance of Theorem \ref{t7.4} is that the trace class condition \eqref{7.25} permits one to define a spectral shift function for the pair $(H,H_0)$.  To make this precise, we introduce the class of 
functions $\gF_m(\bbR)$, $m \in \bbN$, by
\begin{align}
& \gF_m(\bbR) := \big\{f \in C^2(\bbR) \, \big| \, 
f^{(\ell)} \in L^{\infty}(\bbR); \text{ there exists } 
\varepsilon >0 \text{ and } f_0 = f_0(f) \in \bbC    \no \\
& \quad  \text{ such that } 
\big(d^{\ell}/d \lambda^{\ell}\big)\big[f(\lambda) - f_0 \lambda^{-m}\big] \underset{|\lambda|\to \infty}{=} 
\Oh\big(|\lambda|^{- \ell - m - \varepsilon}\big), \; \ell = 0,1,2 \big\}.     \lb{8.1}
\end{align} 
(It is implied that $f_0 = f_0(f)$ is the same as $\lambda \to \pm \infty$.) One observes that 
$C_0^{\infty}(\mathbb{R}) \subset \mathfrak F_m(\mathbb{R})$, $m \in \bbN$.

In \cite{Kr62}, M.~Krein established the existence of a spectral shift function corresponding to any pair of resolvent comparable self-adjoint operators.  Specifically, Krein proved that if $S_0$ and $S$ are self-adjoint and satisfy
\begin{equation}
\big[(S - zI_{\cH})^{-1} - (S_0 - zI_{\cH})^{-1}\big] \in \cB_1(\cH)   \lb{8.2}
\end{equation}
for some (and, hence, for all) $z\in \bbC\backslash\bbR$, then
\begin{equation}
[f(S)-f(S_0)] \in \cB_1(\cH),\quad f\in \gF_1(\bbR),
\end{equation}
and there exists a real-valued spectral shift function
\begin{equation}
\xi(\,\cdot\,;S,S_0)\in L^1\big(\bbR,(1+|\lambda|)^{-2}\, d\lambda\big)
\end{equation}
so that
\begin{equation}
\tr_{\cH}(f(S)-f(S_0)) = \int_{\bbR}d\lambda\, \xi(\lambda;S,S_0)f'(\lambda),\quad f\in \gF_1(\bbR). 
\lb{eq_classical_tace_formula}
\end{equation}

One limitation to Krein's theory is that the condition \eqref{8.2} generally does not hold for Schr\"odinger operators in dimensions $n\geq 4$.  Similar difficulties are encountered for the polyharmonic operator (cf. \cite[\S 3.4]{Ya10}) and the Dirac operator (cf.~\cite[\S 3.5.3]{Ya10} and Theorem \ref{t7.4}).  In these cases, only the difference of higher powers of the resolvents belongs to the trace class (cf. \cite[Remark 3.3.3]{Ya10}).  Using the theory of double operator integrals, Yafaev \cite{Ya05} proved the existence of a spectral shift function under the weaker assumption that the difference of an odd power of the resolvents belongs to the trace class.

\begin{theorem}[{\cite[Theorem 2.2]{Ya05}}] \lb{t8.1}
Let $m \in \bbN$, $m$ odd, and suppose that $S_0$ and $S$ are self-adjoint operators in $\cH$ with
\begin{equation} \lb{8.6}
\big[(S - zI_{\cH})^{-m} - (S_0 - zI_{\cH})^{-m}\big] \in \cB_1(\cH),\quad z\in \bbC\backslash \bbR.
\end{equation}
Then
\begin{equation}
[f(S) - f(S_0)] \in \cB_1(\cH),\quad f\in \gF_m(\bbR),
\end{equation}
and there exists a function
\begin{equation} \lb{8.8}
\xi(\,\cdot\,;S,S_0)\in L^1\big(\bbR; (1+|\lambda|)^{-m-1}\, d\lambda\big)
\end{equation}
such that the following trace formula holds, 
\begin{equation}
\tr_{\cH}(f(S) - f(S_0)) = \int_{\bbR}d\lambda\, \xi(\lambda;S,S_0)f'(\lambda),\quad f\in \gF_m(\bbR).
\end{equation}
In particular, one has 
\begin{equation} \lb{8.8A}
\tr_{\cH}\big((S-zI_{\cH})^{-m} - (S_0-zI_{\cH})^{-m}\big) = -m \int_{\bbR} 
\frac{\xi(\lambda;S,S_0) d\lambda}{(\lambda - z)^{m+1}},   \quad z\in \bbC\backslash \bbR.
\end{equation}
\end{theorem}

\begin{remark} \lb{r8.2} 
The above theorem, together with Theorem \ref{t7.4} guarantees that for Dirac operators $H$ and $H_0$ 
in $[L^2(\bbR^n)]^N$ the spectral shift function $\xi(\,\cdot\,; H,H_0)$ exists. However, for the representation of the spectral shift function in terms of a regularized perturbation determinant it is desirable to take the regularized determinant $\Det_{\cH,p}((H-zI_\cH)(H_0-zI_\cH)^{-1})$ with $p$ equal to $n+1$. Theorem \ref{t8.1} permits this in odd space dimensions $n$. In even space dimensions Theorem \ref{t8.1} does not guarantee the appropriate integrability of the spectral shift function $\xi(\,\cdot\,; H,H_0)$ and so one would be forced to consider a regularized determinant with $p=n+2$. To avoid this drawback, we prove that under a certain stronger condition (satisfied for Dirac operators $H$ and $H_0$ considered in Section \ref{s3}) an analogue of Theorem \ref{t8.1} holds for any $m\in\bbN$. 
\hfill $\diamond$
\end{remark}

\begin{hypothesis}\lb{hyp_even_ssf}
Let $m \in \bbN$ and assume that $S$ and $S_0$ are self-adjoint operators in $\cH$ 
with a common dense domain, such that 
\begin{equation}
(S-S_0) \in \cB(\cH), 
\end{equation}
and for some $0<\varepsilon<1/2$, 
\begin{equation} \lb{hyp_even_epsilon}
(S-S_0)(S_0^2+I_\cH)^{- (m/2) - \varepsilon}\in \cB_1(\cH).
\end{equation}
\end{hypothesis}

\begin{remark} \lb{r8.4} $(i)$ Assuming Hypothesis \ref{hyp_even_ssf}, it follows that 
\begin{equation}
(S-S_0)(S_0-zI_\cH)^{-m-1}\in \cB_1(\cH). 
\end{equation} 
Since $(S-S_0) \in \cB(\cH)$, it follows from the three line theorem that 
\begin{equation} \lb{8.12a}
(S-S_0)(S_0 - zI_{\cH})^{-j}\in \cB_{(m+1)/j}(\cH),\quad j\in \bbN,\; 1\leq j\leq m+1,
\end{equation}
Furthermore, another application of the three line theorem implies  
\begin{equation} \lb{8.12b}
(S_0 - zI_{\cH})^{-j_1}(S-S_0)(S_0 - zI_{\cH})^{-j_2}\in \cB_{(m+1)/(j_1+j_2)}(\cH)
\end{equation}
for all $j_1, j_2\in \bbN,$ with $1\leq j_1+j_2\leq m+1$. \\[1mm] 
$(ii)$ For the proof of Theorem \ref{t8.2} we will only need \eqref{hyp_even_epsilon} and \eqref{8.12a}. We assumed boundedness of $S-S_0$ only to get \eqref{8.12a} as a consequence of \eqref{hyp_even_epsilon}. It is possible to go beyond this boundedness assumption, but we omit further details at this point. \\[1mm]
$(iii)$ The inclusion \eqref{7.41} shows that assumption \eqref{hyp_even_epsilon} holds with $m=n$ for the pair of Dirac operators $(H,H_0)$ as long as $V$ satisfies Hypothesis \ref{h7.1}. 
${}$ \hfill $\diamond$ 
\end{remark}

The following result appeared in \cite[Theorem~2.6]{CLPS20}.

\begin{theorem}\lb{t8.6}
Assume Hypothesis \ref{hyp_even_ssf}. For any $j=1,\dots,m$, one has 
\begin{equation}
\left(S-zI_\cH\right)^{-j} - \left(S_0 -zI_\cH \right)^{-j}\in \cB_{(m+1)/(j+1)}(\cH). 
\end{equation} 
\end{theorem} 

\begin{lemma}\lb{lem_weighted_resol_diff}
Assume Hypothesis \ref{hyp_even_ssf}. For  any $j=1,\dots, m$, and $z\in\bbC \backslash \bbR$,  
one has  
\begin{equation} 
\big[(S-zI_\cH)^{-j}-(S_0-zI_\cH)^{-j}\big](S_0+zI_\cH)^{-m+j}\in\cB_1(\cH).
\end{equation} 
\end{lemma}
\begin{proof}
We prove the claim by induction on $j$. Let $j=1$. Using the resolvent identity twice one writes
\begin{align}\lb{tech_lemma_eq1}
\big[(S&-zI_\cH)^{-1}-(S_0-zI_\cH)^{-1}\big](S_0+zI_\cH)^{-m+1}   \no \\
&=-(S-zI_\cH)^{-1}(S-S_0)(S_0-zI_\cH)^{-1} (S_0+zI_\cH)^{-m+1}    \no \\
&=(S-zI_\cH)^{-1}(S-S_0)(S_0-zI_\cH)^{-1}(S-S_0)(S_0-zI_\cH)^{-1} (S_0+zI_\cH)^{-m+1}   \no \\
&\quad -(S_0-zI_\cH)^{-1}(S-S_0)(S_0-zI_\cH)^{-1} (S_0+zI_\cH)^{-m+1}.
\end{align}
By \eqref{8.12b} one obtains 
\begin{equation} 
(S_0-zI_\cH)^{-1}(S-S_0)(S_0-zI_\cH)^{-m}\in\cB_1(\cH),
\end{equation} 
and therefore the second term on the right-hand side of \eqref{tech_lemma_eq1} is a trace-class operator.
By \eqref{8.12a}, 
\begin{equation} 
(S-S_0)(S_0-zI_\cH)^{-1}\in \cB_{m+1}(\cH),\quad (S-S_0)(S_0-zI_\cH)^{-m}\in\cB_{(m+1)/m}(\cH),
\end{equation} 
guaranteeing that the first term on the right-hand side of \eqref{tech_lemma_eq1} is also a trace-class operator. Thus, one concludes that 
\begin{equation} 
\big[(S-zI_\cH)^{-1}-(S_0-zI_\cH)^{-1}\big](S_0+zI_\cH)^{-m+1}\in \cB_1(\cH),
\end{equation} 
proving the first induction step. 

Next, suppose that 
\begin{equation} 
\big[(S-zI_\cH)^{-j}-(S_0-zI_\cH)^{-j}\big](S_0+zI_\cH)^{-m+j}\in\cB_1(\cH)
\end{equation} 
for some $j=1,\dots m-1$. Writing 
\begin{align}\lb{eq_even_ssf_ugly_I_II}
\big[(S&-zI_\cH)^{-j-1}-(S_0-zI_\cH)^{-j-1}\big](S_0+zI_\cH)^{-m+j+1}   \no \\
&=\big[(S-zI_\cH)^{-j}-(S_0-zI_\cH)^{-j}\big](S-zI_\cH)^{-1}(S_0+zI_\cH)^{-m+j+1}    \no \\
&\quad+ (S_0-zI_\cH)^{-j}\big[(S-zI_\cH)^{-1}-(S_0-zI_\cH)^{-1}\big](S_0+zI_\cH)^{-m+j+1}     \\
&\quad := (I) + (II),   \no 
\end{align}
we will treat the terms $(I)$ and $(II)$ separately in the following. 

For $(I)$ on the right-hand side of \eqref{eq_even_ssf_ugly_I_II} one gets 
\begin{align}
&\big[(S-zI_\cH)^{-j}-(S_0-zI_\cH)^{-j}\big](S-zI_\cH)^{-1}(S_0+zI_\cH)^{-m+j+1}  \no \\
&\quad =\big[(S-zI_\cH)^{-j}-(S_0-zI_\cH)^{-j}\big]    \no \\
&\qquad \quad \times\big[(S-zI_\cH)^{-1}-(S_0-zI_\cH)^{-1}\big](S_0+zI_\cH)^{-m+j+1}   \no \\
&\qquad +\big[(S-zI_\cH)^{-j}-(S_0-zI_\cH)^{-j}\big](S_0-zI_\cH)^{-1}(S_0+zI_\cH)^{-m+j+1}    \no \\
&\quad = - \big[(S-zI_\cH)^{-j}-(S_0-zI_\cH)^{-j}\big](S-zI_\cH)^{-1}    \no \\
&\qquad \quad \times(S-S_0)(S_0-zI_\cH)^{-1}(S_0+zI_\cH)^{-m+j+1}     \\
&\qquad +\big[(S-zI_\cH)^{-j}-(S_0-zI_\cH)^{-j}\big](S_0-zI_\cH)^{-1}(S_0+zI_\cH)^{-m+j+1}.   \no 
\end{align}
By the induction hypothesis one concludes that 
\begin{equation} 
\big[(S-zI_\cH)^{-j}-(S_0-zI_\cH)^{-j}\big](S_0+zI_\cH)^{-m+j} \in \cB_1(\cH),
\end{equation} 
and therefore also 
\begin{equation} 
\big[(S-zI_\cH)^{-j}-(S_0-zI_\cH)^{-j}\big](S_0-zI_\cH)^{-1}(S_0+zI_\cH)^{-m+j+1}  \in \cB_1(\cH).
\end{equation} 
By Theorem \ref{t8.6} one obtains 
\begin{equation} 
\big[(S-zI_\cH)^{-j}-(S_0-zI_\cH)^{-j}\big]\in\cB_{(m+1)/(j+1)}(\cH),
\end{equation} 
and by \eqref{8.12a},  
\begin{equation} 
(S-S_0)(S_0+zI_\cH)^{-m+j}\in \cB_{(m+1)/(m-j)}(\cH).
\end{equation} 
Therefore, 
\begin{align}
& \big[(S-zI_\cH)^{-j}-(S_0-zI_\cH)^{-j}\big](S-zI_\cH)^{-1}  (S-S_0)(S_0-zI_\cH)^{-1}   \no \\
&\quad\times (S_0+zI_\cH)^{-m+j+1} \in \cB_{(m+1)/(j+1)}(\cH)\cdot  \cB_{(m+1)/(m-j)}(\cH)\subset \cB_1(\cH), \\
& \hspace*{8.68cm} 1 \leq j \leq m-1.  \no 
\end{align}
Thus, $(I) \in \cB_1(\cH)$. 

To show that also $(II)$ on the right-hand side of \eqref{eq_even_ssf_ugly_I_II} is a trace-class operator, one writes 
\begin{align}
& (S_0-zI_\cH)^{-j}\big[(S-zI_\cH)^{-1}-(S_0-zI_\cH)^{-1}\big](S_0+zI_\cH)^{-m+j+1}   \no \\
& \quad =-(S_0-zI_\cH)^{-j}(S-zI_\cH)^{-1}(S-S_0)(S_0-zI_\cH)^{-1}(S_0+zI_\cH)^{-m+j+1}   \no \\
& \quad =-(S_0-zI_\cH)^{-j}(S_0-zI_\cH)^{-1}(S-S_0)(S-zI_\cH)^{-1}    \no \\
&\qquad \quad \times(S-S_0)(S_0-zI_\cH)^{-1}(S_0+zI_\cH)^{-m+j+1}    \\
&\qquad+(S_0-zI_\cH)^{-j}(S_0-zI_\cH)^{-1}(S-S_0)(S_0-zI_\cH)^{-1}(S_0+zI_\cH)^{-m+j+1}. \no
\end{align}
By \eqref{8.12a} one infers 
\begin{align} 
\begin{split} 
& (S_0-zI_\cH)^{-j-1}(S-S_0)\in \cB_{(m+1)/(j+1)}(\cH),    \\
& (S-S_0)(S_0+zI_\cH)^{-m+j}\in\cB_{(m+1)/(m-j)}(\cH),
\end{split} 
\end{align} 
and hence, 
\begin{align} 
\begin{split} 
& (S_0-zI_\cH)^{-j}(S_0-zI_\cH)^{-1}(S-S_0)(S-zI_\cH)^{-1}(S-S_0)(S_0-zI_\cH)^{-1}    \\
& \quad \times (S_0+zI_\cH)^{-m+j+1} \in \cB_1(\cH). 
\end{split}
\end{align} 
Furthermore, by \eqref{8.12b}, 
\begin{equation} 
(S_0+zI_\cH)^{-j-1}(S-S_0)(S_0+zI_\cH)^{-m+j}\in \cB_1(\cH), \quad 1 \leq j \leq m - 1.
\end{equation} 
Thus, also $(II)$ is a trace-class operator. Combining this with the fact that $(I) \in \cB_1(\cH)$ and referring to \eqref{eq_even_ssf_ugly_I_II}, one concludes that 
\begin{equation} 
\big[(S-zI_\cH)^{-j-1}-(S_0-zI_\cH)^{-j-1}\big](S_0+zI_\cH)^{-m+j+1}\in\cB_1(\cH).
\end{equation} 
\end{proof}

From this point on we assume Hypothesis \ref{hyp_even_ssf} for even $m = 2k$ for the remainder of this section. Introducing the function 
\begin{equation}\lb{def_even_phi}
\phi(t)=t \big(1+t^2\big)^{\frac{m-1}{2}}=t \big(1+t^2\big)^{k-(1/2)}, \quad t\in\bbR,
\end{equation}
we aim at proving that 
\begin{equation}\lb{even_resol_comp}
  \big[(\phi(S)+iI_{\cH})^{-1}-(\phi(S_0)+iI_{\cH})^{-1}\big] \in \cB_1(\cH), 
  \end{equation}
guaranteeing that the spectral shift function $\xi(\,\cdot\,; \phi(S), \psi(S_0))$ is well-defined. Since 
\begin{equation} 
\phi'(t)= \big(1+t^2\big)^{\frac{m-3}{2}} \big(1+mt^2\big)\geq 1>0, 
\end{equation} 
it follows that $\phi$ is a strictly monotone increasing function on $\bbR$. 
Therefore, one can use the invariance principle for the spectral shift function (see \cite[Section 8.11]{Ya92}) 
to introduce $\xi(\,\cdot\,;S,S_0)$ by setting 
\begin{equation} 
\xi(\lambda;S,S_0)=\xi(\phi(\lambda);\phi(S),\phi(S_0)) \, \text{ for a.e.~$ \lambda\in \bbR$.}
\end{equation} 
The choice of $\phi$ and integrability properties of $\xi(\,\cdot\,; \phi(S), \phi(S_0))$ will imply an appropriate integrability condition for $\xi(\,\cdot\,; S, S_0)$.

A crucial result in the proof of the inclusion \eqref{even_resol_comp} is the following result.
We recall that the H\"older space $C^{1,\alpha}([0,1])$, $0\leq \alpha\leq 1$, is the class of functions $f$ on $[0,1]$ such that 
\begin{equation} 
\|f\|_{C^{1,\alpha}([0,1])}=\|f\|_{C^1([0,1])}+\sup_{t_1,t_2\in[0,1]}\frac{|f'(t_1)-f'(t_2)|}{|t_1-t_2|^\alpha}<\infty.
\end{equation} 

\begin{theorem}\lb{thm_PS}\cite[Theorem 4 and Corollary 2]{PS09}
Suppose that $A$ and $B$ are self-adjoint operators on a Hilbert space, such that $(A-B) \in \cB_1(\cH)$ and 
$\sigma(A)\cup\sigma(B)\subset [0,1]$. For any function $f\in C^{1,\alpha}([0,1])$ with $0< \alpha\leq 1$ one 
has 
\begin{equation} 
[f(A)-f(B)] \in\cB_1(\cH) 
\end{equation} 
and 
\begin{equation} 
\|f(A)-f(B)\|_{\cB_1(\cH)}\leq C \|A-B\|_{\cB_1(\cH)},
\end{equation} 
where the constant $C$ is independent of $A$ and $B$. 
\end{theorem}

Assuming Hypothesis \ref{hyp_even_ssf} with $m=2k$, $k \in \bbN$, we intend to use Theorem \ref{thm_PS} for the operators $A= \big(S^2+I_\cH\big)^{-k}$ and $B= \big(S_0^2+I_\cH\big)^{-k}.$ In the following Lemma \ref{lem_even_square} we 
will show that with this choice of operators $A,B$ the condition $(A-B) \in\cB_1(\cH)$ of Theorem~\ref{thm_PS} is satisfied. 

\begin{lemma}\lb{lem_even_square}
Assume Hypothesis \ref{hyp_even_ssf} with $m=2k$ for some $k \in \bbN$. Then 
\begin{equation} 
\Big[\big(S^2+I_\cH\big)^{-k} - \big(S_0^2+I_\cH\big)^{-k}\Big] \in \cB_1(\cH).
\end{equation} 
\end{lemma}
\begin{proof}
One writes  
\begin{align}
& \big(S^2+I_\cH\big)^{-k} - \big(S_0^2+I_\cH\big)^{-k}    \no \\
& \quad =(S+iI_\cH)^{-k}(S-iI_\cH)^{-k}-(S_0+iI_\cH)^{-k}(S_0-iI_\cH)^{-k}  \no \\
& \quad =\big[(S+iI_\cH)^{-k}-(S_0+iI_\cH)^{-k}\big] \big[(S-iI_\cH)^{-k}-(S_0-iI_\cH)^{-k}\big]  \no \\
&\qquad+\big[(S+iI_\cH)^{-k}-(S_0+iI_\cH)^{-k}\big](S_0-iI_\cH)^{-k}     \lb{8.41AA}  \\
&\qquad+(S_0+iI_\cH)^{-k}\big[(S-iI_\cH)^{-k}-(S_0-iI_\cH)^{-k}\big].   \no 
\end{align}
By Lemma \ref{lem_weighted_resol_diff}, the second and the third terms are trace-class operators. 
By Theorem \ref{t8.6} one infers that 
\begin{equation} 
\big[(S+iI_\cH)^{-k}-(S_0+iI_\cH)^{-k}\big] \in \cB_{(m+1)/(k+1)}(\cH). 
\end{equation} 
Therefore, the first term on the right-hand side of \eqref{8.41AA} is a trace-class operator too.
\end{proof}

\begin{lemma}\lb{lem_good_functions}
Let $k\in\bbN$ and introduce the functions
\begin{equation} 
h_1(t)=\frac1{t^2 \big(1+t^2\big)^{2k-1}+1},\quad h_2(t)=\frac{\big(1+t^2\big)^k}{t^2 \big(1+t^2\big)^{2k-1}+1},\quad t\in\bbR.
\end{equation} 
There exist $f_1, f_2\in C^{1,\frac1k}([0,1])$ such that 
\begin{equation} 
h_1(t)=f_1\Big(\big(1+t^2\big)^{-k}\Big),\quad h_2(t)=f_2\Big(\big(1+t^2\big)^{-k}\Big),  \quad t\in\bbR.
\end{equation} 
\end{lemma}
\begin{proof}
We set 
\begin{equation} 
f_1(u)=\frac{u^2}{1-u^{1/k}+u^2}, \quad f_2(u)=\frac{u}{1-u^{1/k}+u^2},\quad u\in[0,1].
\end{equation}  
A direct verification shows that 
\begin{equation} 
h_1(t)=f_1\Big(\big(1+t^2\big)^{-k}\Big),\quad h_2(t)=f_2\Big(\big(1+t^2\big)^{-k}\Big), \quad t\in\bbR.
\end{equation} 
Since 
\begin{equation} 
1-u^{1/k}+u^2>0, \quad u\in[0,1],
\end{equation} 
$f_1,f_2 \in C([0,1])$. By the fact $f_1(u)=uf_2(u), u\in[0,1]$, it suffices to show that $f_2\in C^{1,\frac1k}([0,1])$. 
One verifies that 
 \begin{align}
 f'_2(u)&=\frac{\big[1-u^{1/k}+u^2\big] - u\big[-\frac1ku^{1/k-1}+2u\big]}{\big[1-u^{1/k}+u^2\big]^2}     \no \\
 &=\frac{1-[1-(1/k)]u^{1/k}-u^2}{\big[1-u^{1/k}+u^2\big]^2}.
 \end{align}
 Clearly $f_j\in C([0,1])$, $j = 1,2$. Furthermore, since the map $u\mapsto u^{1/k}$ is of H\"older class 
 $C^{0,\frac1k}([0,1])$ and the map $u\mapsto \big[1-u^{1/k}+u^2\big]^{-2}$ is bounded on $[0,1]$, it follows that 
 $f_2'\in C^{0,\frac1k}([0,1])$, that is, $f_2\in C^{1, \frac1k}([0,1])$, as required.
\end{proof}

\begin{lemma}\lb{lem_function_g}
Assume Hypothesis \ref{hyp_even_ssf} with $m=2k$ for some $k \in \bbN$. Let $h_2$ be as in Lemma \ref{lem_good_functions} and introduce 
\begin{equation} 
g(t)=\frac{t}{\big(1+t^2\big)^{1/2}},\quad t\in\bbR.
\end{equation} 
Then, 
\begin{equation} 
[g(S)-g(S_0)] h_2(S_0)\in \cB_1(\cH).
\end{equation} 
\end{lemma}
\begin{proof} 
Since
\begin{equation} 
h_2(t)=\frac{\big(1+t^2\big)^k}{t^2\big(1+t^2\big)^{2k-1}+1}\underset{|t|\to\infty}{=} \Oh\Big(\big(1+t^2\big)^{-k}\Big), 
\end{equation} 
 it suffices to show that 
\begin{equation} 
[g(S)-g(S_0)] \big(S_0^2+I_\cH\big)^{-k}\in \cB_1(\cH).
\end{equation} 
By \cite[Lemma 3.1]{CGGLPSZ16b},  
\begin{equation} 
g(S)-g(S_0)=\frac1\pi \Re\bigg( \int_0^\infty \frac{d\lambda}{\lambda^{1/2}}\big[(S+i(\lambda+1)^{1/2}I_\cH)^{-1}-(S_0+i(\lambda+1)^{1/2}I_\cH)^{-1}\big]\bigg),
\end{equation} 
with a convergent Bochner integral in $\cB(\cH)$. The substitution $\theta=(1+\lambda)^{1/2}$ then yields 
\begin{equation} 
g(S)-g(S_0)=\frac1\pi \Re\bigg( \int_1^\infty \frac{\theta d\theta}{(\theta^2-1)^{1/2}}\big[(S+i\theta I_\cH)^{-1}-(S_0+i\theta I_\cH)^{-1}\big]\bigg).
\end{equation}
Therefore, it suffices to prove that 
\begin{equation} 
 \int_1^\infty \frac{\theta d\theta}{(\theta^2-1)^{1/2}}\big[(S\pm i\theta I_\cH)^{-1}-(S_0\pm i\theta I_\cH)^{-1}\big] 
 \big(S_0^2+I_\cH\big)^{-k},
 \end{equation} 
are convergent integrals in $\cB_1(\cH)$. 

The resolvent identity implies 
\begin{align}
& \int_1^\infty \frac{\theta d\theta}{(\theta^2-1)^{1/2}}\big[(S\pm i\theta I_\cH)^{-1}-(S_0\pm i\theta I_\cH)^{-1}\big]  \big(S_0^2+I_\cH\big)^{-k}  \no \\
& \quad  =-\int_1^\infty \frac{\theta d\theta}{(\theta^2-1)^{1/2}}(S\pm i\theta I_\cH)^{-1}(S-S_0)(S_0\pm i\theta I_\cH)^{-1}  \big(S_0^2+I_\cH\big)^{-k}.
\end{align}
Let $0<\varepsilon<1/2$ be as in Hypothesis \ref{hyp_even_ssf}, that is,  
\begin{equation} 
(S-S_0) \big(S_0^2+I_\cH\big)^{-k-\varepsilon}\in \cB_1(\cH).
\end{equation} 
One estimates  
\begin{align}
& \Big\|(S\pm i\theta I_\cH)^{-1}(S-S_0)(S_0\pm i\theta I_\cH)^{-1} \big(S_0^2+I_\cH\big)^{-k}\Big\|_{\cB_1(\cH)}  \no \\
& \quad \leq \big\|(S\pm i\theta I_\cH)^{-1}\big\|_{\cB(\cH)} 
\Big\|(S-S_0)\big(S_0^2+I_\cH\big)^{-k-\varepsilon}\Big\|_{\cB_1(\cH)} 
\Big\|\big(S_0^2+\theta^2I_\cH\big)^{-1/2+\varepsilon}\Big\|_{\cB(\cH)}   \no \\
& \quad \leq \theta^{-2+2\varepsilon} \Big\|(S-S_0)\big(S_0^2+I_\cH\big)^{-k-\varepsilon}\Big\|_{\cB_1(\cH)},
\end{align}
implying, 
\begin{align}
& \bigg\|\int_1^\infty \frac{\theta d\theta}{(\theta^2-1)^{1/2}}(S\pm i\theta I_\cH)^{-1}
(S-S_0)(S_0\pm i\theta I_\cH)^{-1} \big(S_0^2+I_\cH\big)^{-k}\bigg\|_{\cB_1(\cH)}   \no \\
& \quad \leq \int_1^\infty \frac{d\theta}{(\theta^2-1)^{1/2}\theta^{1-2\varepsilon}} 
\Big\|(S-S_0)\big(S_0^2+I_\cH\big)^{-k-\varepsilon}\Big\|_{\cB_1(\cH)}.
\end{align}
Since $\varepsilon< 1/2$, it follows that $\int_1^\infty d\theta (\theta^2-1)^{-1/2}\theta^{2\varepsilon - 1}<\infty$ and thus the integral 
\begin{equation} 
\int_1^\infty \frac{\theta d\theta}{(\theta^2-1)^{1/2}}(S\pm i\theta I_\cH)^{-1}(S-S_0)(S_0\pm i\theta I_\cH)^{-1} \big(S_0^2+I_\cH\big)^{-k}
\end{equation} 
converges in $\cB_1(\cH)$. 
\end{proof}

\begin{lemma}\lb{lem_even_phi_trace_class}
 Assume Hypothesis \ref{hyp_even_ssf} with $m = 2k$ for some $k \in \bbN$. For the function $\phi$ introduced in \eqref{def_even_phi} one concludes that 
\begin{equation} 
\big[(\phi(S)+iI_\cH)^{-1}-(\phi(S_0)+iI_\cH)^{-1}\big] \in \cB_1(\cH).
 \end{equation} 
\end{lemma}
\begin{proof} 
One writes  
\begin{align}
[\phi(t)+i]^{-1} &= \Big[t\big(1+t^2\big)^{k-\frac12}+i\Big]^{-1}   \no \\
& =\frac{t\big(1+t^2\big)^{k-\frac12}}{t^2\big(1+t^2\big)^{2k-1}+1}-i\frac1{t^2\big(1+t^2\big)^{2k-1}+1}   \no \\
& =\frac{t}{\big(t^2+1\big)^{1/2}} \, \frac{\big(1+t^2\big)^{k}}{t^2\big(1+t^2\big)^{2k-1}+1}-i\frac1{t^2\big(1+t^2\big)^{2k-1}+1},
\end{align}
that is,  
\begin{equation} 
[\phi(t)+i]^{-1}=g(t)h_1(t)-ih_2(t),\quad t\in\bbR,
\end{equation} 
where $h_1,h_2$ are introduced in Lemma \ref{lem_good_functions} and $g$ in Lemma \ref{lem_function_g}.
Therefore,
\begin{align}
& \big[(\phi(S)+iI_\cH)^{-1} - (\phi(S_0)+iI_\cH)^{-1}\big]    \no \\
& \quad =g(S)h_1(S)-g(S_0)h_1(S_0)-i[h_2(S)-h_2(S_0)]   \no \\
& \quad = [g(S)-g(S_0)] h_1(S_0)+ g(S)[h_1(S)-h_1(S_0)] -i [h_2(S)-h_2(S_0)].
\end{align}
and by Lemma \ref{lem_function_g} one concludes that 
\begin{equation} 
[g(S)-g(S_0)] h_1(S_0)\in\cB_1(\cH).   
\end{equation} 
Thus, Lemma \ref{lem_good_functions} implies 
\begin{equation} 
h_j(S)-h_j(S_0)=f_j\Big(\big(S^2+I_\cH\big)^{-k}\Big)-f_j\Big(\big(S_0^2+I_\cH\big)^{-k}\Big), \quad j=1,2,
\end{equation} 
with $f_j\in C^{1,\frac1k}([0,1])$, $j =1,2$. Lemma \ref{lem_even_square} then yields 
\begin{equation} 
\Big[\big(S^2+I_\cH\big)^{-k} - \big(S_0^2+I_\cH\big)^{-k}\Big] \in\cB_1(\cH).
\end{equation} 
Thus, by Theorem \ref{thm_PS} one obtains  
\begin{equation} 
\Big[f_j\Big(\big(S^2+I_\cH\big)^{-k}\Big) - f_j\Big(\big(S_0^2+I_\cH\big)^{-k}\Big)\Big] \in \cB_1(\cH), \quad j=1,2,
\end{equation} 
and hence,
\begin{equation} 
\big[(\phi(S)+iI_\cH)^{-1} - (\phi(S_0)+iI_\cH)^{-1}\big] \in \cB_1(\cH),
 \end{equation} 
as required.
\end{proof}

The following theorem improves the integrability condition in \eqref{8.8} for even $m\in\bbN$.  

\begin{theorem}\lb{t8.2}
Assume Hypothesis \ref{hyp_even_ssf} with $m=2k$ for some $k \in \bbN$. For any $f\in \gF_m(\bbR)$ one has 
\begin{equation}
[f(S) - f(S_0)] \in \cB_1(\cH),
\end{equation}
and there exists a function
\begin{equation} \lb{eq_ssf_even_integrability}
\xi(\,\cdot\,;S,S_0)\in L^1\big(\bbR; (1+|\lambda|)^{-m-1}\, d\lambda\big)
\end{equation}
such that the following trace formula holds, 
\begin{equation}\lb{eq_trace_ssf_even}
\tr_{\cH}(f(S) - f(S_0)) = \int_{\bbR}d\lambda\, \xi(\lambda;S,S_0)f'(\lambda),\quad f\in \gF_m(\bbR).
\end{equation}
In particular, one has 
\begin{equation} 
\big[(S-zI_{\cH})^{-m} - (S_0-zI_{\cH})^{-m}\big] \in\cB_1(\cH), \quad z \in \bbC \backslash \bbR,  \lb{8.11A}
\end{equation} 
and 
\begin{equation} \lb{8.11}
\tr_{\cH}\big((S-zI_{\cH})^{-m} - (S_0-zI_{\cH})^{-m}\big) = -m \int_{\bbR} \frac{\xi(\lambda;S,S_0) d\lambda}{(\lambda - z)^{m+1}},   \quad z\in \bbC\backslash \bbR.
\end{equation}
\end{theorem}
\begin{proof}
%
Let $\phi$ be as in \eqref{def_even_phi}.
Then Lemma \ref{lem_even_phi_trace_class} implies that 
\begin{equation} 
\big[(\phi(S)+iI_\cH)^{-1}-(\phi(S_0)+iI_\cH)^{-1}\big] \in \cB_1(\cH), 
 \end{equation} 
and hence there exists the spectral shift function 
\begin{equation} 
\xi(\,\cdot\,;\phi(S),\phi(S_0))\in L^1\big([0,1]; (1+|\lambda|)^{-2}d\lambda\big)
\end{equation} 
 for the pair $(\phi(S),\phi(S_0))$.
 Since  
 \begin{equation} 
 \phi'(t)= \big(1+t^2\big)^{\frac{m-3}{2}} \big(1+mt^2\big)\geq 1>0, 
 \end{equation} 
 it follows that $\phi$ is strictly monotone increasing on $\bbR$. Hence, we introduce the spectral shift function 
 $\xi(\,\cdot\,;S,S_0)$ by setting 
 \begin{equation} 
 \xi(\lambda;S,S_0)=\xi(\phi(\lambda); \phi(S),\phi(S_0)) \, \text{ for a.e.~$\lambda\in\bbR$.} 
 \end{equation} 
  Since $\xi(\,\cdot\,;\phi(S),\phi(S_0))\in L^1\big([0,1]; (1+|\lambda|)^{-2}d\lambda\big)$, the definition of $\phi$ implies that  
 \begin{equation} 
 \xi(\,\cdot\,;S,S_0)\in L^1\big(\bbR; (1+|\lambda|)^{-m-1}\, d\lambda\big).
 \end{equation} 
 Next, let $f\in \gF_m(\bbR)$. Then $f\circ \phi^{-1}\in\gF_1(\bbR)$, and hence \eqref{eq_classical_tace_formula} 
 implies  
 \begin{align}
 \tr_{\cH}\big(f(S)-f(S_0)\big)&=\tr_\cH\big( (f\circ \phi^{-1})(\phi(S))-(f\circ \phi^{-1})(\phi(S_0))\big)   \no \\
 &=\int_0^1\xi(\mu; \phi(S),\phi(S_0)) f'(\phi^{-1}(\mu)) \frac{d\mu}{\phi'(\phi^{-1}(\mu))}   \no \\
 &=\int_\bbR \xi(\phi(\lambda); \phi(S),\phi(S_0)) f'(\lambda)d\lambda   \no \\
 &=\int_\bbR d\lambda \, \xi(\lambda; S, S_0) f'(\lambda),
 \end{align}
 proving \eqref{eq_trace_ssf_even}.
 
 Since for any $z\in\bbC \backslash \bbR$ the map $\lambda\mapsto (\lambda-z)^{-m}$, $\lambda\in\bbR$, belongs to the class $\gF_m(\bbR)$, the trace formula \eqref{8.11} is a particular case of formula \eqref{eq_trace_ssf_even}.
\end{proof}

\section{Representing $\xi(\, \cdot \,;S,S_0)$ in Terms of Regularized Fredholm Determinants}  \lb{s8a}

In this section we establish the representation of $\xi(\, \cdot \,;S,S_0)$ in terms of regularized Fredholm determinants. 

\begin{hypothesis}\lb{h8.3}
Let $S_0$ and $S$ be self-adjoint operators in $\cH$ with $(S-S_0) \in \cB(\cH)$. \\[1mm]
$(i)$ If $m \in \bbN$ is odd, assume \eqref{8.6}, that is,
\begin{equation} \lb{9.1A}
\big[(S - zI_{\cH})^{-m} - (S_0 - zI_{\cH})^{-m}\big] \in \cB_1(\cH),\quad z\in \bbC\backslash \bbR.
\end{equation} 
and 
\begin{equation} \lb{8.12}
(S-S_0)(S_0 - zI_{\cH})^{-j}\in \cB_{(m+1)/j}(\cH),\quad j\in \bbN,\; 1\leq j\leq m+1. 
\end{equation}
$(ii)$ If $m \in \bbN$ is even, assume the remaining conditions in Hypothesis \ref{hyp_even_ssf}, that is, 
for some $0<\varepsilon<1/2$, 
\begin{equation} \lb{hyp_even_epsilonA}
(S-S_0) \big(S_0^2+I_\cH\big)^{- (m/2) - \varepsilon}\in \cB_1(\cH).
\end{equation}
\end{hypothesis}

By Remark \ref{r8.4}\,$(i)$, \eqref{8.12} holds for odd and even $m$. 

\begin{remark}
In the applications to multidimensional Dirac operators to be considered in the sequel, the number $m$ in \eqref{8.12} is the dimension of the underlying Euclidean space $\bbR^n$, that is, $m = n$, $n \in \bbN$, $n \geq 2$, 
as detailed in Lemma \ref{l7.2}. \hfill $\diamond$
\end{remark}

By Theorem \ref{t8.6}, Hypothesis \ref{h8.3} implies 
\begin{equation}
\big[(S-zI_{\cH})^{-m} - (S_0-zI_{\cH})^{-m} \big]\in \cB_1(\cH),\quad z\in \bbC\backslash\bbR,  \lb{9.8AA} 
\end{equation}
for odd and even $m$. 

By Theorem \ref{t8.2} the spectral shift function $\xi(\,\cdot\,;S,S_0)$ exists and \eqref{8.8} and \eqref{8.11} hold.  The main aim of the present section is to obtain an almost everywhere representation for $\xi(\,\cdot\,;S,S_0)$ in terms of the regularized perturbation determinants of the operators $(S-S_0)(S_0-zI_{\cH})^{-1}$, $z\in \bbC\backslash \bbR$.

To set the stage, we begin by recalling some basic definitions and results pertaining to regularized determinants to be used in the sequel.  For detailed discussions of regularized determinants, we refer to \cite[Chapters 3, 5, and 9]{Si05} and \cite[\S1.7]{Ya92}.

Let $\{h_n\}_{n=1}^{\infty}$ denote an orthonormal basis for $\cH$ and suppose $A\in\cB_1(\cH)$.  For each 
$N \in \bbN$, let $M_N \in \bbC^{N \times N}$ denote the matrix with entries
\begin{equation}
\delta_{j,k} + (h_j,Ah_k)_{\cH},\quad 1\leq j,k\leq N.
\end{equation}
The sequence $\{\Det_{\bbC^{N\times N}}(M_N)\}_{N=1}^{\infty}$ has a limit as $N \to \infty$, and its value does not depend on the orthonormal basis chosen.  One defines the Fredholm determinant
\begin{equation}
\Det_{\cH}(I_{\cH}+A):=\lim_{N \to\infty}\Det_{\bbC^{N \times N}}(M_N).
\end{equation}
The Fredholm determinant is continuous with respect to $\|\,\cdot\,\|_{\cB_1(\cH)}$.  That is, if $\{A_n\}_{n=1}^{\infty}\subset \cB_1(\cH)$ and $\lim_{n\to \infty} \|A_n - A\|_{\cB_1(\cH)}=0$, then
\begin{equation} \lb{8.22a}
\lim_{n\to \infty} \Det_{\cH}(I_{\cH} + A_n) = \Det_{\cH}(I_{\cH} + A).
\end{equation}
In fact, the Fredholm determinant is Fr\'echet differentiable with respect to $A$ (cf., e.g., \cite[Theorem 5.2]{Si05}).
Moreover, if $\Omega \subseteq \bbC$ is an open set and $A:\Omega \to \cB_1(\cH)$ is analytic, then the function $\Det_{\cH}(I_{\cH} + A(\,\cdot\,))$ is analytic in $\Omega$, and
\begin{equation}
\frac{d}{dz} \log\big(\Det_{\cH}(I_{\cH}+A(z))\big) = \tr_{\cH}\big((I_{\cH}+A(z))^{-1}A'(z)\big).
\end{equation}

The definition of the Fredholm determinant given in \eqref{8.22a} is generally not meaningful if $A\in \cB_p(\cH)$ with $p\in \bbN\backslash \{1\}$. To give meaning to the determinant in this case, suitable modifications are needed.  For $p\in \bbN\backslash \{1\}$, one introduces the function $R_p:\cB_p(\cH)\to \cB_1(\cH)$ by
\begin{equation}\lb{9.15a}
R_p(A) = (I_{\cH} +A)e^{\sum_{j=1}^{p-1}(-1)^jj^{-1}A^j} - I_{\cH},\quad A\in \cB_p(\cH).
\end{equation}
Then the regularized (or modified) Fredholm determinant is defined by
\begin{equation}\lb{9.16a}
\Det_{\cH,p}(I_{\cH}+A) = \Det_{\cH}(I_{\cH}+ R_p(A)),\quad A\in \cB_p(\cH),\quad p\in \bbN\backslash\{1\}.
\end{equation}
The Fredholm determinant $\det_{\cH,p} (\, \cdot \,)$ retains many of the properties of the ordinary Fredholm determinant.  For example, $\Det_{\cH,p}(I_{\cH}+A)$ is continuous with respect to $A\in \cB_p(\cH)$: if $A\in \cB_p(\cH)$, $\{A_n\}_{n=1}^{\infty}\subset \cB_p(\cH)$, and $\lim_{n\to \infty}\|A_n - A\|_{\cB_p(\cH)}=0$, then
\begin{equation}
\lim_{n\to \infty} \Det_{\cH,p}(I_{\cH}+A_n) = \Det_{\cH,p}(I_{\cH}+A).
\end{equation}
In addition, if $\Omega\subseteq\bbC$ is open and $A:\Omega\to \cB_p(\cH)$ is analytic in $\Omega$, then the function $\Det_{\cH,p}(I_{\cH}+A(\,\cdot\,))$ is analytic in $\Omega$ and
\begin{equation} \lb{8.27a}
\frac{d}{dz}\log\big(\Det_{\cH,p}(I_{\cH}+A(z))\big) = (-1)^{p-1}\tr_{\cH}\big((I_{\cH}+A(z))^{-1} A^{p-1}(z)A'(z) \big).
\end{equation}

The importance of the regularized determinant stems from the fact that for $A\in \cB_p(\cH)$, the operator $I_{\cH}+A$ is boundedly invertible (i.e., $-1\in \rho(A)$) if and only if $\Det_{\cH,p}(I_{\cH}+A)\neq 0$ (cf., e.g., \cite[Theorem 9.2]{Si05}).  Equivalently, $-1\in \sigma(A)$ if and only if $\Det_{\cH,p}(I_{\cH}+A)=0$.

Finally, we note the cyclicity property of the regularized determinant:  if $A,B\in \cB(\cH)$ are such that $AB,BA\in \cB_p(\cH)$, then
\begin{equation} \lb{det_product}
\Det_{\cH,p}(I_{\cH}+AB) = \Det_{\cH,p}(I_{\cH}+BA).
\end{equation}
In particular, if $A\in \cB_p(\cH)$ and $B\in \cB(\cH)$, then \eqref{det_product} holds. Similarly, if $A,B \in \cB(\cH)$ are such that $AB,BA\in \cB_1(\cH)$, then
\begin{equation} \lb{trace_product}
\tr_{\cH}(AB) = \tr_{\cH}(BA).
\end{equation}

With these preliminaries in hand, we start with introducing the regularized determinant associated with the (non-symmetrized) Birman--Schwinger-type operator 
\begin{equation} 
B(z) := (S-S_0)(S_0-zI_{\cH})^{-1}, \quad z\in\bbC\backslash\bbR.    \lb{8.B(z)} 
\end{equation}

\begin{lemma}\lb{l8.7}
Assume Hypothesis \ref{h8.3}. The map 
\begin{equation} \lb{8.19}
B(z)=(S-S_0)(S_0-zI_{\cH})^{-1},\quad z\in \bbC\backslash \bbR,
\end{equation}
is a $\cB_{m+1}(\cH)$-valued analytic function. 
\end{lemma}
\begin{proof}
Let $z_0\in \bbC\backslash \bbR$ be fixed.  Then (see e.g. \cite[Theorem 5.14]{We80}) 
\begin{equation}
(S_0-zI_{\cH})^{-1}=\sum_{k=0}^{\infty}(S_0-z_0I_{\cH})^{-k-1}(z-z_0)^k,
\end{equation}
where the series converges with respect to the $\cB(\cH)$-norm for all $z\in\bbC\backslash\bbR$ such that $|z-z_0|<\|(S_0-z_0I_{\cH})^{-1}\|_{\cB(\cH)}^{-1}\leq |\Im(z_0)|^{-1}$.
Therefore,
\begin{equation}
(S-S_0)(S_0-zI_{\cH})^{-1}=\sum_{k=0}^{\infty}(S-S_0)(S_0-z_0I_{\cH})^{-k-1}(z-z_0)^k.
\end{equation}
We claim that the latter series converges in the ball $\{z\in \bbC\,|\,|z-z_0|<|\Im(z_0)|\}$ in the $\cB_{m+1}(\cH)$ norm. By \eqref{8.12}, the operator $B(z) = (S-S_0)(S_0-z_0I_{\cH})^{-1}\in \cB_{m+1}(\cH)$, so that
\begin{align}
&\big\|(S-S_0)(S_0-z_0I_{\cH})^{-k-1}\big\|_{\cB_{m+1}(\cH)}\\
&\quad\leq \big\|(S-S_0)(S_0-z_0I_{\cH})^{-1}\big\|_{\cB_{m+1}(\cH)} \big\|(S_0-z_0I_{\cH})^{-k}\big\|_{\cB(\cH)}\no\\
&\quad\leq \big\|(S-S_0)(S_0-z_0I_{\cH})^{-1}\big\|_{\cB_{m+1}(\cH)} |\Im(z_0)|^{-k}.\no
\end{align}
This proves the convergence.
\end{proof}

\begin{lemma}\lb{l8.8}
Assume Hypothesis \ref{h8.3}. The function
\begin{equation} \lb{8.23}
F_{S,S_0}(z) := \ln\big(\Det_{\cH,m+1}\big((S-zI_{\cH})(S_0-zI_{\cH})^{-1}\big)\big),\quad z\in \bbC\backslash \bbR,
\end{equation}
is well-defined, in fact, analytic in $\bbC \backslash \bbR$.
\end{lemma}
\begin{proof}
By the second resolvent identity,
\begin{equation}
(S-zI_{\cH})(S_0-zI_{\cH})^{-1}=I_{\cH}+B(z),\quad z\in \bbC\backslash \bbR,
\end{equation}
where $B(\, \cdot \,)$ is the $\cB_{m+1}(\cH)$-valued analytic function defined in \eqref{8.19}. Hence, by the properties of regularized determinants the function 
\begin{equation}
z\to\Det_{\cH,m+1}(I_{\cH}+B(z))
\end{equation}
is analytic in $\bbC_{\pm}$. Combining the Cauchy integral theorem and \cite[Theorem V.4.1]{Pa91}, one infers that the function
\begin{equation}
z\to\ln\big(\Det_{\cH,m+1}(I_{\cH}+B(z))\big)
\end{equation}
is a well-defined analytic function in $\bbC_{\pm}$, provided that $\Det_{\cH,m+1}(I_{\cH}+B(z))\neq 0$ for all $z\in\bbC\backslash \bbR$.   Thus, it remains to show, that $\Det_{\cH,m+1}(I_{\cH}+B(z))\neq0$ for every $z\in\bbC\backslash \bbR$. If $\Det_{\cH,m+1}(I_{\cH}+B(z))=0$ for some $z\in \bbC\backslash \bbR$,  then $-1$ is in the spectrum of $B(z).$ By compactness of $B(\, \cdot \,)$, $-1$ is an eigenvalue of $B(z),$ and therefore, $I_{\cH}+B(z)$ has a nontrivial kernel. Since $S_0$ is self-adjoint and hence $z \in \bbC \backslash \bbR$ cannot be an eigenvalue of $S_0$, $z$ is an eigenvalue for $S$, which, once more, cannot be the case since $S$ is also self-adjoint.
\end{proof}

To correlate the function $\ln\big(\Det_{\cH,m+1}(I_{\cH}+B(\,\cdot\,))\big)$ with the spectral shift function for the pair $(S,S_0)$, we need to introduce an auxiliary function.

\begin{lemma}\lb{l8.9}
Assume Hypothesis \ref{h8.3} and let $B(\,\cdot\,)$ be defined by \eqref{8.19}.  There exists an analytic function $G_{S,S_0}(\dott)$ in $\bbC \backslash \bbR$ such that
\begin{equation} \lb{8.28}
\frac{d^m}{d z^m}G_{S,S_0}(z)=\tr_{\cH}\Bigg(\frac{d^{m-1}}{d z^{m-1}}\sum_{j=0}^{m-1}(-1)^{m-j}(S_0-zI_{\cH})^{-1}B(z)^{m-j}\Bigg),\quad z\in \bbC\backslash \bbR.
\end{equation}
\end{lemma}
\begin{proof} 
It suffices to prove that each of the terms
\begin{equation}\lb{9.31b}
\tr_{\cH}\bigg(\frac{d^{m-1}}{d z^{m-1}}(S_0-zI_{\cH})^{-1}B(z)^{m-j}\bigg),\quad z\in \bbC\backslash \bbR,\; j\in \bbN_0,\; 0\leq j\leq m-1,
\end{equation}
defines an analytic function.  To analyze the operator under the trace in \eqref{9.31b}, we introduce multi-indices. Recalling $\bbN_0=\bbN \cup \{0\}$, for $\nu\in \bbN$ an element $\ul k\in \bbN_0^{\nu}$ is called a multi-index which we express componentwise as
\begin{equation}\lb{9.32c}
\ul k=(k_1,\ldots,k_{\nu}) \in \bbN_0^{\nu}, \quad k_p\in \bbN_0,\; 1\leq p \leq \nu.
\end{equation}
The order of the multi-index $\ul k\in \bbN_0^{\nu}$ is defined to be
\begin{equation}\lb{9.33c}
|\ul k|:=k_1+\cdots+k_{\nu}.
\end{equation}

For each fixed $j\in \bbN_0$ with $0\leq j\leq m-1$,
\begin{align}
&\frac{d^{m-1}}{d z^{m-1}}(S_0-zI_{\cH})^{-1}B(z)^{m-j} \lb{9.34b}\\
&\quad = \sum_{\substack{\ul k\in \bbN_0^{m-j+1} \\ |\ul k|=m-1}} c_{j,\ul k}(S_0-zI_{\cH})^{-(k_1+1)} \prod_{\ell=2}^{m-j+1}(S-S_0)(S_0-I_{\cH})^{-(k_{\ell}+1)},    \no 
\end{align}
for an appropriate set of $z$-independent scalars
\begin{equation}
c_{j,\ul k}\in \bbR,\quad \ul k\in \bbN_0^{m-j+1},\; |\ul k|=m-1.
\end{equation}
The assumption \eqref{8.12} and the analog of H\"older's inequality for trace ideals (see \cite[Theorem 2.8]{Si05}) imply that each term in the sum in \eqref{9.34b} is a trace class operator.  In particular, \eqref{9.34b} implies that the operator $\frac{d^{m-1}}{d z^{m-1}}(S_0-zI_{\cH})^{-1} B(z)^{m-j}$ is a trace class operator.  Repeating the argument in Lemma \ref{l8.7} and employing \eqref{8.12}, one concludes that the map 
\begin{equation}
\bbC \backslash \bbR \ni z\to \frac{d^{m-1}}{d z^{m-1}}(S_0-zI_{\cH})^{-1} B(z)^{m-j}, \quad 0 \leq j \leq m-1,
\end{equation}
is a $\cB_1(\cH)$-valued analytic function.
\end{proof}

The following lemma is the main result, which allows  to correlate the regularized determinant of the operator 
$I_{\cH} + (S-S_0)(S_0-zI_{\cH})^{-1}$ and the spectral shift function $\xi(\,\cdot\,,S,S_0)$ (see Theorem \ref{t8.13} below).

\begin{lemma}\lb{l8.10}
 Assume Hypothesis \ref{h8.3}.  If $F_{S,S_0}$ and $G_{S,S_0}$ denote the analytic functions in $\bbC\backslash \bbR$ introduced in \eqref{8.23} and \eqref{8.28}, respectively, then there exist polynomials $P_{\pm, m-1}$ of degree less than or equal to $m -1$ such that
\begin{equation}
F_{S,S_0}(z)=(z-i)^m \int_{\mathbb{R}} \frac{\xi(\lambda;S,S_0)d\lambda}{(\lambda-i)^m} \f{1}{\lambda-z} 
+ G_{S,S_0}(z)+P_{\pm, m-1}(z),\quad z\in \bbC_{\pm}.
\end{equation}
\end{lemma}
\begin{proof}
One recalls the $\cB_{m+1}(\cH)$-valued analytic function $B(\,\cdot\,)$ defined in \eqref{8.19}.
By the second resolvent identity,
\begin{align}
\big(I_{\cH}+B(z)\big)^{-1}&=\big(I_{\cH}+(S-S_0)(S_0-zI_{\cH})^{-1}\big)^{-1}\\
&=\big((S-zI_{\cH})(S_0-zI_{\cH})^{-1}\big)^{-1}\no\\
&=(S_0-zI_{\cH})(S-zI_{\cH})^{-1},\quad z\in \bbC\backslash \bbR.\no
\end{align}
In addition, 
\begin{equation}
B'(z)=B(z)(S_0-zI_{\cH})^{-1},\quad z\in \bbC\backslash \bbR.
\end{equation} 
Applying \eqref{8.27a}, one obtains
\begin{align}
\frac{d}{dz}F_{S,S_0}(z)&=\frac{d}{dz}\ln\big(\Det_{\cH,m+1}(I_{\cH}+B(z))\big)\\
&=(-1)^{m}\tr_{\cH}\big((S_0-zI_{\cH})(S-zI_{\cH})^{-1} B(z)^m B(z)(S_0-zI_{\cH})^{-1}\big)\no\\
&=(-1)^{m}\tr_{\cH}\big((S_0-zI_{\cH})(S-zI_{\cH})^{-1} B(z)^{m+1}(S_0-zI_{\cH})^{-1}\big),\no\\
&\hspace*{7.9cm}z\in \bbC\backslash \bbR.\no
\end{align}
For $z\in \bbC\backslash \bbR$, $(S_0-zI_{\cH})(S-zI_{\cH})^{-1}\in\cB(\cH)$ and $B(z)^{m+1}\in\cB_1(\cH)$ (cf. \eqref{8.12}), so that
\begin{align}
\frac{d}{dz}F_{S,S_0}(z)&=(-1)^{m}\tr_{\cH}\big(B(z)^{m+1}(S_0-zI_{\cH})^{-1}(S_0-zI_{\cH})(S-zI_{\cH})^{-1}\big)\nonumber\\
&=(-1)^{m}\tr_{\cH}\big(B(z)^{m+1}(S-zI_{\cH})^{-1}\big)\no\\
&=(-1)^{m}\tr_{\cH}\big((S-zI_{\cH})^{-1} B(z)^{m+1}\big),\quad z\in \bbC\backslash \bbR.\lb{eq_det_and_ssf1}
\end{align}
By the second resolvent identity,
\begin{align}
(S-z)^{-1}B(z)&=(S-z)^{-1}(S-S_0)(S_0-z)^{-1}\no\\
&=(S_0-z)^{-1}-(S-z)^{-1},\quad z\in \bbC\backslash \bbR,\lb{8.46}
\end{align}
and repeated application of \eqref{8.46} yields
\begin{align}
&(-1)^{m}(S-zI_{\cH})^{-1} B(z)^{m+1}\lb{8.47}\\
&\quad =(-1)^m \big((S_0-zI_{\cH})^{-1}-(S-zI_{\cH})^{-1}\big) B(z)^{m}\no\\
&\quad =(S_0-zI_{\cH})^{-1}\sum_{j=0}^{m-1}(-1)^{m-j} B(z)^{m-j} + (S_0-zI_{\cH})^{-1}-(S-zI_{\cH})^{-1},\no\\
&\hspace*{9.4cm} z\in \bbC\backslash \bbR.\no
\end{align}
Hence, combining \eqref{eq_det_and_ssf1} with \eqref{8.47}, one obtains
\begin{align}
&\frac{d}{dz}F_{S,S_0}(z)\lb{8.48}\\
&\quad =\tr_{\cH}\bigg((S_0-zI_{\cH})^{-1}-(S-zI_{\cH})^{-1}
+(S_0-zI_{\cH})^{-1}\sum_{j=0}^{m-1}(-1)^{m-j} B(z)^{m-j}\bigg),\no\\
&\hspace*{10.87cm}z\in \bbC\backslash \bbR.\no
\end{align}
Differentiating \eqref{8.48} $m-1$ times,
\begin{align}
\frac{d^m}{d z^m}F_{S,S_0}(z)&=\frac{d^{m-1}}{d z^{m-1}}\tr_{\cH}\bigg((S_0-zI_{\cH})^{-1}-(S-zI_{\cH})^{-1}\lb{8.49}\\
&\hspace*{2.5cm}+(S_0-zI_{\cH})^{-1}\sum_{j=0}^{m-1}(-1)^{m-j} B(z)^{m-j}\bigg)\no\\
&=\tr_{\cH}\bigg((m-1)!\Big((S_0-zI_{\cH})^{-m}-(S-zI_{\cH})^{-m}\Big)\no\\
&\hspace*{1.5cm}+\frac{d^{m-1}}{d z^{m-1}}(S_0-zI_{\cH})^{-1}\sum_{j=0}^{m-1}(-1)^{m-j}B(z)^{m-j}\bigg),\quad z\in \bbC\backslash \bbR.\no
\end{align}
By \eqref{9.8AA} and Lemma \ref{l8.9}, \eqref{8.49} may be recast as
\begin{align}
\frac{d^m}{d z^m}F_{S,S_0}(z)&=(m-1)!\tr_{\cH}\Big((S_0-zI_{\cH})^{-m}-(S-zI_{\cH})^{-m}\Big)\\
&\quad +\frac{d^m}{d z^m}G_{S,S_0}(z),\quad z\in \bbC\backslash \bbR,\no
\end{align}
and an application of Theorem \ref{t8.2} yields
\begin{equation} \lb{8.51}
\frac{d^m}{d z^m}F_{S,S_0}(z)= m! \int_{\mathbb{R}}\frac{\xi(\lambda; S,S_0) \, d\lambda}{(\lambda-z)^{m+1}} +\frac{d^m}{d z^m}G_{S,S_0}(z),\quad z\in \bbC\backslash \bbR.
\end{equation}
Repeated application of the elementary identity
\begin{equation}
\frac{k!}{(\lambda-z)^{k+1}}=\frac{d}{dz}\Bigg(\frac{(k-1)!}{(\lambda-z)^k}-\frac{1}{(\lambda-i)^k}\Bigg),\quad z\in \bbC\backslash \bbR,\; \lambda\in \bbR,\; k\in \bbN,
\end{equation}
yields
\begin{align}
\frac{m!}{(\lambda-z)^{m+1}}&=\frac{d^m}{dz^m}\Bigg(\frac1{\lambda-z}-\sum_{j=1}^{m}\frac{(z-i)^{j-1}}{(\lambda-i)^j}\Bigg)\no\\
&=\frac{d^m}{d z^m}\Bigg(\frac{1}{\lambda-z}-\frac{(z-i)^m-(\lambda-i)^m}{(z-\lambda)(\lambda-i)^m}\Bigg)\no\\
&=\frac{d^m}{d z^m}\Bigg(\frac{(z-i)^m}{(\lambda-z)(\lambda-i)^m}\Bigg),\quad z\in \bbC\backslash \bbR,\; 
\lambda\in \bbR.\lb{8.53}
\end{align}
Therefore, \eqref{8.51} and \eqref{8.53} imply
\begin{equation}
\frac{d^m}{d z^m}F_{S,S_0}(z)=\frac{d^m}{d z^m}\int_\bbR \f{\xi(\lambda; S,S_0)d\lambda \, (z-i)^m}{(\lambda-i)^m (\lambda-z)} +\frac{d^m}{d z^m}G_{S,S_0}(z),\quad z\in \bbC\backslash \bbR,
\end{equation}
completing the proof.
\end{proof}

\begin{remark} \lb{r8.11}
For bookkeeping purposes we have thus far worked with the non-symmetrized Birman--Schwinger-type operator 
$B(z) := (S-S_0)(S_0-zI_{\cH})^{-1}$, $z\in\bbC\backslash\bbR$ and avoided a factorization of $S-S_0$ (and similarly we exploited $V$ without its factorization \eqref{3.25aa} in Section \ref{s7}). In the concrete case of massless Dirac operators $H_0, H$ we will eventually switch over to a symmetrized analog (cf. \eqref{3.19}--\eqref{3.26a}, \eqref{9.82}).  
\hfill $\diamond$
\end{remark}

The main result of this section provides a means for recovering the spectral shift function $\xi(\,\cdot\,; S,S_0)$ almost everywhere in terms of the normal (or nontangential) boundary values of the functions $F_{S,S_0}(\,\cdot\,)$ and $G_{S,S_0}(\,\cdot\,)$ when the latter exist.  Its proof relies on the following (special case of) Privalov's theorem (see, e.g., \cite[Theorem 1.2.5]{Ya92}).

\begin{theorem}\lb{t8.12}
Let $\theta\in L^1\big(\mathbb{R};(1+ |\lambda|)^{-1}\,d\lambda\big)$.  If
\begin{equation}
H(z):=\int_{\mathbb{R}}\frac{\theta(\lambda) d\lambda}{\lambda-z},\quad z\in \bbC\backslash\bbR,
\end{equation}
then
\begin{equation}
\lim_{\varepsilon\downarrow0}H(\lambda \pm i\varepsilon)=\pm \pi i\theta(\lambda)+p.v.\int_{\bbR} 
\frac{\theta(\lambda') d\lambda'}{\lambda'-\lambda}\,  \text{ for a.e.~$\lambda \in \bbR$},     \lb{8.57}
\end{equation}
where $p.v. \, (\, \cdot \,)$ abbreviates the principal value operation. In particular, one obtains the following special case of the Stieltjes inversion theorem,
\begin{equation}
\theta(\lambda) = (2 \pi i)^{-1} \lim_{\varepsilon \downarrow 0} [H(\lambda + i \varepsilon)) 
- H(\lambda - i \varepsilon))] \, \text{ for a.e.~$\lambda \in \bbR$.}    \lb{8.57a} 
\end{equation} 
Moreover, the normal limits in \eqref{8.57} and \eqref{8.57a} can be replaced by nontangential limits. 
\end{theorem}

\begin{theorem} \lb{t8.13}
Assume Hypothesis \ref{h8.3} and let $F_{S,S_0}$ and $G_{S,S_0}$ denote analytic functions in $\bbC\backslash \bbR$ satisfying \eqref{8.23} and \eqref{8.28}, respectively.  If $F_{S,S_0}$ and $G_{S,S_0}$ have normal boundary values on $\bbR$, then for a.e.~$\lambda\in\bbR$, 
\begin{equation} \lb{eq_main}
\xi(\lambda;S,S_0)= \pi^{-1}\Im (F_{S,S_0}(\lambda+i0)) - \pi^{-1}\Im (G_{S,S_0}(\lambda+i0)) 
+ P_{m-1}(\lambda)\,\text{ for a.e.~$\lambda\in \bbR$},
\end{equation}
where $P_{m-1}$ is a polynomial of degree less than or equal to $m - 1$. 
\end{theorem}
\begin{proof}
By Lemma \ref{l8.10}, 
\begin{equation} \lb{8.56}
F_{S,S_0}(z)=(z-i)^m \int_{\mathbb{R}}\frac{\xi(\lambda;S,S_0) d\lambda}{(\lambda-z)(\lambda-i)^m}\, +G_{S,S_0}(z)+P_{\pm,m-1}(z),\quad z\in \bbC_{\pm}.
\end{equation}
If
\begin{equation}
H_{\pm}(z):=F_{S,S_0}(z)-G_{S,S_0}(z)-P_{\pm,m-1}(z),\quad z\in \bbC_{\pm},
\end{equation}
then \eqref{8.56} may be recast as
\begin{equation}
\frac{H_{\pm}(z)}{(z-i)^m}=\int_{\mathbb{R}}\frac{\xi(\lambda;S,S_0) d\lambda}{(\lambda-z)(\lambda-i)^m}, 
\quad z\in \bbC_{\pm}, \; z\neq i.
\end{equation}
By \eqref{8.8}, the spectral shift function $\xi(\,\cdot\,;S,S_0)$ satisfies
\begin{equation}
\xi(\,\cdot\,;S,S_0)(\,\cdot\,-i)^{-m}\in L^1\big(\bbR; (1+ |\lambda|)^{-1}\,d\lambda\big).
\end{equation}
An application of Theorem \ref{t8.12} then yields
\begin{align}
\frac{\xi(\lambda;S,S_0)}{(\lambda - i)^m} = (2\pi i)^{-1} \lim_{\varepsilon\downarrow 0} \Bigg[\frac{H_+(\lambda+i\varepsilon)}{(\lambda +i\varepsilon-i)^m} - \frac{H_-(\lambda-i\varepsilon)}{(\lambda -i\varepsilon-i)^m} \Bigg]\, \text{ for a.e.~$\lambda\in \bbR$}, 
\end{align}
and hence,
\begin{align} \lb{8.63}
\xi(\lambda;S,S_0) &= (2\pi i)^{-1} \lim_{\varepsilon\downarrow 0} [H_+(\lambda-i\varepsilon)-H_-(\lambda+i\varepsilon)]  \, \text{ for a.e.~$\lambda\in \bbR$}.
\end{align} 
It follows from the definition of the functions $F_{S,S_0}$ and $G_{S,S_0}$ (cf.~Lemmas \ref{l8.8} and \ref{l8.9}) that
\begin{equation}
F_{S,S_0}(\bar{z})=\overline{F_{S,S_0}(z)},\quad G_{S,S_0}(\bar{z})=\overline{G_{S,S_0}(z)},\quad z\in \bbC\backslash \bbR.
\end{equation}
Thus, by \eqref{8.63}, 
\begin{align}
\xi(\lambda;S,S_0)&= (2\pi i)^{-1} \lim_{\varepsilon\downarrow 0} [F_{S,S_0}(\lambda+i\varepsilon)-F_{S,S_0}(\lambda-i\varepsilon)]   \no \\
&\quad - (2\pi i)^{-1} \lim_{\varepsilon\downarrow 0} [G_{S,S_0}(\lambda+i\varepsilon) 
+ G_{S,S_0}(\lambda-i\varepsilon)]     \no\\
&\quad - (2\pi i)^{-1} [P_{+,m-1}(\lambda)-P_{-,m-1}(\lambda)]      \no\\
&= \pi^{-1} \Im (F_{S,S_0}(\lambda+i0)) - \pi^{-1} \Im (G_{S,S_0}(\lambda+i0))+P_{m-1}(\lambda)   \\
& \hspace*{7.05cm} \text{ for a.e.~$\lambda\in \bbR$},  \no
\end{align}
where the polynomial $P_{m-1}:= (2\pi i)^{-1}[P_{-,m-1} - P_{+,m-1}]$ has degree less than or equal to $m - 1$ (since 
$P_{\pm,m-1}$ have degree less than or equal to $m - 1$).
\end{proof}

\section{Analysis of $\Im(F_{H,H_0}(\lambda + i0))$, $\lambda \in \bbR$} \lb{s9}

The principal purpose of this section is to analyze continuity properties of the function $\Im(F_{H,H_0}(\lambda+i0))$,  $\lambda\in \bbR$. 

One recalls (see Lemma \ref{l8.8}) that 
\begin{equation} 
F_{H,H_0}(z)=\ln \big({\det}_{[L^2(\bbR^n)]^N, n+1}\big(I_{[L^2(\bbR^n)]^N} 
+ V(H_0 -zI_{[L^2(\bbR^n)]^N})^{-1}\big)\big), \quad z\in \bbC_+.  
\end{equation}
Using a factorization $V=V_1^*V_2$ (see Hypothesis \ref{h9.5} for the details of the factorization) and elementary properties of regularized determinants, the analysis of the function $F_{H,H_0}(z)$ reduces to an analysis of 
\begin{equation} 
\ln \big({\det}_{[L^2(\bbR^n)]^N, n+1}\big(I_{[L^2(\bbR^n)]^N}+ V_2(H_0-zI_{L_2(\bbR^n)]^N})^{-1}V_1^*\big)\big), 
\quad z\in \bbC_+. 
\end{equation} 
Theorem \ref{t6.9} then guarantees that the Birman--Schwinger-type operator 
\begin{equation} 
V_2(H_0-zI_{L_2(\bbR^n)]^N})^{-1}V_1^*, \quad z \in \bbC_+,  
\end{equation} 
extends to a continuous $\cB_{n+1}\big([L^2(\bbR^n)]^N\big)$-valued function for $z$ in the closed upper-half plane 
$\ol{\bbC_+}$, provided that $V_1$ and $V_2$ are decaying sufficiently fast. In particular, the boundary values of  the regularized Fredholm determinant, 
\begin{equation}
{\det}_{[L^2(\bbR^n)]^N, n+1}\big(I_{[L^2(\bbR^n)]^N} 
+ \ol{V_2 (H_0 - (\lambda + i 0) I_{[L^2(\bbR^n)]^N})^{-1} V_1^*}\big), 
\end{equation}
exist and are continuous for all $\lambda\in \bbR$. This means that the function $F_{H,H_0}(\lambda+i0)$ has normal boundary values and is continuous at any point $\lambda$ in $\bbR$ such that 
\begin{equation} 
{\det}_{[L^2(\bbR^n)]^N, n+1}\big(I_{[L^2(\bbR^n)]^N} + \ol{V_2 (H_0 - (\lambda + i 0) I_{[L^2(\bbR^n)]^N})^{-1} V_1^*}\big)\neq 0.
\end{equation}  

By Theorem \ref{t3.4}, the latter holds for $\lambda \in \bbR \backslash \{0\}$ if and only if $\lambda$ is not an eigenvalue of $H$. By \cite{KOY15} one can exclude nonzero eigenvalues by assuming \eqref{4.3A} and \eqref{4.4} (see Theorem \ref{t4.1}).  In particular, under these assumptions the function $\Im(F_{H,H_0}(\lambda+i0))$ is continuous for 
$\lambda \in \bbR \backslash\{0\}$. Thus, the only point where the behavior $\Im(F_{H,H_0}(\lambda+i0))$, $\lambda\in \bbR$, remains to be studied is the ``threshold point''  $\lambda=0$, and hence the majority of this section is devoted to an analysis of the latter.

We start with a series of well-known preliminary results which we state without proof closely following the general outline in the paper by Jensen and Nenciu \cite{JN01}.

\begin{lemma} \lb{l9.1} ${}$ \\[1mm]
$(i)$ $($cf.\ \cite{JN01}$)$. Let $A$ be a densely defined closed operator and $P$ a projection in $\cH$. Suppose that $(A+P)^{-1} \in \cB(\cH)$ and denote by $a := P - P(A+P)^{-1}P$ an operator in $P \cH$. Then 
\begin{equation}
A^{-1} \in \cB(\cH) \, \text{ if and only if } \, a^{-1} \in \cB(P \cH).   \lb{9.1}
\end{equation}
In particular, if $a^{-1} \in \cB(P \cH)$ then
\begin{equation}
A^{-1} = (A+P)^{-1} + (A+P)^{-1} P a^{-1}P(A+P)^{-1}. 
\end{equation}
$(ii)$ $($cf.\ \cite{GH87}, \cite[Sect.~III.6.5]{Ka80}$)$. Let $A$ be a densely defined closed operator in $\cH$ and 
$\lambda_0 \in \bbC$ an isolated point in $\sigma(A)$ with $P_{\lambda_0}$ the Riesz projection in $\cH$ 
associated with $A$ and $\lambda_0$. If the quasi-nilpotent operator associated with $A$ and $\lambda_0$ vanishes, that is, $D_0 := (A - \lambda_0 I_{\cH}) P_{\lambda_0} =0$, then 
\begin{equation}
(A - \lambda_0 I_{\cH} + P_{\lambda_0})^{-1} = P_{\lambda_0} + S_{\lambda_0} \in \cB(\cH),
\end{equation} 
where
\begin{equation}
S_{\lambda_0} = \nlim_{\substack{z \to \lambda_0 \\ z \neq \lambda_0}} (A - z I_{\cH})^{-1} [I_{\cH} - P_{\lambda_0}] \in \cB(\cH). 
\end{equation}
$(iii)$ $($cf.\ \cite{Ra82}$)$. Let $A$ be a compact operator in $\cH$ and $\lambda_0 \in \bbC$ an isolated point in $\sigma(A)$ with $P_{\lambda_0}$ the Riesz projection in $\cH$ 
associated with $A$ and $\lambda_0$. Then 
\begin{equation}
(A - \lambda_0 I_{\cH} + P_{\lambda_0})^{-1} \in \cB(\cH). 
\end{equation}
$(iv)$ $($cf.\ \cite{JN01}$)$. Suppose that $\cH = \cH_1 \oplus \cH_2$ and $\gB$ in $\cH$ is the block operator 
matrix 
\begin{equation}
\gB = \begin{pmatrix} b_{1,1} & b_{1,2} \\ b_{2,1} & b_{2,2} \end{pmatrix}, 
\end{equation}
where 
\begin{align}
\begin{split} 
& b_{j,j} \, \text{ are densely defined, closed operators in $\cH_j$, $j=1,2$,}  \\
& b_{1,2} \in \cB(\cH_2,\cH_1), \quad b_{2,1} \in \cB(\cH_1,\cH_2).
\end{split} 
\end{align}
In addition, assume that $b_{2,2}^{-1} \in \cB(\cH_2)$. Then
\begin{equation}
\gB^{-1} \in \cB(\cH) \, \text{ if and only if } \, \big[b_{1,1} - b_{1,2} \, b_{2,2}^{-1} \, b_{2,1}\big]^{-1} \in \cB(\cH_1).     \lb{10.8} 
\end{equation} 
In particular, abbreviating 
\begin{equation} 
b := \big[b_{1,1} - b_{1,2} \, b_{2,2}^{-1} \, b_{2,1}\big], 
\end{equation} 
if $b^{-1} \in \cB(\cH_1)$,  
then
\begin{equation}
\gB^{-1} = \begin{pmatrix} b^{-1} & - b^{-1} \, b_{1,2} \, b_{2,2}^{-1} \\
- b_{2,2}^{-1} \, b_{2,1} \, b^{-1} & b_{2,2}^{-1} 
+ b_{2,2}^{-1} \, b_{2,1} \, b^{-1} \, b_{1,2} \, b_{2,2}^{-1}\end{pmatrix}.    \lb{10.10} 
\end{equation}
\end{lemma}

We emphasize that $b$ in Lemma \ref{l9.1}\,$(iv)$ is also known as a Schur complement (see, e.g., \cite[Sect.~1.6]{Tr08}) and formula \eqref{10.10} is a variant of the so-called Feshbach formula (see, e.g., \cite{DJ01}). In particular, Lemma \ref{l9.1}\,$(iv)$ is especially useful in the context of two-dimensional Schr\"odinger operators (cf.\ \cite{JN01}) as well as two-dimensional massless Dirac operators (cf.\ \cite{EGG20}). 

\begin{lemma}[{\cite{JN01}}] \lb{l9.2} Suppose that $\Omega \subset \bbC$ has zero as an accumulation point. \\[1mm]
$(i)$ Let $A(\zeta) = A_0 + \zeta A_1(\zeta)$, $\zeta \in \Omega$, be a family of $\cB(\cH)$-valued operators, 
with $A_1(\, \cdot \,)$ uniformly bounded for $\zeta \in \Omega$ sufficiently small. Suppose that $0$ is an 
isolated point in $\sigma(A_0)$ and denote by $P_0$ the Riesz projection in $\cH$ associated with $A_0$ 
and $0$. If $A_0 P_0 = 0$ $($i.e., the quasi-nilpotent operator associated with $A_0$ and $0$ vanishes\,$)$, 
then for $\zeta \in \Omega$ sufficiently small, the operator $B(\, \cdot \,)$ in $P_0 \cH$, defined by 
\begin{equation}
B(\zeta) := \zeta^{-1} \big\{P_0 - P_0 [A(\zeta) + P_0]^{-1} P_0\big\} 
= \sum_{j \in \bbN_0} (- \zeta)^j P_0 \big[A_1(\zeta) (A_0 + P_0)^{-1}\big]^{j+1} P_0, 
\end{equation}
is uniformly bounded as $\zeta \to 0$. Moreover, for $\zeta \in \Omega$ sufficiently small, 
\begin{equation} 
A(\zeta)^{-1} \in \cB(\cH) \, \text{ if and only if } \, B(\zeta)^{-1} \in \cB(P_0 \cH). 
\end{equation}
In particular, if $B(\zeta)^{-1} \in \cB(P_0 \cH)$ for $\zeta \in \Omega$ sufficiently small, then 
\begin{equation}
A(\zeta)^{-1} = [A(\zeta) + P_0]^{-1} + \zeta^{-1} [A(\zeta) + P_0]^{-1} P_0 B(\zeta)^{-1} P_0 [A(\zeta) + P_0]^{-1}.     \lb{10.13} 
\end{equation}
\end{lemma}

\begin{remark} \lb{r9.3} 
A combined application of Lemma \ref{l9.1}\,$(iv)$ and Lemma \ref{l9.2} can be realized in the following 
scenario: Suppose
\begin{align}
& b_{1,1}(\zeta) = \zeta^{-1} [b_0 + \beta(\zeta)], \, \text{ with $b_0^{-1} \in \cB(\cH_1)$ and 
$\| \beta(\zeta)\|_{\cB(\cH_1)} \underset{\substack{\zeta \to 0 \\ \zeta \in \Omega}}{=} \oh(1)$},  \no \\
& b_{2,2}(\zeta)^{-1} \in \cB(\cH_2) \, \text{ is uniformly bounded for $\zeta \in \Omega$ sufficiently small,} \\
& b_{1,2}(\zeta) \in \cB(\cH_2,\cH_1), \quad b_{2,1}(\zeta) \in \cB(\cH_1,\cH_2) \, 
\text{ are uniformly bounded for $\zeta \in \Omega$}    \no \\
& \quad \text{sufficiently small.}     \no 
\end{align}
Then 
\begin{equation}
b_{1,1}(\zeta)^{-1} = \zeta b_0^{-1} \big[I_{\cH_1} + \beta(\zeta) b_0^{-1}\big]^{-1}, \quad 
\big\|b_{1,1}(\zeta)^{-1}\big\| \underset{\substack{\zeta \to 0 \\ \zeta \in \Omega}}{=} \Oh(\zeta), 
\end{equation}
and under these circumstances one then infers, with 
\begin{equation} 
b(\zeta) := \big[b_{1,1}(\zeta) - b_{1,2}(\zeta) \, b_{2,2}^{-1}(\zeta) \, b_{2,1}(\zeta)\big] 
\end{equation} 
(cf.\ \eqref{10.8}), that for $\zeta \in \Omega$ sufficiently small, 
\begin{equation}
b(\zeta)^{-1} = b_{1,1}(\zeta)^{-1} \big[I_{\cH_1} - b_{1,2}(\zeta) \, b_{2,2}(\zeta)^{-1} \, b_{2,1}(\zeta) \, b_{1,1}(\zeta)^{-1}\big]^{-1}.  
\end{equation} 
\end{remark}
${}$ \hfill $\diamond$

At this point one can summarize the strategy in deriving threshold expansions of resolvents described in Jensen and Nenciu \cite{JN01} (see also Murata \cite{Mu82}), in fact, in our context, expansions of  
\begin{align}
\begin{split} 
& V_2 (H - z I_{[L^2(\bbR^n)]^N})^{-1} V_1^* =  I_{[L^2(\bbR^n)]^N}     \\
& \quad - 
\big[I_{[L^2(\bbR^n)]^N} + V_2 (H_0 - z I_{[L^2(\bbR^n)]^N})^{-1} V_1^*\big]^{-1}, \quad z \in \bbC \backslash \bbR,    \lb{9.17} 
\end{split} 
\end{align}
in terms of the (symmetrized) Birman--Schwinger-type operator 
\begin{equation}
V_2 (H_0 - z I_{[L^2(\bbR^n)]^N})^{-1} V_1^*     \lb{9.18}
\end{equation} 
(cf.\ Theorem \ref{t3.4}) around $z=0$ as follows: \\[1mm]
$(\alpha)$ One notes upon combining \eqref{5.8}--\eqref{5.8f} and \eqref{5.17} that treating even dimensions $n$ is considerably more involved than the case of odd dimensions $n$ due to the presence of the logarithm\footnote{This is even more pronounced in the case of Schr\"odinger operators for $n=2$ due to the logarithmic blowup of the Green's function \eqref{5.2} as $z \to 0$. Actually, in the Schr\"odinger context even the one-dimensional case exhibits a $z^{-1/2}$ singularity at $z=0$, rendering both cases more involved than $n \geq 3$. Since the Dirac Green's matrix never exhibits a blowup as $z \to 0$ in all dimensions $n \in \bbN$, $n \geq 2$ (cf.\ \eqref{5.18}), this renders the massless Dirac situation technically a bit simpler than the case of one and two-dimensional Schr\"odinger operators (considered in great detail in \cite{JN01}).} in \eqref{5.8c}. At any rate, formulas \eqref{5.8}--\eqref{5.8f} and \eqref{5.17} permit one to expand the Birman--Schwinger-type operator 
\eqref{9.18} around $z=0$ assuming sufficient decay of $V_1^*(x), V_2(x)$ as 
$|x| \to \infty$. This step is cumbersome, but poses no further difficulties. What might cause difficulties is an expansion of the left-hand side of \eqref{9.17}, or, equivalently, an expansion of the inverse 
$\big[I_{[L^2(\bbR^n)]^N} + V_2 (H_0 - z I_{[L^2(\bbR^n)]^N})^{-1} V_1^*\big]^{-1}$ on the right-hand side of 
\eqref{9.17}. \\[1mm]
$(\beta)$ If this inverse exists boundedly in a sufficiently small neighborhood of $z=0$, that is, if 
\begin{equation}
\big[I_{[L^2(\bbR^n)]^N} + V_2 (H_0 - z I_{[L^2(\bbR^n)]^N})^{-1} V_1^*\big]^{-1} 
\in \cB\big([L^2(\bbR^n)]^N\big)     \lb{9.19} 
\end{equation}
for $|z|$ sufficiently small, then no difficulty arises and a geometric series argument yields the existence of such an expansion in norm (cf.\ Section \ref{s5}), given sufficient decay of $V_1^*(x), V_2(x)$ as $|x| \to \infty$ also in appropriate trace ideal norms (cf.\ the detailed treatment in Section \ref{s6}). This is actually the generic case where $H$ has no zero-energy eigenvalue and no zero-energy resonance (the latter is defined as giving rise to an eigenvalue $-1$ of the Birman--Schwinger-type operator \eqref{9.18} but with no associated 
$L^2$-eigenfunction in the domain of $H$). At this point all that remains is a computation of the expansion coefficients, but the latter is of limited urgency in our present context as we will primarily rely on the leading order in all expansions. 
\\[1mm]
$(\gamma)$ If the inverse in \eqref{9.19} does not exist boundedly in a sufficiently small neighborhood of $z=0$, that is, if the compact operator $V_2 (H_0 - z I_{[L^2(\bbR^n)]^N})^{-1} V_1^*$ has an eigenvalue $-1$, the situation changes drastically. In this case $H$ either has an eigenvalue $0$, or zero-energy resonances, or possibly both, a zero-energy eigenvalue and zero-energy resonances (all of them possibly degenerate) in the worst case scenario. In any of these (exceptional) situations the norm of 
\begin{equation}
\big[I_{[L^2(\bbR^n)]^N} + V_2 (H_0 - z I_{[L^2(\bbR^n)]^N})^{-1} V_1^*\big]^{-1} 
\end{equation}
and hence that of 
\begin{equation} 
V_2 (H - z I_{[L^2(\bbR^n)]^N})^{-1} V_1^* 
\end{equation} 
will exhibit a singularity as $z \to 0$. Without going into details in this summary (see, however, Theorem \ref{t9.11}), we note that the blowup in the case of zero-energy eigenvalue(s) is of the order $z^{-1}$, and in the presence of zero-energy resonances (but no zero-energy eigenvalues) is of a less singular structure, for instance, like $z^{-1} [\ln(z)]^{-1}$, $z^{-1/2}$, or $\ln(z)$, etc., the details now depending crucially on the space dimension $n \in \bbN$, $n \geq 2$, and whether Schr\"odinger or Dirac operators (massive or massless) are considered.   

But even though $\big[I_{[L^2(\bbR^n)]^N} + V_2 (H_0 - z I_{[L^2(\bbR^n)]^N})^{-1} V_1^*\big]$ does not possess a bounded inverse as $z \to 0$, the operator 
\begin{equation}
\big[I_{[L^2(\bbR^n)]^N} + V_2 (H_0 - z I_{[L^2(\bbR^n)]^N})^{-1} V_1^* + P_0\big],
\end{equation}
where $P_0$ is the (finite-dimensional) Riesz projection associated with the operator 
\begin{equation} 
I_{[L^2(\bbR^n)]^N} + \ol{V_2 (H_0 - (0 + i 0) I_{[L^2(\bbR^n)]^N})^{-1} V_1^*},     \lb{9.24} 
\end{equation}
and its eigenvalue $0$, the norm limit of 
\begin{equation}
I_{[L^2(\bbR^n)]^N} + V_2 (H_0 - i \varepsilon I_{[L^2(\bbR^n)]^N})^{-1} V_1^* 
\end{equation}
as $\varepsilon \downarrow 0$, and its eigenvalue $0$, actually has a bounded inverse according to Lemma \ref{l9.1}\,$(ii)$. (Assuming compactness of the operators 
$\ol{V_2 (H_0 - (0 +i 0) I_{[L^2(\bbR^n)]^N})^{-1} V_1^*}$ as well as 
$V_2 (H_0 - i \varepsilon I_{[L^2(\bbR^n)]^N})^{-1} V_1^*$, one concludes that $\dim(\ran(P_0)) < \infty$.) Lemma \ref{l9.2} then demonstrates the key reduction step where the inverse of $A(\zeta)$ 
in $\cH$ is now reduced to the inverse of $B(\zeta)$ in the finite-dimensional Hilbert space $P_0 \cH$. \\[1mm] 
$(\delta)$ At this point one iterates the procedure ending up localizing the singularity in subspaces of decreasing dimensions. With each step the singularity is increased. However, since 
\begin{equation}
z \big[\ol{V_2 (H_0 - z I_{[L^2(\bbR^n)]^N})^{-1} V_1^*}\big]
\end{equation} 
stays bounded for $z = i \varepsilon$ as $\varepsilon \downarrow 0$, the reduction process must stop after a finite number of steps, leading to invertibility of a reduced operator so that again a geometric series argument as in step $(\beta)$ applies. This completes the process resulting in an expansion in appropriate variables involving $z$, $z^{1/2}$, $\ln(z)$, or $[c + \ln(z)]$ for appropriate $c \in \bbC\backslash \{0\}$ (again, depending on spatial dimension $n$ and whether Schr\"odinger or Dirac operators are involved). We refer once more to \cite{JN01} for the somewhat involved details (and the difficulties associated with expansions involving 
$\sum_{k=-1}^{\infty} \sum_{\ell=-\infty}^{\infty} \zeta^k [\ln(\zeta)]^{\ell}$ which cannot be asymptotic in nature) in the case of Schr\"odinger operators and to \cite{EGG20} in the case of two-dimensional massless Dirac 
operators. Much of the threshold analysis in \cite{EGG20} readily extends to dimensions $n \geq 3$ as we will see later in this section.

\begin{remark} \lb{r9.4}
In outlining steps $(\alpha)$--$(\delta)$ above, we deliberately sidestepped verifying the condition $A_0 P_0 = 0$ necessary for Lemma \ref{l9.2} to hold. The condition is equivalent to the statement that the algebraic and geometric multiplicities of the eigenvalue $0$ of $A_0$ coincide. Since by \eqref{5.18}, $G_0(0+i\,0;x,y)$ is purely imaginary, but also involves the scalar product $\alpha\cdot (x-y)$, employing the polar decomposition for the self-adjoint $N \times N$ matrix $V(\dott)$ (i.e., $V(\dott) = U_V(\dott)|V(\dott)|$) in the form (cf.\ \cite{GMMN09}) 
\begin{align}
\begin{split} 
& V(x) = |V(x)|^{1/2}U_V(x) |V(x)|^{1/2} = V_1(x)^* V_2(x) \, \text{ for a.e.~$x \in \bbR^n$,}      \\
& V_1 = V_1^* = |V|^{1/2},  \quad V_2 = U_V |V|^{1/2} = U_V V_1, \quad U_V^2 = I_N,      \lb{9.27}
\end{split} 
\end{align}
making the choice that 
\begin{equation} 
\text{$U_V$ is unitary and self-adjoint}      \lb{9.27a} 
\end{equation}
(the choice of $U_V$ is nonunique if $V$ has a kernel and we simply fix $U_V$ to be the identity operator on $\ker(V)$), the matrix-valued integral kernel 
\begin{equation}
|V(x)|^{1/2}(x) G_0(0+i\,0;x,y) |V(y)|^{1/2} 
\end{equation} 
generates a self-adjoint operator. Hence, the elegant device used in \cite{JN01} that reduces their analysis to a self-adjoint operator $A_0$ in Lemma \ref{l9.2}, so that $A_0 P_0 =0$ is automatically satisfied, applies also in the massless Dirac operator context. (Naturally, this approach of \cite{JN01} also applies in the massive case, where $H_0(m)$, $m > 0$, has the spectral gap $(-m, m)$.) In essence, Jensen and Nenciu \cite{JN01} replace  the operator   
\begin{equation}
I_{[L^2(\bbR^n)]^N} + V_2 (H_0 - z I_{[L^2(\bbR^n)]^N})^{-1} V_1^*, \quad z \in \bbC\backslash\bbR, 
\lb{9.28} 
\end{equation}
by its modification 
\begin{equation}
U_V I_{[L^2(\bbR^n)]^N} + V_1 (H_0 - z I_{[L^2(\bbR^n)]^N})^{-1} V_1^*, \quad z \in \bbC\backslash\bbR, 
\lb{9.29} 
\end{equation}
and show that the formalism displayed in \eqref{3.19}--\eqref{3.26a} instantly extends to the setup in \eqref{9.29}. 
In particular, the norm limit
\begin{equation}
U_V I_{[L^2(\bbR^n)]^N} + \ol{V_1 (H_0 - (0 + i 0) I_{[L^2(\bbR^n)]^N})^{-1} V_1^*} 
\lb{9.30} 
\end{equation}
is now self-adjoint and hence the analog of the condition 
\begin{equation} 
A_0 P_0 = 0      \lb{9.30a}
\end{equation}
thus holds automatically. 
Due to this fact we can, without loss of generality, safely disregard the distinction between \eqref{9.28} and \eqref{9.29} in much of the remainder of this manuscript. 

Finally, by an abuse of notation, we also denote the Riesz projection associated with the self-adjoint operator  \eqref{9.30} and its eigenvalue $0$ by $P_0$. Assuming compactness of the operator 
\begin{equation} 
\ol{V_1 (H_0 - (0+i0) I_{[L^2(\bbR^n)]^N})^{-1} V_1^*},      \lb{9.31a} 
\end{equation} 
the fact that $\sigma(U_V) \subseteq \{1, -1\}$ implies that zero is an isolated point in the spectrum of the operator in \eqref{9.30} and hence 
\begin{equation}
\dim(\ran(P_0)) < \infty.     \lb{9.32a} 
\end{equation} 
(In the concrete context of \eqref{9.27} one has in addition that $V_1=V_1^*$, but this simplification is not needed to conclude \eqref{9.30a} and \eqref{9.32a}.) \hfill $\diamond$ 
\end{remark}

Applying the resolvent equation \eqref{3.21}, \eqref{3.22} to the pair $H, H_0$ results in 
\begin{align}
&(H - z I_{[L^2(\bbR^n]^N})^{-1} = (H_0 - z I_{[L^2(\bbR^n]^N})^{-1} 
- \big[V_1  (H_0 - {\ol z} I_{[L^2(\bbR^n]^N})^{-1}\big]^*    \no \\ 
& \quad \times \big[I_{[L^2(\bbR^n)]^N} + V_2 (H_0 - z I_{[L^2(\bbR^n)]^N})^{-1} V_1^*\big]^{-1} 
V_2 (H_0 - z I_{[L^2(\bbR^n]^N})^{-1},    \\
& \hspace*{9.55cm} z \in \bbC \backslash \bbR,    \no
\end{align} 
To analyze the possible singularity of $(H - z I_{[L^2(\bbR^n]^N})^{-1}$ as $z \to 0$, we choose 
arbitrary
\begin{equation}
\psi_j \in C_0^{\infty}(\bbR^n) \, \text{ real-valued, } \,  j=1,2, 
\end{equation}
and consider
\begin{align}
& \psi_2 I_N (H - z I_{[L^2(\bbR^n]^N})^{-1} \psi_1 I_N =  \psi_2 I_N(H_0 - z I_{[L^2(\bbR^n]^N})^{-1}  \psi_1 I_N 
\no \\
& \quad - \big[V_1  (H_0 - {\ol z} I_{[L^2(\bbR^n]^N})^{-1}  \psi_2 I_N\big]^*   \no \\ 
& \qquad \times \big[I_{[L^2(\bbR^n)]^N} + V_2 (H_0 - z I_{[L^2(\bbR^n)]^N})^{-1} V_1^*\big]^{-1}    \no \\
& \qquad \times V_2 (H_0 - z I_{[L^2(\bbR^n]^N})^{-1}  \psi_1 I_N, \quad z \in \bbC \backslash \bbR.   
\end{align} 
As long as 
\begin{equation}
\big||V|^{1/2}_{\ell,m}(x)\big| = |V_{1,\ell,m}(x)| \leq C \langle x \rangle^{- 1} \, 
\text{ for a.e.~$x \in \bbR^n$, $1 \leq \ell,m \leq N$,}    \lb{9.31}
\end{equation}
Theorem \ref{t5.7}\,$(iii)$ implies that 
\begin{equation}
V_j  (H_0 - z I_{[L^2(\bbR^n]^N})^{-1}  \psi_{j'} I_N \in \cB\big([L^2(\bbR^n)]^N\big), \quad z \in \ol{\bbC_+}, 
\; j, j' \in \{1,2\}, 
\end{equation}
since obviously 
\begin{equation}
\psi_2 I_N(H_0 - z I_{[L^2(\bbR^n]^N})^{-1}  \psi_1 I_N \in \cB\big([L^2(\bbR^n)]^N\big), \quad z \in \ol{\bbC_+}, 
\end{equation}
(in fact, Theorem \ref{t6.6} implies trace ideal properties) one also has 
\begin{equation}
\psi_2 I_N(H - z I_{[L^2(\bbR^n]^N})^{-1}  \psi_1 I_N \in \cB\big([L^2(\bbR^n)]^N\big), \quad z \in \ol{\bbC_+}.  
\end{equation}
Thus, since $\psi_j \in C_0^{\infty}(\bbR^n)$, $j=1,2$, are arbitrary (apart from being real-valued for 
simplicity), one thus concludes that 
\begin{align}
& \psi_2 I_N (H - z I_{[L^2(\bbR^n]^N})^{-1} \psi_1 I_N \in \cB\big([L^2(\bbR^n)]^N\big) \, 
\text{ for $|z|$ sufficiently small} \no \\
& \text{if and only if}   \\
& \big[I_{[L^2(\bbR^n)]^N} + \ol{V_2 (H_0 - z I_{[L^2(\bbR^n)]^N})^{-1} V_1^*}\big]^{-1} 
\in \cB\big([L^2(\bbR^n)]^N\big)   \no \\
& \quad \text{ for $|z|$ sufficiently small.}   \no 
\end{align}

Given the extensive treatment in \cite{JN01} in the case of Schr\"odinger operators in dimensions $n \in \bbN$ (especially, in the most difficult of cases $n=1,2$), and in \cite{EGG20} in the case of massless Dirac operators in dimension $n=2$, and given the fact that dimensions $n \in \bbN$, 
$n \geq 3$, subordinate in difficulty to the case $n=2$ in the massless context, we now briefly discuss the threshold (i.e., $z=0$) behavior of massless Dirac operator in dimensions $n \geq 2$.  

We start by making the following assumptions on the matrix-valued potential $V$.

\begin{hypothesis} \lb{h9.5}
Let $n \in \bbN$, $n \geq 2$, and $\varepsilon > 0$. Assume the a.e.~self-adjoint matrix-valued potential 
$V = \{V_{\ell,m}\}_{1 \leq \ell,m \leq N}$ satisfies for some fixed $\varepsilon \in (0,1)$, $C \in (0,\infty)$, 
\begin{align}
\begin{split} 
& V \in [L^{\infty} (\bbR^n)]^{N \times N},       \\
& |V_{\ell,m}(x)| \leq C \langle x \rangle^{- 2 (1 + \varepsilon)} \, 
\text{ for a.e.~$x \in \bbR^n$, $1 \leq \ell,m \leq N$.}    \lb{9.31A}
\end{split} 
\end{align}
In accordance with the factorization based on the polar decomposition of $V$ discussed in \eqref{9.27} we 
suppose that 
\begin{equation}
V = V_1^* V_2 = |V|^{1/2} U_V |V|^{1/2}, \, \text{ where } \, 
V_1 = V_1^* = |V|^{1/2}, \; V_2 = U_V |V|^{1/2}.
\end{equation}  
\end{hypothesis} 

We continue with the threshold, that is, the $z=0$ behavior of $H$: 

\begin{definition} \lb{d9.6} Assume Hypothesis \ref{h9.5} with $\varepsilon = 0$ in \eqref{9.31A}. \\[1mm] 
$(i)$ The point $0$ is called a zero-energy eigenvalue of $H$ if $H \psi = 0$ has a distributional solution $\psi$ satisfying 
\begin{equation} 
\psi \in \dom(H) = [W^{1,2}(\bbR^n)]^N     
\end{equation}
$($equivalently, $\ker(H) \supsetneqq \{0\}$$)$. \\[1mm] 
$(ii)$ The point $0$ is called a zero-energy $($or threshold\,$)$ resonance of $H$ if 
\begin{equation}
\ker\big(\big[I_{[L^2(\bbR^n)]^N} + \ol{V_2 (H_0 - (0 + i 0) I_{[L^2(\bbR^n)]^N})^{-1} V_1^*}\big]\big)   
\supsetneqq \{0\}, 
\end{equation}
and if there exists $0 \neq \phi \in \ker\big(\big[I_{[L^2(\bbR^n)]^N} + \ol{V_2 (H_0 - (0 + i 0) I_{[L^2(\bbR^n)]^N})^{-1} V_1^*}\big]\big)$ such that $\psi$ defined by 
\begin{align}
\begin{split} 
& \psi(x) = - \big((H_0 - (0 + i 0) I_{[L^2(\bbR^n)]^N})^{-1} V_1^* \phi\big)(x)     \lb{9.46A} \\
& \hspace*{8.5mm} = - i 2^{-1} \pi^{- n/2} \Gamma(n/2) \int_{\bbR^n} d^n y \, |x-y|^{-n} 
[\alpha\cdot (x - y)] V_1(y)^* \phi(y)    
\end{split} 
\end{align}
$($for a.e.~$x \in \bbR^n$, $n \geq 2$$)$ is a distributional solution of $H u = 0$ satisfying 
\begin{equation} 
\psi \notin [L^2(\bbR^n)]^N.
\end{equation} 
$(iii)$ $0$ is called a regular point for $H$ if it is neither a zero-energy eigenvalue nor a zero-energy 
resonance of $H$.  
\end{definition}

Additional properties of $\psi$ are isolated in Theorem \ref{t9.7}.

While the point $0$ being regular for $H$ is the generic situation, zero-energy eigenvalues and/or resonances are exceptional cases. 

For future purposes we recall the asymptotic Green's matrix expansion as $z \to 0$ in the following form,
\begin{align}
G_0(z;x,y) & \underset{\substack{z \to 0 \\ z \in \ol{\bbC_+}\backslash\{0\}}}{=} 
i 2^{-1} \pi^{-n/2} \Gamma(n/2) \alpha \cdot \f{(x - y)}{|x - y|^n}    \no \\
&  \qquad \quad \; - \delta_{n,2} (2 \pi)^{-1} z \ln(z) I_N   \no \\ 
&  \qquad \quad \; + \delta_{n,2} (2\pi)^{-1} [\gamma_{E-M} - i (\pi/2) + \ln(|x - y|/2))] z I_N      \no \\ 
& \qquad \quad \; + [1 - \delta_{n,2}] (n-2)^{-1} 2^{-1} \pi^{-n/2} \Gamma(n/2) |x - y|^{2-n} z I_N      \\
&  \qquad \quad \; + \delta_{n,2}  \Oh\big(z^2 |x - y| \ln(z |x - y|)\big) + \delta_{n,3} \Oh\big(z^2\big)
+ \Oh\big(z^2 |x - y|^2\big)    \no \\
& \underset{\substack{z \to 0 \\ z \in \ol{\bbC_+}\backslash\{0\}}}{=}  
R_{0,0}(x-y) + z R_{1,0}(x-y)   \no \\
&  \qquad \quad \; + z \big[- (2 \pi)^{-1} \ln(z/2) - (2\pi)^{-1} \gamma_{E-M} +i 4^{-1}\big] \delta_{n,2} R_{1,1}(x-y)   
\no \\
&  \qquad \quad \;  + \delta_{n,2} \Oh\big(z^2 |x - y| \ln(z |x - y|)\big) + \delta_{n,3} \Oh\big(z^2\big) 
+ \Oh\big(z^2 |x - y|^2\big), 
\end{align}
where we introduced the following convenient abbreviations (for $x, y \in \bbR^n$, $x \neq y$):
\begin{align}
R_{0,0}(x-y) &= G_0(0+i0;x,y) = i 2^{-1} \pi^{-n/2} \Gamma(n/2) \, \alpha \cdot \f{(x - y)}{|x - y|^n} \no \\
& = \begin{cases} (2 \pi)^{-1} i \alpha \cdot \nabla_x \ln(|x-y|), & n=2, \\
- i \alpha \cdot \nabla_x g_0(0;x,y), & n\geq 3, 
\end{cases}    \lb{9.46} \\
R_{1,0}(x-y) &= \begin{cases} -(2 \pi)^{-1} \ln(|x-y|) I_N, & n=2,    \lb{9.47} \\
g_0(0;x,y) I_N = \f{1}{(n-2) \omega_{n-1}} |x - y|^{2-n} I_N , & n \geq 3,  
\end{cases}  \\
& \hspace*{4.55cm} \omega_{n-1} = 2 \pi^{n/2}/\Gamma(n/2),      \no \\ 
R_{1,1}(x-y) &= 1, \quad n \geq 2.    \lb{9.48} 
\end{align}

\begin{theorem} \lb{t9.7}
Assume Hypothesis \ref{h9.5} with $\varepsilon = 0$ in \eqref{9.31A}.  \\[1mm] 
$(i)$ If $n=2$, there are precisely four possible cases: \\[1mm]
Case $(I)$: $0$ is regular for $H$. \\[1mm]
Case $(II)$: $0$ is a $($possibly degenerate\footnote{We will recall in Lemma \ref{l9.10a}\,$(i)$ that if $n=2$, the degeneracy in Case $(II)$ is at most two.}\,$)$ resonance of $H$. In this case the resonance 
functions $\psi$ satisfy 
\begin{align}
\begin{split} 
& \psi \in [L^q(\bbR^2)]^2, \quad q \in (2, \infty) \cup\{\infty\}, \quad \nabla \psi \in [L^2(\bbR^2)]^{2 \times 2},   \\ 
& \psi \notin [L^2(\bbR^2)]^2.     \lb{3.60b}
\end{split} 
\end{align}
Case $(III)$: $0$ is a $($possibly degenerate\,$)$ eigenvalue of $H$. In this case the corresponding 
eigenfunctions $\psi \in \dom(H) = \big[W^{1,2}(\bbR^2)\big]^2$ of $H \psi = 0$ also satisfy
\begin{equation} 
\psi \in [L^q(\bbR^2)]^2, \quad q \in [2, \infty) \cup \{\infty\}.     \lb{3.60c} 
\end{equation} 
Case $(IV)$: A possible mixture of Cases $(II)$ and $(III)$. \\[1mm]
$(ii)$ If $n \in \bbN$, $n \geq 3$, there are precisely two possible cases: \\[1mm]
Case $(I)$: $0$ is regular for $H$. \\[1mm]
Case $(II)$: $0$ is a $($possibly degenerate\,$)$ eigenvalue of $H$. In this case, the corresponding eigenfunctions $\psi \in \dom(H) = \big[W^{1,2}(\bbR^n)\big]^N$ of $H \psi = 0$ also satisfy
\begin{equation}
\psi \in \big[L^q(\bbR^n)\big]^N, \quad q \in \begin{cases} (3/2, \infty) \cup \{\infty\}, & n=3, \\
(4/3,4), & n=4, \\
(2n/(n+2), 2n/(n-2)), & n \geq 5.
\end{cases}     \lb{9.55A} 
\end{equation}
In particular, there are no zero-energy resonances of $H$ in dimension $n \geq 3$. \\[1mm]  
$(iii)$ The point $0$ is regular for $H$ if and only if 
\begin{equation}
\ker\big(\big[I_{[L^2(\bbR^n)]^N} + \ol{V_2 (H_0 - (0 + i 0) I_{[L^2(\bbR^n)]^N})^{-1} V_1^*}\big]\big) 
= \{0\}. 
\end{equation}
\end{theorem}
\begin{proof}
Since $G_0(0+i0;x,y)$, $x\neq y$, exists for all $n \geq 2$ (cf.\ \eqref{5.18}), the 
Birman--Schwinger eigenvalue equation (cf. \eqref{9.24}) 
\begin{equation} 
\big[I_{[L^2(\bbR^n)]^N} + \ol{V_2 (H_0 - (0 +i 0) I_{[L^2(\bbR^n)]^N})^{-1} V_1^*}\big] \phi_0 = 0,  \quad 
0 \neq \phi_0 \in [L^2(\bbR^n)]^N,   \lb{9.36} 
\end{equation}
gives rise to a distributional zero-energy solution $\psi_0 \in [L^1_{\loc}(\bbR^n)]^N$ of $H \psi_0 = 0$ in 
terms of $\phi_0$ of the form (for a.e.~$x \in \bbR^n$, $n \geq 2$),
\begin{align}
& \psi_0(x) = - \big((H_0 - (0 + i 0) I_{[L^2(\bbR^n)]^N})^{-1} V_1^* \phi_0\big)(x)     \no \\
& \hspace*{8.5mm} = - [R_{0,0} * (V_1^* \phi_0)](x)    \lb{9.36a} \\
& \hspace*{8.5mm} = - i 2^{-1} \pi^{- n/2} \Gamma(n/2) \int_{\bbR^n} d^n y \, |x-y|^{-n} [\alpha\cdot (x - y)] V_1(y)^* \phi_0(y),   \lb{9.37} \\
& \hspace*{8.5mm} = - i 2^{-1} \pi^{- n/2} \Gamma(n/2) \int_{\bbR^n} d^n y \, |x-y|^{-n} 
[\alpha\cdot (x - y)] V_1(y)^* V_2 (y) \psi_0(y),   \lb{9.37a} \\
& \phi_0(x) = (V_2 \psi_0)(x).       \lb{9.38}
\end{align}
In particular, one concludes that $\psi_0 \neq 0$. Thus, one estimates, 
with $\|V_1(\dott)\|_{\bbC^{N \times N}} \leq c \langle \dott \rangle^{-1}$ and some constant $d_n \in (0, \infty)$, 
\begin{align}
\begin{split}
\|\psi_0(x)\|_{\bbC^N} &\leq d_n \int_{\bbR^n} d^n y \, |x - y|^{1-n} \langle y \rangle^{-1}
\|\phi_0(y)\|_{\bbC^N}   \lb{9.41} \\
& = d_n \cR_{1,n}\big(\langle \dott \rangle^{-1} \|\phi_0(\dott)\|_{\bbC^N}\big)(x),  \quad x \in \bbR^n.  
\end{split}
\end{align}
Invoking the Riesz potential $\cR_{1,n}$ (cf.\ Theorem \ref{t5.4A}), 
one obtains (for some constant $\wti d_n \in (0,\infty)$)
\begin{align}
\|\psi_0(x)\|_{\bbC^N} & \leq d_n \int_{\bbR^n} d^n y \, |x - y|^{1-n} \langle y \rangle^{-1} \|\phi_0(y)\|_{\bbC^N}  \no \\
& \leq \wti d_n \cR_{1,n} \big(\langle \dott \rangle^{-1} \|\phi_0(\dott)\|_{\bbC^N}\big)(x),  \quad x \in \bbR^n, 
\lb{9.65A} 
\end{align}
and hence \eqref{5.51} implies (for some constant $\wti C_{p,q,n} \in (0,\infty)$)  
\begin{align}
\|\psi_0\|_{[L^q(\bbR^n)]^N} & 
\leq \wti d_n \big\|\cR_{1,n} \big(\langle \dott \rangle^{-1} \|\phi_0(\dott)\|_{\bbC^N}\big)
\big\|_{L^q(\bbR^n)}    \no \\
& \leq \wti C_{p,q,n} \big\|\langle \dott \rangle^{-1} \|\phi_0(\dott)\|_{\bbC^N}\big\|_{L^p(\bbR^n)}   \no \\
& \leq \wti C_{p,q,n} \big\|\langle \dott \rangle^{-1} \big\|_{L^s(\bbR^n)} \|\|\phi_0(\dott)\|_{\bbC^N}\|_{L^2(\bbR^n)} 
\no \\
& = \wti C_{p,q,n} \big\|\langle \dott \rangle^{-1} \big\|_{L^s(\bbR^n)} \|\phi_0\|_{[L^2(\bbR^n)]^N},   \lb{9.61} \\
& \hspace*{-3.4cm} 1 < p < q < \infty, \; p^{-1} = q^{-1} + n^{-1}, \; s = 2qn [2n + 2q - qn]^{-1} \geq 1.  \no
\end{align}
In particular, 
\begin{equation}
p = qn/(n + q), \quad 2n + 2q - qn > 0. 
\end{equation}
\noindent 
$(a)$ The case $n=2$: Then one can choose $q \in (2,\infty)$, hence  
$p = 2q/(q+2) \in (1,2)$, and $s = q > 2$. Thus, \eqref{9.61} and 
$\big\|\langle \dott \rangle^{-1}\big\|_{L^s(\bbR^2)} < \infty$ imply 
\begin{equation}
\psi_0 \in [L^q(\bbR^2)]^N, \quad q \in (2,\infty). 
\end{equation}

Recalling $R_{0,0}(x-y)$ in \eqref{9.46}, this implies 
\begin{align}
\begin{split} 
- i \alpha \cdot \nabla_x R_{0,0}(x-y) = - \Delta_x g_{0}(0;x,y) I_N = \delta(x-y) I_N,&   \lb{9.73}\\
x, y \in \bbR^n, \; x \neq y, \; n \geq 2,&
\end{split} 
\end{align}
in the sense of distributions. Here we abused notation a bit  and denoted also in the case $n=2$, 
\begin{equation}
g_0(0;x,y) = - (2 \pi)^{-1} \ln(|x-y|), \quad x, y \in \bbR^2, \; x \neq y, \; n=2.    \lb{9.74} 
\end{equation} 
Thus, one obtains 
\begin{align}
& i \alpha \cdot (\nabla \psi_0)(x) = 
- i \alpha \cdot \nabla_x [R_{0,0} * (V_1^* \phi_0)](x) = - i \alpha \cdot \nabla_x [- i (\alpha \cdot \nabla_x g_0 * (V_1^* \phi_0)](x)   \no \\
& \quad = [(- \Delta_x g_{0} I_N) * (V_1^* \phi_0)](x) = (V_1^* \phi_0)(x)  \in [L^2(\bbR^2)]^2,   \lb{9.75} 
\end{align}
proving $\nabla \psi_0 \in [L^2(\bbR^2)]^{2 \times 2}$, upon employing the fact that $[\alpha \cdot p]^2 = I_N |p|^2$, 
$p \in \bbR^n$.    

To prove that $\psi_0\in [L^{\infty}(\bbR^2)]^2$ in \eqref{3.60b} and \eqref{3.60c}, one applies \eqref{9.38} to the inequality in \eqref{9.41}, and then employs the condition $\|V_2(\dott)\|_{\bbC^{2\times 2}}\leq C\langle \dott\rangle^{-1}$ for some constant $C\in (0,\infty)$ to obtain
\begin{equation}
\|\psi_0(x)\|_{\bbC^2} \leq \widetilde{d}_2 \int_{\bbR^2}d^2y\, |x-y|^{-1}\langle y\rangle^{-2}\|\psi_0(y)\|_{\bbC^2},\quad x\in \bbR^2,
\end{equation}
where $\widetilde{d}_2\in (0,\infty)$ is an appropriate $x$-independent constant.  By H\"older's inequality,
\begin{equation} \lb{3.81b}
\|\psi_0(x)\|_{\bbC^2} \leq \widetilde{d}_2\bigg(\int_{\bbR^2}d^2y\, |x-y|^{-3/2}\langle y\rangle^{-3}\bigg)^{2/3} \bigg(\int_{\bbR^2}d^2y\, \|\psi_0(y)\|_{\bbC^2}^3 \bigg)^{1/3},\quad x\in \bbR^2.
\end{equation}
The second integral on the right-hand side in \eqref{3.81b} is finite since $\psi_0 \in [L^3(\bbR^2)]^2$.  Choosing $x_1=x$, $\alpha = n - (3/2)$, $\beta=n$, $\gamma=2$, and $\varepsilon = 1$ in Lemma \ref{l3.12}, one infers that
\begin{equation} \lb{3.82b}
\int_{\bbR^2}d^2y\, |x-y|^{-3/2}\langle y\rangle^{-3} \leq C_{2,3/2,0,2,1},\quad x\in \bbR^2. 
\end{equation}
Hence, the containment $\psi_0\in [L^{\infty}(\bbR^2)]^2$ follows from \eqref{3.81b} and \eqref{3.82b}.

\noindent 
$(b)$ The case $n \geq 3$: An application of Theorem \ref{t5.6}\,$(ii)$ with $c=0$, $d=1$, $p=p'=2$, and the inequality $1 < n/2$, combined with 
$\|\phi_0(\, \cdot \,)\|_{\bbC^N} \in L^2(\bbR^n)$, yield 
\begin{equation}
\|\psi_0(\, \cdot \,)\|_{\bbC^N} \in L^2(\bbR^n) \, \text{ and hence, } \, \psi_0 \in [L^2(\bbR^n)]^N, 
\quad n \geq 3.     \lb{9.42} 
\end{equation} 
To prove that actually $\psi_0 \in \dom(H) = [W^{1,2}(\bbR^n)]^N$, it suffices to argue as follows:  
\begin{equation}
i \alpha \cdot \nabla \psi_0 = - V \psi_0 \in [L^2(\bbR^n)]^N    \lb{9.43}
\end{equation}     
in the sense of distributions since $V \in [L^{\infty} (\bbR^n)]^{N \times N}$ and $\psi_0 \in [L^2(\bbR^n)]^N$.  
Given the fact $\dom(H_0) = \big[W^{1,2}(\bbR^n)\big]^N$ (cf.\ \eqref{2.2}), one concludes  
\begin{equation}
\psi_0 \in \big[W^{1,2}(\bbR^n)\big]^N, \quad n \geq 3.     \lb{9.45}
\end{equation}

By \eqref{9.45} we know that $\psi_0 \in \big[W^{1,2}(\bbR^n)\big]^N$. Employing 
the fact that $\phi_0 = V_2 \psi_0$ in the first line of \eqref{9.65A}, one obtains 
\begin{align}
\begin{split} 
\|\psi_0(x)\|_{\bbC^N} & \leq \wti D_n \int_{\bbR^n} d^n y \, |x - y|^{1-n} \langle y \rangle^{-2} \|\psi_0(y)\|_{\bbC^N}   \\
& =D_n \cR_{1,n} \big(\langle \dott \rangle^{-2} \|\psi_0(\dott)\|_{\bbC^N}\big)(x),  \quad x \in \bbR^n, 
\lb{9.77A} 
\end{split} 
\end{align}
for some constants $\wti D_n, D_n \in (0,\infty)$. Thus, as in \eqref{9.61}, \eqref{5.51} implies for $n \geq 3$, 
\begin{align}
\|\psi_0\|_{[L^q(\bbR^n)]^N} & \leq D_n 
\big\|\cR_{1,n} \big(\langle \dott \rangle^{-2} \|\psi_0(\dott)\|_{\bbC^N}\big)
\big\|_{L^q(\bbR^n)}    \no \\
& \leq \wti D_{p,q,n} \big\|\langle \dott \rangle^{-2} \|\psi_0(\dott)\|_{\bbC^N}\big\|_{L^p(\bbR^n)}   \no \\
& \leq \wti D_{p,q,n} \big\|\langle \dott \rangle^{-2} \big\|_{L^s(\bbR^n)} \|\|\psi_0(\dott)\|_{\bbC^N}\|_{L^2(\bbR^n)} 
\no \\
& = \wti D_{p,q,n} \big\|\langle \dott \rangle^{-2}\big\|_{L^s(\bbR^n)} \|\psi_0\|_{[L^2(\bbR^n)]^N},   \lb{9.78A} \\
& \hspace*{-2.75cm} 1 < p < q < \infty, \; p^{-1} = q^{-1} + n^{-1}, \; s = 2qn [2n + 2q - qn]^{-1} \geq 1,   \no
\end{align}
for some constant $\wti D_{p,q,n} \in (0,\infty)$. In particular, one again has $p = qn/(n + q)$ and 
$2n + 2q - qn > 0$. The latter condition implies $q < 2n/(n-2)$. The requirement $p > 1$ results in $q > n/(n-1)$, and the condition $s \geq 1$ yields $q \geq 2n/(3n-2)$ which, however, is superseded by $q > p > 1$. 
Moreover, the requirement $\big\|\langle \dott \rangle^{-2}\big\|_{L^s(\bbR^n)} < \infty$ yields $q > 2n/(n+2)$. 
Putting it all together implies \eqref{9.55A}. 

To prove the containment $\psi\in [L^{\infty}(\bbR^3)]^4$ in \eqref{9.55A}, one invokes the inequality in \eqref{9.77A} with $n=3$.  Indeed, applying H\"older's inequality (with conjugate exponents $q'=27/20$ and $q=27/7$) to the integral on the right-hand side of the inequality in \eqref{9.77A}, one infers that
\begin{align}
& \|\psi_0(x)\|_{\bbC^4} \leq d_3\bigg(\int_{\bbR^3}d^3y\, |x-y|^{-27/10}\langle y\rangle^{-27/10} 
\bigg)^{20/27} 
\bigg(\int_{\bbR^3}d^3y\, \|\psi_0(y)\|_{\bbC^4}^{2/77} \bigg)^{7/27},\no \\
& \hspace*{9.7cm}  x\in \bbR^3.     \lb{3.86b}
\end{align}
The second integral in \eqref{3.86b} is finite since $\psi_0\in [L^{27/7}(\bbR^3)]^4$, and the first integral in \eqref{3.86b} may be estimated by taking $x_1=x$, $\alpha = n - (27/10)$, $\beta = n$, $\gamma=2$, and 
$\varepsilon = 7/10$ in Lemma \ref{l3.12}, 
\begin{equation} \lb{3.87b}
\int_{\bbR^3}d^3y\, |x-y|^{-\frac{27}{10}}\langle y\rangle^{-\frac{27}{10}}\leq C_{3,\frac{27}{10},0,2,7/10}, 
\quad x\in \bbR^3\backslash\{0\}.
\end{equation}
Hence, the containment $\psi_0\in [L^{\infty}(\bbR^3)]^4$ follows from \eqref{3.86b} and \eqref{3.87b}.


Returning to arbitrary $n\geq 2$, we show (following the proof of \cite[Lemma~7.4]{EGG20}) that if $\ker(H) \supsetneqq \{0\}$ then also 
\begin{equation} 
\ker \big(\big[I_{[L^2(\bbR^n)]^N} + \ol{V_2 (H_0 - (0 + i0) I_{[L^2(\bbR^n)]^N})^{-1} V_1^*}\big]\big) 
\supsetneqq \{0\}. 
\end{equation}
Indeed, if $0 \neq \psi_0 \in \ker(H)$, then 
$\phi_0 := V_2 \psi_0 = U_V V_1 \psi_0 \in [L^2(\bbR^n)]^N$ and hence 
$V_1^* \phi_0 \in [L^2(\bbR^n)]^N$. Then, $H \psi_0 = 0$ yields 
$i \alpha \cdot \nabla \psi_0 = V \psi_0 = V_1^* V_2 \psi_0 = V_1^* \phi_0$. 

Thus, applying \eqref{9.46}, \eqref{9.73}--\eqref{9.74} once again, one obtains for all $n \geq 2$, 
\begin{align}
& - i \alpha \cdot \nabla \big[\psi_0 + (H_0 - (0 + i 0) I_{[L^2(\bbR^n)]^N})^{-1} V_1^* \phi_0\big] (x)    \no \\
& \quad = - i [\alpha \cdot \nabla \psi_0](x) 
- i \alpha \cdot \nabla_x [R_{0,0} * (V_1^* \phi_0)](x)     \no \\
& \quad = - i [\alpha \cdot \nabla \psi_0](x) 
- i \alpha \cdot \nabla_x [- i (\alpha \cdot \nabla_x g_0 * (V_1^* \phi_0)](x)    \no \\
& \quad = - i [\alpha \cdot \nabla \psi_0](x) + (- \Delta_x g_{0} I_N) * (V_1^* \phi_0)](x)    \no \\
& \quad = - i [\alpha \cdot \nabla \psi_0](x) + (V_1^* \phi_0)(x)    \no \\
& \quad = - V(x) \psi_0(x) + V(x) \psi_0(x) = 0. 
\end{align}
Consequently,  
\begin{equation}
- i \alpha \cdot \nabla [\psi_0 + (H_0 - (0 + i 0) I_{[L^2(\bbR^n)]^N})^{-1} V_1^* \phi_0] = 0,
\end{equation}
implying 
\begin{equation}
\psi_0 + (H_0 - (0 + i 0) I_{[L^2(\bbR^n)]^N})^{-1} V_1^* \phi_0 = c^\top, 
\end{equation}
for some $c \in \bbC^N$. Since $\psi_0 \in [L^2(\bbR^n)]^N$, and by exactly the same arguments employed in 
\eqref{9.41}--\eqref{9.42}, also $R_{0,0} * (V_1^* \phi_0) \in [L^2(\bbR^n)]^N$, one concludes that $c = 0$ and hence
\begin{equation}
\psi_0 = - (H_0 - (0 + i 0) I_{[L^2(\bbR^n)]^N})^{-1} V_1^* \phi_0. 
\end{equation}
Thus, $\phi_0 \neq 0$, and 
\begin{align}
0 = & V_2 \psi_0 + V_2 (H_0 - (0 + i 0) I_{[L^2(\bbR^n)]^N})^{-1}V_1^* \phi_0   \no \\
= & \big[I_{[L^2(\bbR^n)]^N} + \ol{V_2 (H_0 - (0 + i0) I_{[L^2(\bbR^n)]^N})^{-1} V_1^*}\big] \phi_0, 
\end{align}
that is, 
\begin{equation} 
0 \neq \phi_0 \in 
\ker \big(\big[I_{[L^2(\bbR^n)]^N} + \ol{V_2 (H_0 - (0 + i0) I_{[L^2(\bbR^n)]^N})^{-1} V_1^*}\big]\big).
\end{equation}  
This concludes the proof. 
\end{proof}

Recalling results of \cite{EGG20}, we will revisit the basic elements in the proof of item $(i)$ of 
Theorem \ref{t9.7} in Lemma \ref{l9.10a}.

\begin{remark} \lb{r9.8} 
$(i)$ In physical notation, the zero-energy resonances in Cases $(II)$ and $(IV)$ for $n=2$ correspond to eigenvalues $\pm 1/2$ of the spin-orbit operator (cf.\ the operator $S$ in \cite{KOY15}, \cite{KY01}) when 
 $V$ is spherically symmetric, see the discussion in \cite{EGG20}. \\[1mm] 
$(ii)$ For basics on the Birman--Schwinger principle in an abstract context, especially, if $0 \in \rho(H_0)$, we refer to \cite{GLMZ05} (cf.\ also \cite{BGHN16}, \cite{GHN15}) and the extensive literature cited therein. 
In the concrete case of Schr\"odinger operators, relations \eqref{9.37}, \eqref{9.38} are discussed at length in \cite{AGH82}, \cite{BGD88}, \cite{BGDW86}, \cite{EGG14}, \cite{EG13}, \cite{ES04}, \cite{ES06},  
\cite{GH87}, \cite{Je80}, \cite{Je84}, \cite{JK79}, \cite{JN01}, \cite{Mu82}, \cite{To17} (see also the list of references quoted therein), and in \cite{EG17}, \cite{EGG19}, \cite{EGG20}, \cite{EGT18}, \cite{EGT19} in the case of (massive and massless) Dirac operators. \\[1mm] 
$(iii)$ As mentioned in Remark \ref{r5.1}\,$(ii)$, the absence of zero-energy resonances is well-known in the 
three-dimensional case $n=3$, see \cite{Ai16}, \cite[Sect.~4.4]{BE11}, \cite{BES08}, \cite{BGW95}, \cite{SU08}, 
\cite{SU08a}, \cite{ZG13}. In fact, for $n=3$ the absence of zero-energy resonances has been shown 
under the weaker decay $|V_{j,k}| \leq C\langle x\rangle^{-1 - \varepsilon}$, $x \in \bbR^3$, in \cite{Ai16}. 
The absence of zero-energy resonances for massless Dirac operators in dimensions $n \geq 4$ as contained in Theorem \ref{t9.7}\,$(ii)$ appears to have gone unnoticed in the literature and was only recently observed in \cite{GN20a}.
${}$ \hfill $\diamond$
\end{remark}

To determine the leading order behavior of 
\begin{equation}
\big[U_V I_{[L^2(\bbR^n)]^N} + \ol{V_1 (H_0 - z I_{[L^2(\bbR^n)]^N})^{-1} V_1^*}\big]^{-1} \, 
\text{ as $z \to 0$, $z \in \ol{\bbC_+}$,} 
\end{equation} 
in all possible cases discussed in Theorem \ref{t9.7}, it is convenient to introduce some more notation:
\begin{align}
& T(z) := U_V I_{[L^2(\bbR^n)]^N} + V_1 (H_0 - z I_{[L^2(\bbR^n)]^N})^{-1} V_1^*, \quad z \in \bbC_+, \lb{10.90ww}\\
& T(\lambda) := U_V I_{[L^2(\bbR^n)]^N} + \ol{V_1 (H_0 - (\lambda + i0) I_{[L^2(\bbR^n)]^N})^{-1} V_1^*}, 
\quad \lambda \in \bbR. 
\end{align}
Next, we split $P_0$ in \eqref{9.32a} according to all possible cases in Theorem \ref{t9.7} as follows: If $n=2$, we write 
\begin{align}
P_0 = P_{0,1} \oplus P_{0,2},     \lb{9.48a} 
\end{align}
where $P_{0,1}$ represents case $(II)$, $P_{0,2}$ represents case $(III)$, and if $P_{0,1}$ and $P_{0,2}$ are both nonzero, $P_0$ represents case $(IV)$. Similarly, if $n \geq 3$, $P_0 \neq 0$ represents case $(II)$. (Again, we remark that we will 
discuss in Lemma \ref{l9.10a}\,$(i)$ that $\dim(\ran(P_{0,1})) \leq 2$.) 

In the following we denote the integral operators in $[L^2(\bbR^n)]^N$ generated by the integral 
kernels $R_{j,k}(\, \cdot\,,\,\cdot\,)$ in \eqref{9.46}--\eqref{9.48} by $R_{j,k}$, $j,k \in \{0,1\}$. In particular,
\begin{equation}
T(0) = U_V I_{[L^2(\bbR^n)]^N} + \ol{V_1 R_{0,0} V_1^*}. 
\end{equation}

In order to study asymptotics as $z\to 0$ of the Birman--Schwinger-type operators, we strengthen Hypothesis \ref{h9.5} as follows.

\begin{hypothesis}\lb{h9.5mod}
Let $n \in \bbN$, $n \geq 2$, and $\varepsilon > 0$. Assume the a.e.~self-adjoint matrix-valued potential 
$V = \{V_{\ell,m}\}_{1 \leq \ell,m \leq N}$ satisfies for some fixed $\varepsilon \in (0,1)$, $C \in (0,\infty)$, 
\begin{align}
\begin{split} 
& V \in [L^{\infty} (\bbR^n)]^{N \times N},       \\
& |V_{\ell,m}(x)| \leq C \langle x \rangle^{- n (1 + \varepsilon)} \, 
\text{ for a.e.~$x \in \bbR^n$, $1 \leq \ell,m \leq N$.}    \lb{9.31Amod}
\end{split} 
\end{align}
In accordance with the factorization based on the polar decomposition of $V$ discussed in \eqref{9.27} we 
suppose that 
\begin{equation}
V = V_1^* V_2 = |V|^{1/2} U_V |V|^{1/2}, \, \text{ where } \, 
V_1 = V_1^* = |V|^{1/2}, \; V_2 = U_V |V|^{1/2}.
\end{equation}  
\end{hypothesis}

We note that, in accordance with \eqref{9.31Amod}, the entries of $V_1(\dott)$ satisfy
\begin{equation}
|(V_1)_{\ell,m}(x)|\leq \wti C\langle x\rangle^{-n(1+\varepsilon)/2}\, \text{ for a.e.~$x \in \bbR^n$, $1 \leq \ell,m \leq N$,}
\end{equation}
for a constant $\wti C\in (0,\infty)$.

\begin{lemma} \lb{l9.9}
Assume Hypothesis \ref{h9.5mod}. Then $($cf.\ \eqref{9.46}--\eqref{9.48}$)$ 
\begin{align}
\begin{split} \lb{10.94}
& \ol{V_1 (H_0 - z I_{[L^2(\bbR^n)]^N})^{-1} V_1^*} 
\underset{\substack{z \to 0 \\ z \in \ol{\bbC_+}\backslash\{0\}}}{=} 
\ol{V_1 R_{0,0} V_1^*} + z \ol{V_1 R_{1,0} V_1^*}  \\
& \quad + z \big[- (2 \pi)^{-1} \ln(z/2) - (2\pi)^{-1} \gamma_{E-M} +i 4^{-1}\big] \delta_{n,2} 
\ol{V_1 R_{1,1} V_1^*} + E(z), 
\end{split} 
\end{align}
where 
\begin{equation}\lb{10.95}
\|E(z)\|_{\cB([L^2(\bbR^n)]^N)} \underset{\substack{z \to 0 \\ z \in \ol{\bbC_+}\backslash\{0\}}}{=} 
\Oh\big(|z|^{1 + \varepsilon}\big)
\end{equation}
$($with $0 < \varepsilon$ taken as in Hypothesis \ref{h9.5mod}$)$.
\end{lemma} 
\begin{proof}
In order to prove \eqref{10.94} and \eqref{10.95} it suffices to show
\begin{align}
&\big\|V_1(x)G_0(z;x,y)V_1^*(y)-V_1(x)R_{0,0}(x-y)V_1^*(y)-zV_1(x)R_{1,0}(x-y)V_1^*(y)\big\|_{\cB(\bbC^N)}\no\\
&\quad \leq c_0|z|^{1+\varepsilon}k(x,y),\quad x,y\in \bbR^n,\, x\neq y,\, z\in \ol{\bbC_+}\backslash \{0\}, |z|\leq 1,\lb{10.108}
\end{align}
for some positive $(z,x,y)$-independent constant $c_0$ and for some $z$-independent function $k(\dott,\dott)$ which generates a bounded integral operator in $L^2(\bbR^n)$. In the following we treat separately the cases {\bf (I)} $n$ odd and {\bf (II)} $n$ even.

\medskip
\noindent
{\bf (I)} $n$ odd.  In order to prove \eqref{10.108}, we estimate
\begin{equation}
\begin{split}
G_0(z;x,y)-R_{0,0}(x-y)-zR_{1,0}(x-y),&\\
x,y\in \bbR^n,\, x\neq y,\, z\in \ol{\bbC_+}\backslash \{0\}, |z|\leq 1,&
\end{split}
\end{equation}
separately in the regimes $|z||x-y|\leq 1$ and $|z||x-y|>1$.

The expansion \eqref{B.5} implies
\begin{align}
&\big\|G_0(z;x,y)-R_{0,0}(x-y)-zR_{1,0}(x-y)\big\|_{\cB(\bbC^N)}\no\\
&\quad \leq c_1\big[|z|^2 + |z|^2|z-y|^{3-n}\big]\no\\
&\quad \leq c_1\big[|z|^2 + |z|^{1+\varepsilon}|z-y|^{(2+\varepsilon)-n}\big]\no\\
&\quad \leq c_1|z|^{1+\varepsilon}\big[1+|x-y|^{(2+\varepsilon)-n}\big],\lb{10.110}\\
& x,y\in \bbR^n,\, x\neq y,\, z\in \ol{\bbC_+}\backslash \{0\}, |z|\leq 1,\, |z||x-y|\leq 1,\no
\end{align}
for some $(z,x,y)$-independent constant $c_1\in (0,\infty)$.  By Lemma \ref{lB.6},
\begin{align}
&G_0(z;x,y)   \no \\
&\quad=i4^{-1}(2\pi)^{(2-n)/2}|x-y|^{(2-n)/2}z^{n/2}e^{iz|x-y|}\omega_{\frac{n-2}{2}}(z|x-y|)I_N \lb{10.113} \\
&\qquad-4^{-1}(2\pi)^{(2-n)/2}|x-y|^{(2-n)/2}z^{n/2}e^{iz|x-y|}\omega_{\frac{n}{2}}(z|x-y|)\alpha\cdot\frac{(x-y)}{|x-y|},\no\\
&\hspace*{6.3cm} x,y\in \bbR^n,\, x\neq y,\, z\in \ol{\bbC_+}\backslash \{0\},\no
\end{align}
with
\begin{align}
|x-y|^{(2-n)/2}|z|^{n/2}\big|\omega_{\nu}(z|x-y|)\big|&\leq c_2 |x-y|^{1-\frac{n}{2}}|z|^{n/2}(1+|z||x-y|)^{-1/2}\no\\
&\leq c_2 |z|^{n-1}(1+|z||x-y|)^{-1/2}\no\\
&\leq c_2|z|^2(1+|z||x-y|)^{-1/2},\lb{10.114}\\
&\hspace*{-3.7cm}x,y\in \bbR^n,\, x\neq y,\, z\in \ol{\bbC_+}\backslash \{0\}, |z|\leq 1,\, |z||x-y|\geq 1,\no
\end{align}
for some $(z,x,y)$-independent constant $c_2\in (0,\infty)$.  The representation \eqref{10.113} and the estimate \eqref{10.114} combine to yield
\begin{align}\lb{10.115}
\begin{split}
& \big\|G_0(z;x,y)\big\|_{\cB(\bbC^N)}\leq c_3 |z|^2(1+|z||x-y|)^{-1/2},     \\
& x,y\in \bbR^n,\, x\neq y,\, z\in \ol{\bbC_+}\backslash \{0\}, |z|\leq 1,\, |z||x-y|\geq 1, 
\end{split}
\end{align}
for some $(z,x,y)$-independent constant $c_3\in (0,\infty)$, and it follows that
\begin{align}
&\big\|G_0(z;x,y)-R_{0,0}(x-y)-zR_{1,0}(x-y)\big\|_{\cB(\bbC^N)}\no\\
&\quad \leq c_4 \big(|z|^2(1+|z||x-y|)^{-1/2}+|x-y|^{1-n}+|z||x-y|^{2-n}\big)\no\\
&\quad \leq c_4|z|^2,\lb{10.116}\\
&\text{for a.e.}\, x,y\in \bbR^n,\, x\neq y,\, z\in \ol{\bbC_+}\backslash \{0\}, |z|\leq 1,\, |z||x-y|\geq 1,\no
\end{align}
for some $(z,x,y)$-independent constant $c_4\in (0,\infty)$.

By combining \eqref{10.110} and \eqref{10.116}, one obtains
\begin{align}
&\big\|G_0(z;x,y)-R_{0,0}(x-y)-zR_{1,0}(x-y)\big\|_{\cB(\bbC^N)}\no\\
&\quad \leq c_5\big[|z|^2+|z|^{1+\varepsilon}|x-y|^{(2+\varepsilon)-n}\big]\no\\
&\quad \leq c_5|z|^{1+\varepsilon}\big[1+|x-y|^{(2+\varepsilon)-n}\big],\lb{10.105z}\\
&\hspace*{-.2cm}\text{for a.e.}\, x,y\in \bbR^n,\, x\neq y,\, z\in \ol{\bbC_+}\backslash \{0\}, |z|\leq 1,\no
\end{align}
for some $(z,x,y)$-independent constant $c_5\in (0,\infty)$.  Hence, \eqref{10.108} holds for some constant $c_0\in (0,\infty)$ and
\begin{align}
k(x,y)&= \langle x \rangle^{- n (1 + \varepsilon)/2}\langle y \rangle^{- n (1 + \varepsilon)/2}\lb{10.118}\\
&\quad + [1+|x|]^{-n(1+\varepsilon)/2}|x-y|^{(2+\varepsilon)-n}[1+|y|]^{-n(1+\varepsilon)/2},\quad x,y\in \bbR^n,\, x\neq y.\no
\end{align}
In deducing the form of $k(\dott,\dott)$ in \eqref{10.118}, one uses
\begin{align}
\|V_1(x)\|_{\cB(\bbC^N)} &\leq C' \langle x\rangle^{-n(1+\varepsilon)/2}\no\\
&\leq C'' [1+|x|]^{-1-(\varepsilon/2)}\, \text{ for a.e.~$x\in \bbR^n$}\lb{10.107a}
\end{align}
for appropriate $x$-independent constants $C',C''\in (0,\infty)$.

The first term on the right-hand side in \eqref{10.118} generates a Hilbert--Schmidt integral operator in $L^2(\bbR^n)$, since $\langle \dott \rangle^{- n (1 + \varepsilon)/2}\in L^2(\bbR^n)$.  The second term on the right-hand side in \eqref{10.118} generates a bounded integral operator in $L^2(\bbR^n)$ as a consequence of Theorem \ref{t5.6} $(ii)$ with the choices $c=d=1+(\varepsilon/2)$ and $p=p'=2$, since $1+(\varepsilon/2)<3/2\leq n/2$.  Thus, $k(\dott,\dott)$ generates a bounded integral operator in $L^2(\bbR^n)$.

\medskip
\noindent
{\bf (II)} $n$ even.  The case $n=2$ is treated in detail in \cite[Lemma 5.1]{EGG20}, so we consider $n\geq 4$ here.  It suffices to verify the inequality in \eqref{10.108}.  The expansion \eqref{B.8} implies
\begin{align}
&\big\|G_0(z;x,y)-R_{0,0}(x-y)-zR_{1,0}(x-y)\big\|_{\cB(\bbC^N)}\no\\
&\quad \leq c_1\big[|z|^{n-1}+|z|^2|x-y|^{3-n}+|z|^{n-1}|\ln(z|x-y|)|\big]\no\\
&\quad \leq c_1\big[|z|^{n-1}+|z|^{1+\varepsilon}|x-y|^{(2+\varepsilon)-n}+|z|^{n-1}|\ln(z|x-y|)|\big]\no\\
&\quad \leq c_1|z|^{1+\varepsilon}\big[1 + |x-y|^{(2+\varepsilon)-n} + |x-y|^{-1}\big],\lb{10.108a}\\
& x,y\in \bbR^n,\, x\neq y,\, z\in \ol{\bbC_+}\backslash \{0\}, |z|\leq 1,\, |z||x-y|\leq 1,\no
\end{align}
for some $(z,x,y)$-independent constant $c_1(\varepsilon)\in (0,\infty)$, and an argument entirely analogous to \eqref{10.113}--\eqref{10.115} shows that \eqref{10.116} extends to the current case where $n$ is even.  Combining \eqref{10.116} and \eqref{10.108a}, one obtains
\begin{align}
&\big\|G_0(z;x,y)-R_{0,0}(x-y)-zR_{1,0}(x-y)\big\|_{\cB(\bbC^N)}\no\\
&\quad \leq c_3|z|^{1+\varepsilon}\big[1+|x-y|^{(2+\varepsilon)-n}+|x-y|^{-1}\big],\\
&\text{for a.e.}\, x,y\in \bbR^n,\, x\neq y,\, z\in \ol{\bbC_+}\backslash \{0\}, |z|\leq 1,\no
\end{align}
Hence, \eqref{10.108} holds for some constant $c_0\in (0,\infty)$ and
\begin{align}
k(x,y)&= \langle x \rangle^{- n (1 + \varepsilon)/2}\langle y \rangle^{- n (1 + \varepsilon)/2}    \no \\
&\quad + [1+|x|]^{-n(1+\varepsilon)/2}|x-y|^{(2+\varepsilon)-n}[1+|y|]^{-n(1+\varepsilon)/2}    \lb{10.110a} \\
&\quad + [1+|x|]^{(1-n)/2}|x-y|^{-1}[1+|y|]^{(1-n)/2},\quad x,y\in \bbR^n,\, x\neq y.\no
\end{align}
In deducing the form of $k(\dott,\dott)$ in \eqref{10.110a}, we used \eqref{10.107a} and the elementary bound $[1+|x|]^{-n(1+\varepsilon)/2}\leq [1+|x|]^{-(n-1)/2}$, $x\in \bbR^n$.  The fact that the first two terms on the right-hand side in \eqref{10.110a} generate bounded integral operators in $L^2(\bbR^n)$ was established in {\bf (I)} above.  The third term on the right-hand side in \eqref{10.110a} generates a bounded integral operator in $L^2(\bbR^n)$ by Theorem \ref{t5.6} (ii) with the choices $c=d=(n-1)/2$ and $p=p'=2$, since $(n-1)/2 < n/2$, $c+d=n-1>0$, and $n-(c+d) = 1$.
\end{proof}
\begin{lemma} \lb{l9.10}
Assume Hypothesis \ref{h9.5mod}.  If $T(\dott)$ is defined by \eqref{10.90ww} and $P_0$ denotes the $($finite-dimensional\,$)$ Riesz projection associated to the operator \eqref{9.24}, then
\begin{align}
[T(z) + P_0]^{-1} & \underset{\substack{z \to 0 \\ z \in \ol{\bbC_+}\backslash\{0\}}}{=} 
[T(0) + P_0]^{-1}    \no \\
& \hspace*{1.4cm} - z \big[- (2 \pi)^{-1} \ln(z/2) - (2\pi)^{-1} \gamma_{E-M} +i 4^{-1}\big]    \no \\ 
& \hspace*{1.4cm} \quad \times \delta_{n,2} [T(0) + P_0]^{-1} \ol{V_1 R_{1,1} V_1^*} [T(0) + P_0]^{-1}     
\lb{10.111zz} \\ 
& \hspace*{1.4cm} - z [T(0) + P_0]^{-1} \ol{V_1 R_{1,0} V_1^*} [T(0) + P_0]^{-1} + E_1(z),     \no 
\end{align}
where 
\begin{equation}\lb{10.112zz}
\begin{split}
&\|E_1(z)\|_{\cB([L^2(\bbR^n)]^N)}\\
&\quad \underset{\substack{z \to 0 \\ z \in \ol{\bbC_+}\backslash\{0\}}}{=}
\begin{cases}
\Oh\big(|z|^{1+ k}\big)\, \text{ for any $0<k<\min\{1,\varepsilon\}$},& n=2,\\
\Oh\big(|z|\big),& n\geq 3 
\end{cases}
\end{split}
\end{equation}
$($with $0 < \varepsilon$ taken as in Hypothesis \ref{h9.5mod}$)$. 
\end{lemma}
\begin{proof}
The case $n=2$ is treated in detail in \cite[Lemma 5.2]{EGG20}, so we consider $n\geq 3$ here.  By Lemma \ref{l9.9},
\begin{align}
[T(z)+P_0]^{-1} &= \Big[T(0)+P_0 + z\overline{V_1R_{1,0}V_1^*}+E(z) \Big]^{-1}\no\\
&=\Big(I_{[L^2(\bbR^n)]^N}+[T(0)+P_0]^{-1}z\overline{V_1R_{1,0}V_1^*} + [T(0)+P_0]^{-1}E(z) \Big)^{-1}\no\\
&\qquad \times[T(0)+P_0]^{-1},\quad z\in \overline{\bbC_+}\backslash\{0\},\, 0<|z|\ll 1,\lb{10.122z}
\end{align}
where $E(\dott)$ satisfies \eqref{10.95}.  By \eqref{9.47} and \eqref{10.105z},
\begin{align}
&\big\|V_1(x)G_0(z;x,y)V_1^*(y)-V_1(x)R_{0,0}(x-y)V_1^*(y)\big\|_{\cB(\bbC^N)}\no\\
&\quad \leq c_1|z|\Big[V_1(x)V_1^*(y) + [1+|x|]^{(1-n)/2}|x-y|^{-1}[1+|y|]^{(1-n)/2}\no\\
&\hspace*{1.7cm} +[1+|x|]^{-1}|x-y|^{2-n}[1+|y|]^{-1}\Big],\lb{10.123}\\
&\hspace*{-.2cm}\text{for a.e.}\, x,y\in \bbR^n,\, x\neq y,\, z\in \ol{\bbC_+}\backslash \{0\}, |z|\leq 1,\no
\end{align}
for some $(z,x,y)$-independent constant $c_1\in (0,\infty)$.  The kernel
\begin{align}
\begin{split} 
k(x,y)&=V_1(x)V_1^*(y) + [1+|x|]^{(1-n)/2}|x-y|^{-1}[1+|y|]^{(1-n)/2}\lb{10.124z}\\
&\quad +[1+|x|]^{-1}|x-y|^{2-n}[1+|y|]^{-1},\quad x,y\in \bbR^n,\, x\neq y,     
\end{split} 
\end{align}
generates a bounded integral operator in $L^2(\bbR^n)$.  The first term on the right-hand side in \eqref{10.124z} generates a Hilbert--Schmidt operator due to the containment $\|V_1(\dott)\|_{\cB(\bbC^N)}\in L^2(\bbR^n)$.  The fact that the second term generates a bounded operator is explained in the proof of Lemma \ref{l9.9} in connection with \eqref{10.118}.  Finally, the third term in \eqref{10.124z} generates a bounded integral operator by an application of Theorem \ref{t5.6} $(ii)$ with $a=b=1$.  It follows that
\begin{align}
\|T(z)-T(0)\|_{\cB([L^2(\bbR^n)]^N)}\leq c_2 |z|,\quad z\in \overline{\bbC_+},\, |z|\leq 1,\lb{10.125z}
\end{align}
for some $z$-independent constant $c_2\in (0,\infty)$.  The estimate in \eqref{10.125z} implies that, for $z\in \overline{\bbC_+}$ with $0<|z|\ll 1$, a Neumann series may be used to obtain
\begin{align}
[T(z)+P_0]^{-1}&=[T(0)+P_0]^{-1}- z [T(0)+P_0]^{-1} \overline{V_1R_{1,0}V_1^*}[T(0)+P_0]^{-1} \no \\
&\quad -[T(0)+P_0]^{-1}E(z)[T(0)+P_0]^{-1}     \lb{10.126zz} \\
&\quad + \sum_{n=2}^{\infty}(-1)^nA(z)^n[T(0)+P_0]^{-1},\quad z\in \overline{\bbC_+}\backslash\{0\},\, 0<|z|\ll 1,\no
\end{align}
where
\begin{equation}\lb{10.128ww}
A(z):= z [T(0)+P_0]^{-1} \overline{V_1R_{1,0}V_1^*}+[T(0)+P_0]^{-1}E(z)\underset{\substack{z \to 0 \\ z \in \ol{\bbC_+}\backslash\{0\}}}{=}\Oh(|z|),
\end{equation}
apllying \eqref{10.95}.  In particular,
\begin{equation}\lb{10.128zz}
\bigg\|\sum_{n=2}^{\infty}(-1)^nA(z)^n[T(0)+P_0]^{-1}\bigg\|_{\cB([L^2(\bbR^n)]^N)}\underset{\substack{z \to 0 \\ z \in \ol{\bbC_+}\backslash\{0\}}}{=}\Oh(|z|).
\end{equation}
Hence, \eqref{10.111zz} follows from \eqref{10.126zz} with
\begin{align}
& E_1(z):=-[T(0)+P_0]^{-1}E(z)[T(0)+P_0]^{-1}+ \sum_{n=2}^{\infty}(-1)^nA(z)^n[T(0)+P_0]^{-1},    \no\\
& \hspace*{7.5cm} z\in \overline{\bbC_+},\, 0<|z|\ll1.
\end{align}
Thus, the $\Oh(|z|)$ relation in \eqref{10.112zz} for $n \geq 3$ follows from \eqref{10.95} and \eqref{10.128zz}.
\end{proof}

\begin{lemma}[{\cite[Lemmas~5.2, 7.1--7.6]{EGG20}}] \lb{l9.10a}
Assume Hypothesis \ref{h9.5mod} and $n=2$.  The following statements $(i)$--$(iv)$ hold.  \\[1mm] 
$(i)$ If $\phi_0 \in \ker(T(0))$, then $\phi_0 = U_V V_1\psi_0$, with $\psi_0$ a distributional solution of $H \psi_0 = 0$ satisfying 
$\psi_0 \in [L^p(\bbR^2)]^2$ for all $p \in (2, \infty) \cup \{\infty\}$. Moreover,
\begin{equation}
\psi_0 (x) = - i \alpha \cdot x \big[2 \pi \langle x \rangle^2\big]^{-1} (R_{1,1} V_1^* \phi_0) + \psi_1(x),
\end{equation}
where
\begin{equation}
(R_{1,1} V_1^* \phi_0) = \int_{\bbR^2} d^2y \, V_1^*(y) \phi_0 (y) \, \text{ and } \, \psi_1 \in [H^1(\bbR^2)]^2. 
\end{equation}
In particular, 
\begin{equation}
\psi_0 \in [H^1(\bbR^2)]^2 \, \text{ if and only if } \, 
(R_{1,1} V_1^* \phi_0) = \int_{\bbR^2} d^2y \, V_1^*(y) \phi_0 (y) = 0.
\end{equation}
Moreover, the rank of $P_0$ is at most two plus the dimension of the eigenspace of $H$ at energy zero, that is,
\begin{align}
P_0 = P_{0,1} \oplus P_{0,2}, \, \text{ with } \, \dim(\ran(P_{0,1})) \leq 2  
\end{align}
in \eqref{9.48a}. \\[1mm] 
$(ii)$ If $\psi_0 \in [L^2(\bbR^2)]^2 + \bigcap_{p \in (2,\infty) \cup \{\infty\}} [L^p(\bbR^2)]^2$, then 
\begin{equation}
\phi_0 = U_V V_1 \psi_0 \in \ker(T(0)).
\end{equation} 
$(iii)$ If $\phi_0 = U_V V_1\psi_0 \in \ker(T(0))$, then $\phi_0 \in \ran(P_{0,2})$ if and only if 
$\psi_0 \in [H^1(L^2(\bbR^2))]^2$. Thus, $\phi_0 \in \ran(P_{0,2})$ if and only if 
$\phi_0 \in \ker(P_0V_1 R_{1,1} V_1^* P_0)$. \\[1mm]
$(iv)$ If $n=2$ and $\phi_0 \in \ran(P_{0,2})$, then 
\begin{equation}
(R_{0,0}V_1^*\phi_0, R_{0,0}V_1^*\phi_0)_{[L^2(\bbR^2)]^2} = (V_1^*\phi_0, R_{1,0}V_1^*\phi_0)_{[L^2(\bbR^2)]^2} 
\lb{9.68}
\end{equation}
and 
\begin{equation}
\ker(P_{0,2} V_1 R_{1,0} V_1^* P_{0,2}) = \{0\}.    \lb{9.69}
\end{equation} 
\end{lemma}

\begin{lemma} \lb{l9.10b}
Assume Hypothesis \ref{h9.5mod} and $n\geq 3$.  The following statements $(i)$ and $(ii)$ hold.  \\[1mm]
$(i)$ $\phi_0 = U_V V_1\psi_0 \in \ker(T(0))$ $($i.e., $\phi_0 \in \ran(P_0)$$)$ if and only if 
$\psi_0 \in [H^1(L^2(\bbR^2))]^2$. Thus, $\phi_0 \in \ran(P_0)$ if and only if 
$\phi_0 \in \ker(P_0V_1 R_{1,1} V_1^* P_0)$. \\[1mm]
$(ii)$ $\phi_0 \in \ran(P_0)$, then 
\begin{equation}
(R_{0,0}V_1^*\phi_0, R_{0,0}V_1^*\phi_0)_{[L^2(\bbR^n)]^N} = (V_1^*\phi_0, R_{1,0}V_1^*\phi_0)_{[L^2(\bbR^n)]^N} 
\lb{9.70}
\end{equation}
and 
\begin{equation}
\ker(P_0 V_1 R_{1,0} V_1^* P_0) = \{0\}.    \lb{9.71}
\end{equation} 
\end{lemma}
\begin{proof}
Item $(i)$ is just a rephrasing of the proof of Theorem \ref{t9.7} for $n \geq 3$. Item $(ii)$ is proved exactly along the lines of \cite[Lemma~7.6]{EGG20}; we briefly sketch the argument.  By item $(i)$, 
$\psi_0 = - R_{0,0} V_1^* \phi_0 \in [L^2(\bbR^2)]^2$ and hence, applying Fourier transforms, 
\begin{align}
& (R_{0,0} V_1^* \phi_0, R_{0,0} V_1^* \phi_0)_{[L^2(\bbR^n)]^N} 
= \int_{\bbR^n} d^n p \, |p|^{-4} \big((\alpha \cdot p) (V_1^* \phi_0)^{\wedge}, (\alpha \cdot p) (V_1^* \phi_0)^{\wedge}\big)_{\bbC^N} 
\no \\
& \quad = \int_{\bbR^n} d^n p \, |p|^{-2} \big((V_1^* \phi_0)^{\wedge}, (V_1^* \phi_0)^{\wedge}\big)_{\bbC^N}. 
\lb{9.72A}
\end{align}
On the other hand, employing the monotone convergence theorem, 
\begin{align}
& (V_1^* \phi_0, R_{1,0} V_1^* \phi_0)_{[L^2(\bbR^n)]^N} = (V_1^* \phi_0, (- \Delta) V_1^* \phi_0)_{[L^2(\bbR^n)]^N} \no \\
& \quad = \lim_{\varepsilon \downarrow 0} 
\int_{\bbR^n} d^n p \, \big[|p|^2 + \varepsilon^2\big]^{-1} \big((V_1^* \phi_0)^{\wedge}, (V_1^* \phi_0)^{\wedge}\big)_{\bbC^N}
\no \\
& \quad = \int_{\bbR^n} d^n p \, |p|^{-2} \big((V_1^* \phi_0)^{\wedge}, (V_1^* \phi_0)^{\wedge}\big)_{\bbC^N}, 
\end{align}
proving \eqref{9.70}.  Finally, assume that $\phi_0 \in \ker(P_0 V_1 R_{1,0} V_1^* P_0)$. Then \eqref{9.70} yields
\begin{equation}
\|\psi_0\|_{[L^2(\bbR^n)]^N} = (R_{0,0}V_1^*\phi_0, R_{0,0}V_1^*\phi_0)_{[L^2(\bbR^n)]^N} = 0, 
\end{equation}
implying $\psi_0 =0$ and thus $\phi_0 = U_V V_1 \psi_0 = 0$. 
\end{proof}

One of the principal results of this section then reads as follows:

\begin{theorem} \lb{t9.11}
Assume Hypothesis \ref{h9.5}. \\[1mm] 
$(i)$ Suppose $n=2$. Then 
\begin{align}
& T(z)^{-1} = \big[U_V I_{[L^2(\bbR^n)]^N} + \ol{V_1 (H_0 - z I_{[L^2(\bbR^n)]^N})^{-1} V_1^*}\big]^{-1}    \no \\[1mm]  
& \quad \;\, \underset{\substack{z \to 0 \\ z \in \ol{\bbC_+}\backslash\{0\}}}{=} 
\begin{cases}
T(0)^{-1} - T(0)^{-1} [\Oh(|z \ln(z)|] T(0)^{-1} \;\; \text{in Case $(I)$}, \\[1mm] 
[z \ln(z)]^{-1} P_{0,1} A P_{0,1} + P_{0,1} \big[\Oh\big(\big[|z|^{-1} |\ln(z)|^{-2}\big)\big] P_{0,1} \;\; \text{in Case $(II)$}, 
\\[1mm] 
z^{-1} P_{0,2} [P_{0,2} V_1 R_{1,0} V_1^* P_{0,2}] P_{0,2}    \\
\; + P_{0,2} \big[\Oh\big(|z|^{-1 + \varepsilon}\big)\big] P_{0,2}  
\;\; \text{in Case $(III)$}, \\[2mm] 
z^{-1} P_0 \left(\begin{smallmatrix} 0 & \;\; 0 \\[1mm] 0 & \;\; P_{0,2} V_1 R_{1,0} V_1^* P_{0,2} \end{smallmatrix}\right) P_0    \\
\; + P_0 \big[\Oh\big(|z \ln(z)|^{-1}\big)\big] P_0  
\;\; \text{in Case $(IV)$},
\end{cases}
\end{align}
where
\begin{align}
& T(0) = U_V + V_1 R_{0,0} V_1^*, \quad T(0)^{-1} \in \cB\big([L^2(\bbR^n)]^N\big) \, \text{ in Case $(I)$},   \\
& A \in \bbR\backslash\{0\} \, \text{ if } \, \dim(\ran(P_{0,1})) =1, \text{ in Case $(II)$},  \\
& {\det}_{\bbC^2}(A) \neq 0  \, \text{ if } \, \dim(\ran(P_{0,1})) =2 \, \text{ in Case $(II)$}.   
\end{align}
$(ii)$ Suppose $n \in \bbN$, $n \geq 3$. Then 
\begin{align}
& T(z)^{-1} = \big[U_V I_{[L^2(\bbR^n)]^N} + \ol{V_1 (H_0 - z I_{[L^2(\bbR^n)]^N})^{-1} V_1^*}\big]^{-1}    \no \\[1mm]  
& \quad \;\, \underset{\substack{z \to 0 \\ z \in \ol{\bbC_+}\backslash\{0\}}}{=} 
\begin{cases}
T(0)^{-1} - T(0)^{-1} [\Oh(|z \ln(z)|] T(0)^{-1} \;\; \text{in Case $(I)$}, \\[1mm] 
z^{-1} P_0 [P_0 V_1 R_{1,0} V_1^* P_0] P_0  
+ P_0 \big[\Oh\big(|z|^{-1 + \varepsilon}\big)\big] P_0  
\;\; \text{in Case $(II)$}, \\[2mm] 
\end{cases}
\end{align}
where, again,
\begin{equation}
T(0) = U_V + V_1 R_{0,0} V_1^*, \quad T(0)^{-1} \in \cB\big([L^2(\bbR^n)]^N\big) \, \text{ in Case $(I)$}. 
\end{equation}

Moreover, in both items $(i)$ and $(ii)$, the coefficients of all singular terms in the expansion of $T(z)^{-1}$ at $z=0$ 
$($i.e., in cases different from $(I)$$)$ are finite-rank operators acting in $($subspaces of\,$)$ $P_0 [L^2(\bbR^n)]^N$.

Here, $\big[\Oh\big(|\zeta|^a\big)\big]$, $a \in \bbR$, abbreviate estimates with respect to the operator norm.  
\end{theorem} 
\begin{proof}
Item $(i)$ for $n=2$ has been treated in detail \cite[Sect.~5]{EGG20} on the basis of the Jensen and Nenciu method 
\cite{JN01} outlined in Lemmas \ref{l9.1}, \ref{l9.2}, Remarks \ref{r9.3}, \ref{r9.4}, and our summary in items $(\alpha)$--$(\delta)$ 
following Remark \ref{r9.3}. Item $(ii)$ for $n \geq 3$ parallels Cases $(I)$ and $(III)$ for $n=2$.  
\end{proof}

\begin{remark} \lb{r9.12}
A comparison of the threshold behavior of massless Dirac operators (\cite{EGG20}, Theorem 9.10\,$(i)$) and 
Schr\"odinger operators (\cite{BGD88}, \cite{BGDW86}, \cite{JN01}, \cite{Mu82}) demonstrates that in both situations zero-energy resonances produce a logarithmically weaker singularity of the order $\Oh\big(|z \ln(z)|^{-1}\big)$ than the zero-energy eigenvalues which produce the expected $\Oh\big(|z|^{-1}\big)$ singularity. 
\hfill $\diamond$
 \end{remark}

Finally, returning to $F_{H,H_0}$, we again introduce the strengthened assumptions made in Hypothesis \ref{h7.1} and Corollary \ref{c4.4}\,$(ii)$. 

\begin{hypothesis} \lb{h9.13}
Let $n \in \bbN$ and suppose that $V= \{V_{\ell,m}\}_{1 \leq \ell,m \leq N}$ satisfies for some constants  
$C \in (0,\infty)$ and $\varepsilon > 0$, 
\begin{equation}
V \in [L^{\infty} (\bbR^n)]^{N \times N}, \quad 
|V_{\ell,m}(x)| \leq C \langle x \rangle^{- n - \varepsilon} \, \text{ for a.e.~$x \in \bbR^n$, $1 \leq \ell,m \leq N$.}    \lb{9.81}
\end{equation}
In addition, assume that $V(x)= \{V_{\ell,m}(x)\}_{1 \leq \ell,m \leq N}$ is self-adjoint for a.e.~$x \in \bbR^n$. 
In accordance with the factorization based on the polar decomposition of $V$ discussed in \eqref{9.27} we 
suppose that $V = V_1^* V_2 = |V|^{1/2} U_V |V|^{1/2}$, where $V_1 = V_1^* = |V|^{1/2}$, $V_2 = U_V |V|^{1/2}$. 

In addition we assume that $V$ satisfies \eqref{4.3A} and \eqref{4.4}\footnote{The first condition in \eqref{4.4} is superseded by assumption \eqref{9.81}.}.  
\end{hypothesis} 

According to Remark \ref{r8.11}, we now use the symmetrized version of the Birman--Schwinger operator in connection 
with \eqref{8.23} and hence write 
\begin{align}
F_{H,H_0} (z) &= \ln\big({\det}_{[L^2(\bbR^n)]^N, n}\big((H -z I_{[L^2(\bbR^n)]^N}) (H_0 -z I_{[L^2(\bbR^n)]^N})^{-1}\big)\big) 
\no \\
&= \ln\big({\det}_{[L^2(\bbR^n)]^N, n+1}\big(I_{[L^2(\bbR^n)]^N} + V(H_0 -z I_{[L^2(\bbR^n)]^N})^{-1}\big)\big)  \no \\
&= \ln\big({\det}_{[L^2(\bbR^n)]^N, n+1}\big(I_{[L^2(\bbR^n)]^N} + V_2(H_0 -z I_{[L^2(\bbR^n)]^N})^{-1} V_1^*\big)\big) \no \\
&= \ln\big({\det}_{[L^2(\bbR^n)]^N, n+1}\big(U_V\big\{U_V I_{[L^2(\bbR^n)]^N} 
+ V_1(H_0 -z I_{[L^2(\bbR^n)]^N})^{-1} V_1^*\big\}\big)\big)   \no \\ 
&= \ln({\det}_{[L^2(\bbR^n)]^N, n+1}(U_V T(z))), \quad z \in \bbC_{\pm},     \lb{9.82} 
\end{align} 
employing $U_V^2 = I_N$. 

Next, we briefly recall a few facts on continuous (resp., analytic) logarithms and continuous arguments of complex-valued functions (see \cite[p.~40--46]{AN07} for details): 

If $S \subseteq \bbC$ and 
$f \colon S \to \bbC \backslash \{0\}$, then $g$ is called a 
{\it continuous logarithm} of $f$ on $S$ if $g$ is continuous on $S$ and $f(z) = e^{g(z)}$, $z \in S$. Similarly, 
$\theta \colon S \to \bbR$ is called a {\it continuous argument} of $f$ on $S$ if $\theta$ is continuous on $S$ and $f(z) = |f(z)| e^{i \theta(z)}$, $z \in S$. 

$\bullet$ If $g$ is a continuous logarithm of $f$ on $S$, then $\Im(g)$ is a continuous argument of $f$ on $S$. 

$\bullet$ If $\theta$ is a continuous argument of $f$, then $\ln(|f|) + i \theta$ is a continuous logarithm of $f$ on $S$. 

$\bullet$ Thus, $f$ has a continuous logarithm on $S$ if and only if $f$ has a continuous argument on $S$.  

If $\Omega \subseteq \bbC$ is open and $f \colon \Omega \to \bbC \backslash \{0\}$ is analytic, then 
$g \colon \Omega \to \bbC$ is called an analytic logarithm of $f$ on $\Omega$ if $g$ is analytic on $\Omega$ and 
$f(z) = e^{g(z)}$, $z \in \Omega$. 

$\bullet$ If $\Omega \subseteq \bbC$ is open and starlike and $f \colon \Omega \backslash \{0\}$ is analytic, then $f$ has an analytic logarithm on $\Omega$. 

$\bullet$ Suppose $\Omega$ is open and $f \colon \Omega \to \bbC \backslash \{0\}$ is analytic with $g$ a continuous logarithm of $f$ on $\Omega$. Then $g$ is analytic on $\Omega$. 

$\bullet$ Let $a, b, c, d \in \bbR$, $R = \{z = x + iy \, | \, a \leq x \leq b, \; c \leq y \leq d\}$, and 
$f \colon R \to \bbC \backslash \{0\}$ continuous. Then $f$ has a continuous logarithm on $R$. 

$\bullet$ $f \colon \bbC_+ \to \bbC \backslash \{0\}$ analytic, $f \colon \ol{\bbC_+} \to \bbC \backslash \{0\}$ 
continuous, then $f$ has an analytic logarithm on $\bbC_+$ which is continuous on $\ol{\bbC_+}$. More generally, 
$f \colon \bbC_+ \to \bbC \backslash \{0\}$ analytic, $f \colon \ol{\bbC_+} \to \bbC$ 
continuous, then $f$ has an analytic logarithm on $\bbC_+$ which is continuous at $x_0 \in \bbR$ if 
$f(x_0) \neq 0$.

This yields the final and principal result of this section.

\begin{theorem} \lb{t9.14}
Let $n \in \bbN$, $n \geq 2$, and assume Hypothesis \ref{h9.13}. Then $F_{H,H_0}$, $z \in \bbC_{\pm}$, has normal boundary values on $\bbR\backslash \{0\}$. In addition, the boundary values to $\bbR$ 
of the function $\Im(F_{H,H_0}(z))$, $z \in \bbC_+$, are continuous on $(-\infty, 0) \cup (0,\infty)$, 
\begin{equation}
\Im(F_{H,H_0}(\lambda + i0)) \in C((-\infty,0) \cup (0,\infty)),      \lb{9.83} 
\end{equation}
and the left and right limits at zero, 
\begin{equation}
\Im(F_{H,H_0}(0_{\pm} + i 0)) = \lim_{\varepsilon \downarrow 0} \Im(F_{H,H_0}(\pm \varepsilon + i 0)), 
\lb{9.84}
\end{equation}   
exist. In particular, if $0$ is a regular point for $H$ according to Definition \ref{d9.6}\,$(iii)$ and Theorem 
\ref{t9.7}\,$(iii)$ $($this corresponds to case $(I)$ in Theorem \ref{t9.7}\,$(i), (ii)$$)$, then 
\begin{equation}
\Im(F_{H,H_0}(\lambda + i0)) \in C(\bbR).     \lb{10.121} 
\end{equation}
\end{theorem}
\begin{proof}
Applying Theorem \ref{t3.4}, Corollary \ref{c4.4}\,$(i)$, and Theorem \ref{t6.9}, the function 
${\det}_{[L^2(\bbR^n)]^N, n+1}(U_V T(z))$, $z \in \bbC_{\pm}$, in 
\eqref{9.82} continuously extends to $z \in \ol{\bbC_{\pm}} \big \backslash\{0\}$ and does not vanish there. 
In particular, $F_{H,H_0}$ has normal boundary values on $\bbR\backslash \{0\}$. Moreover, combining 
Theorem \ref{t6.9} and  \cite[Theorem~3.1.7]{AN07}, and especially, by \cite[Exercise~3.2.6 on p.~46]{AN07}, the function 
\begin{align}
& {\det}_{[L^2(\bbR^n)]^N, n+1}(U_V T(z))     \no \\
& \quad ={\det}_{[L^2(\bbR^n)]^N, n+1}\big(U_V\big\{U_V I_{[L^2(\bbR^n)]^N} 
+ \ol{V_1(H_0 -z I_{[L^2(\bbR^n)]^N})^{-1} V_1^*}\big\}\big), \\
& \hspace*{8.95cm} z \in \ol{\bbC_+} \big \backslash\{0\},     \no 
\end{align}  
has a continuous argument in any rectangle of the form  
\begin{equation}
\{z = x + iy \,|\, x \in [a,b] \subset (-\infty,0) \cup (0, \infty), \; y \in [0,c]\}, \quad c > 0, 
\end{equation}
in $\ol{\bbC_+} \big \backslash \{0\}$, proving \eqref{9.83}. Thus $\lambda = 0$ is the only possible exception to continuity of 
$\Im(F_{H,H_0}(\, \cdot \, + i0))$ on $\bbR$. 

If $0$ is a regular point for $H$, that is, if 
\begin{equation}
\ker\big(\big[I_{[L^2(\bbR^n)]^N} + \ol{V_2 (H_0 - (0 + i 0) I_{[L^2(\bbR^n)]^N})^{-1} V_1^*}\big]\big) 
= \{0\}.     \lb{10.124} 
\end{equation}
then 
\begin{equation}
{\det}_{[L^2(\bbR^n)]^N, n+1}(U_V T(0)) \neq 0
\end{equation}
and hence ${\det}_{[L^2(\bbR^n)]^N, n+1}\big(U_V T(z)\big)$ has a continuous argument in any 
rectangle of the form  
\begin{equation}
\{z = x + iy \,|\, x \in [a,b] \subset \bbR, \; y \in [0,c]\}, \quad c > 0, 
\end{equation}
proving \eqref{10.121}. 

If
\begin{equation}
\ker\big(\big[I_{[L^2(\bbR^n)]^N} + \ol{V_2 (H_0 - (0 + i 0) I_{[L^2(\bbR^n)]^N})^{-1} V_1^*}\big]\big) 
\supsetneqq \{0\},      \lb{10.127} 
\end{equation} 
 denote by $P_{0,+}$ the projection onto the (finite-dimensional) eigenspace of the compact 
 operator $\ol{V_2 (H_0 - (0 + i 0) I_{[L^2(\bbR^n)]^N})^{-1} V_1^*}$ corresponding to the eigenvalue $-1$. 
 By Lemma \ref{l9.1}\,$(iii)$,
 \begin{equation}
\big(I_{[L^2(\bbR^n)]^N} + \ol{V_2 (H_0 - (0 + i 0) I_{[L^2(\bbR^n)]^N})^{-1} V_1^*} + P_{0,+}\big)^{-1} 
\in \cB\big([L^2(\bbR^n)]^N\big)
 \end{equation}
 and hence, 
 \begin{equation}
 {\det}_{[L^2(\bbR^n)]^N, n+1}\big(I_{[L^2(\bbR^n)]^N} + \ol{V_2 (H_0 - (0 + i 0) I_{[L^2(\bbR^n)]^N})^{-1} V_1^*} + P_{0,+}\big) \neq 0
 \end{equation}
 and 
 \begin{align}
&  {\det}_{[L^2(\bbR^n)]^N, n+1}\big(I_{[L^2(\bbR^n)]^N} + \ol{V_2 (H_0 - (0 + i 0) I_{[L^2(\bbR^n)]^N})^{-1} V_1^*}\big)    \no \\
& \quad = {\det}_{[L^2(\bbR^n)]^N, n+1}\big(I_{[L^2(\bbR^n)]^N} + \ol{V_2 (H_0 - (0 + i 0) I_{[L^2(\bbR^n)]^N})^{-1} V_1^*}    \no \\
& \hspace*{3.3cm} + P_{0,+} - P_{0,+}\big)    \no \\
& \quad = {\det}_{[L^2(\bbR^n)]^N, n+1}\Big(\big[I_{[L^2(\bbR^n)]^N} + \ol{V_2 (H_0 - (0 + i 0) I_{[L^2(\bbR^n)]^N})^{-1} V_1^*} + P_{0,+}\big]      \no \\
& \qquad \times \Big\{I_{[L^2(\bbR^n)]^N} - \big[I_{[L^2(\bbR^n)]^N} + \ol{V_2 (H_0 - (0 + i 0) I_{[L^2(\bbR^n)]^N})^{-1} V_1^*} + P_{0,+}\big]^{-1}\Big\}     \no \\
& \qquad \times P_{0,+}\Big).\lb{10.153a}
 \end{align} 
 
Applying Theorem \ref{tD.1} in \eqref{10.153a} one obtains
  \begin{align}
&  {\det}_{[L^2(\bbR^n)]^N, n+1}\big(I_{[L^2(\bbR^n)]^N} + \ol{V_2 (H_0 - (0 + i 0) I_{[L^2(\bbR^n)]^N})^{-1} V_1^*}\big)    \no \\
& \quad = {\det}_{[L^2(\bbR^n)]^N, n+1}\big(I_{[L^2(\bbR^n)]^N} + \ol{V_2 (H_0 - (0 + i 0) I_{[L^2(\bbR^n)]^N})^{-1} V_1^*} + P_{0,+}\big)     \no \\
& \qquad \times {\det}_{[L^2(\bbR^n)]^N, n+1} \Big(I_{[L^2(\bbR^n)]^N}      \no \\
& \hspace*{2cm} 
- \big[I_{[L^2(\bbR^n)]^N} + \ol{V_2 (H_0 - (0 + i 0) I_{[L^2(\bbR^n)]^N})^{-1} V_1^*} + P_{0,+}\big]^{-1}  
 P_{0,+}\Big)  \no \\
 & \qquad \times e^{ {\tr}_{[L^2(\bbR^n)]^N}(X_{n+1})}   \no \\
 & \quad = {\det}_{[L^2(\bbR^n)]^N, n+1}\big(I_{[L^2(\bbR^n)]^N} + \ol{V_2 (H_0 - (0 + i 0) I_{[L^2(\bbR^n)]^N})^{-1} V_1^*} + P_{0,+}\big)     \no \\
& \qquad \times {\det}_{[L^2(\bbR^n)]^N, n+1} \Big(I_{[L^2(\bbR^n)]^N}      \no \\
& \hspace*{1.5cm} 
- P_{0,+} \big[I_{[L^2(\bbR^n)]^N} + \ol{V_2 (H_0 - (0 + i 0) I_{[L^2(\bbR^n)]^N})^{-1} V_1^*} + P_{0,+}\big]^{-1}  
 P_{0,+}\Big)  \no \\
 & \qquad \times \exp({\tr}_{[L^2(\bbR^n)]^N}(X_{n+1})). 
 \end{align}

 Here ${\tr}_{[L^2(\bbR^n)]^N}(X_{n+1})$ is a finite sum of traces of products of the operators 
 \begin{equation} 
 \big[\ol{V_2 (H_0 - (0 + i 0) I_{[L^2(\bbR^n)]^N})^{-1} V_1^*} + P_{0,+}\big] 
 \end{equation}
 and 
 \begin{equation}
 - \big[I_{[L^2(\bbR^n)]^N} + \ol{V_2 (H_0 - (0 + i 0) I_{[L^2(\bbR^n)]^N})^{-1} V_1^*} + P_{0,+}\big]^{-1}  
 P_{0,+}
 \end{equation}
of at least $n+1$ factors (in various orders) as described in detail in Appendix \ref{sD}, in particular,
\begin{equation}
\exp({\tr}_{[L^2(\bbR^n)]^N}(X_{n+1})) \neq 0. 
\end{equation} 
Thus, the structure of the zero of the modified Fredholm determinant 
\begin{equation} 
{\det}_{[L^2(\bbR^n)]^N, n+1}\big(I_{[L^2(\bbR^n)]^N} + \ol{V_2 (H_0 - (0 + i 0) I_{[L^2(\bbR^n)]^N})^{-1} V_1^*}\big) 
\end{equation} 
as $z \to 0$, $ z \in \ol{\bbC_+}\backslash \{0\}$, is identical to the structure of the zero of the modified Fredholm  determinant (see, e.g., \cite[Theorem 9.2(d)]{Si05})
\begin{align}
& {\det}_{[L^2(\bbR^n)]^N, n+1} \Big(I_{[L^2(\bbR^n)]^N}      \no \\
& \hspace*{1.5cm} 
- P_{0,+} \big[I_{[L^2(\bbR^n)]^N} + \ol{V_2 (H_0 - (0 + i 0) I_{[L^2(\bbR^n)]^N})^{-1} V_1^*} + P_{0,+}\big]^{-1}  
 P_{0,+}\Big)      \no \\
& \quad = {\det}_{[P_{0,+} L^2(\bbR^n)]^N} \Big(P_{0,+} I_{[L^2(\bbR^n)]^N}      \no \\
& \hspace*{1.3cm} 
- P_{0,+} \big[I_{[P_{0,+} L^2(\bbR^n)]^N} + \ol{V_2 (H_0 - (0 + i 0) I_{[L^2(\bbR^n)]^N})^{-1} V_1^*} + P_{0,+}\big]^{-1}  P_{0,+}\Big)   \no \\
& \qquad \times \exp{\bigg(\sum_{j=1}^n j^{-1} {\tr}_{[L^2(\bbR^n)]^N}
\Big(\Big[P_{0,+}    } \no \\
& \hspace*{1.3cm}
\times {\big[I_{[P_{0,+} L^2(\bbR^n)]^N} + \ol{V_2 (H_0 - (0 + i 0) I_{[L^2(\bbR^n)]^N})^{-1} V_1^*} + P_{0,+}\big]^{-1}  P_{0,+}\Big]^j\Big)\bigg)},\lb{10.159a}
\end{align}
which now reduces to a finite-dimensional determinant. The behavior of the latter as $z \to 0$, $z \in \ol{\bbC_+}\backslash\{0\}$, 
\begin{align}
& {\det}_{[P_{0,+} L^2(\bbR^n)]^N} \Big(P_{0,+} I_{[L^2(\bbR^n)]^N}      \no \\
& \hspace*{1.3cm} 
- P_{0,+} \big[I_{[P_{0,+} L^2(\bbR^n)]^N} + \ol{V_2 (H_0 - z I_{[L^2(\bbR^n)]^N})^{-1} V_1^*} + P_{0,+}\big]^{-1}  P_{0,+}\Big) 
\end{align} 
in turn, is governed by Lemma \ref{l9.10} and hence in leading order 
is a polynomial $\cP(\dott,\dott)$ in the two variables $z \ln(z)$ and $z$ (the $z\ln(z)$ part being absent in odd space dimensions). By \eqref{10.127}, $\cP(\dott,\dott)$ has no constant term and hence its leading order is of the form 
\begin{equation}
\cP(z\ln(z),z) \underset{\substack{z \to 0 \\ z \in \ol{\bbC_+}\backslash\{0\}}}{=}  c z^{M_1} [\ln(z)]^{M_2}
[1 + \oh(1)], \quad M_1 \in \bbN, \; M_2 \in \bbN_0, \; c \in \bbC.     \lb{10.137} 
\end{equation}
Setting $z = \varepsilon e^{i\varphi}$, $\varphi \in [0, \pi]$, and letting $\varepsilon \downarrow 0$ in 
\eqref{10.137} then readily yields
\begin{equation}
\Im(\ln(\cP(z\ln(z),z))) \underset{\substack{z \to 0 \\ z \in \bbR\backslash\{0\}}}{=}  
\Im(\ln(c)) + \begin{cases} 0, & \varphi =0, \\
M_1 \pi, & \varphi = \pi, 
\end{cases} 
\end{equation} 
and hence proves the claim \eqref{9.84}.
\end{proof}

\section{Analysis of $G_{H,H_0}$} \lb{s10}

In this section we analyze $G_{H,H_0}(z)$, $z \in \bbC \backslash \bbR$, and its limiting behavior on $\bbR$.

One recalls from \eqref{8.28} (with $S=H$, $S_0=H_0$, and $m=n$, cf.\ Remark \ref{r8.4}\,$(iii)$), that $G_{H,H_0}$ is of the form 
\begin{align}
&\frac{d^n}{d z^n}G_{H,H_0}(z)=\tr_{[L^2(\bbR^n)]^N}\Bigg(\frac{d^{n-1}}{dz^{n-1}}\sum_{j=0}^{n-1}(-1)^{n-j}(H_0-zI_{[L^2(\bbR^n)]^N})^{-1}A(z)^{n-j}\Bigg),\no\\
&\hspace*{9.5cm}z\in \bbC\backslash \bbR,\lb{B.27y}
\end{align}
where
\begin{equation}
A(z) = V (H_0 - zI_{[L^2(\bbR^n)]^N})^{-1},\quad z\in \bbC\backslash \bbR.
\end{equation}

To analyze the trace in \eqref{B.27y}, we use multi-indices (see \eqref{9.32c} and \eqref{9.33c}).  For each fixed $j\in \bbN_0$ with $0\leq j\leq n-1$,
\begin{align}
&\frac{d^{n-1}}{dz^{n-1}}(H_0-zI_{[L^2(\bbR^n)]^N})^{-1}A(z)^{n-j}     \lb{B.25a} \\
&\quad = \sum_{\substack{\ul k\in \bbN_0^{n-j+1} \\ |\ul k|=n-1}} c_{j,\ul k}(H_0-zI_{[L^2(\bbR^n)]^N})^{-(k_1+1)} \prod_{\ell=2}^{n-j+1}V(H_0-zI_{[L^2(\bbR^n)]^N})^{-(k_{\ell}+1)},    \no 
\end{align}
for an appropriate set of $z$-independent scalars
\begin{equation}
c_{j,\ul k}\in \bbR,\quad \ul k\in \bbN_0^{n-j+1},\; |\ul k|=n-1.
\end{equation}
Therefore, applying the cyclicity property of the trace, one infers
\begin{align}
\frac{d^n}{d z^n}G_{H,H_0}(z) 
&= \sum_{j=0}^{n-1}(-1)^{n-j}\sum_{\substack{\ul k\in \bbN_0^{n-j+1} \\ |\ul k|=n-1}} c_{j,\ul k}   \no \\
& \quad \times \tr_{[L^2(\bbR^n)]^N}\Bigg(\Bigg[\prod_{\ell=2}^{n-j}
V(H_0-zI_{[L^2(\bbR^n)]^N})^{-(k_{\ell}+1)}\Bigg]      \lb{11.6} \\
& \hspace*{2.7cm} \times V(H_0-zI_{[L^2(\bbR^n)]^N})^{-(k_1+k_{n - j + 1}+2)}\Bigg),     \no 
\end{align}
and hence it suffices to analyze the trace
\begin{align}
& \tr_{[L^2(\bbR^n)]^N}\Bigg(\Bigg[\prod_{\ell=2}^{n-j}
V(H_0-zI_{[L^2(\bbR^n)]^N})^{-(k_{\ell}+1)}\Bigg]     \no \\
& \hspace*{2cm} \times V(H_0-zI_{[L^2(\bbR^n)]^N})^{-(k_1+k_{n - j + 1}+2)}\Bigg)    \no \\
& \quad = \int_{\bbR^n} d^n x_1 \cdots \int_{\bbR^n} d^n x_{n-j} \, V(x_1) \f{1}{{k_2}!} 
\bigg[\f{\partial^{k_2}}{\partial z^{k_2}}G_0(z;x_1,x_2)\bigg] 
\no \\
& \qquad \, \cdots \times V(x_{n-j-1}) \f{1}{k_{n-j}!}
\bigg[\f{\partial^{k_{n-j}}}{\partial z^{k_{n-j}}} 
G_0(z;x_{n-j-1},x_{n-j})\bigg]   \no \\
&  \qquad \, \cdots \times V(x_{n-j})  \f{1}{(k_1 + k_{n-j+1} + 1)!}
\bigg[\f{\partial^{k_1 + k_{n-j+1} + 1}}{\partial z^{k_1 + k_{n-j+1} + 1}}
G_0(z;x_{n-j},x_1)\bigg],  \lb{11.7A} \\
&\hspace*{3.37cm} z \in \bbC_+,  \; \ul k\in \bbN_0^{n-j+1}, \;  |\ul k|=n-1, \; 0 \leq j \leq n - 1,  \no
\end{align}
and its properties as $\Im(z) \downarrow 0$. 

Here we employed the fact that the integral kernel of 
\begin{equation} 
(H_0-zI_{[L^2(\bbR^n)]^N})^{-(1+r)} = \f{1}{r!} \f{d^r}{dz^r} (H_0-zI_{[L^2(\bbR^n)]^N})^{-1}, 
\quad z \in \bbC_+, \; r \in \bbN_0, 
\end{equation}
is of the form
\begin{equation}
\f{1}{r!} \f{\partial^{r}}{\partial z^{r}}G_0(z;x,y), \quad z \in \bbC_+, \; r \in \bbN_0, \; x, y \in \bbR^n, \; x \neq y.    \lb{11.9A}
\end{equation} 

Next, we recall the asymptotic relations proved in Appendix \ref{sC} and the estimates \eqref{C.33}, \eqref{C.34}. In particular, the estimates \eqref{C.33} and \eqref{C.34} as $|z||x-y| \geq 1$ necessitate the following 
strengthening of the estimate \eqref{9.81} in Hypothesis \ref{h9.13}:

\begin{hypothesis} \lb{h11.1a}
Let $n \in \bbN$ and suppose that $V= \{V_{\ell,m}\}_{1 \leq \ell,m \leq N}$ satisfies for some constant 
$C \in (0,\infty)$ and $\varepsilon > 0$, 
\begin{equation}
V \in [L^{\infty} (\bbR^n)]^{N \times N}, \quad 
|V_{\ell,m}(x)| \leq C \langle x \rangle^{- n - 1 - \varepsilon} \, \text{ for a.e.~$x \in \bbR^n$, $1 \leq \ell,m \leq N$.}    \lb{11.10}
\end{equation}
\end{hypothesis} 

This yields the following result.

\begin{theorem} \lb{t11.1} 
Assume Hypothesis \ref{h11.1a}, \\[1mm]
$(i)$ Let $n \in \bbN$ be odd, $n \geq 3$. Then $\frac{d^n}{d z^n}G_{H,H_0}(\dott)$ is analytic in $\bbC_+$ and 
continuous in $\ol{\bbC_+}$. \\[1mm] 
$(ii)$  Let $n \in \bbN$ be even. Then $\frac{d^n}{d z^n}G_{H,H_0}(\dott)$ is analytic in $\bbC_+$, continuous in $\ol{\bbC_+}\backslash\{0\}$.  If $n\geq 4$, then
\begin{align}
\bigg\|\frac{d^n}{dz^n}G_{H,H_0}(\dott)\bigg\|_{\cB(\bbC^N)}\underset{\substack{z \to 0, \\ z \in \ol{\bbC_+} \backslash\{0\}}}{=}O\big(|z|^{-[n - (n/(n-1))]}\big).
\end{align}
If $n=2$, then for any $\delta\in (0,1)$,
\begin{align}
\bigg\|\frac{d^2}{dz^2}G_{H,H_0}(\dott)\bigg\|_{\cB(\bbC^2)}\underset{\substack{z \to 0, \\ z \in \ol{\bbC_+} \backslash\{0\}}}{=} O\big(|z|^{-(1+\delta)}\big).
\end{align}
\end{theorem}
\begin{proof}
By Lemma \ref{l8.9} it suffices to focus on the boundary values of $\frac{d^n}{d z^n}G_{H,H_0}(z)$ as 
$\Im(z) \downarrow 0$. Utilizing the asymptotic relations  
\eqref{B.6a}, \eqref{B.12y}, \eqref{B.14y}, \eqref{B.13y}, \eqref{B.15z}, \eqref{B.22z}, \eqref{B.23zz}, \eqref{B.24z}, \eqref{B.23z},  \eqref{B.31z},  \eqref{B.32z}, and the fact that 
$\f{\partial^{k}}{\partial z^{k}}G_0(z;x,y)$, $k \in \bbN_0$, $0 \leq k \leq n$, is continuous in $z \in \ol{\bbC_+}$, 
$x, y \in \bbR^n$, $x \neq y$, the stated continuity of $\frac{d^n}{d z^n}G_{H,H_0}(\dott)$ in $\ol{\bbC_+}$ follows once we derive a $z$-independent integrable majorant of the integrand in \eqref{11.7A}, appealing to Lebesgue's dominated convergence theorem. \\[1mm]
\noindent 
$(i)$ Specializing to $n \in \bbN$ odd, $n\geq 3$, and employing \eqref{11.10} and \eqref{C.33}, one obtains 
from \eqref{11.7A}, 
\begin{align}
& \bigg\|\int_{\bbR^n} d^n x_1 \cdots \int_{\bbR^n} d^n x_{n-j-1} \int_{\bbR^n} d^n x_{n-j} \, V(x_1) \f{1}{{k_2}!} 
\bigg[\f{\partial^{k_2}}{\partial z^{k_2}}G_0(z;x_1,x_2)\bigg] 
\no \\
& \qquad \, \cdots \times V(x_{n-j-1}) \f{1}{k_{n-j}!}
\bigg[\f{\partial^{k_{n-j}}}{\partial z^{k_{n-j}}} 
G_0(z;x_{n-j-1},x_{n-j})\bigg]   \no \\
&  \qquad \, \cdots \times V(x_{n-j})  \f{1}{(k_1 + k_{n-j+1} + 1)!}
\bigg[\f{\partial^{k_1 + k_{n-j+1} + 1}}{\partial z^{k_1 + k_{n-j+1} + 1}}
G_0(z;x_{n-j},x_1)\bigg] \bigg\|_{\cB(\bbC^N)}  \no \\
& \quad \leq C_n \int_{\bbR^n} d^n x_1 \cdots \int_{\bbR^n} d^n x_{n-j-1} \int_{\bbR^n} d^n x_{n-j}   \no \\
& \hspace*{1.6cm} \times \langle x_1 \rangle^{-n-1-\varepsilon} 
\big\{|x_1-x_2|^{k_2+1-n} \chi_{[0,1]}(|z||x_1-x_2|)    \no \\
& \hspace*{1.6cm} \quad + |z|^{(n-1)/2} 
\big[|x_1|^{(2k_2+1-n)/2} + |x_2|^{(2k_2+1-n)/2}\big] \chi_{[1,\infty)}(|z||x_1-x_2|)\big\}
  \no \\ 
&  \hspace*{1.6cm} \; \, \vdots   \no \\
& \hspace*{1.6cm} \times \langle x_{n-j-1} \rangle^{-n-1-\varepsilon}  
\big\{|x_{n-j-1}-x_{n-j}|^{k_{n-j}+1-n} \chi_{[0,1]}(|z||x_{n-j-1}-x_{n-j}|)    \no \\
& \hspace*{1.6cm} \quad + |z|^{(n-1)/2} 
\big[|x_{n-j-1}|^{(2k_{n-j}+1-n)/2} + |x_{n-j}|^{(2k_{n-j}+1-n)/2}\big]   \no \\ 
& \hspace*{1.6cm} \qquad \times \chi_{[1,\infty)}(|z||x_{n-j-1}-x_{n-j}|)\big\}  \no \\ 
&  \hspace*{1.6cm} \times \langle x_{n-j} \rangle^{-n-1-\varepsilon}  
\big\{|x_{n-j}-x_1|^{k_1+k_{n-j+1}+2-n} \chi_{[0,1]}(|z||x_{n-j}-x_1|)    \no \\
& \hspace*{1.6cm} \quad + |z|^{(n-1)/2} 
\big[|x_{n-j}|^{(2k_1+2k_{n-j+1}+3-n)/2} + |x_1|^{(2k_1+2k_{n-j+1}+3-n)/2}\big]    \no \\
& \hspace*{1.6cm} \qquad \times \chi_{[1,\infty)}(|z||x_{n-j}-x_1|)\big\}  \no \\ 
& \quad \leq \wti C_n \int_{\bbR^n} d^n x_1 \cdots \int_{\bbR^n} d^n x_{n-j-1} \int_{\bbR^n} d^n x_{n-j}   \no \\
& \hspace*{1.6cm} \times \langle x_1 \rangle^{-n-1-\varepsilon} 
\big\{|x_1-x_2|^{k_2+1-n}    \no \\
& \hspace*{1.6cm} \quad + |z|^{(n-1)/2} [1+|x_1|]^{(2k_2+1-n)/2} [1+|x_2|]^{(2k_2+1-n)/2}\big\}    \no \\ 
&  \hspace*{1.6cm} \;\, \vdots   \no \\
& \hspace*{1.6cm} \times \langle x_{n-j-1} \rangle^{-n-1-\varepsilon}  
\big\{|x_{n-j-1}-x_{n-j}|^{k_{n-j}+1-n}    \no \\
& \hspace*{1.6cm} \quad + |z|^{(n-1)/2} 
[1+|x_{n-j-1}|]^{(2k_{n-j}+1-n)/2} [1+|x_{n-j}|]^{(2k_{n-j}+1-n)/2}\big\}   \no \\ 
&  \hspace*{1.6cm} \times \langle x_{n-j} \rangle^{-n-1-\varepsilon}  
\big\{|x_{n-j}-x_1|^{k_1+k_{n-j+1}+2-n}    \no \\
& \hspace*{1.6cm} \quad + |z|^{(n-1)/2} [1+|x_{n-j}|]^{(2k_1+2k_{n-j+1}+3-n)/2}     \no \\
& \hspace*{3.75cm} \times [1+|x_1|]^{(2k_1+2k_{n-j+1}+3-n)/2}\big\},     \lb{11.11} \\
&\hspace*{1.3cm} z \in \bbC_+,  \; \ul k\in \bbN_0^{n-j+1}, \;  |\ul k|=n-1, \; 0 \leq j \leq n - 2,    \no
\end{align}
where $C_n, \wti C_n \in (0,\infty)$ are suitable constants and we removed all characteristic functions in the last step (a very crude estimate, but sufficient for our purpose). 

We postpone a discussion of the case $j=n-1$ to the end of the proof of part $(i)$. 

Next, one notes that all terms originally multiplied by an ``exterior'' characteristic function $\chi_{[1,\infty)}(|z||\dott|)$, that is, all terms of the type
\begin{align} 
&|z|^{(n-1)/2} [1+|x_{n-j-1}|]^{(2k_{n-j}+1-n)/2} [1+|x_{n-j}|]^{(2k_{n-j}+1-n)/2},   \\
& \quad \dots |z|^{(n-1)/2} [1+|x_{n-j}|]^{(2k_1+2k_{n-j+1}+3-n)/2} [1+|x_1|]^{(2k_1+2k_{n-j+1}+3-n)/2},   \lb{11.12a}
\end{align}
can be grouped together with 
\begin{equation} 
\langle x_{n-j-1} \rangle^{-n-1-\varepsilon}, \langle x_{n-j} \rangle^{-n-1-\varepsilon}, \dots , 
\langle x_{n-j} \rangle^{-n-1-\varepsilon}, \langle x_{1} \rangle^{-n-1-\varepsilon},
\end{equation} 
due to the decay assumptions imposed in \eqref{11.10}, and hence we can simply disregard all these contributions in the following as they lead to finite integrals. To illustrate this fact we look at the extreme case where only these terms are considered. Indeed, ignoring all numerical constants and the factors $|z|^{(n-1)/2}$ for simplicity, this leads to the integral, 
\begin{align}
& \int_{\bbR^n} d^n x_1 \cdots \int_{\bbR^n} d^n x_{n-j-1} \int_{\bbR^n} d^n x_{n-j}   \no \\
& \qquad \times \langle x_1 \rangle^{-n-1-\varepsilon} 
[1+|x_1|]^{(2k_2+1-n)/2} [1+|x_2|]^{(2k_2+1-n)/2}     \no \\ 
& \qquad \;\, \vdots   \no \\
& \qquad \times \langle x_{n-j-1} \rangle^{-n-1-\varepsilon}  
[1+|x_{n-j-1}|]^{(2k_{n-j}+1-n)/2} [1+ |x_{n-j}|]^{(2k_{n-j}+1-n)/2}   \no \\ 
&  \qquad \times \langle x_{n-j} \rangle^{-n-1-\varepsilon}  
[1+|x_{n-j}|]^{(2k_1+2k_{n-j+1}+3-n)/2}     \no \\
& \hspace*{3.35cm} \times [1+|x_1|]^{(2k_1+2k_{n-j+1}+3-n)/2}    \no \\
& \quad \leq \int_{\bbR^n} d^n x_1 \, \langle x_1 \rangle^{-n - 1 - \varepsilon} [1+|x_1|]^{k_1+k_2+k_{n-j+1}+2-n}    
\no \\ 
& \qquad \times \int_{\bbR^n} d^n x_2 \, \langle x_2 \rangle^{-n - 1 - \varepsilon} [1+|x_2|]^{k_2+k_3+1-n}  \no \\
& \qquad \;\, \vdots   \no \\
& \qquad \times \int_{\bbR^n} d^n x_{n-j-1} \, \langle x_{n-j-1} \rangle^{-n - 1 - \varepsilon} 
[1+|x_{n-j-1}|]^{k_{n-j-1}+k_{n-j}+1-n}  \no \\ 
& \qquad \times \int_{\bbR^n} d^n x_{n-j} \, \langle x_{n-j} \rangle^{-n - 1 - \varepsilon} 
[1+|x_{n-j}|]^{k_1+k_{n-j}+k_{n-j+1}+2-n}  \no \\ 
& \quad \leq \int_{\bbR^n} d^n x_1 \, \langle x_1 \rangle^{-n - 1 - \varepsilon} [1+|x_1|]  
\bigg[\int_{\bbR^n} d^n y \, \langle y \rangle^{-n - 1 - \varepsilon}\bigg]^{n-j-2}   \no \\ 
& \qquad \times \int_{\bbR^n} d^n x_{n-j} \, \langle x_{n-j} \rangle^{-n - 1 - \varepsilon} [1+|x_{n-j}|]    \no \\ 
& \quad < \infty, \quad \ul k\in \bbN_0^{n-j+1}, \;  |\ul k|=n-1, \; 0 \leq j \leq n-2,       \lb{11.14} 
\end{align} 
employing \eqref{11.10}.

Thus, without loss of generality, we now focus on the terms originally multiplied by an ``interior'' characteristic 
function $\chi_{[0,1]}(|z||\dott|)$ and hence arrive at the need to estimate the integral 
\begin{align}
& \int_{\bbR^n} d^n x_1 \cdots \int_{\bbR^n} d^n x_{n-j-1} \int_{\bbR^n} d^n x_{n-j} \, 
\langle x_1 \rangle^{-n-1-\varepsilon} |x_1-x_2|^{k_2+1-n}    \no \\
& \qquad \;\, \vdots   \no \\
& \qquad \times \langle x_{n-j-1} \rangle^{-n-1-\varepsilon}  
|x_{n-j-1}-x_{n-j}|^{k_{n-j}+1-n}    \no \\
&  \qquad \times \langle x_{n-j} \rangle^{-n-1-\varepsilon}  
|x_{n-j}-x_1|^{k_1+k_{n-j+1}+2-n},      \lb{11.15} \\
&\hspace*{1.15cm} \ul k\in \bbN_0^{n-j+1}, \;  |\ul k|=n-1, \; 0 \leq j \leq n - 2.  \no 
\end{align}
For this purpose we recall the following special case of Lemma \ref{l3.12}, 
\begin{align}
& \int_{\bbR^n} d^ny \, |y_1 - y|^{\alpha - n} \langle y \rangle^{-\gamma - \varepsilon} |y - y_2|^{\beta - n}  \no \\ 
& \quad \leq \wti C_{n,\alpha,\beta,\gamma,\varepsilon} \begin{cases}
|y_1 - y_2|^{\min(n,\alpha + \beta) - n}, & |y_1 - y_2| \leq 1, \\
|y_1 - y_2|^{\max(\alpha, \beta) - n}, & |y_1 - y_2| \geq 1,
\end{cases}     \no \\
& \quad \leq \wti C_{n,\alpha,\beta,\gamma,\varepsilon} \begin{cases}
|y_1 - y_2|^{\min(n,\alpha + \beta) - n}, & |y_1 - y_2| \leq 1, \\
1, & |y_1 - y_2| \geq 1,
\end{cases}     \no \\
& \quad \leq C_{n,\alpha,\beta,\gamma,\varepsilon} \big[|y_1 - y_2|^{\min(n,\alpha + \beta) - n} + 1\big],   
\lb{11.16} \\
& \hspace*{7.5mm} \alpha, \beta \in (0,n], \; \gamma > (\alpha + \beta) -n, \; \varepsilon > 0,   \no 
\end{align}
for appropriate constants 
$\wti C_{n,\alpha,\beta,\gamma,\varepsilon}, C_{n,\alpha,\beta,\gamma,\varepsilon} \in (0,\infty)$. 

Hence,
\begin{align}
& \int_{\bbR^n} d^nx_2 \, |x_1-x_2|^{k_2+1-n} \langle x_2 \rangle^{-n -1 - \varepsilon} |x_2-x_3|^{k_3+1-n}  \no \\
& \quad \leq c_{2,n} \big[|x_1-x_3|^{\min(n,k_2+k_3+2)-n} + 1\big],    \lb{11.17} 
\end{align} 
for some $c_{2,n} \in (0,\infty)$. For precisely the same reason as in the context of \eqref{11.14}, we will simply 
disregard the additive term $+1$ on the right-hand side of \eqref{11.17} as the latter is bounded and we want to 
focus on the possibly most singular contribution to the integral in \eqref{11.15} when probing whether or not this integral is finite. 

Thus, with these simplifications of ignoring $1$'s and at the same time focusing on the possibly most singular contribution, the next integral over $x_3$ becomes
\begin{align}
& \int_{\bbR^n} d^nx_3 \, |x_1-x_3|^{k_2+k_3+2-n} 
\langle x_3 \rangle^{-n -1 - \varepsilon} |x_3-x_4|^{k_4+1-n}  \no \\
& \quad \leq c_{3,n} \big[|x_1-x_4|^{\min(n,k_2+k_3+k_4+3)-n} + 1\big],    \lb{11.18} 
\end{align} 
for some $c_{3,n} \in (0,\infty)$. Repeating this process (again disregarding $1$'s at each step and focusing on the possibly most singular contributions only) leads to 
\begin{align}
& \int_{\bbR^n} d^nx_{n-j-1} \, |x_1-x_{n-j-1}|^{k_2+k_3+ \cdots+k_{n-j-1}+(n-j-2)-n} 
\langle x_{n-j-1} \rangle^{-n -1 - \varepsilon}   \no \\
& \qquad \times |x_{n-j-1}-x_{n-j}|^{k_{n-j}+1-n}  \no \\
& \quad \leq c_{n-j-1,n} \big[|x_1-x_{n-j}|^{\min(n,k_2+k_3+ \cdots + k_{n-j} + (n-j-1) )- n} + 1\big],   \lb{11.19} 
\end{align} 
for some $c_{n-j-1,n} \in (0,\infty)$. Thus, disregarding once more the additive constant $+1$ in \eqref{11.19} results in the following integral over $x_{n-j}$, $0 \leq j \leq n-2$, 
\begin{align}
& \int_{\bbR^n} d^nx_{n-j} \, \Bigg[ \begin{cases} |x_1-x_{n-j}|^{k_2+k_3+ \cdots+k_{n-j}-j-1}, 
& \big(\sum_{q=2}^{n-j}k_q\big) -j-1 \leq 0, \\
1, & \big(\sum_{q=2}^{n-j}k_q\big)-j-1 \geq 0 \end{cases} \Bigg]   \no \\
& \qquad \times \langle x_{n-j} \rangle^{-n -1 - \varepsilon} |x_{n-j}-x_1|^{k_1+k_{n-j+1}+2-n}  \no \\
& \quad \leq  c \int_{\bbR^n} d^nx_{n-j} \, \langle x_{n-j} \rangle^{- n -1 - \varepsilon} 
|x_1-x_{n-j}|^{k_1+k_2+ \cdots+k_{n-j+1}-j+1-n}    \no \\ 
& \qquad + d \int_{\bbR^n} d^nx_{n-j} \, \langle x_{n-j} \rangle^{-n -1 - \varepsilon} 
|x_{n-j}-x_1|^{k_1+k_{n-j+1}+2-n}  \no \\ 
& \quad \leq C \int_{\bbR^n} d^nx_{n-j} \, \langle x_{n-j} \rangle^{-n -1 - \varepsilon} 
\big[|x_1-x_{n-j}|^{-j} + |x_1-x_{n-j}|^{-m}\big]  \no \\ 
& \hspace*{10.5mm} \text{for some $-1 \leq m \leq n-2$, where $m=k_1+k_{n-j+1}+2-n$},   \lb{11.20}
\end{align}

for appropriate $c, d, C \in (0,\infty)$, employing again that 
\begin{equation} 
\ul k = (k_1,\dots,k_{n-j+1}) \in \bbN_0^{n-j+1}, \quad  |\ul k| = k_1+k_2 + \cdots k_{n-j+1} = n-1. 
\end{equation} 
At this point we invoke the special case $\alpha = n$ in \eqref{11.16}, resulting in 
\begin{align}
\begin{split} 
\int_{\bbR^n} d^nx_{n-j} \, \langle x_{n-j} \rangle^{-n -1 - \varepsilon} 
\big[|x_1-x_{n-j}|^{-j} + |x_1-x_{n-j}|^{-m}\big] \leq C_{j,m},& \\
0 \leq j \leq n-2, \; 0 \leq m \leq n-2,&    \lb{11.22} 
\end{split} 
\end{align} 
for some $C_{j,m} \in (0,\infty)$. The remaining case $m = - 1$ in \eqref{11.20} leads to 
\begin{align} 
& \int_{\bbR^n} d^nx_{n-j} \, \langle x_{n-j} \rangle^{-n -1 - \varepsilon} |x_1-x_{n-j}|    \no \\
& \quad \leq \int_{\bbR^n} d^nx_{n-j} \, \langle x_{n-j} \rangle^{-n -1 - \varepsilon} [1 + |x_1| + |x_{n-j}|] 
\no \\
& \quad \leq C_n + D_n [1 + |x_1|],    \lb{11.23} 
\end{align} 
for some $C_n, D_n \in (0,\infty)$, since 
\begin{equation}
\int_{\bbR^n} d^n y \, \langle y \rangle^{-n -1 - \varepsilon} [1 + |y|] < \infty.     \lb{11.24} 
\end{equation}
Thus, altogether, \eqref{11.20}--\eqref{11.23} finally yield 
\begin{equation}
\eqref{11.20} \leq C_0 [1 + |x_1|], \quad 0 \leq j \leq n-2,    \lb{11.25}
\end{equation}
for appropriate $C_0 \in (0,\infty)$. Hence, applying \eqref{11.24} once more, the integral \eqref{11.15} is finite.

If $j = n-1$ in \eqref{B.25a}, \eqref{11.6} one is left to consider $\ul k = (k_1,k_2)$, $|\ul k| = k_1+k_2 = n-1$, and hence obtains
\begin{align}
& \tr_{[L^2(\bbR^n)]^N} \bigg(\f{d^{n-1}}{dz^{n-1}} (H_0 - z I_{[L^2(\bbR^n)]^N})^{-1} A(z)\bigg)    \no \\
& \quad =  \tr_{[L^2(\bbR^n)]^N} \bigg(\f{d^{n-1}}{dz^{n-1}} (H_0 - z I_{[L^2(\bbR^n)]^N})^{-1} 
V (H_0 - z I_{L^2(\bbR^n)]^N})^{-1}\bigg)    \no \\
& \quad =  \tr_{[L^2(\bbR^n)]^N} \bigg(V (H_0 - z I_{L^2(\bbR^n)]^N})^{-(k_1+k_2 +2)}\bigg)    \no \\
& \quad =  \tr_{[L^2(\bbR^n)]^N} \bigg(V (H_0 - z I_{L^2(\bbR^n)]^N})^{-(n+1)}\bigg), \quad z \in \bbC_+.  
\lb{11.26} 
\end{align}
Since by \eqref{B.10}--\eqref{B.11Z} 
\begin{equation}
\f{d^n}{dz^n} G_0(z;x,y) \underset{|x-y| \to 0}{=} \Oh(|x-y|),      \lb{11.27} 
\end{equation}
the trace in \eqref{11.23} vanishes and hence extends continuously to $z \in \ol{\bbC_+}$. \\[1mm] 
$(ii)$ Next, we specialize to $n \in \bbN$ even.  We investigate each term in \eqref{11.6} separately.  To this end, let $0\leq j\leq n-1$ and $\ul k\in \bbN_0^{n-j+1}$ with $|\ul k|=n-1$ be fixed.   We distinguish the following cases:

\begin{description}
   \item[{\it Case~1}] $n\geq 4$ with $0\leq j\leq n-3$ and $k_1+k_{n-j+1}\neq n-1$. 
   \item[{\it Case~2}] $n\geq 4$ with $0\leq j\leq n-3$ and $k_1+k_{n-j+1}= n-1$.
   \item[{\it Case~3}] $n\geq 2$ with $j=n-2$ and $k_1+k_3\neq n-1$.
   \item[{\it Case~4}] $n\geq 2$ with $j=n-2$ and $k_1+k_3= n-1$.
   \item[{\it Case~5}] $n\geq 2$ with $j=n-1$.
\end{description}

We begin with {\it Case~1}.  The assumptions in {\it Case~1} imply
\begin{equation} \lb{11.28}
\text{$0\leq k_\ell \leq n-1$ for all $2\leq \ell\leq n-j$ and $k_1+k_{n-j+1}+1<n$}.
\end{equation}
Define the quantity $\delta=\delta(n,j)$ by
\begin{equation} \lb{11.29}
\delta:=\frac{n-j-2}{n-j-1},
\end{equation}
so that $\delta\in (0,1)$ and
\begin{equation} \lb{11.30}
\delta \geq \frac{n-j-\ell}{n-j-\ell+1},\quad 2\leq \ell\leq n-j-1.
\end{equation}
Invoking the final estimate in \eqref{C.34}, one obtains
\begin{align}
& \bigg\|\int_{\bbR^n} d^n x_1 \cdots \int_{\bbR^n} d^n x_{n-j-1} \int_{\bbR^n} d^n x_{n-j} \, V(x_1) \f{1}{{k_2}!} 
\bigg[\f{\partial^{k_2}}{\partial z^{k_2}}G_0(z;x_1,x_2)\bigg] 
\no \\
& \qquad \, \cdots \times V(x_{n-j-1}) \f{1}{k_{n-j}!}
\bigg[\f{\partial^{k_{n-j}}}{\partial z^{k_{n-j}}} 
G_0(z;x_{n-j-1},x_{n-j})\bigg]   \no \\
&  \hspace*{1.25cm} \times V(x_{n-j})  \f{1}{(k_1 + k_{n-j+1} + 1)!}
\bigg[\f{\partial^{k_1 + k_{n-j+1} + 1}}{\partial z^{k_1 + k_{n-j+1} + 1}}
G_0(z;x_{n-j},x_1)\bigg] \bigg\|_{\cB(\bbC^N)}  \no \\
& \quad \leq C_{n,j,\delta}|z|^{-(n-j)\delta} \int_{\bbR^n} d^n x_1 \cdots \int_{\bbR^n} d^n x_{n-j-1} \int_{\bbR^n} d^n x_{n-j}   \no \\
& \hspace*{1.4cm} \times \langle x_1 \rangle^{-n-1-\varepsilon} 
\big\{|x_1-x_2|^{k_2+1-\delta-n} \chi_{[0,1]}(|z||x_1-x_2|)    \no \\
& \hspace*{1.6cm} \quad + |z|^{(n-1+2\delta)/2} 
\big[|x_1|^{(2k_2+1-n)/2} + |x_2|^{(2k_2+1-n)/2}\big] \chi_{[1,\infty)}(|z||x_1-x_2|)\big\}
  \no \\ 
&  \hspace*{1.4cm} \; \, \vdots   \no \\
& \hspace*{1.4cm} \times \langle x_{n-j-1} \rangle^{-n-1-\varepsilon}  
\big\{|x_{n-j-1}-x_{n-j}|^{k_{n-j}+1-\delta-n} \chi_{[0,1]}(|z||x_{n-j-1}-x_{n-j}|)    \no \\
& \hspace*{1.4cm} \quad + |z|^{(n-1+2\delta)/2} 
\big[|x_{n-j-1}|^{(2k_{n-j}+1-n)/2} + |x_{n-j}|^{(2k_{n-j}+1-n)/2}\big]   \no \\ 
& \hspace*{1.4cm} \qquad \times \chi_{[1,\infty)}(|z||x_{n-j-1}-x_{n-j}|)\big\}  \no \\ 
&  \hspace*{1.4cm} \times \langle x_{n-j} \rangle^{-n-1-\varepsilon}  
\big\{|x_{n-j}-x_1|^{k_1+k_{n-j+1}+2-\delta-n} \chi_{[0,1]}(|z||x_{n-j}-x_1|)    \no \\
& \hspace*{1.4cm} \quad + |z|^{(n-1+2\delta)/2} 
\big[|x_{n-j}|^{(2k_1+2k_{n-j+1}+3-n)/2} + |x_1|^{(2k_1+2k_{n-j+1}+3-n)/2}\big]    \no \\
& \hspace*{1.4cm} \qquad \times \chi_{[1,\infty)}(|z||x_{n-j}-x_1|)\big\}  \no \\ 
& \quad \leq \wti C_{n,j,\delta}|z|^{-(n-j)\delta} \int_{\bbR^n} d^n x_1 \cdots \int_{\bbR^n} d^n x_{n-j-1} \int_{\bbR^n} d^n x_{n-j}   \lb{11.31} \\
& \hspace*{1.4cm} \times \langle x_1 \rangle^{-n-1-\varepsilon} 
\big\{|x_1-x_2|^{k_2+1-\delta-n}     \no \\
& \hspace*{1.4cm} \quad + |z|^{(n-1+2\delta)/2} 
\big[|x_1|^{(2k_2+1-n)/2} + |x_2|^{(2k_2+1-n)/2}\big]\big\}
  \no \\ 
&  \hspace*{1.4cm} \; \, \vdots   \no \\
& \hspace*{1.4cm} \times \langle x_{n-j-1} \rangle^{-n-1-\varepsilon}  
\big\{|x_{n-j-1}-x_{n-j}|^{k_{n-j}+1-\delta-n}   \no \\
& \hspace*{1.4cm} \quad + |z|^{(n-1+2\delta)/2} 
\big[|x_{n-j-1}|^{(2k_{n-j}+1-n)/2} + |x_{n-j}|^{(2k_{n-j}+1-n)/2}\big]\big\}   \no \\
&  \hspace*{1.4cm} \times \langle x_{n-j} \rangle^{-n-1-\varepsilon}  
\big\{|x_{n-j}-x_1|^{k_1+k_{n-j+1}+2-\delta-n}    \no \\
& \hspace*{1.4cm} \quad + |z|^{(n-1+2\delta)/2} 
\big[|x_{n-j}|^{(2k_1+2k_{n-j+1}+3-n)/2} + |x_1|^{(2k_1+2k_{n-j+1}+3-n)/2}\big]\big\},    \no\\
&\hspace*{11.4cm} z \in \bbC_+,    \no
\end{align}
where $C_{n,j,\delta}, \wti C_{n,j,\delta} \in (0,\infty)$ are suitable constants and we removed all characteristic functions in the last step (again, a very crude estimate, but sufficient for our purpose).

We claim that for each bounded subset $\Omega\subset \bbC_+$, the integrand under the iterated integral on the right-hand side in \eqref{11.31} is uniformly bounded with respect to $z\in \Omega$ by an integrable function of the variables $x_1,...,x_{n-j}$.  Since $|z|^{(n-1+2\delta)/2}$ is locally bounded, to justify the claim, it suffices to establish convergence of the following integral:
\begin{align}
&\int_{\bbR^n} d^n x_1 \cdots \int_{\bbR^n} d^n x_{n-j-1} \int_{\bbR^n} d^n x_{n-j}   \lb{11.32} \\
& \hspace*{1.4cm} \times \langle x_1 \rangle^{-n-1-\varepsilon} 
\big\{|x_1-x_2|^{k_2+1-\delta-n}     \no \\
& \hspace*{1.4cm} \quad + 
\big[|x_1|^{(2k_2+1-n)/2} + |x_2|^{(2k_2+1-n)/2}\big]\big\}
  \no \\ 
&  \hspace*{1.4cm} \; \, \vdots   \no \\
& \hspace*{1.4cm} \times \langle x_{n-j-1} \rangle^{-n-1-\varepsilon}  
\big\{|x_{n-j-1}-x_{n-j}|^{k_{n-j}+1-\delta-n}   \no \\
& \hspace*{1.4cm} \quad +
\big[|x_{n-j-1}|^{(2k_{n-j}+1-n)/2} + |x_{n-j}|^{(2k_{n-j}+1-n)/2}\big]\big\}   \no \\
&  \hspace*{1.4cm} \times \langle x_{n-j} \rangle^{-n-1-\varepsilon}  
\big\{|x_{n-j}-x_1|^{k_1+k_{n-j+1}+2-\delta-n}    \no \\
& \hspace*{1.4cm} \quad + |
\big[|x_{n-j}|^{(2k_1+2k_{n-j+1}+3-n)/2} + |x_1|^{(2k_1+2k_{n-j+1}+3-n)/2}\big]\big\}.\no
\end{align}
In turn, as in the argument for the proof of part $(i)$, it suffices to focus on the most singular term in \eqref{11.32} and thus disregard the terms originally multiplied by the factor $|z|^{(n-1+2\delta)/2}$ in \eqref{11.31} (following the same line of reasoning used throughout \eqref{11.12a}--\eqref{11.14}).  With this simplification, the claim reduces to establishing convergence of the integral
\begin{align}
&\int_{\bbR^n} d^n x_1 \cdots \int_{\bbR^n} d^n x_{n-j-1} \int_{\bbR^n} d^nx_{n-j}\lb{11.34}\\
&\hspace*{1.4cm}\times \langle x_1 \rangle^{-n-1-\varepsilon}|x_1-x_2|^{k_2+1-\delta-n}     \no \\
&\hspace{1.55cm}\vdots\no\\
& \hspace*{1.4cm} \times \langle x_{n-j-1} \rangle^{-n-1-\varepsilon}  |x_{n-j-1}-x_{n-j}|^{k_{n-j}+1-\delta-n}   \no \\
&  \hspace*{1.4cm} \times \langle x_{n-j} \rangle^{-n-1-\varepsilon}  |x_{n-j}-x_1|^{k_1+k_{n-j+1}+2-\delta-n}.\no
\end{align}
The integrals over the inner variables $x_2,\ldots,x_{n-j-1}$ in \eqref{11.34} can be estimated successively as follows.  Beginning with the integral with respect to $x_2$, an application of \eqref{11.16} with the choices
\begin{equation} \lb{11.35}
\alpha = k_2+1-\delta,\quad \beta=k_3+1-\delta,\quad \gamma=n+1,
\end{equation}
implies
\begin{align}
&\int d^nx_2\, |x_1 - x_2|^{k_2+1-\delta-n} \langle x_2 \rangle^{-n-1-\varepsilon} |x_2 - x_3|^{k_3+1-\delta-n} \lb{11.36}\\
&\quad \leq c_{2,n,\delta} \Big[|x_1 - x_3|^{k_2+k_3+2(1-\delta)-n} + 1\Big],\no
\end{align}
for some $c_{2,n,\delta}\in (0,\infty)$.  The conditions on $\alpha$ and $\beta$ in \eqref{11.16} are satisfied by the choices in \eqref{11.35} since \eqref{11.28} implies
\begin{equation}
0\leq k_2 + 1 - \delta \leq n\quad \text{and}\quad 0 < k_3 + 1 - \delta \leq n,
\end{equation}
together with
\begin{align} \lb{11.38}
(\alpha + \beta) - n &= k_2 + k_3 + 2(1-\delta) - n\leq  n - 2\delta < n+1 = \gamma.
\end{align}
The inequality in \eqref{11.30} with $\ell=n-j-1$ implies $1-2\delta\leq 0$, so that
\begin{align}
k_2 + k_3 + 2(1-\delta) \leq n - 1 + 2(1-\delta) = n + 1 - 2\delta \leq n,
\end{align}
which yields
\begin{equation}
\min(n, \alpha+\beta) = \min(n,k_2 + k_3 + 2(1-\delta)) = k_2 + k_3 + 2(1-\delta),
\end{equation}
and the estimate in \eqref{11.36} follows.  If $j=n-3$, then $x_2$ is the only inner variable, and the integration over the inner variables is complete with \eqref{11.36}.  For $j\leq n-4$ the process continues and there are $n-j-3$ remaining inner integrals to estimate.  Applying \eqref{11.36} in \eqref{11.34}, the next inner integral is with respect to $x_3$:
\begin{align}
&\int_{\bbR^n}d^nx_2\, \int_{\bbR^n}d^nx_3\, |x_1-x_2|^{k_2+1-\delta-n}\langle x_2\rangle^{-n-1-\varepsilon}|x_2-x_3|^{k_3+1-\delta-n}\\
&\hspace*{2.9cm} \times \langle x_3\rangle^{-n-1-\varepsilon}|x_3-x_4|^{k_4+1-\delta-n}\no\\
&\quad \leq c_{2,n,\delta} \int_{\bbR^n}d^nx_3\, \Big[|x_1 - x_3|^{k_2+k_3+2(1-\delta)-n} + 1\Big]\langle x_3\rangle^{-n-1-\varepsilon}|x_3-x_4|^{k_4+1-\delta-n}\no\\
&\quad =c_{2,n,\delta}\bigg[\int_{\bbR^n}d^nx_3\, |x_1 - x_3|^{k_2+k_3+2(1-\delta)-n}\langle x_3\rangle^{-n-1-\varepsilon}|x_3-x_4|^{k_4+1-\delta-n} \no\\
&\hspace*{1.8cm}+ \int_{\bbR^n}d^nx_3\, \langle x_3\rangle^{-n-1-\varepsilon}|x_3-x_4|^{k_4+1-\delta-n}\bigg]\no\\
&\quad=: c_{2,n,\delta}[\cI_1(x_1,x_4) + \cI_2(x_4)].\lb{11.42}
\end{align}
An application of \eqref{11.16} with the choices
\begin{equation}
\alpha = n,\quad \beta = k_4+1-\delta,\quad \gamma=n+1,
\end{equation}
immediately yields (note that in this case $\min(n,\alpha+\beta)=n$)
\begin{equation} \lb{11.44}
\cI_2(x_4) \leq c_{3,n,\delta}'',\quad x_4\in \bbR^n,
\end{equation}
for some $c_{3,n,\delta}''\in (0,\infty)$.  Another application of \eqref{11.16}, this time with the choices
\begin{equation} \lb{11.45}
\alpha=k_2+k_3+2(1-\delta),\quad \beta=k_4+1-\delta,\quad \gamma=n+1,
\end{equation}
implies
\begin{align}
\cI_1(x_1,x_4) \leq c_{3,n,\delta}'\Big[|x_1-x_4|^{k_2+k_3+k_4+3(1-\delta)-n}+1 \Big]\lb{11.46}
\end{align}
for some $c_{3,n,\delta}'\in (0,\infty)$.  The conditions on $\alpha$ and $\beta$ in \eqref{11.16} are satisfied by the choices in \eqref{11.45} since (cf.~\eqref{11.38})
\begin{equation}
0 < k_2 + k_3 + 2(1-\delta) \leq n-1 + 2 - 2\delta = n+1-2\delta \leq n
\end{equation}
and
\begin{equation}
0 < k_4 + 1-\delta \leq n-\delta \leq n
\end{equation}
together with
\begin{align}
(\alpha + \beta)-n &= \underbrace{k_2+k_3}_{\leq n-1} + \underbrace{k_4}_{\leq n-1} + 3(1-\delta) - n \\
&\leq 2(n-1) + 3(1-\delta) - n \no\\
&= n+1-3\delta \no\\
&< n+1 = \gamma.\no
\end{align}
The inequality in \eqref{11.30} with $\ell=n-j-2$ implies $2-3\delta \leq 0$, so that
\begin{align}
k_2+k_3+k_4+3(1-\delta) \leq n-1 +3(1-\delta) = n+2-3\delta \leq n,
\end{align}
which yields
\begin{equation}
\min(n,\alpha+\beta) = \min(n,k_2+k_3+k_4+3(1-\delta))=k_2+k_3+k_4+3(1-\delta),
\end{equation}
and the estimate in \eqref{11.46} follows.  Finally, combining \eqref{11.42}, \eqref{11.44}, and \eqref{11.46}, one obtains
\begin{align}
&\int_{\bbR^n}d^nx_2\, \int_{\bbR^n}d^nx_3\, |x_1-x_2|^{k_2+1-\delta-n}\langle x_2\rangle^{-n-1-\varepsilon}|x_2-x_3|^{k_3+1-\delta-n}\no\\
&\hspace*{2.9cm} \times \langle x_3\rangle^{-n-1-\varepsilon}|x_3-x_4|^{k_4+1-\delta-n} \lb{11.53}\\
&\quad \leq c_{3,n,\delta}\Big[|x_1-x_4|^{k_2+k_3+k_4+3(1-\delta)-n}+1 \Big]\no
\end{align}
for some $c_{3,n,\delta}\in (0,\infty)$.  Continuing systematically in this way, one obtains
\begin{align}
&\int_{\bbR^n} d^n x_2 \int_{\bbR^n}d^nx_3 \cdots \int_{\bbR^n} d^n x_{n-j-1}\lb{11.54} \\
&\qquad\quad \;\;\,
\times |x_1-x_2|^{k_2+1-\delta-n} \langle x_2\rangle^{-n-1-\varepsilon}|x_2-x_3|^{k_3+1-\delta-n} \no\\
&\qquad\quad \;\;\, \times \langle x_3\rangle^{-n-1-\varepsilon}|x_3-x_4|^{k_4+1-\delta-n}\no\\
&\qquad\;\;\,\hspace*{.5cm} \vdots\no\\
&\qquad\quad \;\;\,
\times \langle x_{n-j-1} \rangle^{-n-1-\varepsilon}  |x_{n-j-1}-x_{n-j}|^{k_{n-j}+1-\delta-n} \no\\
&\quad \leq c_{n-j-2,n,\delta} \int_{\bbR^n}d^nx_{n-j-1}\, \Big[|x_1-x_{n-j-1}|^{k_2+\cdots+ k_{n-j-1}+1(n-j-2)(1-\delta)-n}+1 \Big]\no\\
&\hspace*{3.4cm}\times \langle x_{n-j-1}\rangle^{-n-1-\varepsilon}|x_{n-j-1}-x_{n-j}|^{k_{n-j}+1-\delta-n}\no\\
&\quad = c_{n-j-2,n,\delta} \bigg[\int_{\bbR^n}d^nx_{n-j-1}\,|x_1-x_{n-j-1}|^{k_2+\cdots+ k_{n-j-1}+1(n-j-2)(1-\delta)-n}\no\\
&\hspace*{3.6cm}\times\langle x_{n-j-1}\rangle^{-n-1-\varepsilon}|x_{n-j-1}-x_{n-j}|^{k_{n-j}+1-\delta-n} \no\\
&\hspace*{2.5cm}+ \int_{\bbR^n}d^nx_{n-j-1}\, \langle x_{n-j-1}\rangle^{-n-1-\varepsilon}|x_{n-j-1}-x_{n-j}|^{k_{n-j}+1-\delta-n} \bigg]\no\\
&\quad =: c_{n-j-2,n,\delta}[\cI_1(x_1,x_{n-j})+\cI_2(x_{n-j})]\no
\end{align}
for some $c_{n-j-2,n,\delta}\in (0,\infty)$.  An application of \eqref{11.16} with the choices
\begin{equation}
\alpha=n,\quad \beta=k_{n-j}+1-\delta,\quad \gamma=n+1
\end{equation}
immediately yields (note that in this case $\min(n,\alpha+\beta)=n$)
\begin{equation} \lb{11.56}
\cI_2(x_{n-j}) \leq c_{n-j-1,n,\delta}'',\quad x_{n-j}\in \bbR^n,
\end{equation}
for some $c_{n-j-1,n,\delta}''\in (0,\infty)$.  Another application of \eqref{11.16}, this time with the choices
\begin{equation} \lb{11.57}
\alpha=k_2+\cdots+k_{n-j-1}+(n-j-2)(1-\delta),\quad \beta=k_{n-j}+1-\delta,\quad \gamma=n+1,
\end{equation}
implies
\begin{equation} \lb{11.58}
\cI_1(x_1,x_{n-j}) \leq c_{n-j-1,n,\delta}'\Big[|x_1-x_{n-j}|^{k_2+\cdots+k_{n-j}+(n-j-1)(1-\delta)-n} + 1 \Big]
\end{equation}
for some $c_{n-j-1,n,\delta}'\in(0,\infty)$.  The conditions on $\alpha$ and $\beta$ in \eqref{11.16} are satisfied by the choices in \eqref{11.57} since
\begin{align}
0&<k_2+\cdots+k_{n-j-1}+(n-j-2)(1-\delta)\no\\
&\leq n-1 + (n-j-2)(1-\delta)\no\\
&= n + (n-j-3)-(n-j-2)\delta\no\\
&\leq n, \lb{11.59}
\end{align}
and
\begin{equation}
0< k_{n-j}+1-\delta \leq n-\delta \leq n.
\end{equation}
The final inequality in \eqref{11.59} follows by choosing $\ell=3$ in \eqref{11.30}.  In addition,
\begin{align}
(\alpha+\beta)-n &= k_2+\cdots+k_{n-j}+(n-j-1)(1-\delta)-n\\
&= \underbrace{k_2+\cdots+k_{n-j}}_{\leq n-1} - (j+1)-(n-j-1)(1-\delta)\no\\
&\leq n-2-j-(n-j-2)(1-\delta)\no\\
&<n+1\no\\
&= \gamma.\no
\end{align}
The inequality in \eqref{11.31} with $\ell=2$ implies $(n-j-2)-(n-j-1)\delta \leq 0$, so that
\begin{align}
k_2+\cdots+k_{n-j} + (n-j-1)\delta&=  k_2+\cdots+k_{n-j} + (n-j-1)+(n-j-1)\delta\no\\
&\leq n + (n-j-2)-(n-j-1)\delta\no\\
&\leq n,
\end{align}
which yields
\begin{equation}
\min(n,\alpha+\beta) = k_2+\cdots+k_{n-j} + (n-j-1)\delta,
\end{equation}
and the estimate in \eqref{11.58} follows.  Finally, combining \eqref{11.54}, \eqref{11.56}, and \eqref{11.58}, one obtains
\begin{align}
&\int_{\bbR^n} d^n x_2 \int_{\bbR^n}d^nx_3 \cdots \int_{\bbR^n} d^n x_{n-j-1}\no\\
&\qquad\quad \;\;\,
\times |x_1-x_2|^{k_2+1-\delta-n} \langle x_2\rangle^{-n-1-\varepsilon}|x_2-x_3|^{k_3+1-\delta-n} \no\\
&\qquad\quad \;\;\,
\times \langle x_3\rangle^{-n-1-\varepsilon}|x_3-x_4|^{k_4+1-\delta-n}\no\\
&\hspace*{1.47cm} \vdots\no\\
&\qquad\quad \;\;\,
\times \langle x_{n-j-1} \rangle^{-n-1-\varepsilon}  |x_{n-j-1}-x_{n-j}|^{k_{n-j}+1-\delta-n} \no\\
&\quad \;\;\,
\leq c_{n-j-1,n,\delta} \Big[|x_1-x_{n-j}|^{k_2+\cdots+k_{n-j}+(n-j-1)(1-\delta)-n} +1 \Big]\lb{11.63}
\end{align}
for some $c_{n-j-1,n,\delta}\in(0,\infty)$.

The estimate in \eqref{11.63} implies
\begin{align}
\eqref{11.34}&\leq c_{n-j-1,n,\delta}\int_{\bbR^n}d^nx_1\int_{\bbR^n}d^nx_{n-j}\no\\
&\hspace*{1cm}\times \langle x_1 \rangle^{-n-1-\varepsilon}\Big[|x_1-x_{n-j}|^{k_2+\cdots+k_{n-j}+(n-j-1)(1-\delta)-n} +1 \Big]\no\\
&\hspace*{1cm}\times\langle x_{n-j} \rangle^{-n-1-\varepsilon}  |x_{n-j}-x_1|^{k_1+k_{n-j+1}+2-\delta-n}.\lb{11.64}
\end{align}
Focusing on the integral over $x_{n-j}$ in \eqref{11.64},
\begin{align}
&\int_{\bbR^n}d^nx_{n-j}\, \Big[|x_1-x_{n-j}|^{k_2+\cdots+k_{n-j}+(n-j-1)(1-\delta)-n} +1 \Big]\no\\
&\hspace*{1.95cm}\times\langle x_{n-j} \rangle^{-n-1-\varepsilon}  |x_{n-j}-x_1|^{k_1+k_{n-j+1}+2-\delta-n}\no\\
&\quad= \int_{\bbR^n}d^nx_{n-j}\, |x_1-x_{n-j}|^{(n-j)(1-\delta)-n}\langle x_{n-j} \rangle^{-n-1-\varepsilon} \no\\
&\qquad+\int_{\bbR^n}d^nx_{n-j}\,\langle x_{n-j} \rangle^{-n-1-\varepsilon}  |x_{n-j}-x_1|^{k_1+k_{n-j+1}+2-\delta-n}\no\\
&\quad= \cI_1(x_1)+\cI_2(x_1),\lb{11.65}
\end{align}
an application of \eqref{11.16} with the choices
\begin{equation} \lb{11.66}
\alpha=(n-j)(1-\delta),\quad \beta=n,\quad \gamma=n+1,
\end{equation}
yields
\begin{equation} \lb{11.67}
\cI_1(x_1) \leq c_{n-j,n,\delta}',\quad x_1\in \bbR^n,
\end{equation}
for some $c_{n-j,n,\delta}'\in(0,\infty)$.  The conditions on $\alpha$ and $\beta$ in \eqref{11.16} are satisfied by the choices in \eqref{11.66} since $\alpha=(n-j)(1-\delta),\beta=n\in (0,n]$ and
\begin{equation}
(\alpha+\beta)-n=\alpha = (n-j)(1-\delta) \leq n < n+1 = \gamma.
\end{equation}
Since $\min(n,\alpha+\beta)=n$, the estimate in \eqref{11.16} results in \eqref{11.67}.  To estimate 
$\cI_2(\dott)$, one applies \eqref{11.16} with the choices
\begin{equation} \lb{11.69}
\alpha=n,\quad \beta=k_1+k_{n-j+1}+2-\delta,\quad \gamma=n+1,
\end{equation}
to obtain
\begin{equation} \lb{11.70}
\cI_2(x_1) \leq c_{n-j,n,\delta}'',\quad x_1\in \bbR^n,
\end{equation}
for some $c_{n-j,n,\delta}''\in(0,\infty)$. The conditions on $\alpha$ and $\beta$ in \eqref{11.16} are satisfied by the choices in \eqref{11.69} since $\alpha = n\in (0,n]$,
\begin{equation}
0<\beta = \underbrace{k_1+k_{n-j+1}}_{\leq n-2} + 2-\delta \leq n-\delta \leq n,
\end{equation}
and $(\alpha+\beta)-n = \beta \leq n<n+1$.  In this case, $\min(n,\alpha+\beta)= n$, and \eqref{11.16} results in \eqref{11.70}.  Combining \eqref{11.65}, \eqref{11.67}, and \eqref{11.70}, one obtains
\begin{equation}
\eqref{11.65} \leq c_{n-j,n,\delta},\quad x_1\in \bbR^n,
\end{equation}
for some $c_{n-j,n,\delta}\in(0,\infty)$.  As a consequence,
\begin{equation}
\eqref{11.64} \leq c_{n-j-1,n,\delta}c_{n-j,n,\delta}\int_{\bbR^n}d^nx_1\, \langle x_1 \rangle^{-n-1-\varepsilon} <\infty.
\end{equation}
Hence, this establishes the claim that for each bounded subset $\Omega\subset \bbC_+$, the integrand under the iterated integral on the right-hand side in \eqref{11.31} is uniformly bounded with respect to $z\in \Omega$ by an integrable function of the variables $x_1,...,x_{n-j}$.  As a consequence of this claim, \eqref{11.31} implies that for each bounded subset $\Omega\subset \bbC_+$, the following estimate holds:
\begin{equation}
\eqref{11.31}\leq C_{n,j,\delta,\Omega}|z|^{-(n-j)\delta},\quad z\in \Omega,
\end{equation}
for some $C_{n,j,\delta,\Omega}\in (0,\infty)$.
In summary, for {\it Case~1}, one has that for any bounded subset $\Omega\subset \bbC_+$,
\begin{align}
& \bigg\|\int_{\bbR^n} d^n x_1 \cdots \int_{\bbR^n} d^n x_{n-j-1} \int_{\bbR^n} d^n x_{n-j} \, V(x_1) \f{1}{{k_2}!} 
\bigg[\f{\partial^{k_2}}{\partial z^{k_2}}G_0(z;x_1,x_2)\bigg] 
\no \\
& \qquad \, \cdots \times V(x_{n-j-1}) \f{1}{k_{n-j}!}
\bigg[\f{\partial^{k_{n-j}}}{\partial z^{k_{n-j}}} 
G_0(z;x_{n-j-1},x_{n-j})\bigg]   \no \\
&  \hspace*{1.25cm} \times V(x_{n-j})  \f{1}{(k_1 + k_{n-j+1} + 1)!}
\bigg[\f{\partial^{k_1 + k_{n-j+1} + 1}}{\partial z^{k_1 + k_{n-j+1} + 1}}
G_0(z;x_{n-j},x_1)\bigg] \bigg\|_{\cB(\bbC^N)}  \no \\
&\quad \leq C_{n,j,\delta,\Omega}|z|^{-(n-j)\delta},\quad z\in \Omega,\lb{11.75}
\end{align}
where $\delta$ is defined by \eqref{11.29}.  In addition, since $|z|^{-(n-j)\delta}$ is bounded in $\overline{\Omega}$ if $0\notin \overline{\Omega}$, Lebesgue's dominated convergence theorem implies that
\begin{align}
&\int_{\bbR^n} d^n x_1 \cdots \int_{\bbR^n} d^n x_{n-j-1} \int_{\bbR^n} d^n x_{n-j} \, V(x_1) \f{1}{{k_2}!} 
\bigg[\f{\partial^{k_2}}{\partial z^{k_2}}G_0(z;x_1,x_2)\bigg] \\
& \qquad \, \cdots \times V(x_{n-j-1}) \f{1}{k_{n-j}!}
\bigg[\f{\partial^{k_{n-j}}}{\partial z^{k_{n-j}}} 
G_0(z;x_{n-j-1},x_{n-j})\bigg]   \no \\
&  \hspace*{1.25cm} \times V(x_{n-j})  \f{1}{(k_1 + k_{n-j+1} + 1)!}
\bigg[\f{\partial^{k_1 + k_{n-j+1} + 1}}{\partial z^{k_1 + k_{n-j+1} + 1}}
G_0(z;x_{n-j},x_1)\bigg],\quad z\in \Omega,\no
\end{align}
is analytic in $\Omega$ and extends continuously to $\overline{\Omega}$ if $0\notin \overline\Omega$.  This settles {\it Case~1}.

Next, we treat {\it Case~2}.  The assumptions in {\it Case~2} imply
\begin{equation} \lb{11.77}
\text{$k_\ell =0$ for all $2\leq \ell\leq n-j$ and $k_1+k_{n-j+1}+1=n$}.
\end{equation}
Let $\delta\in (0,1)$ be fixed.  Applying the final estimate in \eqref{C.34}, one obtains:
\begin{align}
& \bigg\|\int_{\bbR^n} d^n x_1 \cdots \int_{\bbR^n} d^n x_{n-j-1} \int_{\bbR^n} d^n x_{n-j} \, V(x_1) \f{1}{{k_2}!} 
\bigg[\f{\partial^{k_2}}{\partial z^{k_2}}G_0(z;x_1,x_2)\bigg] 
\no \\
& \qquad \, \cdots \times V(x_{n-j-1}) \f{1}{k_{n-j}!}
\bigg[\f{\partial^{k_{n-j}}}{\partial z^{k_{n-j}}} 
G_0(z;x_{n-j-1},x_{n-j})\bigg]   \no \\
&  \hspace*{1.25cm} \times V(x_{n-j})  \f{1}{(k_1 + k_{n-j+1} + 1)!}
\bigg[\f{\partial^{k_1 + k_{n-j+1} + 1}}{\partial z^{k_1 + k_{n-j+1} + 1}}
G_0(z;x_{n-j},x_1)\bigg] \bigg\|_{\cB(\bbC^N)}  \no \\
&\quad = \bigg\|\int_{\bbR^n} d^n x_1 \cdots \int_{\bbR^n} d^n x_{n-j-1} \int_{\bbR^n} d^n x_{n-j} \, V(x_1) G_0(z;x_1,x_2)
\no \\
&\hspace*{2.4cm} \cdots \times V(x_{n-j-1}) G_0(z;x_{n-j-1},x_{n-j})   \no \\
&\hspace*{2.86cm} \times V(x_{n-j})  \f{1}{n!}
\bigg[\f{\partial^{n}}{\partial z^{n}}
G_0(z;x_{n-j},x_1)\bigg] \bigg\|_{\cB(\bbC^N)}  \no \\
&\quad \leq C_{n,j,\delta}|z|^{-(n-j-1)\delta-1} \int_{\bbR^n} d^n x_1 \cdots \int_{\bbR^n} d^n x_{n-j-1} \int_{\bbR^n} d^n x_{n-j} \no\\
&\quad \quad\quad \times \langle x_1\rangle^{-n-1-\varepsilon}\Big[|x_1-x_2|^{1-\delta-n}+|z|^{(n-1+2\delta)/2}|x_1-x_2|^{(1-n)/2}\Big]\no\\
&\hspace*{1.25cm}\vdots\no\\
&\quad \quad\quad \times \langle x_{n-j-1}\rangle^{-n-1-\varepsilon}\Big[|x_{n-j-1}-x_{n-j}|^{1-\delta-n}\no\\
&\quad\quad\quad\quad+|z|^{(n-1+2\delta)/2}|x_{n-j-1}-x_{n-j}|^{(1-n)/2}\Big]\no\\
&\quad\quad\quad\times \langle x_{n-j}\rangle^{-n-1-\varepsilon}\Big[1+|z|^{(n+1)/2}|x_{n-j}-x_1|^{(n+1)/2}\Big]\no\\
&\quad \leq  \wti C_{n,j,\delta}|z|^{-(n-j-1)\delta-1} \int_{\bbR^n} d^n x_1 \cdots \int_{\bbR^n} d^n x_{n-j-1} \int_{\bbR^n} d^n x_{n-j} \lb{11.78}\\
&\quad \quad \quad \times \langle x_1\rangle^{-n-1-\varepsilon}\Big[|x_1-x_2|^{1-\delta-n}\no\\
&\quad\quad\quad\quad+|z|^{(n-1+2\delta)/2}[1+|x_1|]^{(1-n)/2}[1+|x_2|]^{(1-n)/2}\Big]\no\\
&\hspace*{1.25cm}\vdots\no\\
&\quad \quad\quad \times \langle x_{n-j-1}\rangle^{-n-1-\varepsilon}\Big[|x_{n-j-1}-x_{n-j}|^{1-\delta-n}\no\\
&\quad\quad\quad\quad+|z|^{(n-1+2\delta)/2}[1+|x_{n-j-1}|]^{(1-n)/2}[1+|x_{n-j}|]^{(1-n)/2}\Big]\no\\
&\quad\quad\quad\times \langle x_{n-j}\rangle^{-n-1-\varepsilon}\Big[1+|z|^{(n+1)/2}[1+|x_{n-j}|]^{(n+1)/2}[1+|x_1|]^{(n+1)/2}\Big],\no\\
&\hspace*{10.08cm} z\in \bbC_+,\no
\end{align}
where $C_{n,j,\delta}, \wti C_{n,j,\delta}\in (0,\infty)$ are suitable constants.  We claim that for each bounded subset $\Omega\subset \bbC_+$, the integrand under the iterated integral on the right-hand side in \eqref{11.78} is uniformly bounded with respect to $z\in \Omega$ by an integrable function of the variables $x_1,...,x_{n-j}$.  Since $|z|^{(n-1+2\delta)/2}$ and $|z|^{(n+1)/2}$ are locally bounded, to justify the claim, it suffices to establish convergence of the following integral:
\begin{align}
&\int_{\bbR^n} d^n x_1 \cdots \int_{\bbR^n} d^n x_{n-j-1} \int_{\bbR^n} d^n x_{n-j} \lb{11.79}\\
&\quad \quad \quad \times \langle x_1\rangle^{-n-1-\varepsilon}\Big[|x_1-x_2|^{1-\delta-n}+[1+|x_1|]^{(1-n)/2}[1+|x_2|]^{(1-n)/2}\Big]\no\\
&\hspace*{1.25cm}\vdots\no\\
&\quad \quad\quad \times \langle x_{n-j-1}\rangle^{-n-1-\varepsilon}\Big[|x_{n-j-1}-x_{n-j}|^{1-\delta-n}\no\\
&\quad\quad\quad\quad+[1+|x_{n-j-1}|]^{(1-n)/2}[1+|x_{n-j}|]^{(1-n)/2}\Big]\no\\
&\quad\quad\quad\times \langle x_{n-j}\rangle^{-n-1-\varepsilon}\Big[1+[1+|x_{n-j}|]^{(n+1)/2}[1+|x_1|]^{(n+1)/2}\Big].\no
\end{align}
As with {\it Case~1}, it suffices to focus on the most singular term in \eqref{11.79} and thus disregard the terms originally multiplied by $|z|^{(n-1+2\delta)/2}$ or $|z|^{(n+1)/2}$ in \eqref{11.78} (following the same line of reasoning used throughout \eqref{11.12a}--\eqref{11.14}).  With this simplification, the claim reduces to establishing convergence of the integral
\begin{align}
&\int_{\bbR^n} d^n x_1 \cdots \int_{\bbR^n} d^n x_{n-j-1} \int_{\bbR^n} d^n x_{n-j} \lb{11.80}\\
&\quad \quad  \times \langle x_1\rangle^{-n-1-\varepsilon}|x_1-x_2|^{1-\delta-n}\times \cdots \times \langle x_{n-j-1}\rangle^{-n-1-\varepsilon}|x_{n-j-1}-x_{n-j}|^{1-\delta-n} \no\\
&\quad\quad\times \langle x_{n-j}\rangle^{-n-1-\varepsilon}.\no
\end{align}
In analogy to {\it Case~1}, one successively estimates the integrals over the inner variables $x_2,\ldots,x_{n-j-1}$ in \eqref{11.80} as follows.  Beginning with the integral with respect to $x_2$, an application of \eqref{11.16} with the choices
\begin{equation} \lb{11.81}
\alpha=\beta=1-\delta,\quad \gamma=n+1,
\end{equation}
yields
\begin{align}
&\int_{\bbR^n}d^nx_2\, |x_1-x_2|^{1-\delta-n}\langle x_2\rangle^{-n-1-\varepsilon}|x_2-x_3|^{1-\delta-n}\no\\
&\quad \leq c_{2,n}\Big[|x_1-x_3|^{2(1-\delta)-n}+1 \Big].\lb{11.82}
\end{align}
The assumptions on $\alpha$ and $\beta$ in \eqref{11.16} are satisfied by the choices in \eqref{11.81}.  In fact,
\begin{equation}
\alpha=\beta=1-\delta\in (0,n]\quad \text{and}\quad (\alpha+\beta)-n = 2(1-\delta)-n<n+1 = \gamma.
\end{equation}
Finally, $\min(n,2(1-\delta))=2(1-\delta)$ and \eqref{11.16} results in \eqref{11.82}.  If $j=n-3$, then $x_2$ is the only inner variable, and the integration over the inner variables is complete with \eqref{11.82}.  For $j\leq n-4$ the process continues and there are $n-j-3$ remaining inner integrals to estimate.  Applying \eqref{11.82} in \eqref{11.80}, the next inner integral is with respect to $x_3$:
\begin{align}
&\int_{\bbR^n}d^nx_2\int_{\bbR^n}d^nx_3\, |x_1-x_2|^{1-\delta-n}\langle x_2\rangle^{-n-1-\varepsilon}|x_2-x_3|^{1-\delta-n}\no\\
&\qquad \times \langle x_3\rangle^{-n-1-\varepsilon}|x_3-x_4|^{1-\delta-n}\no\\
&\quad \leq c_{2,n,\delta}\int_{\bbR^n}d^nx_3\,\Big[|x_1-x_3|^{2(1-\delta)-n}+1 \Big]\langle x_3\rangle^{-n-1-\varepsilon}|x_3-x_4|^{1-\delta-n}\no\\
&\quad= c_{2,n,\delta}\bigg[\int_{\bbR^n}d^nx_3\,|x_1-x_3|^{2(1-\delta)-n}\langle x_3\rangle^{-n-1-\varepsilon}|x_3-x_4|^{1-\delta-n}\no\\
&\hspace*{2cm} + \int_{\bbR^n}d^nx_3\,\langle x_3\rangle^{-n-1-\varepsilon}|x_3-x_4|^{1-\delta-n}\bigg]\no\\
&\quad =: c_{2,n,\delta}[\cI_1(x_1,x_4)+\cI_2(x_4)]\lb{11.84}
\end{align}
for some $c_{2,n,\delta}\in (0,\infty)$.  An application of \eqref{11.16} with the choices
\begin{equation} \lb{11.85}
\alpha=2(1-\delta),\quad \beta=1-\delta,\quad \gamma=n+1,
\end{equation}
yields
\begin{equation} \lb{11.86}
\cI_1(x_1,x_4) \leq c_{3,n,\delta}' \Big[|x_1-x_4|^{3(1-\delta)-n}+1\Big]
\end{equation}
for some $c_{3,n,\delta}'\in (0,\infty)$.  With the choices in \eqref{11.85}, it is clear that $\alpha,\beta\in (0,n]$ and $(\alpha+\beta)-n = 3(1-\delta)-n < n+1$ since
\begin{align}
n+1 - [3(1-\delta)-n] = 2(n-1)+ 3\delta >0,
\end{align}
so the assumptions on $\alpha$ and $\beta$ in \eqref{11.16} are satisfied.  Finally, $\min(n,3(1-\delta))=3(1-\delta)$, and \eqref{11.16} results in \eqref{11.86}.  A second application of \eqref{11.16}, this time with the choices 
\begin{equation}
\alpha=n,\quad \beta=1-\delta,\quad \gamma=n+1
\end{equation}
yields
\begin{equation} \lb{11.89}
\cI_2(x_4) \leq c_{3,n,\delta}'',\quad x_4 \in \bbR^n,
\end{equation}
for some $c_{3,n,\delta}''\in (0,\infty)$.  As a result, \eqref{11.84}, \eqref{11.86}, and \eqref{11.89} imply 
\begin{align}
&\int_{\bbR^n}d^nx_2\int_{\bbR^n}d^nx_3\, |x_1-x_2|^{1-\delta-n}\langle x_2\rangle^{-n-1-\varepsilon}|x_2-x_3|^{1-\delta-n}\no\\
&\qquad \times \langle x_3\rangle^{-n-1-\varepsilon}|x_3-x_4|^{1-\delta-n}\no\\
&\quad \leq c_{3,n,\delta}\Big[|x_1-x_4|^{3(1-\delta)-n}+1\Big]\lb{11.90}
\end{align}
for some $c_{3,n,\delta}\in (0,\infty)$.  Continuing systematically in this way, one obtains
\begin{align}
&\int_{\bbR^n} d^n x_2 \cdots \int_{\bbR^n} d^n x_{n-j-1} |x_1-x_2|^{1-\delta-n}\langle x_2\rangle^{-n-1-\varepsilon}\no\\
&\hspace*{1cm} 
\times \cdots \times |x_{n-j-2}-x_{n-j-1}|^{1-\delta-n}\langle x_{n-j-1}\rangle^{-n-1-\varepsilon}|x_{n-j-1}-x_{n-j}|^{1-\delta-n} \no\\
&\quad \leq c_{n-j-2,n,\delta}\int_{\bbR^n}d^nx_{n-j-1}\, \Big[ |x_1-x_{n-j-1}|^{(n-j-2)(1-\delta)-n}+1\Big]\no\\
&\hspace*{4.63cm}\times \langle x_{n-j-1}\rangle^{-n-1-\varepsilon}|x_{n-j-1}-x_{n-j}|^{1-\delta-n}\no\\
&\quad = c_{n-j-2,n,\delta}\bigg[\int_{\bbR^n}d^nx_{n-j-1}\, |x_1-x_{n-j-1}|^{(n-j-2)(1-\delta)-n}\no\\
&\hspace*{2.5cm} 
+\int_{\bbR^n}d^nx_{n-j-1}\, \langle x_{n-j-1}\rangle^{-n-1-\varepsilon}|x_{n-j-1}-x_{n-j}|^{1-\delta-n}\bigg]\no\\
&\quad =: c_{n-j-2,n,\delta}[\cI_1(x_1,x_{n-j}) +\cI_2(x_{n-j})]\lb{11.91}
\end{align}
for some $c_{n-j-2,n,\delta}\in (0,\infty)$.  Applying \eqref{11.16} with the choices
\begin{equation} \lb{11.92}
\alpha=(n-j-2)(1-\delta),\quad \beta=1-\delta,\quad \gamma=n+1,
\end{equation}
yields
\begin{equation} \lb{11.93}
\cI_1(x_1,x_{n-j}) \leq c_{n-j-1,n,\delta}'\Big[|x_1-x_{n-j}|^{(n-j-1)(1-\delta)-n}+1\Big].
\end{equation}
The assumptions on $\alpha$ and $\beta$ in \eqref{11.16} are satisfied by the choices in \eqref{11.92}.  In fact,
\begin{equation}
0<\alpha=(n-j-2)(1-\delta)\leq n-2\leq n\quad\text{and}\quad 0<\beta=1-\delta\leq n,
\end{equation}
while
\begin{equation}
(\alpha+\beta)-n = (n-j-1)(1-\delta)-n \leq 0 < n+1 = \gamma.
\end{equation}
Finally, $\min(n,(n-j-1)(1-\delta))=(n-j-1)(1-\delta)$ and \eqref{11.16} results in \eqref{11.93}.  A second application of \eqref{11.16}, this time with the choices
\begin{equation} \lb{11.96}
\alpha=n,\quad \beta=1-\delta,\quad \gamma=n+1,
\end{equation}
yields
\begin{equation} \lb{11.97}
\cI_2(x_{n-j})\leq c_{n-j-1,n,\delta}'',\quad x_{n-j} \in \bbR^n,
\end{equation}
for some $c_{n-j-1,n,\delta}''\in (0,\infty)$.  Combining \eqref{11.91}, \eqref{11.93}, and \eqref{11.97}, one obtains
\begin{align}
&\int_{\bbR^n} d^n x_2 \cdots \int_{\bbR^n} d^n x_{n-j-1} |x_1-x_2|^{1-\delta-n}\langle x_2\rangle^{-n-1-\varepsilon}\no\\
&\hspace*{1.04cm} 
\times \cdots \times |x_{n-j-2}-x_{n-j-1}|^{1-\delta-n}\langle x_{n-j-1}\rangle^{-n-1-\varepsilon}|x_{n-j-1}-x_{n-j}|^{1-\delta-n} \no\\
&\quad \leq c_{n-j-1,n,\delta}\Big[|x_1-x_{n-j}|^{(n-j-1)(1-\delta)-n}+1\Big]\lb{11.98}
\end{align}
for some $c_{n-j-1,n,\delta}\in (0,\infty)$.

The estimate in \eqref{11.98} implies
\begin{align}
\eqref{11.80}&\leq c_{n-j-1,n,\delta}\int_{\bbR^n}d^nx_1\int_{\bbR^n}d^nx_{n-j}\lb{11.99}\\
&\hspace*{1cm}\times\langle x_1\rangle^{-n-1-\varepsilon}\Big[|x_1-x_{n-j}|^{(n-j-1)(1-\delta)-n}+1\Big]\langle x_{n-j}\rangle^{-n-1-\varepsilon}\no
\end{align}
Focusing on the integral over $x_{n-j}$ in \eqref{11.99},
\begin{align}
&\int_{\bbR^n}d^nx_{n-j}\, \Big[|x_1-x_{n-j}|^{(n-j-1)(1-\delta)-n}+1\Big]\langle x_{n-j}\rangle^{-n-1-\varepsilon}\no\\
&\quad= \int_{\bbR^n}d^nx_{n-j}\, |x_1-x_{n-j}|^{(n-j-1)(1-\delta)-n}\langle x_{n-j}\rangle^{-n-1-\varepsilon}\no\\
&\qquad + \int_{\bbR^n}d^nx_{n-j}\, \langle x_{n-j}\rangle^{-n-1-\varepsilon}\no\\
&\quad=: \cI_1(x_1)+\cI_2,\lb{11.100}
\end{align}
one infers that
\begin{equation} \lb{11.101}
\cI_2=c_n<\infty.
\end{equation}
An application of \eqref{11.16} with the choices
\begin{equation}
\alpha=(n-j-1)(1-\delta),\quad \beta=n,\quad \gamma=n+1
\end{equation}
yields
\begin{equation} \lb{11.103}
\cI_1(x_1) \leq c_{n-j,n,\delta}',\quad x_1\in \bbR^n.
\end{equation}
Thus, \eqref{11.100}, \eqref{11.101}, and \eqref{11.103} imply
\begin{equation} \lb{11.104}
\eqref{11.100}\leq c_{n-j,n,\delta},\quad x_1\in \bbR^n,
\end{equation}
for some $c_{n-j,n,\delta}\in (0,\infty)$.  As a result, \eqref{11.99} and \eqref{11.104} imply
\begin{equation}
\eqref{11.80} \leq c_{n-j-1,n,\delta}c_{n-j,n,\delta}\int_{\bbR^n}d^nx_1\langle x_1\rangle^{-n-1-\varepsilon}<\infty.
\end{equation}
Hence, this establishes the claim that for each bounded subset $\Omega\subset \bbC_+$, the integrand under the iterated integral on the right-hand side in \eqref{11.78} is uniformly bounded with respect to $z\in \Omega$ by an integrable function of the variables $x_1,...,x_{n-j}$.  As a consequence of this claim, \eqref{11.78} implies that for each $\delta\in (0,1)$ and each bounded subset $\Omega\subset \bbC_+$, the following estimate holds:
\begin{equation} 
\eqref{11.78}\leq C_{n,j,\delta,\Omega} |z|^{-(n-j-1)\delta-1},\quad z\in \Omega,   \lb{11.106}
\end{equation}
for some $C_{n,j,\delta,\Omega}\in (0,\infty)$.  In summary, for {\it Case~2}, one has that for any $\delta\in (0,1)$ and any bounded subset $\Omega\subset \bbC_+$,
\begin{align}
&\bigg\|\int_{\bbR^n} d^n x_1 \cdots \int_{\bbR^n} d^n x_{n-j-1} \int_{\bbR^n} d^n x_{n-j} \, V(x_1) G_0(z;x_1,x_2)     \no \\
&\hspace*{1.55cm} \cdots \times V(x_{n-j-1}) G_0(z;x_{n-j-1},x_{n-j})   \no \\
&\hspace*{2.04cm} \times V(x_{n-j})  \f{1}{n!}
\bigg[\f{\partial^{n}}{\partial z^{n}}
G_0(z;x_{n-j},x_1)\bigg] \bigg\|_{\cB(\bbC^N)}\no\\
&\quad \leq C_{n,j,\delta,\Omega} |z|^{-(n-j-1)\delta-1},\quad z\in \Omega.\lb{11.107}
\end{align}
In addition, since $|z|^{-(n-j-1)(1-\delta)-1}$ is bounded in $\overline{\Omega}$ if $0\notin \overline{\Omega}$, Lebesgue's dominated convergence theorem implies that
\begin{align}
&\int_{\bbR^n} d^n x_1 \cdots \int_{\bbR^n} d^n x_{n-j-1} \int_{\bbR^n} d^n x_{n-j} \, V(x_1) G_0(z;x_1,x_2)
\no \\
& \hspace*{1.4cm} \cdots \times V(x_{n-j-1}) G_0(z;x_{n-j-1},x_{n-j})   \\
& \hspace*{1.86cm} \times V(x_{n-j})  \f{1}{n!}
\bigg[\f{\partial^{n}}{\partial z^{n}}
G_0(z;x_{n-j},x_1)\bigg]\no
\end{align}
is analytic in $\Omega$ and extends continuously to $\overline{\Omega}$ if $0\notin \overline\Omega$.  This settles {\it Case~2}.

Turning to {\it Case~3}, we assume that $j=n-2$.  In this case, $\underline{k}=(k_1,k_2,k_3)$ with $k_1+k_3\neq n-1$.  Let $\delta\in(0,1)$ be fixed.  Invoking the final estimate in \eqref{11.34}, one obtains
\begin{align}
&\bigg\|\int_{\bbR^n}d^nx_1\int_{\bbR^n}d^nx_2\, V(x_1)\frac{1}{k_2!}\bigg[\frac{\partial^{k_2}}{\partial z^{k_2}}G_0(z;x_1,x_2) \bigg]\no\\
&\hspace*{1.5cm} \times V(x_2)\frac{1}{(k_1+k_3+1)!}\bigg[ \frac{\partial^{k_1+k_3+1}}{\partial z^{k_1+k_3+1}}G_0(z;x_2,x_1)\bigg]\bigg\|_{\cB(\bbC^N)}\no\\
&\quad \leq C_{n,j,\delta}|z|^{-2\delta}\int_{\bbR^n}d^nx_1\int_{\bbR^n}d^nx_2\no\\
&\hspace*{1cm}\times\langle x_1\rangle^{-n-1-\varepsilon}\big\{|x_1-x_2|^{k_2+1-\delta-n} \chi_{[0,1]}(|z||x_1-x_2|)    \no \\
& \hspace*{1cm} \quad + |z|^{(n-1+2\delta)/2} 
\big[|x_1|^{(2k_2+1-n)/2} + |x_2|^{(2k_2+1-n)/2}\big] \chi_{[1,\infty)}(|z||x_1-x_2|)\big\}\no\\
&\hspace*{1cm}\times \langle x_2 \rangle^{-n-1-\varepsilon}  
\big\{|x_2-x_1|^{k_1+k_3+2-\delta-n} \chi_{[0,1]}(|z||x_2-x_1|)    \no \\
& \hspace*{1cm} \quad + |z|^{(n-1+2\delta)/2} 
\big[|x_2|^{(2k_1+2k_3+3-n)/2} + |x_1|^{(2k_1+2k_3+3-n)/2}\big]    \no \\
& \hspace*{1.4cm} \qquad \times \chi_{[1,\infty)}(|z||x_2-x_1|)\big\}\no\\
& \quad \leq \wti C_{n,j,\delta}|z|^{-2\delta} \int_{\bbR^n} d^n x_1\int_{\bbR^n} d^n x_2 \lb{11.109}\\
& \hspace*{1.4cm} \times \langle x_1 \rangle^{-n-1-\varepsilon} 
\big\{|x_1-x_2|^{k_2+1-\delta-n}     \no \\
& \hspace*{1.4cm} \quad + |z|^{(n-1+2\delta)/2} 
\big[|x_1|^{(2k_2+1-n)/2} + |x_2|^{(2k_2+1-n)/2}\big]\big\}
  \no \\ 
&  \hspace*{1.4cm} \times \langle x_2 \rangle^{-n-1-\varepsilon}  
\big\{|x_2-x_1|^{k_1+k_3+2-\delta-n}    \no \\
& \hspace*{1.4cm} \quad + |z|^{(n-1+2\delta)/2} 
\big[|x_2|^{(2k_1+2k_3+3-n)/2} + |x_1|^{(2k_1+2k_3+3-n)/2}\big]\big\}, \quad z \in \bbC_+,    \no
\end{align}
where $C_{n,n-2,\delta}, \wti C_{n,n-2,\delta} \in (0,\infty)$ are suitable constants.  We claim that for each bounded subset $\Omega\subset \bbC_+$, the integrand under the iterated integral on the right-hand side in \eqref{11.109} is uniformly bounded with respect to $z\in \Omega$ by an integrable function of the variables $x_1,x_2$.  Since $|z|^{(n-1+2\delta)/2}$ is locally bounded, to justify the claim, it suffices to establish convergence of the following integral:
\begin{align}
&\int_{\bbR^n} d^n x_1\int_{\bbR^n} d^n x_2\, \langle x_1 \rangle^{-n-1-\varepsilon} 
\big\{|x_1-x_2|^{k_2+1-\delta-n}     \no \\
& \hspace*{2.45cm} \qquad +  
\big[|x_1|^{(2k_2+1-n)/2} + |x_2|^{(2k_2+1-n)/2}\big]\big\}   \no \\ 
&  \hspace*{2.45cm} \quad \times \langle x_2 \rangle^{-n-1-\varepsilon}  
\big\{|x_2-x_1|^{k_1+k_3+2-\delta-n}       \lb{11.110} \\
& \hspace*{2.45cm} \qquad + 
\big[|x_2|^{(2k_1+2k_3+3-n)/2} + |x_1|^{(2k_1+2k_3+3-n)/2}\big]\big\}.    \no
\end{align}
In turn, as in the argument for the proof of part $(i)$ and Cases 1 and 2, it suffices to focus on the most singular term in \eqref{11.110} and thus disregard the terms originally multiplied by the factor $|z|^{(n-1+2\delta)/2}$ in \eqref{11.109} (following the same line of reasoning used throughout \eqref{11.12a}--\eqref{11.14}).  With this simplification, the claim reduces to establishing convergence of the integral:
\begin{align}
&\int_{\bbR^n} d^n x_1\int_{\bbR^n} d^n x_2\, \langle x_1 \rangle^{-n-1-\varepsilon} 
|x_1-x_2|^{k_2+1-\delta-n} \no\\
&\hspace*{2.85cm}\times\langle x_2 \rangle^{-n-1-\varepsilon}  
|x_2-x_1|^{k_1+k_3+2-\delta-n}\no\\
&\quad =\int_{\bbR^n}d^nx_1\int_{\bbR^n}d^nx_2\langle x_1\rangle^{-n-1-\varepsilon}|x_1-x_2|^{2(1-\delta)-n}\langle x_2\rangle^{-n-1-\varepsilon}.\lb{11.111}
\end{align}
Applying \eqref{11.16} with $\alpha=2(1-\delta)$, $\beta=n$, and $\gamma=n+1$, one infers that
\begin{equation} \lb{11.112}
\int_{\bbR^n}d^nx_2 |x_1-x_2|^{2(1-\delta)-n}\langle x_2\rangle^{-n-1-\varepsilon} \leq c_{2,n,\delta}, 
\quad x_1 \in \bbR^n,
\end{equation}
for some $c_{2,n,\delta}\in (0,\infty)$.  In turn, \eqref{11.112} implies
\begin{equation}
\eqref{11.111}\leq c_{2,n,\delta}\int_{\bbR^n}d^nx_1\,\langle x_1\rangle^{-n-1-\varepsilon}<\infty.  \lb{11.113}
\end{equation}
Hence, this establishes the claim that for each bounded subset $\Omega\subset \bbC_+$, the integrand under the iterated integral on the right-hand side in \eqref{11.109} is uniformly bounded with respect to $z\in \Omega$ by an integrable function of the variables $x_1,...,x_2$.  As a consequence of this claim, \eqref{11.109} implies that for each $\delta\in (0,1)$ and each bounded subset $\Omega\subset \bbC_+$, the following estimate holds:
\begin{equation}
\eqref{11.109}\leq C_{n,n-2,\delta,\Omega}|z|^{-2\delta},\quad z\in \Omega,
\end{equation}
for some $C_{n,n-2,\delta,\Omega}\in (0,\infty)$.  In summary, for {\it Case~3}, one has that for any $\delta\in (0,1)$ and any bounded subset $\Omega\subset\bbC_+$,
\begin{align}
&\bigg\|\int_{\bbR^n}d^nx_1\int_{\bbR^n}d^nx_2\, V(x_1)\frac{1}{k_2!}\bigg[\frac{\partial^{k_2}}{\partial z^{k_2}}G_0(z;x_1,x_2) \bigg]\no\\
&\hspace*{1.6cm} \times V(x_2)\frac{1}{(k_1+k_3+1)!}\bigg[ \frac{\partial^{k_1+k_3+1}}{\partial z^{k_1+k_3+1}}G_0(z;x_2,x_1)\bigg]\bigg\|_{\cB(\bbC^N)}\no\\
&\quad \leq C_{n,n-2,\delta,\Omega}|z|^{-2\delta},\quad z\in \Omega.\lb{11.115}
\end{align}
In addition, since $|z|^{-2\delta}$ is bounded in $\overline{\Omega}$ if $0\notin \overline{\Omega}$, Lebesgue's dominated convergence theorem implies that
\begin{align}
&\int_{\bbR^n}d^nx_1\int_{\bbR^n}d^nx_2\, V(x_1)\frac{1}{k_2!}\bigg[\frac{\partial^{k_2}}{\partial z^{k_2}}G_0(z;x_1,x_2) \bigg]\no\\
&\hspace*{1.4cm} \times V(x_2)\frac{1}{(k_1+k_3+1)!}\bigg[ \frac{\partial^{k_1+k_3+1}}{\partial z^{k_1+k_3+1}}G_0(z;x_2,x_1)\bigg]
\end{align}
is analytic in $\Omega$ and extends continuously to $\overline{\Omega}$ if $0\notin \overline\Omega$.  This settles {\it Case~3}.

Turning to {\it Case~4}, we assume that $j=n-2$.  In this case, $\underline{k}=(k_1,0,k_3)$ with $k_1+k_3= n-1$.  Let $\delta\in(0,1)$ be fixed.  Invoking the final estimate in \eqref{11.34}, one obtains
\begin{align}
&\bigg\|\int_{\bbR^n}d^nx_1\int_{\bbR^n}d^nx_2\, V(x_1)\frac{1}{k_2!}\bigg[\frac{\partial^{k_2}}{\partial z^{k_2}}G_0(z;x_1,x_2) \bigg]\no\\
&\qquad \times V(x_2)\frac{1}{(k_1+k_3+1)!}\bigg[ \frac{\partial^{k_1+k_3+1}}{\partial z^{k_1+k_3+1}}G_0(z;x_2,x_1)\bigg]\bigg\|_{\cB(\bbC^N)}\no\\
&\quad =\bigg\|\int_{\bbR^n}d^nx_1\int_{\bbR^n}d^nx_2\, V(x_1)G_0(z;x_1,x_2)V(x_2)\frac{1}{n!}\bigg[ \frac{\partial^n}{\partial z^n}G_0(z;x_2,x_1)\bigg]\bigg\|_{\cB(\bbC^N)}\no\\
&\quad \leq C_{n,j,\delta}|z|^{-\delta-1}\int_{\bbR^n}d^nx_1\int_{\bbR^n}d^nx_2\no\\
&\hspace*{1cm}\times\langle x_1\rangle^{-n-1-\varepsilon}\big\{|x_1-x_2|^{1-\delta-n} \chi_{[0,1]}(|z||x_1-x_2|)    \no \\
& \hspace*{1cm} \quad + |z|^{(n-1+2\delta)/2} 
\big[|x_1|^{(1-n)/2} + |x_2|^{(1-n)/2}\big] \chi_{[1,\infty)}(|z||x_1-x_2|)\big\}\no\\
&\hspace*{1cm}\times \langle x_2 \rangle^{-n-1-\varepsilon}  
\big\{ \chi_{[0,1]}(|z||x_2-x_1|)    \no \\
& \hspace*{1cm} \quad + |z|^{(n+1)/2} 
\big[|x_2|^{(n+1)/2} + |x_1|^{(n+1)/2}\big]    \no \\
& \hspace*{1.4cm} \qquad \times \chi_{[1,\infty)}(|z||x_2-x_1|)\big\}\no\\
& \quad \leq \wti C_{n,j,\delta}|z|^{-\delta-1} \int_{\bbR^n} d^n x_1\int_{\bbR^n} d^n x_2\lb{11.117} \\
& \hspace*{1.4cm} \times \langle x_1 \rangle^{-n-1-\varepsilon} 
\big\{|x_1-x_2|^{1-\delta-n}   + |z|^{(n-1+2\delta)/2} 
\big[|x_1|^{(1-n)/2} + |x_2|^{(1-n)/2}\big]\big\}
  \no \\ 
&  \hspace*{1.4cm} \times \langle x_2 \rangle^{-n-1-\varepsilon}  
\big\{1+ |z|^{(n+1)/2} 
\big[|x_2|^{(n+1)/2} + |x_1|^{(n+1)/2}\big]\big\}, \quad z \in \bbC_+,    \no
\end{align}
where $C_{n,n-2,\delta}, \wti C_{n,n-2,\delta} \in (0,\infty)$ are suitable constants.  We claim that for each bounded subset $\Omega\subset \bbC_+$, the integrand under the iterated integral on the right-hand side in \eqref{11.117} is uniformly bounded with respect to $z\in \Omega$ by an integrable function of the variables $x_1,x_2$.  Since $|z|^{(n-1+2\delta)/2}$ and $|z|^{(n+1)/2}$ are locally bounded, to justify the claim, it suffices to establish convergence of the following integral:
\begin{align}
&\int_{\bbR^n} d^n x_1\int_{\bbR^n} d^n x_2     \no \\
& \hspace*{1.4cm} \times \langle x_1 \rangle^{-n-1-\varepsilon} 
\big\{|x_1-x_2|^{1-\delta-n}   + |z|^{(n-1+2\delta)/2} 
\big[|x_1|^{(1-n)/2} + |x_2|^{(1-n)/2}\big]\big\}
  \no \\ 
&  \hspace*{1.4cm} \times \langle x_2 \rangle^{-n-1-\varepsilon}  
\big\{1+ |z|^{(n+1)/2} 
\big[|x_2|^{(n+1)/2} + |x_1|^{(n+1)/2}\big]\big\}   \lb{11.118}
\end{align}
In turn, as in the argument for the proof of part $(i)$ and {\it Cases~1, 2}, and {\it 3}, it suffices to focus on the most singular term in \eqref{11.118} and thus disregard the terms originally multiplied by the factor $|z|^{(n-1+2\delta)/2}$ or $|z|^{(n+1)/2}$ in \eqref{11.117} (following the same line of reasoning used throughout \eqref{11.12a}--\eqref{11.14}).  With this simplification, the claim reduces to establishing convergence of the integral:
\begin{align} \lb{11.119}
&\int_{\bbR^n} d^n x_1\int_{\bbR^n} d^n x_2\,\langle x_1 \rangle^{-n-1-\varepsilon} 
|x_1-x_2|^{1-\delta-n}\langle x_2 \rangle^{-n-1-\varepsilon}.
\end{align}
The integral in \eqref{11.119} is similar to the integral in \eqref{11.111}.  An argument entirely analogous to that used throughout \eqref{11.111}--\eqref{11.113} to show the integral in \eqref{11.111} is finite shows that the integral in \eqref{11.119} is finite.  We omit further details at this point.  In summary, for {\it Case~4}, one has that for any $\delta\in (0,1)$ and any bounded subset $\Omega\subset \bbC_+$,
\begin{align}
&\bigg\|\int_{\bbR^n}d^nx_1\int_{\bbR^n}d^nx_2\, V(x_1)\frac{1}{k_2!}\bigg[\frac{\partial^{k_2}}{\partial z^{k_2}}G_0(z;x_1,x_2) \bigg]\no\\
&\hspace*{3cm} \times V(x_2)\frac{1}{(k_1+k_3+1)!}\bigg[ \frac{\partial^{k_1+k_3+1}}{\partial z^{k_1+k_3+1}}G_0(z;x_2,x_1)\bigg]\bigg\|_{\cB(\bbC^N)}\no\\
&\quad \leq C_{n,n-2,\delta,\Omega}|z|^{-\delta-1},\quad z\in \Omega,\lb{11.120}
\end{align}
for some $C_{n,n-2,\delta,\Omega}\in (0,\infty)$.  In addition, since $|z|^{-\delta-1}$ is bounded in $\overline{\Omega}$ if $0\notin \overline{\Omega}$, Lebesgue's dominated convergence theorem implies that
\begin{align}
&\int_{\bbR^n}d^nx_1\int_{\bbR^n}d^nx_2\, V(x_1)\frac{1}{k_2!}\bigg[\frac{\partial^{k_2}}{\partial z^{k_2}}G_0(z;x_1,x_2) \bigg]\no\\
&\hspace*{2.8cm} \times V(x_2)\frac{1}{(k_1+k_3+1)!}\bigg[ \frac{\partial^{k_1+k_3+1}}{\partial z^{k_1+k_3+1}}G_0(z;x_2,x_1)\bigg]
\end{align}
is analytic in $\Omega$ and extends continuously to $\overline{\Omega}$ if $0\notin \overline\Omega$.  This settles {\it Case~4}.

In {\it Case~5}, we assume that $j=n-1$.  In this case, $\underline{k}=(k_1,k_2)$ with $k_1+k_2=n-1$.  Invoking the final estimate in \eqref{C.34}, one obtains
\begin{align}
\bigg\|\int_{\bbR^n}d^nx_1\, V(x_1)\bigg[\frac{\partial^n}{\partial z^n}G_0(z;x_1,x_1) \bigg]\bigg\|_{\cB(\bbC^N)} & \leq C_n|z|^{-1} \int_{\bbR^n}d^nx_1\, \langle x_1\rangle^{-n-1-\varepsilon}\no\\
& = \wti C_n|z|^{-1},\quad z\in \bbC_+,\lb{11.122}
\end{align}
for some $C_n, \wti C_n\in (0,\infty)$.  In addition, since $|z|^{-1}$ is bounded outside any neighborhood of $0$, Lebesgue's dominated convergence theorem implies that
\begin{equation} \lb{11.123}
\int_{\bbR^n}d^nx_1\, V(x_1)\bigg[\frac{\partial^n}{\partial z^n}G_0(z;x_1,x_1) \bigg]
\end{equation}
is analytic in $\Omega$ and extends continuously to $\overline{\bbC_+}\backslash\{0\}$.

Now, looking at the bounds \eqref{11.75}, \eqref{11.107}, \eqref{11.115}, \eqref{11.120}, and \eqref{11.122}, we identify the bound which is the most singular as $z\to 0$ in $\bbC_+$.  The bounds from \eqref{11.75} are (up to $z$-independent constant multiples)
\begin{equation} \lb{11.124}
|z|^{-(n-j)\delta},\quad \delta=\frac{n-j-2}{n-j-1},\, 0\leq j\leq n-3.
\end{equation}
The singularity in \eqref{11.124} is strongest when $(n-j)\delta$, $0\leq j\leq n-3$, is largest.  Since the expression
\begin{equation}
(n-j)\delta = (n-j) - \frac{n-j}{n-j-1}
\end{equation}
is decreasing with respect to the parameter $j$, its maximum value is attained for $j=0$:
\begin{equation}
\max_{0\leq j\leq n-3} (n-j)\delta = n - \frac{n}{n-1}.
\end{equation}
Thus, the strongest singularity in \eqref{11.124} corresponds to $j=0$ and is
\begin{equation} \lb{11.127}
|z|^{-[n - (n/(n-1))]}.
\end{equation}
The bounds from \eqref{11.107} are (up to $z$-independent constant multiples)
\begin{equation} \lb{11.128}
|z|^{-(n-j-1)\delta-1},\quad \delta\in (0,1),\, 0\leq j\leq n-3.
\end{equation}
Choosing $\delta = \eta/(n-j-1)$, $\eta\in (0,1)$, the bound in \eqref{11.128} may be recast as
\begin{equation} \lb{11.129}
|z|^{-(1+\eta)},\quad \eta\in (0,1),\, 0\leq j\leq n-3.
\end{equation}
The bound from \eqref{11.115} is (up to a $z$-independent constant multiple)
\begin{equation} \lb{11.130}
|z|^{-2\delta},\quad \delta\in (0,1),\, j=n-2.
\end{equation}
Choosing $\delta= \eta/2$, $\eta\in (0,1)$, the bound in \eqref{11.130} may be recast as
\begin{equation} \lb{11.131}
|z|^{-\eta},\quad \eta\in (0,1),\, j=n-2.
\end{equation}
The bound from \eqref{11.120} is (up to a $z$-independent constant multiple)
\begin{equation} \lb{11.132}
|z|^{-(1+\delta)},\quad \delta\in (0,1),\, j=n-2.
\end{equation}
The bound from \eqref{11.122} is (up to a $z$-independent constant multiple)
\begin{equation} \lb{11.133}
|z|^{-1},\quad j=n-1.
\end{equation}
If $n\geq 4$, then the strongest singularity from \eqref{11.124}, \eqref{11.129}, \eqref{11.131}, \eqref{11.132}, and \eqref{11.133} is given by \eqref{11.127}.  Therefore, combining the results of {\it Case~1--Case~5} above with \eqref{11.6} and \eqref{11.7A}, one concludes that $\frac{d^n}{dz^n}G_{H,H_0}(\dott)$ is analytic in $\bbC_+$, continuous in $\overline{\bbC_+}\backslash\{0\}$ and
\begin{equation} \lb{11.134}
\bigg\|\frac{d^n}{dz^n}G_{H,H_0}(\dott)\bigg\|_{\cB(\bbC^N)}\underset{\substack{z \to 0, \\ z \in \ol{\bbC_+} \backslash\{0\}}}{=} O\big(|z|^{-[n - (n/(n-1))]}\big).
\end{equation}

If $n=2$, then the strongest singularity in \eqref{11.131} and \eqref{11.133} is $|z|^{-(1+\delta)}$.  Therefore, combining the results of {\it Case~4} and {\it Case~5} above with \eqref{11.6} and \eqref{11.7A}, one concludes that $\frac{d^2}{dz^2}G_{H,H_0}(\dott)$ is analytic in $\bbC_+$, continuous in $\overline{\bbC_+}\backslash\{0\}$ and for any $\delta\in (0,1)$,
\begin{equation} \lb{11.137}
\bigg\|\frac{d^2}{dz^2}G_{H,H_0}(\dott)\bigg\|_{\cB(\bbC^2)}\underset{\substack{z \to 0, \\ z \in \ol{\bbC_+} \backslash\{0\}}}{=} O\big(|z|^{-(1+\delta)}\big).
\end{equation}
\end{proof}

\section{Analysis of $\xi(\, \cdot \,; H,H_0)$ and an Application to the Witten Index for a Class of 
Non-Fredholm operators} \lb{s11}

Combining Hypotheses \ref{h9.13} and \ref{h11.1a} we next make the following assumptions to describe continuity properties of the spectral shift function for the pair $(H,H_0)$. 
 
\begin{hypothesis} \lb{h12.1}
Let $n \in \bbN$ and suppose that $V= \{V_{\ell,m}\}_{1 \leq \ell,m \leq N}$ satisfies for some constants  
$C \in (0,\infty)$ and $\varepsilon > 0$, 
\begin{equation}
V \in [L^{\infty} (\bbR^n)]^{N \times N}, \quad 
|V_{\ell,m}(x)| \leq C \langle x \rangle^{- n - 1 - \varepsilon} \, \text{ for a.e.~$x \in \bbR^n$, $1 \leq \ell,m \leq N$.}    \lb{12.1}
\end{equation}
In addition, assume that $V(x)= \{V_{\ell,m}(x)\}_{1 \leq \ell,m \leq N}$ is self-adjoint for a.e.~$x \in \bbR^n$. 
In accordance with the factorization based on the polar decomposition of $V$ discussed in \eqref{9.27} we 
suppose that $V = V_1^* V_2 = |V|^{1/2} U_V |V|^{1/2}$, where $V_1 = V_1^* = |V|^{1/2}$, $V_2 = U_V |V|^{1/2}$. 

Finally, we assume that $V$ satisfies \eqref{4.3A} and \eqref{4.4}\footnote{The first condition in \eqref{4.4} is superseded by assumption \eqref{12.1}.}.  
\end{hypothesis} 

Thus, combining Theorems \ref{t8.13}, \ref{t9.14}, and \ref{t11.1} yields the principal result of this 
manuscript:

\begin{theorem} \lb{t12.2}
Assume Hypothesis \ref{h12.1}. Then 
\begin{equation}
\xi(\dott;H,H_0) \in C((-\infty,0) \cup (0,\infty)),      \lb{12.3} 
\end{equation}
and the left and right limits at zero, 
\begin{equation}
\xi(0_{\pm}; H,H_0) = \lim_{\varepsilon \downarrow 0} \xi(\pm \varepsilon; H,H_0), 
\lb{12.4}
\end{equation}   
exist. In particular, if $0$ is a regular point for $H$ according to Definition \ref{d9.6}\,$(iii)$ and Theorem 
\ref{t9.7}\,$(iii)$, then 
\begin{equation}
\xi(\dott;H,H_0) \in C(\bbR).     \lb{12.5} 
\end{equation}
\end{theorem} 

In the remainder of this section we describe an application to the Witten index for a class of non-Fredholm operators applicable in the context of multi-dimensional, massless Dirac operators $H$. We develop some necessary preparations and the basic setup next. 

We begin by isolating a bit of notation: Linear operators in the Hilbert space $L^2(\bbR; dt; \cH)$, in short, $L^2(\bbR; \cH)$, will be denoted by calligraphic boldface symbols of the type $\bsT$, to distinguish them from operators $T$ in $\cH$. In particular, operators denoted by 
$\bsT$ in the Hilbert space $L^2(\bbR;\cH)$ represent operators associated with a 
family of operators $\{T(t)\}_{t\in\bbR}$ in $\cH$, defined by
\begin{align}
&(\bsT f)(t) = T(t) f(t) \, \text{ for a.e.\ $t\in\bbR$,}    \no \\
& f \in \dom(\bsT) = \bigg\{g \in L^2(\bbR;\cH) \,\bigg|\,
g(t)\in \dom(T(t)) \text{ for a.e.\ } t\in\bbR;    \lb{12.6}  \\
& \quad t \mapsto T(t)g(t) \text{ is (weakly) measurable;} \, 
\int_{\bbR} dt \, \|T(t) g(t)\|_{\cH}^2 <  \infty\bigg\}.   \no
\end{align}
In the special case, where $\{T(t)\}$ is a family of bounded operators on $\cH$ with 
$\sup_{t\in\bbR}\|T(t)\|_{\cB(\cH)}<\infty$, the associated operator $\bsT$ is a bounded operator on $L^2(\bbR;\cH)$ with $\|\bsT\|_{\cB(L^2(\bbR;\cH))} = \sup_{t\in\bbR}\|T(t)\|_{\cB(\cH)}$.

For brevity we will abbreviate $\bsI := I_{L^2(\bbR; \cH)}$ in the following and note that in the concrete situation of $n$-dimensional, massless Dirac operators at hand, $\cH = [L^2(\bbR^n)]^N$. 

Denoting 
\begin{equation}
A_- = H_0, \quad B_+ = V, \quad A_+ = A_- + B_+ = H, 
\end{equation}
we introduce two families of operators in 
$[L^2(\bbR^n)]^N$ by 
\begin{align}
\begin{split} 
& B(t) = b(t) B_+, \quad t \in \bbR,     \\
& b^{(k)} \in C^{\infty}(\bbR) \cap L^{\infty}(\bbR; dt), \; k \in \bbN_0, \quad b' \in L^1(\bbR; dt),  \\ 
& \lim_{t \to \infty} b(t) = 1, \quad \lim_{t \to - \infty} b(t) = 0,    \lb{12.8} \\
& A(t) = A_- + B(t), \quad t \in \bbR.      
\end{split} 
\end{align}

Next, following the general setups described in \cite{CGGLPSZ16b}, \cite{CGLPSZ16}--\cite{CLPS20}, 
\cite{GLMST11}, \cite{Pu08} we recall the definitions of 
$\bsA$, $\bsB, \bsA '=\bsB'$, given in terms of the families $A(t)$, $B(t)$, and $B'(t)$, $t\in\bbR$, as in \eqref{12.6}. In addition, $\bsA_-$  in $L^2\big(\bbR;[L^2(\bbR^n)]^N\big)$ represents 
the self-adjoint (constant fiber) operator defined by 
\begin{align}
&(\bsA_- f)(t) = A_- f(t) \, \text{ for a.e.\ $t\in\bbR$,}   \no \\
& f \in \dom(\bsA_-) = \bigg\{g \in L^2\big(\bbR;[L^2(\bbR^n)]^N\big) \,\bigg|\,
g(t)\in \dom(A_-) \text{ for a.e.\ } t\in\bbR,    \no \\
& \quad t \mapsto A_- g(t) \text{ is (weakly) measurable,} \,  
\int_{\bbR} dt \, \|A_- g(t)\|_{[L^2(\bbR^n)]^N}^2 < \infty\bigg\}.    \lb{12.9}
\end{align} 
Next, we introduce the operator $\bsD_\bsA^{}$ in $L^2\big(\bbR;[L^2(\bbR^n)]^N\big)$ by 
\begin{equation}
\bsD_\bsA^{} = \f{d}{dt} + \bsA,
\quad \dom(\bsD_\bsA^{})= W^{1,2}\big(\bbR; [L^2(\bbR^n)]^N\big) \cap \dom(\bsA_-),   \lb{12.10}
\end{equation} 
where  
\begin{equation}
\bsA = \bsA_- + \bsB, \quad \dom(\bsA) = \dom(\bsA_-),  
\end{equation}
and  
\begin{equation}
\|\bsB\|_{\cB(L^2(\bbR; [L^2(\bbR^n)]^N))} = \sup_{t \in \bbR} \|B(t)\|_{\cB([L^2(\bbR^n)]^N)} < \infty. 
\end{equation}

Here the operator $d/dt$ in $L^2\big(\bbR;[L^2(\bbR^n)]^N\big)$  is defined by 
\begin{align}
& \bigg(\f{d}{dt}f\bigg)(t) = f'(t) \, \text{ for a.e.\ $t\in\bbR$,}    \no \\
& \, f \in \dom(d/dt) = \big\{g \in L^2\big(\bbR;[L^2(\bbR^n)]^N\big) \, \big|\,
g \in AC_{\loc}\big(\bbR; [L^2(\bbR^n)]^N\big),   \no \\
& \hspace*{7.05cm} g' \in L^2\big(\bbR;[L^2(\bbR^n)]^N\big)\big\}    \no \\
& \hspace*{2.3cm} = W^{1,2} \big(\bbR; [L^2(\bbR^n)]^N\big).     \lb{12.13} 
\end{align} 
By \cite[Lemma 4.4]{GLMST11} (which extends to the present setting), 
$\bsD_\bsA^{}$ is densely defined and closed in $L^2\big(\bbR; [L^2(\bbR^n)]^N\big)$  and the adjoint 
operator $\bsD_\bsA^*$ of $\bsD_\bsA^{}$ is given by
\begin{equation}
\bsD_\bsA^*=- \f{d}{dt} + \bsA, \quad
\dom(\bsD_\bsA^*) = W^{1,2}\big(\bbR; [L^2(\bbR^n)]^N\big) \cap \dom(\bsA_-).   
\end{equation}

This enables one to introduce the nonnegative, self-adjoint operators 
$\bsH_j$, $j=1,2$, in $L^2\big(\bbR;[L^2(\bbR^n)]^N\big)$ by
\begin{equation}
\bsH_1 = \bsD_{\bsA}^{*} \bsD_{\bsA}^{}, \quad 
\bsH_2 = \bsD_{\bsA}^{} \bsD_{\bsA}^{*}.
\end{equation} 

In order to effectively describe the domains of $\bsH_j$, $j=1,2$, we will  
decompose the latter as discussed below:  To this end, one first observes that  
\begin{equation}
\|\bsB'\|_{\cB(L^2(\bbR; [L^2(\bbR^n)]^N))} = \sup_{t \in \bbR} \|B'(t)\|_{\cB([L^2(\bbR^n)]^N)} < \infty.  
\lb{12.16}
\end{equation}
It is convenient to also introduce the operator $\bsH_0$ in $L^2\big(\bbR; [L^2(\bbR^n)]^N\big)$ by  
\begin{equation}
\bsH_0 = - \f{d^2}{dt^2} + \bsA_-^2, \quad \dom(\bsH_0) 
= W^{2,2}\big(\bbR; [L^2(\bbR^n)]^N\big) \cap \dom\big(\bsA_-^2\big).     \lb{12.17}
\end{equation}
Then $\bsH_0$ is self-adjoint by Theorem VIII.33 of \cite{RS80}. Moreover, since 
the operator $ \bsB \bsA_- + \bsA_- \bsB$ is $\bsH_0$-bounded with bound less than one, 
\cite[Theorem~VI.4.3]{Ka80} implies the following decomposition of the operators $\bsH_j$, $j=1,2$, 
\begin{align}
& \bsH_j = - \f{d^2}{dt^2} + \bsA^2 + (-1)^j \bsA^{\prime}    \no \\
& \hspace*{5.5mm} = \bsH_0 + \bsB \bsA_- + \bsA_- \bsB + \bsB^2 + (-1)^j \bsB^{\prime},     \lb{12.18} \\
& \, \dom(\bsH_j) = \dom(\bsH_0), \quad j =1,2.    \no 
\end{align} 

Next, we introduce an approximation procedure as follows: Consider the characteristic function for the interval $[- \ell,\ell] \subset \bbR$,  
\begin{equation}
\chi_{\ell}(\nu) = \chi_{[- \ell,\ell]}(\nu), \quad \nu \in \bbR, \; \ell \in \bbN,   \lb{12.19}
\end{equation} 
and hence 
\begin{equation}
\slim_{\ell \to \infty} \chi_{\ell}(A_-) = I_{[L^2(\bbR^n)]^N}.      \lb{12.20} 
\end{equation}
Introducing  
\begin{align}
\begin{split} 
& A_{\ell}(t) = A_- + \chi_{\ell}(A_-) B(t) \chi_{\ell}(A_-) = A_- + B_{\ell}(t), \\
& \dom(A_{\ell}(t)) = \dom(A_-), \quad \ell \in \bbN, \; t \in \bbR,    \lb{12.21} 
\end{split} \\
& A_{+,\ell} = A_- + \chi_{\ell}(A_-) B_+ \chi_{\ell}(A_-), \quad \dom(A_{+,\ell}) = \dom(A_-), 
\quad \ell \in \bbN,   \lb{12.23} 
\end{align}
where 
\begin{equation} 
B_{\ell}(t) = \chi_{\ell}(A_-) B(t) \chi_{\ell}(A_-), \quad \dom(B_{\ell}(t)) = [L^2(\bbR^n)]^N, 
\quad \ell \in \bbN, \; t \in \bbR,   
\end{equation}
one concludes that 
\begin{align} 
& A_{+,\ell} - A_- = \chi_{\ell}(A_-) B_+ \chi_{\ell}(A_-) \in \cB_1\big([L^2(\bbR^n)]^N\big), 
\quad \ell \in \bbN,    \lb{12.24} \\ 
& A_{\ell}'(t) = B_{\ell}'(t) = \chi_{\ell}(A_-) B'(t) \chi_{\ell}(A_-) \in \cB_1\big([L^2(\bbR^n)]^N\big), 
\quad \ell \in \bbN, \; t \in \bbR.     \lb{12.25} 
\end{align}

As a consequence of \eqref{12.24}, which follows from 
\begin{align}
& \|\chi_{\ell}(A_-) B_+ \chi_{\ell}(A_-)\|_{\cB_1([L^2(\bbR^n)]^N)}    \no \\
& \quad \leq 
\big\|\chi_{\ell}(A_-) B_+ (A_- -i I_{[L^2(\bbR^n)]^N})^{-n-1}\big\|_{\cB_1([L^2(\bbR^n)]^N)}   \no \\
& \qquad \times 
\big\|(A_- -i I_{[L^2(\bbR^n)]^N})^{n+1}  \chi_{\ell}(A_-)\big\|_{\cB([L^2(\bbR^n)]^N)} < \infty 
\lb{12.27} 
\end{align}
(cf.\ \eqref{7.4}), the spectral shift functions 
$\xi(\, \cdot \, ; A_{+,\ell}, A_-)$, $\ell \in \bbN$, exist and are uniquely determined by 
\begin{equation}
\xi(\, \cdot \, ; A_{+,\ell}, A_-) \in L^1(\bbR; d\nu), \quad \ell \in \bbN,    \lb{12.28} 
\end{equation}
implying 
\begin{equation}
\tr_{[L^2(\bbR^n)]^N} (f(A_{+,\ell}) - f(A_-)) = \int_{\bbR} d\nu \, \xi(\nu;A_{+,\ell},A_-) f'(\nu), 
\quad f \in C_0^{\infty}(\bbR). 
\end{equation}

We also note the analogous decompositions,
\begin{align}
& \bsH_{j,\ell} = - \f{d^2}{dt^2} + \bsA_{\ell}^2 + (-1)^j \bsA_{\ell}^{\prime}    \no \\
& \hspace*{6.7mm} = \bsH_0 + \bsB_{\ell} \bsA_- + \bsA_- \bsB_{\ell} 
+ \bsB_{\ell}^2 + (-1)^j \bsB_{\ell}^{\prime}, \\
& \, \dom(\bsH_{j,\ell}) = \dom(\bsH_0) = W^{2,2}\big(\bbR; [L^2(\bbR^n)]^N\big),  
\quad \ell \in \bbN, \;  j =1,2,  \no 
\end{align}
with
\begin{equation}
\bsB_{\ell} = \chi_{\ell}(\bsA_-) \bsB \chi_{\ell}(\bsA_-), \quad 
\bsB_{\ell}^{\prime} = \chi_{\ell}(\bsA_-) \bsB^{\prime} \chi_{\ell}(\bsA_-), \quad \ell \in \bbN,  
\end{equation} 
implying 
\begin{align}
& \bsH_2 - \bsH_1 = 2 \bsB^{\prime},      \lb{12.32} \\
&  \bsH_{2,\ell} - \bsH_{1,\ell} = 2 \bsB_{\ell}^{\prime} 
= 2 \chi_{\ell}(\bsA_-) \bsB^{\prime} \chi_{\ell}(\bsA_-), 
\quad \ell \in \bbN.    \lb{12.33} 
\end{align}

Next, we recall the fact that for $\varepsilon > 0$, 
\begin{equation}
L^2\big(\bbR^n; (1+|x|)^{(n/2) + \varepsilon}d^n x\big) \subset \ell^1(L^2)(\bbR^n)
\end{equation}
(see, e.g., \cite[p.~38]{Si05} for the definition of the Birman--Solomyak space $\ell^1(L^2)(\bbR^n)$) and, given $\alpha > n$, 
\begin{equation}
(1 + |\dott|)^{-\alpha} \in L^2\big(\bbR^n; (1+|x|)^{(n/2) + \varepsilon}d^n x\big) 
\end{equation}
for $0 < \varepsilon$ sufficiently small (depending on $a$). This is of relevance here so that 
\cite[Subsection~7.2.1]{CLPS20} becomes applicable in our context. 

We continue with the following basic result in \cite[Theorems~3.6, 7.1]{CLPS20}:

\begin{theorem} \lb{t12.3}
In addition to Hypothesis \ref{h12.1} suppose that 
\begin{equation}
V_{\ell,m} \in W^{4n,\infty}(\bbR^n), \quad 1 \leq \ell, m \leq N. 
\end{equation}
Then, abbreviating 
\begin{equation} 
q = \lceil n/2 \rceil = \begin{cases} (n+1)/2, & $n$ \text{ odd},  \\
n/2, & $n$\text{ even}, 
\end{cases}  
\end{equation}
one obtains 
\begin{align}
\begin{split} 
& [(\bsH_2 - z \, \bsI)^{-q} - (\bsH_1 - z \, \bsI)^{-q}\big], \, 
\big[(\bsH_{2,\ell} - z \, \bsI)^{-q} - (\bsH_{1,\ell} - z \, \bsI)^{-q}\big]    \lb{12.38} \\
& \quad \in \cB_1\big(L^2\big(\bbR; [L^2(\bbR^n)]^N\big)\big), \quad \ell \in \bbN, 
\; r \in \bbN, \; r \geq q, \; z \in \bbC \backslash [0,\infty), 
\end{split}
\end{align} 
and 
\begin{align}
& \lim_{\ell\to\infty} \big\|\big[(\bsH_{2,\ell} - z \, \bsI)^{-q} - (\bsH_{1,\ell} - z \, \bsI)^{-q}\big]   \no \\
& \hspace*{1cm} - [(\bsH_2 - z \, \bsI)^{-q} - (\bsH_1 - z \, \bsI)^{-q}\big]
\big\|_{\cB_1(L^2(\bbR;[L^2(\bbR^n)]^N))} = 0,     \lb{12.39} \\ 
& \hspace*{8.25cm} z \in \bbC \backslash [0,\infty).   \no
\end{align} 
\end{theorem}

For the fact that $q = \lceil n/2 \rceil$ in \eqref{12.38} can be replaced by any $r\geq q$, $r \in \bbN$, see, for instance, \cite[p.~210]{Ya92}; similarly, \eqref{12.39} extends to $r\geq q$, $r \in \bbN$, by \cite{CGLNPS16}.

Relations \eqref{12.38} together with the fact that $\bsH_j \geq 0$, $\bsH_{j,\ell} \geq 0$, $\ell \in \bbN$, $j=1,2$, implies the existence and uniqueness of spectral shift functions 
$\xi(\, \cdot \,; \bsH_2, \bsH_1)$ and 
$\xi(\, \cdot \,; \bsH_{2,\ell}, \bsH_{1,\ell})$ for the pair of operators $(\bsH_2, \bsH_1)$ and 
$(\bsH_{2,\ell}, \bsH_{1,\ell})$, $\ell \in \bbN$, respectively, employing the normalization 
\begin{equation}
\xi(\lambda; \bsH_2, \bsH_1) = 0, \, 
\xi(\lambda; \bsH_{2,\ell}, \bsH_{1,\ell}) = 0, \quad  \lambda < 0, \; \ell \in \bbN     \lb{12.40} 
\end{equation}
(cf.\ \cite[Sect.~8.9]{Ya92}). Moreover, 
\begin{equation}
\xi(\dott; \bsH_2, \bsH_1) \in L^1\big(\bbR; (1+ |\lambda|)^{-q-1} d\lambda\big).    \lb{12.41}
\end{equation}
Since in analogy to \eqref{12.27},
\begin{align}
& \|A_{\ell}'(\dott)\|_{\cB_1([L^2(\bbR^n)]^N)} =  \|B_{\ell}'(\dott)\|_{\cB_1([L^2(\bbR^n)]^N)}   \no \\ 
& \quad = \|\chi_{\ell}(A_-) B'(\dott) \chi_{\ell}(A_-)\|_{\cB_1([L^2(\bbR^n)]^N)}    \no \\
& \quad \leq 
\big\|\chi_{\ell}(A_-) B_+ (A_- -i I_{[L^2(\bbR^n)]^N})^{-n-1}\big\|_{\cB_1([L^2(\bbR^n)]^N)}    
\lb{12.42} \\
& \qquad \times 
\big\|(A_- -i I_{[L^2(\bbR^n)]^N})^{n+1}  \chi_{\ell}(A_-)\big\|_{\cB([L^2(\bbR^n)]^N)} \, b'(\dott) 
\in L^1(\bbR;dt), \quad \ell \in \bbN,     \no 
\end{align}
employing $b'(\dott) \in L^1(\bbR;dt)$ (cf.\ \eqref{12.8}), one obtains
\begin{equation} 
\int_{\bbR} dt \, \|A_{\ell}'(t)\|_{\cB_1([L^2(\bbR^n)]^N)} < \infty, \quad \ell \in \bbN.   \lb{12.43} 
\end{equation} 
Given \eqref{12.43}, the results in \cite{Pu08} (see also \cite{GLMST11}) 
actually imply that 
\begin{equation}
[(\bsH_{2,\ell} - z \, \bsI)^{-1} - (\bsH_{1,\ell} - z \, \bsI)^{-1}\big] 
\in \cB_1\big(L^2\big(\bbR;[L^2(\bbR^n)]^N\big)\big), \quad \ell \in \bbN,  
\end{equation}
and 
\begin{equation}
\xi(\dott; \bsH_{2,\ell}, \bsH_{1,\ell}) \in L^1\big(\bbR; (1 + |\lambda|)^{-2} d\lambda\big), \quad \ell \in \bbN.  
\end{equation}
In particular, 
\begin{align}
& {\tr}_{L^2(\bbR;[L^2(\bbR^n)]^N)} (f(\bsH_2) - f(\bsH_1))  = 
\int_{[0,\infty)} d \lambda \, \xi(\lambda; \bsH_2, \bsH_1) f'(\lambda),    \no \\
& {\tr}_{L^2(\bbR;[L^2(\bbR^n)]^N)} (f(\bsH_{2,\ell}) - f(\bsH_{1,\ell})) = 
\int_{[0,\infty)} d \lambda \, \xi(\lambda; \bsH_{2,\ell}, \bsH_{1,\ell}) f'(\lambda), \quad \ell \in \bbN,  \no \\
& \hspace*{8.8cm} f \in C_0^{\infty}(\bbR).    
\end{align}

In addition, as derived in \cite{Pu08} (see also, \cite{GLMST11}), \eqref{12.28}, \eqref{12.40}, and 
\eqref{12.43} imply the approximate trace formula,   
\begin{equation}
\int_{[0,\infty)} \f{\xi(\lambda; \bsH_{2,\ell}, \bsH_{1,\ell}) \, d\lambda}{(\lambda - z)^2} 
= \f{1}{2} \int_{\bbR} \f{\xi(\nu; A_{+,\ell}, A_-) \, d\nu}{(\nu^2 -z)^{3/2}}, \quad 
\ell \in \bbN, \; z \in \bbC \backslash [0,\infty),      \lb{12.46}
\end{equation}
which in turn implies Pushnitski's formula \cite{Pu08}, 
\begin{equation}
\xi(\lambda; \bsH_{2,\ell}, \bsH_{1,\ell}) = \begin{cases} \f{1}{\pi} \int_{- \lambda^{1/2}}^{\lambda^{1/2}} 
\f{\xi(\nu; A_{+,\ell}, A_-) \, d \nu}{(\lambda - \nu^2)^{1/2}}, & 
\text{for a.e.~$\lambda > 0$,} \\
0, & \lambda < 0, 
\end{cases} 
\quad \ell \in\bbN,    \lb{12.47} 
\end{equation}
via a Stieltjes inversion argument (cf.\ \cite[Sect.~8]{GLMST11}). 

As shown in \cite{CGPST17}, \eqref{12.39} implies for $f \in C_0^{\infty}(\bbR)$, 
\begin{align}
\begin{split} 
& \lim_{\ell\to\infty} \|[f(\bsH_{2,\ell}) - f(\bsH_{1,\ell})] 
- [f(\bsH_2) - f(\bsH_1)]\|_{\cB_1(L^2(\bbR;[L^2(\bbR^n)]^N))} =0,    \lb{12.48}
\end{split} 
\end{align} 
and hence
\begin{align}
&  \lim_{\ell \to \infty} \int_{[0,\infty)} d \lambda \, \xi(\lambda; \bsH_{2,\ell}, \bsH_{1,\ell}) f'(\lambda) 
= \lim_{\ell \to \infty} {\tr}_{L^2(\bbR;[L^2(\bbR^n)]^N)} (f(\bsH_{2,\ell}) - f(\bsH_{1,\ell}))  \no \\
& \quad = {\tr}_{L^2(\bbR;[L^2(\bbR^n)]^N)} (f(\bsH_2) - f(\bsH_1)) 
= \int_{[0,\infty)} d \lambda \, \xi(\lambda; \bsH_2, \bsH_1) f'(\lambda).     \lb{12.49} 
\end{align} 

Abbreviating 
\begin{equation} 
q_0 = 2 \lfloor n/2 \rfloor + 1 = \begin{cases} n, & $n$ \text{ odd},  \\
n+1, & $n$\text{ even}, 
\end{cases}  
\end{equation}
and assuming Hypothesis \ref{h7.1}, one recalls that Theorem \ref{t7.4} implies
\begin{align} 
\begin{split} 
\big[(A_+ - z I_{[L^2(\bbR^n)]^N})^{-r_0} - (A_- - z I_{[L^2(\bbR^n)]^N})^{-r_0} \big] \in \cB_1\big([L^2(\bbR^n)]^N\big),&    \lb{12.50} \\ 
r_0 \in \bbN, \; r_0 \geq q_0, \; z \in \bbC\backslash\bbR.&
\end{split} 
\end{align} 

Since $q_0$ is always odd, \cite[Theorem~2.2]{Ya05} yields the existence of a spectral shift function 
$\xi(\dott; A_+,A_-)$ for the pair $(A_+, A_-)$ satisfying  
\begin{equation}
\xi(\dott; A_+,A_-) \in L^1\big(\bbR; (1 + |\nu|)^{-q_0-1} d \nu\big)    \lb{12.53}
\end{equation}
and hence
\begin{equation}
\tr_{[L^2(\bbR^n)]^N} (f(A_+) - f(A_-)) = \int_{\bbR} d\nu \, \xi(\nu;A_+,A_-) f'(\nu), 
\quad f \in C_0^{\infty}(\bbR).      \lb{12.54}
\end{equation}
While $\xi(\dott; A_+,A_-)$ in \eqref{12.53}, \eqref{12.54} is not unique, we will select a unique candidate using Theorem \ref{t12.5} below. 

The next result is essentially\cite[Theorem~4.7]{CGLNPS16}; due to its importance we reproduce the proof here. To prepare the stage, we temporarily go beyond the approximation $A_{+,\ell}$ of $A_+$ and now introduce the following path of self-adjoint operators $\{A_+(s)\}_{s \in [0,1]}$, in 
$\big[L^2(\bbR^n)\big]^N$, where 
\begin{align}
& A_+(s) = A_- + P_s B_+ P_s, \quad \dom(A_+(s)) = \dom(A_-),   \quad s \in [0,1],    \lb{12.59} \\
& P_s = \chi_{[-(1-s)^{-1},(1-s)^{-1}]}(A_-), \; s \in [0,1), \quad P_1 = I_{[L^2(\bbR^n)]^N},   \lb{12.60} 
\end{align}
in particular, 
\begin{equation} 
A_+(0) = A_{+,1} \text{ (cf.\ \eqref{12.23} with $\ell = 1$) and } \, A_+(1) = A_+.   \lb{12.61} 
\end{equation} 

\begin{theorem} \lb{t12.5} 
Assume Hypothesis \ref{h12.1} and suppose that 
\begin{equation}
V_{\ell,m} \in W^{4n,\infty}(\bbR^n), \quad 1 \leq \ell, m \leq N. 
\end{equation}
Then there exists a unique spectral shift function $\xi(\dott; A_+,A_-)$ such that 
\begin{equation}
\xi(\dott; A_+,A_-) = \xi(\dott; A_+(1),A_-) = \lim_{\ell \to \infty} \xi(\dott; A_{+,\ell},A_-)  \, \text{ in $L^1\big(\bbR; (1 + |\nu|)^{-q_0-1} d\nu\big)$.}     \lb{12.56} 
\end{equation}
Moreover, assume that $g \in L^{\infty}(\bbR; d\nu)$. Then 
\begin{equation}
\lim_{\ell \to \infty} \|\xi(\dott; A_{+,\ell},A_-) g 
- \xi(\dott; A_+,A_-) g\|_{L^1(\bbR; (1 + |\nu|)^{-q_0-1} d\nu)} = 0,   \lb{12.57}
\end{equation}
and hence,
\begin{equation}
\lim_{\ell \to \infty} \int_{\bbR} d \nu \, \xi(\nu; A_{+,\ell},A_-) h(\nu) = 
\int_{\bbR} d \nu \, \xi(\nu; A_+,A_-) h(\nu)     \lb{12.58} 
\end{equation}
for all $h \in L^{\infty}(\bbR; d\nu)$ such that $\esssup_{\nu \in \bbR} |h(\nu)| (1+|\nu|)^{q_0+1} < \infty$.
\end{theorem} 
\begin{proof}
Since by \eqref{12.24}, $\chi_{\ell}(A-) B_+ \chi_{\ell}(A-) \in \cB_1\big([L^2(\bbR^n)]^N\big)$, also 
\begin{equation} 
A_+(s) - A_- = P_s B_+ P_s \in \cB_1\big([L^2(\bbR^n)]^N\big), \quad s \in [0,1),
\end{equation}
and hence there exists a uniques spectral shift function $\xi(\dott;A_+(s),A_-)$ for the pair 
$(A_+(s),A_-)$ satisfying 
\begin{equation}
\xi(\dott;A_+(s),A_-) \in L^1(\bbR; d\nu). 
\end{equation}
Moreover, in complete analogy to \eqref{12.50}, the family 
$A_+(s)$ depends continuously on $s \in [0,1]$ with respect to the pseudometric 
\begin{equation}
d_{q_0,z}(A,A') = \big\|(A - z I_{[L^2(\bbR^n)]^N})^{-q_0} - (A' - z I_{[L^2(\bbR^n)]^N})^{-q_0}\big\|_{\cB_1([L^2(\bbR^n)]^N)}    \lb{12.64}
\end{equation}
for $A, A'$ in the set of self-adjoint operators which satisfy 
for all $\zeta \in i \bbR \backslash \{0\}$, 
\begin{align}
\begin{split}
& \big[(A - \zeta I_{[L^2(\bbR^n)]^N})^{-q_0} - (A_- - \zeta I_{[L^2(\bbR^n)]^N})^{-q_0}\big],    \\
&\big[(A' - \zeta I_{[L^2(\bbR^n)]^N})^{-q_0} - (A_- - \zeta I_{[L^2(\bbR^n)]^N})^{-q_0}\big]  \in \cB_1\big([L^2(\bbR^n)]^N\big).  
\end{split} 
\end{align}
Thus, the hypotheses of \cite[Theorem~4.7]{CGLNPS16} are satisfied and one concludes the existence of a unique spectral shift function $\xi(\dott; A_+(s),A_-)$ for the pair $(A_+(s),A_-)$ depending continuously on 
$s \in [0,1]$ in the space $L^1\big(\bbR; (1 + |\nu|)^{-q_0 - 1} d \nu\big)$, satisfying
$\xi(\dott; A_+(0),A_-) = \xi(\dott; A_{+,1},A_-)$. Taking $s = (\ell-1)/\ell$, $\ell \in \bbN$, yields
\begin{align}
\begin{split} 
\xi(\dott; A_+,A_-) &= \xi(\dott; A_+(1),A_-) = \lim_{s \uparrow 1} \xi(\dott; A_+(s),A_-)    \\
&= \lim_{\ell \to \infty} \xi(\dott; A_{+,\ell},A_-)  \, \text{ in $L^2\big(\bbR; (1 + |\nu|)^{-q_0-1} d\nu\big)$.}
\end{split} 
\end{align}
Hence an appropriate subsequence, again denoted by $\{\xi(\dott; A_{+,\ell},A_-)\}_{\ell \in \bbN}$, converges 
pointwise a.e.~to $\xi(\dott; A_+,A_-)$ as $\ell \to \infty$. Since each 
$\xi(\dott; A_{+,\ell},A_-) \in L^1(\bbR; d \nu)$, $\ell \in \bbN$, is uniquely defined one obtains a unique 
spectral shift function satisfying \eqref{12.64}. 

The facts \eqref{12.57} and \eqref{12.58} are now evident. 
\end{proof}

In the following we will always employ $\xi(\dott; A_+,A_-)$ as determined by the limiting relation \eqref{12.56} as the spectral shift function for the pair $(A_+,A_-)$. 

The next result is fundamental, it establishes \eqref{12.47} in the limit $\ell \to \infty$.

\begin{theorem}\lb{t4.2}
Assume Hypothesis \ref{h12.1} and suppose that 
\begin{equation}
V_{\ell,m} \in W^{4n,\infty}(\bbR^n), \quad 1 \leq \ell, m \leq N. 
\end{equation} 
Then, 
\begin{equation}
\xi(\lambda; \bsH_2, \bsH_1) = \f{1}{\pi} \int_{- \lambda^{1/2}}^{\lambda^{1/2}} 
\f{\xi(\nu; A_+, A_-) \, d \nu}{(\lambda - \nu^2)^{1/2}} \, \text{ for a.e.~$\lambda > 0$.} 
\lb{12.68} 
\end{equation}
\end{theorem}
\begin{proof} 
We start by multiplying the approximate relation \eqref{12.47} by the derivative $f'$ of a test function 
$f \in C_0^{\infty}(\bbR)$, and integrate to get,
\begin{align}
& \int_{\bbR} d \lambda \, \xi(\lambda; \bsH_{2,\ell}, \bsH_{1,\ell}) f'(\lambda) = 
\int_{[0,\infty)} d \lambda \, \xi(\lambda; \bsH_{2,\ell}, \bsH_{1,\ell}) f'(\lambda)    \no \\
& \quad = \f{1}{\pi} \int_{[0,\infty)} d \lambda \, f'(\lambda) \int_{- \lambda^{1/2}}^{\lambda^{1/2}} 
\f{\xi(\nu; A_{+,\ell}, A_-) \, d \nu}{(\lambda - \nu^2)^{1/2}}    \no \\
& \quad = \f{1}{\pi} \int_{\bbR} d\nu \, \xi(\nu; A_{+,\ell},A_-) F'(\nu), \quad \ell \in \bbN,   \lb{12.69} 
\end{align}
where $F'$ is defined by 
\begin{equation}
F'(\nu) = \int_{\nu^2}^{\infty} d\lambda \, f'(\lambda) (\lambda - \nu^2)^{-1/2}, \quad \nu \in \bbR.
\lb{12.70} 
\end{equation}
We claim that 
\begin{equation}
F' \in C^{\infty}_0(\bbR),
\end{equation}
rendering the manipulations leading to \eqref{12.69} well-defined. Clearly, $F' \in C_0(\bbR)$ 
since $f' \in C^{\infty}_0(\bbR)$. To show that $F' \in C_0^{\infty}(\bbR)$, it suffices to repeatedly integrate by parts and allude to the following representations of $F'$, 
\begin{align}
F'(\nu) &= \int_{\nu^2}^{\infty} d\lambda \, f'(\lambda) (\lambda - \nu^2)^{-1/2}   \no \\
& = - 2 \int_{\nu^2}^{\infty} d\lambda \, f''(\lambda) (\lambda - \nu^2)^{1/2} 
= 2 \f{2}{3} \int_{\nu^2}^{\infty} d\lambda \, f'''(\lambda) (\lambda - \nu^2)^{3/2}    \no \\ 
& \;\,\, \vdots \no \\
& = c_k \int_{\nu^2}^{\infty} d\lambda \, f^{(k)}(\lambda) (\lambda - \nu^2)^{k-(3/2)} 
\quad \nu \in \bbR, \; k \in \bbN,     \lb{12.72}
\end{align}
for appropriate constants $c_k$, $k \in \bbN$. Thus, \eqref{12.69} yields the 
following,
\begin{equation}
\int_{[0,\infty)} d \lambda \, \xi(\lambda; \bsH_{2,\ell}, \bsH_{1,\ell}) f'(\lambda) 
= \f{1}{\pi} \int_{\bbR} d\nu \, \xi(\nu; A_{+,\ell},A_-) F'(\nu), \quad \ell \in \bbN,   \lb{12.73}
\end{equation}
where $f \in C_0^{\infty}(\bbR)$ was arbitrary, and $F' \in C_0^{\infty}(\bbR)$ 
(depending on $f'$) is given by \eqref{12.70} or equivalently, by any of the expressions in \eqref{12.72}.

It remains to control the limits $\ell \to \infty$ on either side of \eqref{12.73}: By \eqref{12.49}, 
the left-hand side of \eqref{12.73} converges as $\ell \to \infty$,
\begin{equation}
\lim_{\ell \to \infty} \int_{[0,\infty)} d \lambda \, \xi(\lambda; \bsH_{2,\ell}, \bsH_{1,\ell}) f'(\lambda) 
= \int_{[0,\infty)} d \lambda \, \xi(\lambda; \bsH_2, \bsH_1) f'(\lambda).    \lb{12.74}
\end{equation} 
For the right-hand side of \eqref{12.73} one applies Theorem \ref{t12.5}, especially, \eqref{12.58}, 
and concludes that
\begin{equation}
\lim_{\ell \to \infty} \f{1}{\pi} \int_{\bbR} d\nu \, \xi(\nu; A_{+,\ell},A_-) F'(\nu) 
= \f{1}{\pi} \int_{\bbR} d\nu \, \xi(\nu; A_+,A_-) F'(\nu),    \lb{12.75} 
\end{equation}
since by \eqref{12.64}
\begin{equation} 
\lim_{\ell \to \infty} \|\xi(\, \cdot \,; A_{+,\ell},A_-) 
- \xi(\, \cdot \,; A_+,A_-)\|_{L^1(\bbR; (1 +|\nu|)^{-q_0-1} d\nu)} =0.
\end{equation} 
Combining \eqref{12.73}--\eqref{12.75} finally yields 
\begin{align}
& \int_{[0,\infty)} d \lambda \, \xi(\lambda; \bsH_2, \bsH_1) f'(\lambda)   
= \f{1}{\pi} \int_{\bbR} d \nu \, \xi(\nu; A_+, A_-) F'(\nu)   \no \\
& \quad = \f{1}{\pi} \int_{\bbR} d \lambda \, f'(\lambda) \int_{- \lambda^{1/2}}^{\lambda^{1/2}} 
\f{\xi(\nu; A_+, A_-) \, d \nu}{(\lambda - \nu^2)^{1/2}} \chi_{[0,\infty)}(\lambda), \quad 
f \in C^{\infty}_0(\bbR).    
\end{align}
An application of the Du Bois--Raymond Lemma (see, e.g., \cite[Theorem\ 6.11]{LL01}), thus 
implies for some constant $c \in \bbR$, 
\begin{equation}
\xi(\lambda; \bsH_2, \bsH_1) = \f{1}{\pi} \int_{- \lambda^{1/2}}^{\lambda^{1/2}} 
\f{\xi(\nu; A_+, A_-) \, d \nu}{(\lambda - \nu^2)^{1/2}} \chi_{[0,\infty)}(\lambda) + c 
\, \text{ for a.e.~$\lambda \in \bbR$.}
\end{equation}
Due to our normalization \eqref{12.40}, $c = 0$, proving \eqref{12.68}.
\end{proof} 

Having established \eqref{12.68}, we turn to the resolvent regularized Witten index of the densely defined and closed operator $\bsD_\bsA^{}$. We refer to \cite{BGGSS87}, \cite{CGGLPSZ16b}, 
\cite{CGLPSZ16}--\cite{CLPS20}, \cite{GLMST11}, \cite{GS88}, \cite{Pu08} and the references therein for a bit of history on this subject. 

Since $\sigma(A_{\pm}) = \bbR$, in particular, 
$0 \notin \rho(A_+) \cap \rho(A_-)$, 
\begin{equation} 
\text{$\bsD_\bsA^{}$ is a non-Fredholm operator.} 
\end{equation} 
This follows from the criterion for Fredholm operators established in \cite[Theorem~2.6]{CGPST17} (which extends to the current setting by replacing the resolvent of $A_{\pm}$ by appropriate powers of the resolvent in the proof). 

In the following we will show that even though $\bsD_\bsA^{}$ is a non-Fredholm operator, its Witten 
index is well-defined and expressible in terms of the spectral shift functions for the pair of operators 
$(\bsH_2, \bsH_1)$ and $(A_+, A_-)$.

To introduce an appropriately (resolvent regularized) Witten index of $\bsD_\bsA^{}$, we consider a densely defined, closed operator $T$ in the complex, separable Hilbert space $\cK$ and assume that for some 
$k \in \bbN$, and all $\lambda < 0$
\begin{equation}
\big[(T^* T - \lambda I_{\cK})^{-k} - (T T^* - \lambda I_{\cK})^{-k}\big] \in \cB_1(\cK).   
\end{equation}
Then the $k$th resolvent regularized Witten index of $T$ is defined by
\begin{equation}
W_{k,r}(T) = \lim_{\lambda \uparrow 0} (- \lambda)^k 
\tr_{\cK} \big((T^* T - \lambda I_{\cK})^{-k} - (T T^* - \lambda I_{\cK})^{-k}\big),
\end{equation}
whenever the limit exists. The case $k=1$ as well as the approach where resolvents are replaced by semigroups has been studied in great detail in \cite{CGPST17}, the extension to $k \geq 2$ was discussed in 
\cite{CLPS20}. 

It is well-known that the (regularized) Witten index is generally not an integer, in fact, it can take on any real value (cf.\ \cite{BGGSS87}, \cite{GS88}). The intrinsic value of $W_{k,r}(T)$ 
lies in its stability properties with respect to additive perturbations,
analogous to stability properties of the Fredholm index. Indeed, as long as one replaces the
familiar relative compactness assumption on an additive perturbation in connection with the
Fredholm index, by appropriate relative trace class conditions in connection with
the resolvent regularized Witten index, stability of the Witten index was proved in \cite{BGGSS87} (for $k=1$, see also \cite{CP74}) and, in connection with the analogous semigroup regularized Witten index, in \cite{GS88} (the semigroup approach then yielding stability for $W_{k,r}(\dott)$, $k \in \bbN$).

The following result, the first of this kind applicable to non-Fredholm operators in a 
partial differential operator setting involving multi-dimensional massless Dirac operators, then characterizes the Witten index of $\bsD_\bsA^{}$ in terms of spectral shift functions: 

\begin{theorem} \lb{t12.7} 
Assume Hypothesis \ref{h12.1} and suppose that 
\begin{equation}
V_{\ell,m} \in W^{4n,\infty}(\bbR^n), \quad 1 \leq \ell, m \leq N. 
\end{equation} 
Then $0$ is a right Lebesgue point of $\xi(\dott; \bsH_2, \bsH_1)$, denoted by 
$\xi_L(0_+; \bsH_2, \bsH_1)$, and 
\begin{equation}
\xi_L(0_+; \bsH_2, \bsH_1) = [\xi(0_+;A_+,A_-) + \xi(0_-;A_+,A_-)]/2.
\end{equation}
In addition, the resolvent regularized Witten index $W_{k,r}(\bsD_\bsA^{})$ of $\bsD_\bsA^{}$ exists for all $k \in \bbN$, $k \geq q$ and 
equals
\begin{align}
\begin{split} 
W_{k,r}(\bsD_\bsA^{}) = \xi_L(0_+; \bsH_2, \bsH_1) &= [\xi(0_+;A_+,A_-) + \xi(0_-;A_+,A_-)]/2    \\
&= [\xi(0_+;H,H_0) + \xi(0_-;H,H_0)]/2.    \lb{12.84}
\end{split} 
\end{align}
\end{theorem}
\begin{proof} 
The key new input for the proof is the existence of $0$ as a left and right Lebesgue point of 
$\xi(\dott;A_+,A_-) = \xi(\dott; H,H_0)$.~This is established in Theorem \ref{t12.2}, in fact, more is proved since left and right limits of $\xi(\dott;H,H_0)$ at $0$ are shown to exist. For $q=1$, the remaining assertions are proved in \cite[Theorem~4.3]{CGPST17}, the extension to 
$q \geq 2$ is discussed in \cite[Sects.~5.2, 7.1]{CLPS20}.
\end{proof}

The actual computation of the right-hand side of \eqref{12.84} in terms of the potential $V$ is left for a future investigation.

\appendix
\section{Some Remarks on Block Matrix Operators} \lb{sA}
\renewcommand{\theequation}{A.\arabic{equation}}
\renewcommand{\thetheorem}{A.\arabic{theorem}}
\setcounter{theorem}{0} \setcounter{equation}{0}

In this appendix we collect some useful (and well-known) material on linear operators in connection with pointwise domination, boundedness, compactness, and the Hilbert-Schmidt property.

\begin{definition} \lb{dA.1}
Let $(M; \cM; \mu)$ be a $\sigma$-finite, separable measure space, $\mu$ a nonnegative, measure with 
$0 < \mu (M) \leq \infty$, and consider the linear operators $A, B \in \cB\big(L^2(M; d\mu)\big)$. Then $B$ pointwise dominates $A$ 
\begin{equation}
\text{if for all $f \in L^2(M; d\mu)$, $|(A f)(\, \cdot \,)| \leq (B |f|)(\, \cdot \,)$ $\mu$-a.e.~on $M$}.   
\lb{A.1}
\end{equation}
\end{definition}

For a linear block operator matrix $T = \{T_{j,k}\}_{1 \leq j,k \leq N}$, $N \in \bbN$, in the Hilbert space 
$[L^2(M; d\mu)]^N$ (where $[L^2(M; d\mu)]^N = L^2(M; d\mu; \bbC^N)$), we recall that 
$T \in \cB_2\big([L^2(M; d\mu)]^N\big)$ if and only if $T_{j,k} \in \cB_2\big(L^2(M; d\mu)\big)$, 
$1 \leq j,k \leq N$. Moreover, we recall that (cf.\ e.g., \cite[Theorem~11.3.6]{BS87}) 
\begin{align}
& \|T\|^2_{\cB_2(L^2(M; d\mu)^N)} = \int_{M \times M} d\mu(x) \, d\mu(y) \, \|T(x,y)\|^2_{\cB_2(\bbC^N)} 
\no \\
& \quad = \int_{M \times M} d\mu(x) \, d\mu(y) \, \sum_{j,k=1}^N |T_{j,k}(x,y)|^2    \no \\
& \quad =  \sum_{j,k=1}^N\int_{M \times M} d\mu(x) \, d\mu(y) \, |T_{j,k}(x,y)|^2    \no \\
& \quad = \sum_{j,k=1}^N \|T_{j,k}\|^2_{\cB_2(L^2(M; d\mu))},      \lb{A.2}
\end{align}
where, in obvious notation, $T(\, \cdot \,, \, \cdot \,)$ denotes the $N \times N$ matrix-valued integral kernel of $T$ in $[L^2(M; d\mu)]^N$, and $T_{j,k}(\, \cdot \,, \, \cdot \,)$ represents the integral kernel of $T_{j,k}$ in $L^2(M; d\mu)$, , $1 \leq j,k \leq N$.

In addition, employing the fact that for any $N \times N$ matrix $D \in \bbC^{N \times N}$, 
\begin{equation}
\|D\|_{\cB(\bbC^N)} \leq \|D\|_{\cB_2(\bbC^N)} \leq N^{1/2} \|D\|_{\cB(\bbC^N)},     \lb{A.3} 
\end{equation}
one also obtains 
\begin{equation}
\|T\|^2_{\cB_2(L^2(M; d\mu)^N)} \leq N \int_{M \times M} d\mu(x) \, d\mu(y) \, 
\|T(x,y)\|^2_{\cB(\bbC^N)}.     \lb{A.4}
\end{equation}

More generally, for $\cH$ a complex separable Hilbert space and $T = \{T_{j,k}\}_{1 \leq j,k \leq N}$, 
$N \in \bbN$, a block operator matrix in $\cH^N$, one confirms that 
\begin{align}
& \text{$T \in \cB\big(\cH^N)$ \big(resp., $T \in \cB_p\big(\cH^N\big)$, $p \in [1,\infty)\cup \{\infty\}$\big)} 
\no \\
& \quad \text{if and only if}       \lb{A.5} \\
& \text{for each $1 \leq j,k \leq N$, $T_{j,k} \in \cB\big(\cH^N)$ \big(resp., $T_{j,k} \in \cB_p\big(\cH^N\big)$, $p \in [1,\infty)\cup \{\infty\}$\big)}.     \no  
\end{align}
In other words, for membership of $T$ in $\cB\big(\cH^N)$ or $\cB_p\big(\cH^N\big)$, 
$p \in [1,\infty)\cup \{\infty\}$, it suffices to focus on each of its matrix elements 
$T_{j,k}$, $1 \leq j,k \leq N$. (For necessity of the last line in \eqref{A.5} it suffices to multiply 
$T$ from the left and right by $N \times N$ diagonal matrices with $I_{\cH}$ on the $j$th and $k$th position, resp., to isolate $T_{j,k}$ and appeal to the ideal property. For sufficiency, it suffices to 
write $T$ as a sum of $N^2$ terms with $T_{j,k}$ at the $j,k$th position and zeros otherwise.)

The next result is useful in connection with Sections \ref{s5} and \ref{s6}. 

\begin{theorem} \lb{tA.2}
Let $N \in \bbN$ and suppose that $T_1, T_2$ are linear $N \times N$ block operator matrices defined on $[L^2(M; d\mu)]^N$, such that for each $1 \leq j,k \leq N$, $T_{2,j,k}$ pointwise dominates $T_{1,j,k}$. Then the following items $(i)$--$(iii)$ hold: \\[1mm]
$(i)$ \hspace*{1.4mm} If $T_2 \in \cB\big([L^2(M; d\mu)]^N\big)$ then  $T_1 \in \cB\big([L^2(M; d\mu)]^N\big)$ and 
\begin{equation}
\|T_1\|_{\cB([L^2(M; d\mu)]^N)} \leq \|T_2\|_{\cB([L^2(M; d\mu)]^N)}.     \lb{A.6}
\end{equation}
$(ii)$ \hspace*{.1mm} If $T_2 \in \cB_{\infty}\big([L^2(M; d\mu)]^N\big)$ then  $T_1 \in \cB_{\infty}\big([L^2(M; d\mu)]^N\big)$ and 
\begin{equation}
\|T_1\|_{\cB([L^2(M; d\mu)]^N)} \leq \|T_2\|_{\cB([L^2(M; d\mu)]^N)}.     \lb{A.7}
\end{equation}
$(iii)$ If $T_2 \in \cB_2\big([L^2(M; d\mu)]^N\big)$ then  $T_1 \in \cB_2\big([L^2(M; d\mu)]^N\big)$ and 
\begin{equation}
\|T_1\|_{\cB_2([L^2(M; d\mu)]^N)} \leq \|T_2\|_{\cB_2([L^2(M; d\mu)]^N)}.     \lb{A.8} 
\end{equation}
\end{theorem} 
\begin{proof}
For item $(ii)$ we refer to \cite{DF79} and \cite{Pi79} (see also \cite{Le82}) combined with \eqref{A.5} as we will not use it in this manuscript. While the proofs of items $(i)$ and $(iii)$ are obviously well-known, we briefly recall them here as we will be using these facts in Sections \ref{s5} and \ref{s6}. Starting with item $(i)$, we introduce the notation $f = (f_1,\dots,f_N) \in [L^2(M; d\mu)]^N$ and 
$|f| = (|f_1|,\dots,|f_N|) \in [L^2(M; d\mu)]^N$ and compute, 
\begin{align}
& \|T_1 f\|^2_{[L^2(M; d\mu)]^N} = \sum_{j=1}^N \|(T_1 f)_j\|^2_{L^2(M; d\mu)} 
= \sum_{j=1}^N ((T_1 f)_j, (T_1 f)_j)_{L^2(M; d\mu)}     \no \\
& \quad = \sum_{j=1}^N \bigg|\sum_{k,\ell=1}^N (T_{1,j,k} f_k, T_{1,j,\ell} f_{\ell})_{L^2(M; d\mu)}\bigg|
\no \\
& \quad \leq \sum_{j=1}^N \sum_{k,\ell=1}^N |(T_{1,j,k} f_k, T_{1,j,\ell} f_{\ell})_{L^2(M; d\mu)}|
\no \\
& \quad \leq \sum_{j=1}^N \sum_{k,\ell=1}^N (|T_{1,j,k} f_k|, |T_{1,j,\ell} f_{\ell}|)_{L^2(M; d\mu)} 
\no \\
& \quad \leq \sum_{j=1}^N \sum_{k,\ell=1}^N (T_{2,j,k} |f_k|, T_{2,j,\ell} |f_{\ell}|)_{L^2(M; d\mu)} 
\no \\
& \quad = \sum_{j=1}^N ((T_2 |f|)_j, (T_2 |f|)_j)_{L^2(M; d\mu)} = \|T_2 |f|\|^2_{[L^2(M; d\mu)]^N}     
\no \\
& \quad \leq \|T_2\|^2_{\cB(L^2(M; d\mu)^N)}  \| |f| \|^2_{[L^2(M; d\mu)]^N}    \no \\
& \quad = \|T_2\|^2_{\cB(L^2(M; d\mu)^N)}  \|f\|^2_{[L^2(M; d\mu)]^N},    \lb{A.9} 
\end{align}
implying item $(i)$. For item $(iii)$ we recall from \cite[Theorem~2.13]{Si05} that 
$T_{1,j,k} \in \cB_2\big(L^2(M; d\mu)\big)$, $1 \leq j,k \leq N$, and 
$\|T_{1,j,k}\|_{\cB_2(L^2(M; d\mu))} \leq \|T_{2,j,k}\|_{\cB_2(L^2(M; d\mu))}$, $1\leq j,k \leq N$, and hence
by \eqref{A.2},
\begin{align}
\begin{split} 
\|T_1\|^2_{\cB_2([L^2(M; d\mu)]^N)} &= \sum_{j,k=1}^N \|T_{1,j,k}\|^2_{\cB_2(L^2(M; d\mu))} 
\leq \sum_{j,k=1}^N \|T_{2,j,k}\|^2_{\cB_2(L^2(M; d\mu))}    \lb{A.10} \\
& = \|T_2\|^2_{\cB_2([L^2(M; d\mu)]^N)}.
\end{split} 
\end{align} 
\end{proof}

\begin{remark} \lb{rA.3}
We complete this appendix with the observation that the subordination assumption 
$|(A f)(\, \cdot \,)| \leq (B |f|)(\, \cdot \,)$ $\mu$-a.e.~on $M$, if $A$ and $B$ are integral operators in 
$L^2(M; d\mu)$ with integral kernels $A(\, \cdot \,, \, \cdot \,)$ and $B(\, \cdot \,, \, \cdot \,)$, respectively, 
is implied by the condition
$|A(\, \cdot \,, \, \cdot \,)| \leq B(\, \cdot \,, \, \cdot \,)$ $\mu \otimes \mu$-a.e.~on $M \times M$ since
\begin{align}
|(Af)(x)| & = \bigg|\int_M d\mu(y) \, A(x,y)f(y))\bigg| \leq \int_{M} d\mu(y) \, |A(x,y)| |f(y)|  \no \\
& \leq \int_{M} d\mu(y) \, B(x,y) |f(y)| = (B |f|)(x)| \, \text{ for a.e.~$x \in M$.}     \lb{A.13} 
\end{align}
In fact, the converse is true as well as shown next in a concrete situation. 
\hfill $\diamond$
\end{remark}

\begin{lemma} \lb{lA.3}
Let $\Omega \subseteq \bbR^n$ open, $n \in \bbN$, suppose $(\Omega, \Sigma, d\sigma)$ represents the standard Lebesgue measure on $\Omega$ $($i.e., $d\sigma = d^n x$$)$, and denote the Lebesgue measure of a set 
$S \in \Sigma$ by $|S|$. Consider bounded, linear integral operators $A, B \in \cB\big(L^2(\Omega; d^n x)\big)$, with integral kernels $A(\, \cdot \,, \, \cdot \,)$ and $B(\, \cdot \,, \, \cdot \,)$, respectively. Then
\begin{align} 
& \text{$B$ pointwise dominates $A$, that is,}    \no \\
& \quad |(A f)(\, \cdot \,)| \leq (B |f|)(\, \cdot \,) \, \text{ $\sigma$-a.e.~on $\Omega$ for all 
$f \in L^2(\Omega; d^n x)$,}   
\lb{A.11} \\
& \text{if and only if}     \no \\ 
& \quad |A(\, \cdot \,, \, \cdot \,)| \leq B(\, \cdot \,, \, \cdot \,) \, \text{ $\sigma \otimes \sigma$-a.e.~on 
$\Omega \times \Omega$.}      
\lb{A.12} 
\end{align}
\end{lemma}
\begin{proof} 
That \eqref{A.12} implies \eqref{A.11} has just been shown in \eqref{A.13}. 

To prove the converse, suppose \eqref{A.11} holds. Since the operators $A$ and $B$ are bounded operators on $L^2(\Omega; d^n x)$, it follows from Tonelli's theorem that the integral kernels $A(\, \cdot\,,\,\cdot\,)$ and $B(\, \cdot\,,\,\cdot\,)$ are locally integrable functions on $\Omega \times \Omega$ (with respect to $\sigma \otimes \sigma$). 

Let $x,y \in \Omega$ be arbitrary and let $\alpha\in \bbQ$. For $\varepsilon>0$ consider the open ball ball $B_\varepsilon(x) \subset \Omega$ of radius $\varepsilon$ with the center at $x \in \Omega$. The fact 
$f_{\varepsilon,x} = \chi_{B_\varepsilon(x)}(\cdot)\in L^2(\Omega; d^n x)$ and assumption \eqref{A.11} imply that
\begin{equation}
\Big|\int_{B_\varepsilon(x)}d^n x' \,  A(x',y)\Big|=\big|(Af_{\varepsilon,x})(y)\big| 
\leq (B|f_{\varepsilon,x}|)(y)\leq \int_{B_\varepsilon(x)}d^n x' \, B(x',y), 
\end{equation}
for $\sigma$-a.e.~$y \in \Omega$. Since $\Re \big(e^{i\alpha}z\big)\leq |z|$ for all $z\in\bbC$, it follows that
\begin{equation}
\int_{B_\varepsilon(x)}d^n x' \, \Re\big(e^{i\alpha} A(x',y)\big)\leq \int_{B_\varepsilon(x)}d^n x' \,  B(x',y),   \lb{A.15} 
\end{equation}
for $\sigma$-a.e.~$y \in \Omega$. Integrating inequality \eqref{A.15} over the ball $B_\varepsilon(y)$ implies
\begin{equation}  
\int_{B_\varepsilon(y)}d^n y' \int_{B_\varepsilon(x)}d^n x' \, 
\Re \big(e^{i\alpha} A(x',y')\big)\leq \int_{B_\varepsilon(y)}d^n y' \int_{B_\varepsilon(x)}d^n x' \, B(x',y'). 
\end{equation}
Since $A(\, \cdot\,,\,\cdot\,)$ and $B(\, \cdot\,,\,\cdot\,)$ are locally integrable functions, an application of Fubini's theorem yields
\begin{align}\lb{A.14}
\begin{split} 
\frac1{|B_\varepsilon(\cdot)|^2}\int_{B_\varepsilon(x)\times B_\varepsilon(y)}d(\sigma \otimes \sigma)(x',y')\, 
\Re \big(e^{i\alpha} A(x',y')\big)     \\
\leq \frac1{|B_\varepsilon(\cdot)|^2}\int_{B_\varepsilon(x)\times B_\varepsilon(y)} 
d(\sigma \otimes \sigma)(x',y') \, B(x',y').
\end{split} 
\end{align} 
Moreover, since both $A(\, \cdot\,,\,\cdot\,)$ and $B(\, \cdot\,,\,\cdot\,)$ are locally integrable functions, it follows that almost every point $(x,y) \in \Omega \times \Omega$ is a Lebesgue point for $B(\, \cdot\,,\,\cdot\,)$ and for $\Re \big(e^{i\alpha} A(\, \cdot\,,\,\cdot\,)\big)$. Hence, letting $\varepsilon\to 0$ in inequality \eqref{A.14}, one infers that for any $\alpha\in\bbQ$ 
\begin{equation}  
\Re \big(e^{i\alpha} A(x,y)\big)\leq B(x,y),
\end{equation}
for $\sigma \otimes \sigma$-a.e.~$(x,y) \in \Omega \times \Omega$.
Taking the supremum over all $\alpha\in\mathbb{Q},$ one obtains 
\begin{equation}  
|A(x,y)|\leq B(x,y),
\end{equation}
for $\sigma \otimes \sigma$-a.e.~$(x,y) \in \Omega \times \Omega$, proving \eqref{A.12}.
\end{proof}

\section{Asymptotic Results for Hankel Functions} \lb{sB}
\renewcommand{\theequation}{B.\arabic{equation}}
\renewcommand{\thetheorem}{B.\arabic{theorem}}
\setcounter{theorem}{0} \setcounter{equation}{0}

In this appendix we collect asymptotic results for Hankel functions in the regions of large and small arguments.  To set the stage, we recall some details on the analytic behavior of $H_{\nu}^{(1)}(\, \cdot \,)$ 
(cf.\ \cite[p.~358--360]{AS72}):
\begin{align}
& H_{\nu}^{(1)} (\zeta) = J_{\nu} (\zeta) + i Y_{\nu} (\zeta), \quad \nu \in \bbC, \; \zeta \in \bbC \backslash \{0\},  
\lb{5.8} \\
& J_{\nu} (\zeta) = (\zeta/2)^{\nu} \sum_{k \in \bbN_0} [k! \Gamma(\nu + k + 1)]^{-1} (-1)^k (\zeta/2)^{2k}, 
\lb{5.8a} \\  
& Y_{\nu} (\zeta) = [\sin(\nu \pi)]^{-1} [J_{\nu}(\zeta)  \cos(\nu \pi) - J_{- \nu} (\zeta)],   \lb{5.8b} \\
& Y_n (\zeta) = - \pi^{-1} (\zeta/2)^{- n} \sum_{k=0}^{n-1} [k!]^{-1} [(n-k - 1)!] (\zeta/2\big)^{2k} 
+ 2 \pi^{-1} J_{n}(\zeta) \ln(\zeta/2)     \no \\
& \hspace*{1.2cm} - \pi^{-1} (\zeta/2)^n \sum_{k \in \bbN_0} [\psi(k+1) + \psi(n + k + 1)] [k! (n+k)!]^{-1} 
(-1)^k (\zeta/2)^{2k},     \no \\ 
& \hspace*{9.8cm} n \in \bbN,    \lb{5.8c} \\
&Y_0(\zeta)=\frac{2}{\pi}\{\ln(\zeta/2)+\gamma_{E-M}\}J_0(\zeta)-\frac{2}{\pi}\sum_{k=1}^{\infty}\bigg(\sum_{\ell=1}^k\frac{1}{\ell}\bigg)\frac{(-4)^{-k}\zeta^{2k}}{(k!)^2},\lb{5.8dd}\\
\begin{split} 
& J_{-n} (\zeta) = (-1)^n J_{n} (\zeta), \quad Y_{-n} (\zeta) = (-1)^n Y_{n} (\zeta), \quad n \in \bbN,   \\
& H_{- \nu}^{(1)} (\zeta) = e^{i \nu \pi} H_{\nu}^{(1)} (\zeta),     \lb{5.8d} 
\end{split} 
\end{align}
where (cf.\ \cite[p.~256]{AS72}) 
\begin{align}
\begin{split}
& \psi(\zeta) = \Gamma'(\zeta)/\Gamma(\zeta),      \\ 
& \psi(1) = - \gamma_{E-M}, \quad \psi(\ell) = - \gamma_{E-M} + \sum_{k=1}^{\ell-1} k^{-1},    \lb{5.8e} 
\end{split}
\end{align} 
and 
\begin{equation}
\gamma_{E-M} := \lim_{m \to \infty} \bigg( - \ln(m) + \sum_{k=1}^m k^{-1}\bigg) = 0.5772156649 \, ..... \lb{5.8f}
\end{equation}
denotes the Euler--Mascheroni constant (cf.\ \cite[p.~255]{AS72}). We also recall the asymptotic behavior 
(cf.\ \cite[p.~360]{AS72}, \cite[p.~723--724]{JN01})
\begin{align}
& H_0^{(1)} (\zeta) \underset{\substack{\zeta\to 0 \\ \zeta\in \bbC \backslash\{0\}}}{=} 
(2i/\pi) \ln(\zeta) + \Oh\big(|\ln(\zeta)| |\zeta|^2\big),  \lb{5.9} \\
& H_{\nu}^{(1)} (\zeta) \underset{\substack{\zeta\to 0 \\ \zeta\in \bbC \backslash\{0\}}}{=} 
- (i/\pi) 2^{\nu} \Gamma(\nu) \zeta^{- \nu} 
+ \begin{cases} \Oh\big(|\zeta|^{\min(\nu, - \nu + 2)}\big), & \nu \notin \bbN, \\
\Oh\big(|\ln(\zeta)| |\zeta|^{\nu}\big) + \Oh\big(\zeta^{ - \nu + 2}\big), & \nu \in \bbN, \end{cases}  \lb{5.10} \\
& \hspace*{9.17cm} \Re(\nu) > 0,  \no \\ 
& H_{\nu}^{(1)} (\zeta) \underset{\zeta \to \infty}{=} (2/\pi)^{1/2} \zeta^{-1/2} e^{i [\zeta - (\nu \pi/2) - (\pi/4)]}, 
\quad \nu \geq 0, \; \Im(\zeta) \geq 0.   \lb{5.11} 
\end{align}

\subsection{Asymptotics of $H_{\nu}^{(1)}(\zeta)$ as $|\zeta| \to \infty$}

\begin{hypothesis}\lb{hD.1}
Let $\nu\in \bbC$ with $\Re(\nu+(1/2))>0$.
\end{hypothesis}

Assuming Hypothesis \ref{hD.1}, the Hankel function $H_{\nu}^{(1)}(\,\cdot\,)$ permits the following representation (cf., e.g., \cite[8.421.9]{GR80}, \cite[6.12(3)]{Wa66})
\begin{equation} \lb{D.1}
H_{\nu}^{(1)}(\zeta) = \bigg(\frac{2}{\pi\zeta}\bigg)^{1/2} \frac{e^{i[\zeta - (\pi/2)\nu - (\pi/4)]}}{\Gamma(\nu+(1/2))} \int_0^{\infty}du\, e^{-u}u^{\nu-(1/2)}\bigg(1 + \frac{iu}{2\zeta} \bigg)^{\nu-(1/2)},
\end{equation}
where $\Re(\nu + (1/2))>0$ and $-\pi/2 < \arg(\zeta) < 3\pi/2$.  We will derive the asymptotic behavior of $H_{\nu}^{(1)}(\zeta)$ as $|\zeta|\to \infty$ closely following the presentation given in \cite[Sect.~7.2]{Wa66}.

The factor in parentheses in the integrand in \eqref{D.1} may be expanded for any $p\in \bbN$ according to
\begin{align}
\bigg(1 + \frac{iu}{2\zeta} \bigg)^{\nu-(1/2)} &= \sum_{m=0}^{p-1} \frac{((1/2)-\nu)_m}{m!}\bigg(\frac{u}{2i\zeta}\bigg)^m\no\\
&\quad + \frac{((1/2)-\nu)_p}{(p-1)!} \bigg(\frac{u}{2i\zeta}\bigg)^p \int_0^1dt\, (1-t)^{p-1}\bigg(1-\frac{ut}{2i\zeta}\bigg)^{\nu-p-(1/2)},\lb{D.2}
\end{align}
where we have employed the Pochhammer symbol,
\begin{equation}
(a)_n = \frac{\Gamma(a+n)}{\Gamma(a)},\quad n\in \bbN,\; a\in \bbC\backslash \{0,-1,-2,-3,\ldots\}.
\end{equation}
We shall assume for convenience that $p\in \bbN$ is chosen sufficiently large to guarantee that 
$\Re(\nu-p-(1/2))\leq 0$, and we will comment on how to remove this restriction later.  Next, fix 
an angle $\delta\in (0,\pi/2)$ which satisfies
\begin{equation}
 |\arg(\zeta)-(\pi/2)|\leq \pi - \delta.
\end{equation}
With $\delta$ so chosen, one infers
\begin{equation}
\bigg|1-\frac{ut}{2i\zeta}\bigg| \geq \sin(\delta),\quad \bigg|\arg\bigg(1-\frac{ut}{2i\zeta}\bigg)\bigg| < \pi,
\end{equation}
for all $t\in [0,1]$ and all $u\in (0,\infty)$.  In particular,
\begin{equation} \lb{D.6}
\bigg|\bigg(1-\frac{ut}{2i\zeta}\bigg)^{\nu-p-(1/2)}\bigg| \leq e^{\pi |\Im(\nu)|}[\sin(\delta)]^{\Re(\nu - p -(1/2))} =: C_{\nu,p,\delta},
\end{equation}
where $C_{\nu,p,\delta}$ is independent of $\zeta$.  Using the expansion \eqref{D.2} in \eqref{D.1}, one obtains
\begin{align}
H_{\nu}^{(1)}(\zeta) &= \bigg(\frac{2}{\pi\zeta}\bigg)^{1/2} \frac{e^{i[\zeta - (\pi/2)\nu - (\pi/4)]}}{\Gamma(\nu+(1/2))}   \no\\
& \quad \times \int_0^{\infty}du\, e^{-u}u^{\nu-(1/2)}\Bigg[\sum_{m=0}^{p-1} \frac{((1/2)-\nu)_m}{m!}\bigg(\frac{u}{2i\zeta}\bigg)^m\no\\
&\hspace*{1.9cm}+\frac{((1/2)-\nu)_p}{(p-1)!} \bigg(\frac{u}{2i\zeta}\bigg)^p \int_0^1dt\, (1-t)^{p-1}\bigg(1-\frac{ut}{2i\zeta}\bigg)^{\nu-p-(1/2)}\Bigg]\no\\
&= \bigg(\frac{2}{\pi\zeta}\bigg)^{1/2} \frac{e^{i[\zeta - (\pi/2)\nu - (\pi/4)]}}{\Gamma(\nu+(1/2))}\Bigg[\sum_{m=0}^{p-1} \frac{((1/2)-\nu)_m\Gamma(\nu+m+(1/2))}{m!(2i\zeta)^m}\no\\
&\quad + \frac{((1/2)-\nu)_p}{(p-1)!}\int_0^{\infty}du\, e^{-u}u^{\nu-(1/2)} \bigg(\frac{u}{2i\zeta}\bigg)^p    \no \\ 
& \hspace*{3.95cm} \times \int_0^1dt\, (1-t)^{p-1}\bigg(1-\frac{ut}{2i\zeta}\bigg)^{\nu-p-(1/2)}\Bigg]\no\\
&= \bigg(\frac{2}{\pi\zeta}\bigg)^{1/2} e^{i[\zeta - (\pi/2)\nu - (\pi/4)]}\Bigg[\sum_{m=0}^{p-1} \frac{((1/2)-\nu)_m(\nu+(1/2))_m}{m!(2i\zeta)^m} + R_{\nu,p}^{(1)}(\zeta)\Bigg],     \lb{D.7}
\end{align}
where
\begin{align}
R_{\nu,p}^{(1)}(\zeta)&:= \frac{((1/2)-\nu)_p}{(p-1)!\Gamma(\nu+(1/2))}      \lb{D.8}\\
&\quad\times\int_0^{\infty}du\, e^{-u}u^{\nu-(1/2)} \bigg(\frac{u}{2i\zeta}\bigg)^p \int_0^1dt\, (1-t)^{p-1}\bigg(1-\frac{ut}{2i\zeta}\bigg)^{\nu-p-(1/2)}.\no
\end{align}
One observes that
\begin{align}
|R_{\nu,p}^{(1)}(\zeta)| & \leq \frac{C_{\nu,p,\delta}}{(p-1)!}\bigg|\frac{((1/2)-\nu)_p}{\Gamma(\nu+(1/2))(2i\zeta)^p}\bigg|  \no\\
&\quad \times \bigg[ \int_0^1dt\, (1-t)^{p-1}\bigg]\bigg[\int_0^{\infty}e^{-u}|u^{\nu+p-(1/2)}|\, du\bigg]\no\\
& = \wti C_{\nu,p,\delta}|\zeta|^{-p}.\lb{D.9}
\end{align}
As a consequence of \eqref{D.7}, \eqref{D.8}, and \eqref{D.9}, one infers that for any fixed $\delta\in (0,\pi/2)$,
\begin{align}
& H_{\nu}^{(1)}(\zeta) \underset{|\zeta|\to\infty}{=} \bigg(\frac{2}{\pi\zeta}\bigg)^{1/2} e^{i[\zeta -(\pi/2)\nu -(\pi/4)]} 
\Bigg[\sum_{m=0}^{p-1} \frac{((1/2)-\nu)_m(\nu+(1/2))_m}{m!(2i\zeta)^m}      \lb{D.10} \\
& \hspace*{2.25cm} + O(|\zeta|^{-p})\Bigg],   \quad  |\arg(\zeta)-(\pi/2)|\leq \pi - \delta.    \no
\end{align}

To obtain similar expansions when $\Re(\nu-p-(1/2))> 0$, one chooses $q\in \bbN$ so large that $\Re(\nu-q-(1/2)) < 0$, which requires $p<q$.  Then \eqref{D.7}, \eqref{D.8}, \eqref{D.9}, and \eqref{D.10} hold with $p$ replaced by $q$.  In particular, by \eqref{D.7}, for any fixed $\delta\in (0,\pi/2)$,
\begin{align}
H_{\nu}^{(1)}(\zeta) &= \bigg(\frac{2}{\pi\zeta}\bigg)^{1/2} e^{i[\zeta - (\pi/2)\nu - (\pi/4)]}\Bigg[\sum_{m=0}^{q-1} \frac{((1/2)-\nu)_m(\nu+(1/2))_m}{m!(2i\zeta)^m} + R_{\nu,q}^{(1)}(\zeta)\Bigg]\no\\
&= \bigg(\frac{2}{\pi\zeta}\bigg)^{1/2} e^{i[\zeta - (\pi/2)\nu - (\pi/4)]}\Bigg[\sum_{m=0}^{p-1} \frac{((1/2)-\nu)_m(\nu+(1/2))_m}{m!(2i\zeta)^m} + \wti R_{\nu,q}^{(1)}(\zeta)\Bigg],     \no \\
&\hspace*{6.1cm}  |\arg(\zeta)-(\pi/2)|\leq \pi - \delta,\lb{D.11}
\end{align}
where
\begin{equation} \lb{D.12}
\wti R_{\nu,p,q}^{(1)}(\zeta) = \sum_{m=p}^{q-1} \frac{((1/2)-\nu)_m(\nu+(1/2))_m}{m!(2i\zeta)^m} + R_{\nu,q}^{(1)}(\zeta).
\end{equation}

The following lemma provides sufficient conditions for the differentiability of an integral depending on a complex parameter.

\begin{lemma}[\cite{Ma01}]\lb{lD.1}
Let $(X,\cM,\mu)$ be a measure space, let $G\subset \bbC$ be an open set, and let $f:G\times X \to \bbC$ be a function which satisfies the following conditions:\\[1mm]
$(i)$ $f(\zeta,\,\cdot\,)$ is $\cM$-measurable for every $\zeta\in G$,\\[1mm]
$(ii)$ $f(\,\cdot\,,x)$ is holomorphic in $G$ for every $x\in X$, and\\[1mm]
$(iii)$ $\int_X d\mu\, |f(\,\cdot\,,x)|$ is locally bounded; that is, for every $\zeta_0\in G$, there exists $\varepsilon(\zeta_0)>0$ such that
\begin{equation}
\sup_{\substack{\zeta\in G \\ |\zeta-\zeta_0|\leq \varepsilon(\zeta_0)}} \int_X d\mu\, |f(\zeta,x)| < \infty.
\end{equation}
Then $\int_X d\mu\, f(\,\cdot\,,x)$ is holomorphic in $G$ and
\begin{equation}
\frac{d^n}{d\zeta^n} \int_X d\mu\, f(\zeta,x) = \int_X d\mu\, \frac{\partial^n}{\partial \zeta^n} f(\zeta,x)\quad \text{in $G$ for every $n\in \bbN$.}
\end{equation}
\end{lemma}

\begin{proposition} \lb{pD.2}
Assume Hypothesis \ref{hD.1}.  Let $p\in \bbN$, $u\in (0,\infty)$, and suppose $\Re(\nu - p -(1/2))\leq0$.  If
\begin{equation} \lb{D.15a}
\Omega_0 := \{\zeta\in \bbC\,|\, |\arg(\zeta) - (\pi/2)| < \pi\},
\end{equation}
then the function $a_{u,p,\nu}:\Omega_0\to \bbC$ defined by
\begin{equation} \lb{D.15}
a_{u,p,\nu}(\zeta) = \int_0^1dt\, (1-t)^p \bigg(1 - \frac{ut}{2i\zeta}\bigg)^{\nu-p-(1/2)},\quad \zeta\in \Omega_0,
\end{equation}
is analytic in $\Omega_0$ and
\begin{equation} \lb{D.16a}
\frac{d^n}{d\zeta^n} a_{u,p,\nu}(\zeta) =  \int_0^1dt\, (1-t)^p \frac{\partial^n}{\partial^n\zeta} \bigg(1 - \frac{ut}{2i\zeta}\bigg)^{\nu-p-(1/2)},\quad \zeta\in \Omega_0.
\end{equation}
\end{proposition}
\begin{proof}
Let $p\in \bbN$, $u\in (0,\infty)$, $\nu \in \bbC$ with $\Re(\nu - p -(1/2))\leq0$.  It suffices to apply Lemma \ref{lD.1} to the function
\begin{equation} \lb{D.16}
f_{u,p,\nu}(\zeta,t) = (1-t)^p \bigg(1 - \frac{ut}{2i\zeta}\bigg)^{\nu-p-(1/2)},\quad \zeta\in \Omega_0,\; t\in (0,1).
\end{equation}
Of course, \eqref{D.16} defines a function which is Lebesgue measurable for each $\zeta\in \Omega_0$ and analytic in $\Omega_0$ for every $t\in (0,1)$.  Therefore, it remains to verify condition $(iii)$ in Lemma \ref{lD.1}.  To this end, let $\zeta_0\in \Omega_0$.  Choose $\delta \in (0,\pi/2)$ such that
\begin{equation}
\zeta_0\in \Omega_{\delta}:=\{\zeta\in \bbC\,|\, |\arg(\zeta)- (\pi/2)| < \pi - \delta\}.
\end{equation}
By \eqref{D.6}, one then infers
\begin{equation} \lb{D.19}
|a_{u,p,\nu}(\zeta)|=\bigg|\int_0^1dt\, f_{u,p,\nu}(\zeta,t)\bigg| \leq \int_0^1 dt\, |f_{u,p,\nu}(\zeta,t)| \leq \frac{C_{\nu,p,\delta}}{p},\quad \zeta\in \Omega_{\delta}.
\end{equation}
In particular, choosing $\varepsilon(\zeta_0)\in (0,1)$ so small that
\begin{equation}
\{\zeta\in \bbC\,|\, |\zeta-\zeta_0|<\varepsilon(\zeta_0)\}\subset \Omega_{\delta},
\end{equation}
one concludes
\begin{equation}
\sup_{\substack{\zeta\in \Omega_0 \\ |\zeta-\zeta_0|<\varepsilon(\zeta_0)}}\bigg|\int_0^1dt\, f_{u,p,\nu}(\zeta,t)\bigg| < \infty.
\end{equation}
Therefore, condition $(iii)$ in Lemma \ref{lD.1} holds, and it follows that $a_{u,p,\nu}$ is analytic in $\Omega_0$.
\end{proof}

\begin{proposition} \lb{pD.4}
Assume Hypothesis \ref{hD.1}, let $p\in \bbN$, and let $\Omega_0$ be defined as in \eqref{D.15a}.  The following statements hold:\\
$(i)$ If $\Re(\nu - p -(1/2))\leq0$, then the function $R_{\nu,p}^{(1)}:\Omega_0\to \bbC$ defined by \eqref{D.8} is analytic in $\Omega_0$.\\
$(ii)$  If $\Re(\nu - p -(1/2))>0$, then the function $\wti R_{\nu,p,q}^{(1)}:\Omega_0\to \bbC$ defined by \eqref{D.12} is analytic in $\Omega_0$ for every $q\in \bbN$ such that $\Re(\nu-q-(1/2)) < 0$.
\end{proposition}
\begin{proof}
Let $p\in \bbN$ and suppose $\Re(\nu - p -(1/2))\leq0$.  We begin with the proof of $(i)$.  It suffices to show that the function $b_{p,\nu}:\Omega_0\to \bbC$ defined by (cf.~\eqref{D.15})
\begin{align}
b_{p,\nu}(\zeta)&=\int_0^{\infty}du\, e^{-u}u^{p+\nu-(1/2)} \int_0^1dt\, (1-t)^{p-1}\bigg(1-\frac{ut}{2i\zeta}\bigg)^{\nu-p-(1/2)}\no\\
&=\int_0^{\infty}du\, e^{-u}u^{p+\nu-(1/2)} a_{u,p,\nu}(\zeta),\quad \zeta\in \Omega_0,\lb{D.22}
\end{align}
is analytic in $\Omega_0$.  The function $e^{-u}u^{p+\nu-(1/2)} a_{u,p,\nu}(\zeta)$ is a measurable function of $u\in (0,\infty)$ for each $\zeta\in \Omega_0$ and is, by Proposition \ref{pD.2}, an analytic function of $\zeta\in \Omega_0$ for each $u\in (0,\infty)$.  Therefore, by Lemma \ref{lD.1}, it suffices to prove that for each $\zeta_0\in \Omega_0$, there exists $\varepsilon(\zeta_0)\in (0,\infty)$ such that
\begin{equation} \lb{D.23}
\sup_{\substack{\zeta\in \Omega_0 \\ |\zeta-\zeta_0|<\varepsilon(\zeta_0)}}\bigg| \int_0^{\infty}du\, e^{-u}u^{p+\nu-(1/2)} a_{u,p,\nu}(\zeta)\bigg|<\infty.
\end{equation}
To this end, let $\zeta_0\in \Omega_0$.  Choose $\delta \in (0,\pi/2)$ such that
\begin{equation}
\zeta_0\in \Omega_{\delta}:=\{\zeta\in \bbC\,|\, |\arg(\zeta)- (\pi/2)| < \pi - \delta\}.
\end{equation}
An application of \eqref{D.19} yields the following estimate:
\begin{align}
\bigg| \int_0^{\infty}du\, e^{-u}u^{p+\nu-(1/2)} a_{u,p,\nu}(\zeta)\bigg| & \leq \frac{C_{\nu,p,\delta}}{p}\int_0^{\infty}du\, e^{-u} u^{\Re(p+\nu-(1/2))}\\
&= \frac{C_{\nu,p,\delta}}{p}\Gamma\bigg(\Re\bigg(p+\nu+(1/2)\bigg)\bigg),\quad \zeta\in \Omega_{\delta}.
\end{align}
Thus, one obtains \eqref{D.23} by choosing $\varepsilon(\zeta_0)\in (0,1)$ so small that
\begin{equation}
\{\zeta\in \bbC\,|\, |\zeta-\zeta_0|<\varepsilon(\zeta_0)\}\subset \Omega_{\delta}.
\end{equation}
Finally, to prove item $(ii)$, suppose that $\Re(\nu - p -(1/2))>0$ and $q\in \bbN$ with $\Re(\nu-q-(1/2)) < 0$.  The first term on the right-hand side in \eqref{D.12} is analytic in $\bbC\backslash\{0\}$, while the second term on the right-hand side in \eqref{D.12} is analytic in $\Omega _0$ by the statement in $(i)$. Hence, the statement in $(ii)$ follows from the subspace property of analytic functions.
\end{proof}

\begin{remark}
Of course, analyticity of $R_{\nu,p}^{(1)}$ (resp., $\wti R_{\nu,p,q}^{(1)}$) follows immediately from \eqref{D.7} (resp., \eqref{D.11}).  However, the proof of Proposition \ref{pD.4} shows that the $\zeta$-derivatives of $R_{\nu,p}^{(1)}$ may be computed by differentiating under the integrals in \eqref{D.8}.  In fact, as a consequence of \eqref{D.16a} and the proof of Proposition \ref{pD.4}, one infers that under the assumptions of Proposition \ref{pD.4},
\begin{align}
& \frac{\partial^n}{\partial^n\zeta} b_{p,\nu}(\zeta)=\int_0^{\infty}du\, e^{-u}u^{p+\nu-(1/2)} \int_0^1dt\, (1-t)^{p-1}\frac{\partial^n}{\partial^n\zeta}\bigg(1-\frac{ut}{2i\zeta}\bigg)^{\nu-p-(1/2)},    \no \\
& \hspace{8.5cm} \zeta\in \Omega_0,\; n\in \bbN.     \lb{D.30}
\end{align}
${}$ \hfill $\diamond$
\end{remark}

In order to state the next result, we introduce $\wti O$-notation.  Recall that if $\Omega\subseteq \bbC$ and $f,g:\Omega\to \bbC$, then one writes
\begin{equation}
f(\zeta)=O(g(\zeta)), \quad \zeta\in \Omega,
\end{equation}
if and only if there exists a constant $C\in (0,\infty)$ (independent of $\zeta \in \Omega$) such that
\begin{equation}
|f(\zeta)|\leq C|g(\zeta)|,\quad \zeta\in \Omega.
\end{equation}
One writes
\begin{equation}
f(\zeta)= \wti O(g(\zeta)), \quad \zeta\in \Omega,
\end{equation}
if and only if for each $n\in \bbN_0$,
\begin{equation} \lb{D.34a}
\frac{d^nf}{d\zeta^n} = O\bigg(\frac{d^ng}{d\zeta^n}\bigg),\quad \zeta\in \Omega.
\end{equation}
It is understood that the constant corresponding to \eqref{D.34a} will, in general, depend on $n\in \bbN_0$.

The principal asymptotic result for $H_{\nu}^{(1)}(\zeta) $ as $|\zeta| \to \infty$ can be summarized as follows: 

\begin{lemma}\lb{lB.6}
Assume Hypothesis \ref{hD.1} holds.  If $\delta\in (0,\pi/2)$, then
\begin{equation}
H_{\nu}^{(1)}(\zeta) = e^{i\zeta}\omega_{\nu}(\zeta),\quad \zeta\in \Omega_{\delta},
\end{equation}
where
\begin{equation} \lb{D.34}
\omega_{\nu}(\zeta) \underset{|\zeta|\to\infty}{=} \wti O\big((1+|\zeta|)^{-1/2}\big), \quad \zeta\in \Omega_{\delta}\cap\{z\in\bbC\,|\,|z|\geq 1\}.
\end{equation}
\end{lemma}
\begin{proof}
Assume Hypothesis \ref{hD.1} holds.  We distinguish two cases: $\Re(\nu-(3/2))\leq 0$ and $\Re(\nu-(3/2))>0$.  If $\Re(\nu-(3/2))\leq 0$, then one may take $p=1$ in \eqref{D.7} to obtain
\begin{align}
H_{\nu}^{(1)}(\zeta) &= \bigg(\frac{2}{\pi\zeta}\bigg)^{1/2} e^{i[\zeta - (\pi/2)\nu - (\pi/4)]}\big[1 + R_{\nu,1}^{(1)}(\zeta)\big]\no\\
&= e^{i\zeta}\omega_{\nu}(\zeta),\quad \zeta\in \Omega_{\delta},
\end{align}
where
\begin{equation} \lb{D.36}
\omega_{\nu}(\zeta) = \bigg(\frac{2}{\pi\zeta}\bigg)^{1/2}e^{i[-(\pi/2)\nu -(\pi/4)]}\big[1 + R_{\nu,1}^{(1)}(\zeta)\big],\quad \zeta\in \Omega_{\delta}.
\end{equation}
It remains to prove $\omega_{\nu}(\,\cdot\,)$ defined by \eqref{D.36} satisfies \eqref{D.34}.  To prove this, it suffices to show
\begin{equation} \lb{D.37}
\zeta^{-1/2}\big[1 + R_{\nu,1}^{(1)}(\zeta)\big] \underset{|\zeta|\to\infty}{=} \wti O\big((1+|\zeta|)^{-1/2}\big),\quad \zeta\in \Omega_{\delta}\cap\{z\in\bbC\,|\,|z|\geq 1\};
\end{equation}
that is,
\begin{align} \lb{D.38}
\frac{d^n}{d\zeta^n}\zeta^{-1/2}\big[1 + R_{\nu,1}^{(1)}(\zeta)\big] \underset{|\zeta|\to\infty}{=} O\bigg(\frac{d^n}{d\zeta^n}(1+|\zeta|)^{-1/2}\bigg),&\\
\zeta\in \Omega_{\delta}\cap\{z\in\bbC\,|\,|z|\geq 1\},\; n\in \bbN_0.&\no
\end{align}
For $n=0$, the relation in \eqref{D.38} follows immediately from \eqref{D.9}.  To treat the derivatives in \eqref{D.38}, one differentiates under the integrals in \eqref{D.8}.  For simplicity, we only treat the case $n=1$ and omit the details for $n\geq 2$.  One computes
\begin{align}
&\frac{d}{d\zeta}\zeta^{-1/2}\big[1+ R_{\nu,1}^{(1)}(\zeta)\big] = -\frac{1}{2\zeta^{3/2}}\big[1+ R_{\nu,1}^{(1)}(\zeta)\big]  + \zeta^{-1/2}\frac{d}{d\zeta} R_{\nu,1}^{(1)}(\zeta)\no\\
&\quad \underset{|\zeta|\to\infty}{=} O\big((1+|\zeta|)^{-3/2}\big)\no\\
& \;\;\quad \qquad-\zeta^{-5/2}\frac{((1/2)-\nu)_1}{2i\Gamma(\nu+(1/2))} \int_0^{\infty}du\, e^{-u}u^{\nu+(1/2)}  \int_0^1dt\, \bigg(1-\frac{ut}{2i\zeta}\bigg)^{\nu-(3/2)}\no\\
& \;\;\quad \qquad -\zeta^{-7/2}\frac{((1/2)-\nu)_1(\nu-(3/2))}{4\Gamma(\nu+(1/2))}   
\int_0^{\infty}du\, e^{-u}u^{\nu+(3/2)}   \no \\
& \hspace*{6.2cm} \times \int_0^1dt\,t\bigg(1-\frac{ut}{2i\zeta}\bigg)^{\nu-(5/2)}\no\\
&\quad \underset{|\zeta|\to\infty}{=} O\big((1+|\zeta|)^{-3/2}\big),\quad \zeta\in \Omega_{\delta}\cap\{z\in\bbC\,|\,|z|\geq 1\}.\lb{D.39}
\end{align}
To obtain the final equality in \eqref{D.39}, one applies \eqref{D.6} to bound the two ($\zeta$-dependent) integrals with respect to $t\in (0,1)$.  This settles the case when $\Re(\nu-(3/2))\leq 0$.

If $\Re(\nu-(3/2))>0$, one chooses $q\in \bbN$ such that $\Re(\nu-q-(1/2))<0$.  Then
\begin{align}
H_{\nu}^{(1)}(\zeta) &= \bigg(\frac{2}{\pi\zeta}\bigg)^{1/2} e^{i[\zeta - (\pi/2)\nu - (\pi/4)]}\big[1 + \wti R_{\nu,1,q}^{(1)}(\zeta)\big]\no\\
&= e^{i\zeta}\omega_{\nu}(\zeta),\quad \zeta\in \Omega_{\delta},
\end{align}
where
\begin{equation}
\omega_{\nu}(\zeta) = \bigg(\frac{2}{\pi\zeta}\bigg)^{1/2}e^{i[-(\pi/2)\nu -(\pi/4)]}\big[1 + \wti R_{\nu,1,q}^{(1)}(\zeta)\big],\quad \zeta\in \Omega_{\delta}.
\end{equation}
Then, as a consequence of \eqref{D.12} and \eqref{D.37}, one obtains
\begin{equation}
\zeta^{-1/2}\big[1 + \wti R_{\nu,1,q}^{(1)}(\zeta)\big] \underset{|\zeta|\to\infty}{=} 
\wti O\big((1+|\zeta|)^{-1/2}\big),\quad \zeta\in \Omega_{\delta}\cap\{z\in\bbC\,|\,|z|\geq 1\}.
\end{equation}
\end{proof}

\subsection{Asymptotics of $H_{\nu}^{(1)}(\zeta)$ as $|\zeta| \to 0$}

Since the asymptotics derived here will be applied to Dirac operators, we only consider $\nu \in [0,\infty)$ from this point on.  We distinguish two cases:\\[1mm]
$(i)$ $\nu\in (0,\infty)\backslash \bbN$, \\[1mm]
and \\[1mm] 
$(ii)$ $\nu \in \bbN_0$.

\medskip
\noindent
Case $(i)$:  $\nu\in (0,\infty)\backslash \bbN$.

In this case, one has the following representation for $H_{\nu}^{(1)}$ in terms of Bessel functions:
\begin{equation}
H_{\nu}^{(1)}(\zeta) = [1+i\cot(\nu \pi)]J_{\nu}(\zeta) - i [\sin(\nu \pi)]^{-1} J_{-\nu}(\zeta),\quad \zeta\in \bbC.
\end{equation}

Repeated term-by-term differentiation of the series representations for $J_{\pm\nu}$ reveals
\begin{align}
\frac{d^k}{d\zeta^k}J_{\pm\nu}(\zeta) & \underset{|\zeta| \to 0}{=} \bigg(\frac{d^k}{d\zeta^k}\bigg[\frac{1}{\Gamma(1\pm\nu)}
(\zeta/2)^{\pm\nu}\bigg]\bigg)\big[1+O\big(|\zeta|^2\big)\big]     \\
& \underset{|\zeta| \to 0}{=} \frac{(\pm\nu)_k}{2^k\Gamma(1\pm\nu)}(\zeta/2)^{\pm\nu-k} 
\big[1+O\big(|\zeta|^2\big)\big],\quad |\zeta|\leq 1,\; k\in \bbN_0.   \no 
\end{align}
As a result,
\begin{align}
&\frac{d^k}{d\zeta^k}H_{\nu}^{(1)}(\zeta) \underset{|\zeta| \to 0}{=} 
[1+i\cot(\nu \pi)]\frac{(\nu)_k}{2^k\Gamma(\nu+1)}(\zeta/2)^{\nu-k}\big[1+O\big(|\zeta|^2\big)\big]     \\
&\quad  - i [\sin(\nu \pi)]^{-1} \frac{(-\nu)_k}{2^k\Gamma(1-\nu)}(\zeta/2)^{-\nu-k}
\big[1+O\big(|\zeta|^2\big)\big],\quad |\zeta|\leq 1,\; k\in \bbN_0,\no
\end{align}
which settles Case $(i)$.

\medskip
\noindent
Case $(ii)$: $\nu \in \bbN_0$.

Since $\nu$ is a nonnegative integer, we write 
\begin{equation} 
\wti n:=\nu\in \bbN_0.  
\end{equation} 
First, we treat the case $\wti n \in \bbN$. Then,
\begin{align} \lb{D.49}
J_{\wti n}(\zeta) = \sum_{m=0}^{\infty} \frac{(-1)^m}{m!(\wti n+m)!}(\zeta/2)^{2m+\wti n},\quad \zeta\in \bbC,
\end{align}
and
\begin{align}
Y_{\wti n}(\zeta)&= -\frac{1}{\pi}(\zeta/2)^{- \wti n}\sum_{m=0}^{\wti n-1}\frac{(\wti n-m-1)!}{m!}(\zeta/2)^{2m} 
+ \frac{2}{\pi}\ln(\zeta/2)J_{\wti n}(\zeta)\no\\
&\quad -\frac{1}{\pi}(\zeta/2)^{\wti n} \sum_{m=0}^{\infty}[\psi(m+1) 
+ \psi(\wti n+m+1)]\frac{(-1)^m}{m!(\wti n+m)!}(\zeta/2)^{2m}\no\\
&=: Y_{1,\wti n}(\zeta) + Y_{2,\wti n}(\zeta) + Y_{3,\wti n}(\zeta),\quad \zeta\in \bbC\backslash\{0\}.
\end{align}
Repeated differentiation of the series representation for $J_{\wti n}$ yields
\begin{align}
\frac{d^k}{d\zeta^k}J_{\wti n}(\zeta) \underset{|\zeta| \to 0}{=} \frac{(\wti n)_k}{2^k \wti n!}(\zeta/2)^{\wti n-k}
\big[1+O\big(|\zeta|^2\big)\big],\quad |\zeta|\leq 1,\; k\in \bbN\cap[0,\wti n],\lb{D.51}
\end{align}
and
\begin{align}
\frac{d^k}{d\zeta^k}J_{\wti n}(\zeta) \underset{|\zeta| \to 0}{=} 
\begin{cases}
\frac{(-1)^{(k- \wti n)/2}k!}{2^k\left((k - \wti n)/2\right)!\left((k + \wti n)/2\right)!}\big[1+O\big(|\zeta|^2\big)\big],& \text{$\wti n+k$ even},\\[2mm]
\frac{(-1)^{(k - \wti n +1)/2}(k+1)!}{2^{k+1}\left((k - \wti n +1)/2\right)!\left((k + \wti n +1)/2\right)!}\zeta\big[1+O\big(|\zeta|^2\big)\big],& \text{$\wti n+k$ odd},
\end{cases}
&\lb{D.52}\\
|\zeta|\leq 1,\,k\in \bbN\cap(\wti n,\infty).&\no
\end{align}

Repeated differentiation of $Y_{1,\wti n}$ yields
\begin{align}
\frac{d^k}{d\zeta^k}Y_{1,\wti n}(\zeta) &\underset{|\zeta| \to 0}{=} -\pi^{-1} 
\bigg(\frac{d^k}{d\zeta^k}\big[(\zeta/2)^{-\wti n}\big]\bigg)\big[1+O\big(|\zeta|^2\big)\big]\no\\
&\underset{|\zeta| \to 0}{=} -\frac{(-\wti n)_k}{2^k\pi}(\zeta/2)^{-\wti n-k}\big[1+O\big(|\zeta|^2\big)\big],
\quad |\zeta|\leq 1,\; k\in \bbN_0. 
\lb{D.53}
\end{align}
In view of \eqref{D.49},
\begin{align}
Y_{2,\wti n}(\zeta) \underset{|\zeta| \to 0}{=} \frac{2}{\pi \wti n!}\ln(\zeta/2)(\zeta/2)^{\wti n} + \wti O\big(\zeta^{\wti n+2}\ln(\zeta)\big), 
\quad |\zeta| \leq 1.\lb{D.54}
\end{align}
Differentiation of the series representation of $Y_{3,\wti n}$ yields
\begin{align}
\frac{d^k}{d\zeta^k}Y_{3,\wti n}(\zeta) \underset{|\zeta| \to 0}{=} 
-\frac{1}{\pi}\frac{[\psi(\wti n+1)-\gamma](\wti n)_k}{2^k \wti n!}
(\zeta/2)^{\wti n-k}\big[1+O\big(|\zeta|^2\big)\big],&\lb{D.55}\\
|\zeta|\leq 1,\; k\in \bbN\cap[0,\wti n],&\no
\end{align}
and
\begin{align}
&\frac{d^k}{d\zeta^k}Y_{3,\wti n}(\zeta)\lb{D.56}\\
&\quad \underset{|\zeta| \to 0}{=} 
\begin{cases}
-\frac{1}{\pi}\frac{\left[\psi\left([(k - \wti n)/2]+1\right)+\psi\left([(k + \wti n)/2]+1\right)\right](-1)^{(k - \wti n)/2} k!}{2^k\left((k - \wti n)/2\right)!\left((k + \wti n)/2\right)!}\big[1+O\big(|\zeta|^2\big)\big],  \\
\hspace*{7.05cm} \text{$\wti n+k$ even},\\
-\frac{1}{\pi}\frac{\left[\psi\left([(k - \wti n +1)/2]+1\right)+\psi\left([(k + \wti n +1)/2]+1\right)\right] 
(-1)^{(k - \wti n +1)/2} k!}{2^k\left((k - \wti n +1)/2\right)!\left((k + \wti n +1)/2\right)!}\zeta\big[1+O\big(|\zeta|^2\big)\big],  \\
\hspace*{8.35cm} \text{$\wti n+k$ odd},
\end{cases}
\no\\
&\hspace*{7.95cm}|\zeta|\leq 1,\; k\in \bbN\cap(\wti n,\infty).\no
\end{align}

In the remaining case $\wti n=0$, one obtains (cf., e.g., \cite[(11), (12)]{EGT18})
\begin{align}
J_0(\zeta) \underset{|\zeta| \to 0}{=} 1 + O\big(\zeta^2\big),\quad |\zeta|\leq 1,\lb{D.57}
\end{align}
with
\begin{align}
\frac{d^k}{d\zeta^k}J_0(\zeta) \underset{|\zeta| \to 0}{=}
\begin{cases}
\frac{(-1)^{k/2}k!}{2^k [(k/2)!]^2}\big[1+O\big(|\zeta|^2\big)\big],& \text{$k$ even},\\[2mm]
\frac{(-1)^{(k + 1)/2}(k+1)!}{2^{k+1} [((k + 1)/2)!]^2} \zeta\big[1+O\big(|\zeta|^2\big)\big],& \text{$k$ odd},
\end{cases}
&\lb{D.58}\\
|\zeta|\leq 1,\; k\in \bbN,&\no
\end{align}
and
\begin{align}
Y_0(\zeta) \underset{|\zeta| \to 0}{=} \frac{2}{\pi}\ln(\zeta/2) 
+ \frac{2\gamma}{\pi} + \wti O\big(\zeta^2\ln(\zeta)\big),\quad |\zeta| \leq 1.\lb{D.59}
\end{align}

Finally, to obtain expressions for the derivatives of $H_{\wti n}^{(1)}$, one applies the representation
\begin{equation}
H_{\wti n}^{(1)}(\zeta) = J_{\wti n}(\zeta) + iY_{\wti n}(\zeta),\quad \zeta\in \bbC\backslash \{0\}, 
\; \wti n\in \bbN_0.
\end{equation}
If $\wti n \in \bbN$, then
\begin{align}
&\frac{d^k}{d \zeta^k}H_{\wti n}^{(1)}(\zeta) \underset{|\zeta| \to 0}{=} \frac{(\wti n)_k}{2^k \wti n!}(\zeta/2)^{\wti n-k}\big[1+O\big(|\zeta|^2\big)\big]    \no \\
& \quad - i \frac{(-\wti n)_k}{2^k\pi}(\zeta/2)^{-\wti n-k}\big[1+O\big(|\zeta|^2\big)\big]\no\\
&\quad + i \frac{2}{\pi \wti n!}\frac{d^k}{d\zeta^k} \big[\ln(\zeta/2) (\zeta/2)^{\wti n}\big] 
+ O\bigg(\frac{d^k}{d\zeta^k}\big[\zeta^{\wti n+2}\ln(\zeta)\big]\bigg)\no\\
&\quad -\frac{i}{\pi}\frac{[\psi(\wti n+1)-\gamma](\wti n)_k}{2^k \wti n!}(\zeta/2)^{\wti n-k}\big[1+O\big(|\zeta|^2\big)\big],\quad |\zeta| \leq 1,\; k\in \bbN\cap[0,\wti n],\no
\end{align}
while
\begin{align}
&\frac{d^k}{d\zeta^k}H_{\wti n}^{(1)}(\zeta) \underset{|\zeta| \to 0}{=} 
\frac{(-1)^{(k - \wti n)/2}k!}{2^k\left((k - \wti n)/2\right)!\left((k + \wti n)/2\right)!}
\big[1+O\big(|\zeta|^2\big)\big]    \no \\
& \quad -i\frac{(-\wti n)_k}{2^k\pi}(\zeta/2)^{-\wti n-k}\big[1+O\big(|\zeta|^2\big)\big]\no\\
&\quad + i \frac{2}{\pi \wti n!}\frac{d^k}{d\zeta^k}\big[(\zeta/2)^{\wti n} \ln(\zeta/2)\big] 
+ O\bigg(\frac{d^k}{d\zeta^k}\big[\zeta^{\wti n+2}\ln(\zeta)\big]\bigg)\no\\
&\quad -\frac{i}{\pi}\frac{\left[\psi\left([(k - \wti n)/2]+1\right) 
+ \psi\left([(k + \wti n)/2]+1\right)\right](-1)^{(k - \wti n)/2}k!}{2^k\left((k - \wti n)/2\right)!\left((k + \wti n)/2\right)!}\big[1+O\big(|\zeta|^2\big)\big],\no\\
&\hspace*{6.2cm} |\zeta| \leq 1,\; k\in \bbN\cap(\wti n,\infty),\; \text{$\wti n+k$ even},\no
\end{align}
and
\begin{align}
& \frac{d^k}{d\zeta^k}H_{\wti n}^{(1)}(\zeta) \underset{|\zeta| \to 0}{=}\frac{(-1)^{(k - \wti n +1)/2}(k+1)!}{2^{k+1}\left((k - \wti n +1)/2\right)!\left((k + \wti n +1)/2\right)!}\zeta\big[1+O\big(|\zeta|^2\big)\big]   \no \\
& \quad -i\frac{(-\wti n)_k}{2^k\pi}(\zeta/2)^{-\wti n-k}\big[1+O\big(|\zeta|^2\big)\big]\no\\
&\quad + i \frac{2}{\pi \wti n!}\frac{d^k}{d\zeta^k}\big[(\zeta/2)^{\wti n} \ln(\zeta/2)\big] 
+ O\bigg(\frac{d^k}{d\zeta^k}\big[\zeta^{\wti n+2}\ln(\zeta)\big]\bigg)\no\\
&\quad -\frac{i}{\pi}\frac{\left[\psi\left([(k - \wti n +1)/2]+1\right)+\psi\left([(k + \wti n +1)/2]+1\right)\right]
(-1)^{(k - \wti n +1)/2} k!}{2^k\left((k - \wti n +1)/2\right)!\left((k + \wti n +1)/2\right)!}\zeta      \no \\
& \qquad  \times \big[1+O\big(|\zeta|^2\big)\big], \quad  |\zeta| \leq 1,\; k\in \bbN\cap(\wti n,\infty),\; 
\text{$\wti n+k$ odd}.\no
\end{align}

In the case $\wti n=0$,
\begin{align}
\frac{d^k}{d\zeta^k}H_0^{(1)}(\zeta) & \underset{|\zeta| \to 0}{=} \frac{(-1)^{k/2}k!}{2^k [(k/2)!]^2}
\big[1+O\big(|\zeta|^2\big)\big]     \no \\ 
& \qquad \; \; + i\frac{d^k}{d\zeta^k}\bigg[\frac{2}{\pi}\ln(\zeta/2) + \frac{2\gamma}{\pi}\bigg] + O\bigg(\frac{d^k}{d\zeta^k}\big[\zeta^2\ln(\zeta)\big]\bigg),     \\
&\hspace*{4.4cm}|\zeta| \leq 1,\; k\in \bbN_0, \, \text{$k$ even},     \no
\end{align}
and
\begin{align}
\frac{d^k}{d\zeta^k}H_0^{(1)}(\zeta) & \underset{|\zeta| \to 0}{=} 
\frac{(-1)^{(k + 1)/2}(k+1)!}{2^{k+1} [((k + 1)/2)!]^2}\zeta\big[1+O\big(|\zeta|^2\big)\big]   \no \\
& \qquad \; \; + i\frac{d^k}{d\zeta^k}\bigg[\frac{2}{\pi}\ln(\zeta/2) + \frac{2\gamma}{\pi}\bigg] + O\bigg(\frac{d^k}{d\zeta^k}\big[\zeta^2\ln(\zeta)\big]\bigg),     \\
& \hspace*{4.7cm} |\zeta| \leq 1,\; k\in \bbN, \, \text{$k$ odd}.\no
\end{align}

\section{Expansions and estimates of the Free Dirac Green's Matrix} \lb{sC}
\renewcommand{\theequation}{C.\arabic{equation}}
\renewcommand{\thetheorem}{C.\arabic{theorem}}
\setcounter{theorem}{0} \setcounter{equation}{0}

In this section, we investigate the behavior of the Green's function \eqref{5.17} of the massless Dirac operator and certain of its partial derivatives with respect to the energy parameter $z\in \bbC\backslash \bbR$.  Throughout, we assume that $n\in \bbN\backslash\{1\}$.

By \eqref{5.17},
\begin{align}
G_0(z;x,y) &= i2^{-1- (n/2)}\pi^{1-(n/2)}z^{n/2}|x-y|^{1-(n/2)}H_{(n/2)-1}^{(1)}(z|x-y|)I_N   \lb{B.1} \\
&\quad - 2^{-1-(n/2)}\pi^{1-(n/2)}z^{n/2}|x-y|^{1-(n/2)}H_{n/2}^{(1)}(z|x-y|) \alpha\cdot \frac{(x-y)}{|x-y|}.   \no 
\end{align}
Due to the distinct difference in the behavior of the Hankel function $H_{\nu}^{(1)}$ for integer and fractional values of $\nu$, we treat separately the cases where: $(\textbf{I})$ $n$ is odd, $(\textbf{II})$ $n\in\bbN\backslash\{2\}$ is even, and $(\textbf{III})$ $n=2$.

$(\textbf{I})$ If $n \in \bbN$ is odd, then $H_{(n/2)-1}^{(1)}$ and $H_{n/2}^{(1)}$ are fractional (half-integer) Hankel functions.  Applying the identity (cf., e.g., \cite[9.1.3]{AS72})
\begin{equation} \lb{B.2}
H_{\nu}^{(1)}(\zeta) = i [\sin(\nu \pi)]^{-1} \big[e^{-\nu \pi i}J_{\nu}(\zeta) - J_{-\nu}(\zeta)\big],\quad \nu\in \bbC,\; \zeta\in \bbC\backslash \{0\},
\end{equation}
one obtains
\begin{align}
H_{(n/2)-1}^{(1)}(\zeta) &= i [\sin(((n/2)-1)\pi)]^{-1} \Big[e^{-((n/2)-1)\pi i}J_{(n/2)-1}(\zeta) - J_{-(n/2)+1}(\zeta) \Big]\no\\
&= i(-1)^{(n+1)/2}\Big[-i (-1)^{(n+1)/2}J_{(n/2)-1}(\zeta) - J_{-(n/2)+1}(\zeta) \Big]\no\\
&= (-1)^{n+1}J_{(n/2)-1}(\zeta) - i (-1)^{(n+1)/2}J_{-(n/2)+1}(\zeta)\no\\
&= J_{(n/2)-1}(\zeta)-i (-1)^{(n+1)/2}J_{-(n/2)+1}(\zeta).\lb{B.3}
\end{align}
Similarly,
\begin{align}
H_{n/2}^{(1)}(\zeta)&= i [\sin((n/2)\pi)]^{-1} \Big[e^{-n\pi i/2}J_{n/2}(\zeta) - J_{-n/2}(\zeta)\Big]\no\\
&= i(-1)^{(n-1)/2}\Big[-i(-1)^{(n-1)/2}J_{n/2}(\zeta) - J_{-n/2}(\zeta)\Big]\no\\
&= (-1)^{n-1}J_{n/2}(\zeta) + i (-1)^{(n+1)/2}J_{-n/2}(\zeta)\no\\
&= J_{n/2}(\zeta) + i(-1)^{(n+1)/2}J_{-n/2}(\zeta).\lb{B.4}
\end{align}
The series representation for $J_{\nu}(\zeta)$ in \eqref{5.8a} then yields the following expansion:
\begin{align}
&G_0(z;x,y) = i2^{-n}\pi^{1-(n/2)}z^{n-1} 
\sum_{k=0}^{\infty}\frac{(-4)^{-k}z^{2k}|x-y|^{2k}}{k!\, \Gamma((n/2) + k))}I_N\no\\
&\quad\quad+4^{-1}(-1)^{(n+1)/2}\pi^{1-(n/2)}z|x-y|^{2-n}\sum_{k=0}^{\infty}\frac{(-4)^{-k}z^{2k}|x-y|^{2k}}{k!\, 
\Gamma(-(n/2)+k+2)}I_N\no\\
&\quad\quad-2^{-1-n}\pi^{1-(n/2)}z^n|x-y|\sum_{k=0}^{\infty} \frac{(-4)^{-k}z^{2k}|x-y|^{2k}}{k!\, \Gamma((n/2) + k +1)}\alpha\cdot \frac{(x-y)}{|x-y|}\no\\
&\quad\quad-i2^{-1}(-1)^{(n+1)/2}\pi^{1-(n/2)}|x-y|^{1-n}\sum_{k=0}^{\infty}\frac{(-4)^{-k}z^{2k}|x-y|^{2k}}{k!\, 
\Gamma(-(n/2)+k+1)}\alpha\cdot \frac{(x-y)}{|x-y|}\no\\
&\quad= i2^{-n}\pi^{1-(n/2)}\sum_{k=0}^{\infty} \frac{(-4)^{-k}z^{2k+n-1}|x-y|^{2k}}{k!\, \Gamma((n/2) + k)}I_N  \no \\
&\quad\quad + 4^{-1}(-1)^{(n+1)/2}\pi^{1-(n/2)}|x-y|^{2-n}\sum_{k=0}^{\infty} \frac{(-4)^{-k}z^{2k+1}|x-y|^{2k}}{k!\, \Gamma(-(n/2)+k+2)}I_N\no\\
&\quad \quad-2^{-1-n}\pi^{1-(n/2)}|x-y|\sum_{k=0}^{\infty}\frac{(-4)^{-k}z^{2k+n}|x-y|^{2k}}{k!\, \Gamma((n/2)+k+1)}\alpha\cdot\frac{(x-y)}{|x-y|}     \lb{B.5} \\
&\quad \quad -i2^{-1}(-1)^{(n+1)/2}\pi^{1-(n/2)}|x-y|^{1-n}\sum_{k=0}^{\infty}\frac{(-4)^{-k}z^{2k}|x-y|^{2k}}{k!\, 
\Gamma(-(n/2)+k+1)}\alpha\cdot\frac{(x-y)}{|x-y|}.\no
\end{align}
The identity in \eqref{B.5} implies
\begin{align}
&G_0(z;x,y) = -\frac{i(-1)^{(n+1)/2}\pi^{1-(n/2)}}{2\, \Gamma(1-(n/2))}|x-y|^{1-n} 
\big[1 + \Oh\big(z^2|x-y|^2\big)\big]\alpha\cdot\frac{(x-y)}{|x-y|}   \no \\
&\quad\quad\quad + \frac{(-1)^{(n+1)/2}\pi^{1-(n/2)}}{4\, \Gamma(2-(n/2))}|x-y|^{2-n}z\big[1 + \Oh\big(z^2|x-y|^2\big)\big]I_N\no\\
&\quad\quad\quad + \frac{i\pi^{1-(n/2)}}{2^n\Gamma(n/2)}z^{n-1}\big[1 + \Oh\big(z^2|x-y|^2\big)\big]I_N\no\\
&\quad\quad\quad -\frac{\pi^{1-(n/2)}}{2^{1+n}\Gamma(1+(n/2))}|x-y|z^n \big[1 + \Oh\big(z^2|x-y|^2\big)\big] \alpha\cdot\frac{(x-y)}{|x-y|}     \lb{B.6a} \\
&\hspace*{1.6cm} 
\text{as $z\to 0$, $z\in \ol{\bbC_+}\big\backslash\{0\}$, $x,y\in \bbR^n$, $x\neq y$, $n\in \bbN\backslash\{1\}$ odd.}\no
\end{align}
One notes that \eqref{B.6a} implies, together with the identity,
\begin{equation}
\Gamma(1-z)\Gamma(z) = \frac{\pi}{\sin(\pi z)},\quad z\in \bbC\backslash\bbZ,
\end{equation}
that
\begin{align}
 \lim_{\substack{z \to 0, \\ z \in \ol{\bbC_+} \backslash\{0\}}} G_0(z;x,y) &= -\frac{i(-1)^{(n+1)/2}\pi^{1-(n/2)}}{2\, 
 \Gamma(1-(n/2))}|x-y|^{1-n}\alpha\cdot\frac{(x-y)}{|x-y|}\no\\
&= -\frac{i(-1)^{(n+1)/2}\pi^{1-(n/2)}}{2\sin(n\pi/2)}\Gamma(n/2)|x-y|^{1-n}\alpha\cdot\frac{(x-y)}{|x-y|}\no\\
&= i2^{-1}\pi^{-n/2}\Gamma(n/2)\alpha\cdot\frac{(x-y)}{|x-y|^n},     \\ 
& \hspace*{-2mm} x,y\in \bbR^n,\; x\neq y, \; n\in \bbN\backslash\{1\} \text{ odd,}    \no 
\end{align}
consistent with \eqref{5.18}.

Partial derivatives of $G_0(z;x,y)$ with respect to $z$ may be computed by differentiating the series representations in \eqref{B.5} term-by-term.  For $n\in \bbN\backslash\{1\}$ odd and $r\in \bbN$ with $1\leq r\leq n$, one obtains
\begin{align}
&\frac{\partial^r}{\partial z^r}G_0(z;x,y) 
=i2^{-n}\pi^{1-(n/2)}\sum_{k=\delta_n(r)}^{\infty} \frac{(-4)^{-k}(2k+n-1)!\, z^{2k+n-1-r}|x-y|^{2k}}{k!\, (2k+n-1-r)!\, \Gamma((n/2) + k)}I_N\no\\
&\quad\quad + 4^{-1}(-1)^{(n+1)/2}\pi^{1-(n/2)}|x-y|^{2-n}\no\\
&\quad\quad\quad \times\sum_{k=k_-(r)}^{\infty}\frac{(-4)^{-k}(2k+1)!\, z^{2k+1-r}|x-y|^{2k}}{k!\, (2k+1-r)!\, 
\Gamma(-(n/2) + k + 2)}I_N\no\\
&\quad\quad - 2^{-1-n} \pi^{1-(n/2)}|x-y|\sum_{k=0}^{\infty} \frac{(-4)^{-k}(2k+n)!\, z^{2k+n-r}|x-y|^{2k}}{k!\, (2k+n-r)!\, \Gamma((n/2) + k +1)}\alpha\cdot\frac{(x-y)}{|x-y|}\no\\
&\quad\quad -i2^{-1}(-1)^{(n+1)/2}\pi^{1-(n/2)}|x-y|^{1-n}\no\\
&\quad\quad\quad \times\sum_{k=k_+(r)}^{\infty}\frac{(-4)^{-k}(2k)!\, z^{2k-r}|x-y|^{2k}}{k!\, (2k-r)!\, \Gamma(-(n/2)+k+1)}\alpha\cdot \frac{(x-y)}{|x-y|},\lb{B.10}
\end{align}
where 
\begin{equation}
k_{\pm}(r):=
\begin{cases}
(r\pm 1)/2,& \text{$r$ odd},\\
r/2,& \text{$r$ even},
\end{cases}  \quad 1 \leq r \leq n,  \lb{B.10Z}
\end{equation}
and $\delta_{n}$ is the Kronecker delta function,
\begin{equation}
\delta_n(r)=
\begin{cases}
1,& r=n,\\
0,& 1 \leq r \leq n-1,
\end{cases}  \quad 1 \leq r \leq n.     \lb{B.11Z} 
\end{equation}

The expansion in \eqref{B.10} implies the following asymptotics of $\frac{\partial^r}{\partial z^r}G_0(z;x,y)$ as $z\to 0$:\\[1mm] 
\noindent
$(i)$  If $n \in \bbN$ is odd and $1\leq r\leq n -2$ is odd, then
\begin{align}
&\frac{\partial^r}{\partial z^r}G_0(z;x,y) 
= \frac{i2^{-n}\pi^{1-(n/2)}(n-1)!}{(n-1-r)!\Gamma(n/2)}z^{n-1-r}\big[1+\Oh\big(z^2|x-y|^2\big)\big]I_N \no\\
&\quad\quad+\frac{(-1)^{(n+1)/2}\pi^{1-(n/2)}(-4)^{-(r+1)/2}r!}{[(r-1)/2]!\, \Gamma(-(n/2) 
+ ((r-1)/2)+2)}|x-y|^{1+r-n}\big[1+\Oh\big(z^2|x-y|^2\big)\big]I_N\no\\
&\quad\quad-\frac{2^{-1-n}\pi^{1-(n/2)}n!}{(n-r)!\, \Gamma(1+(n/2))}|x-y|z^{n-r}\big[1+\Oh\big(z^2|x-y|^2\big)\big]\alpha\cdot\frac{(x-y)}{|x-y|}\no\\
&\quad\quad - \frac{i2^{-1}(-1)^{(n+1)/2}\pi^{1-(n/2)}(-4)^{-(r+1)/2}(r+1)!}{[(r+1)/2]!\, 
\Gamma(-(n/2) + ((r+1)/2)+1)}|x-y|^{2+r-n}z\no\\
&\quad\quad\quad\times\big[1+\Oh\big(z^2|x-y|^2\big)\big]\alpha\cdot\frac{(x-y)}{|x-y|}   \lb{B.12y} \\
&\, \text{as $z\to 0$, $z\in \ol{\bbC_+}\big\backslash\{0\}$, $x,y\in \bbR^n$, $x\neq y$}.\no
\end{align}

\noindent
$(ii)$ If $n \in \bbN$ is odd and $1\leq r\leq n-1$ is even, then
\begin{align}
&\frac{\partial^r}{\partial z^r}G_0(z;x,y) 
= \frac{i2^{-n}\pi^{1-(n/2)}(n-1)!}{(n-1-r)!\Gamma(n/2)}z^{n-1-r}\big[1+\Oh\big(z^2|x-y|^2\big)\big]I_N  \no \\
&\quad\quad+\frac{4^{-1}(-1)^{(n+1)/2}\pi^{1-(n/2)}(-4)^{-r/2}(r+1)!}{(r/2)!\, \Gamma(-(n/2)+(r/2)+2)}|x-y|^{2+r-n} \no \\
&\qquad \quad \times z\big[1+\Oh\big(z^2|x-y|^2\big)\big]I_N\no\\
&\quad\quad - \frac{2^{-1-n}\pi^{1-(n/2)}n!}{(n-r)!\, \Gamma(1+(n/2))}|x-y|z^{n-r}\big[1+\Oh\big(z^2|x-y|^2\big)\big]\alpha\cdot\frac{(x-y)}{|x-y|}\no\\
&\quad\quad -\frac{i2^{-1}(-1)^{(n+1)/2}\pi^{1-(n/2)}(-4)^{-r/2}r!}{(r/2)!\, \Gamma(-(n/2)+(r/2)+1)}|x-y|^{1+r-n}\no\\
&\quad\quad\quad\times\big[1+\Oh\big(z^2|x-y|^2\big)\big]\alpha\cdot\frac{(x-y)}{|x-y|}   \lb{B.14y} \\
&\, \text{as $z\to 0$, $z\in \ol{\bbC_+}\big\backslash\{0\}$, $x,y\in \bbR^n$, $x\neq y$}.\no
\end{align}

\noindent
$(iii)$  If $n \in \bbN$ is odd, then
\begin{align}
&\frac{\partial^n}{\partial z^n}G_0(z;x,y) 
= -\frac{i2^{-n}\pi^{1-(n/2)}}{4\Gamma(1+(n/2))}|x-y|^2z\big[1+\Oh\big(z^2|x-y|^2\big)\big]I_N  \no \\
&\quad\quad +\frac{2(-1)^{(n+1)/2}\pi^{(1-n)/2}(-4)^{-(n+1)/2}n!}{[(n-1)/2]!}|x-y|\big[1+\Oh\big(z^2|x-y|^2\big)\big]I_N\no\\
&\quad\quad-\frac{2^{-1-n}\pi^{1-(n/2)}n!}{\Gamma(1+(n/2))}|x-y|\big[1+\Oh\big(z^2|x-y|^2\big)\big]\alpha\cdot\frac{(x-y)}{|x-y|}\no\\
&\quad\quad -\frac{i(-1)^{(n+1)/2}\pi^{(1-n)/2}(-4)^{-(n+1)/2}(n+1)!}{[(n+1)/2]!}|x-y|^2 z\no\\
&\quad\quad\quad\times\big[1+\Oh\big(z^2|x-y|^2\big)\big]\alpha\cdot\frac{(x-y)}{|x-y|}   \lb{B.13y} \\
&\, \text{as $z\to 0$, $z\in \ol{\bbC_+}\big\backslash\{0\}$, $x,y\in \bbR^n$, $x\neq y$}.\no
\end{align}

$(\textbf{II})$ If $n \in \bbN$ is even, then the indices of the Hankel functions $H_{(n/2)-1}^{(1)}$ and $H_{n/2}^{(1)}$ are nonnegative integers.  Due to the the difference in behavior of $Y_{n}$, $n\in \bbN$, and $Y_0$ (cf.~\eqref{5.8c} and \eqref{5.8dd}) we distinguish two cases: $n \geq 4$ and $n=2$.  First we treat the case $n \geq 4$.

Combining \eqref{5.8}, \eqref{5.8a}, and \eqref{5.8c}, one obtains for $n\geq 4$:
\begin{align}
& H_{(n/2)-1}^{(1)}(\zeta) = J_{(n/2)-1}(\zeta) + i Y_{(n/2)-1}(\zeta)\no\\
&\quad = 2^{1-(n/2)}\zeta^{(n/2)-1}\sum_{k=0}^{\infty}\frac{(-4)^{-k}\zeta^{2k}}{k!\, \Gamma((n/2) + k)}\no\\
&\qquad - i2^{(n/2)-1}\pi^{-1}\zeta^{1-(n/2)}\sum_{k=0}^{(n/2)-2}\frac{((n/2)-k-2)!4^{-k}\zeta^{2k}}{k!}\no\\
&\qquad+i2\pi^{-1}\ln(\zeta/2)2^{1-(n/2)}\zeta^{(n/2)-1}\sum_{k=0}^{\infty} \frac{(-4)^{-k}\zeta^{2k}}{k!\, 
\Gamma((n/2) + k)}\no\\
&\qquad- i2^{1-(n/2)}\pi^{-1}\zeta^{(n/2)-1}\sum_{k=0}^{\infty}[\psi(k+1)
+\psi((n/2)+k)]\frac{(-4)^{-k}\zeta{2k}}{k!\, ((n/2)-1+k)!}\no\\
&\quad = 2^{1-n/2}\zeta^{(n/2)-1}\sum_{k=0}^{\infty}\frac{(-4)^{-k}\zeta^{2k}}{k!\, ((n/2) + k - 1)!}   \no \\
&\qquad - i2^{(n/2)-1}\pi^{-1}\zeta^{1-(n/2)}\sum_{k=0}^{(n/2)-2}\frac{((n/2)-k-2)!4^{-k}\zeta^{2k}}{k!}\no\\
&\qquad+i2\pi^{-1}\ln(\zeta/2)2^{1-(n/2)}\zeta^{(n/2)-1}\sum_{k=0}^{\infty} \frac{(-4)^{-k}\zeta^{2k}}{k!\, ((n/2) + k - 1)!}
\lb{B.6} \\
&\qquad- i2^{1-(n/2)}\pi^{-1}\zeta^{(n/2)-1}\sum_{k=0}^{\infty}[\psi(k+1)
+\psi((n/2)+k)]\frac{(-4)^{-k}\zeta{2k}}{k!\, ((n/2)-1+k)!}.\no
\end{align}

Next, for any even $n\in \bbN$,
\begin{align}
H_{(n/2)}^{(1)}(\zeta) &= 2^{-n/2}\zeta^{n/2}\sum_{k=0}^{\infty} \frac{(-4)^{-k}\zeta^{2k}}{k!\, ((n/2) + k)!}\no\\
&\quad -i2^{n/2}\pi^{-1}\zeta^{-n/2}\sum_{k=0}^{(n/2)-1}\frac{((n/2) - k -1)! 4^{-k}\zeta^{2k}}{k!}\no\\
&\quad +i2^{1-(n/2)}\pi^{-1}\ln(\zeta/2)\zeta^{n/2}\sum_{k=0}^{\infty}\frac{(-4)^{-k}\zeta^{2k}}{k!\, ((n/2) + k)!} \lb{B.7} \\
&\quad -i2^{-n/2}\pi^{-1}\zeta^{n/2}\sum_{k=0}^{\infty}[\psi(k+1)
+\psi(n/2+k+1)]\frac{(-4)^{-k}\zeta^{2k}}{k!\, ((n/2) + k)!}.    \no
\end{align}
Substitution of \eqref{B.6} and \eqref{B.7} into \eqref{B.1} then yields for even $n\geq 4$,
\begin{align}
&G_0(z;x,y) 
= i2^{-n}\pi^{1-(n/2)}z^{n-1} \sum_{k=0}^{\infty} \frac{(-4)^{-k}z^{2k}|x-y|^{2k}}{k!\, ((n/2) + k -1)!}I_N\no\\
&\quad\quad + 4^{-1}\pi^{-n/2}z|x-y|^{2-n} \sum_{k=0}^{(n/2) - 2}\frac{((n/2)-k-2)!\, 4^{-k}z^{2k}|x-y|^{2k}}{k!}I_N\no\\
&\quad\quad -2^{1-n}\pi^{-n/2}z^{n-1}\ln(z|x-y|/2)\sum_{k=0}^{\infty} \frac{(-4)^{-k}z^{2k}|x-y|^{2k}}{k!\, ((n/2) + k -1)!}I_N\no\\
&\quad\quad +2^{-n}\pi^{-n/2}z^{n-1}\sum_{k=0}^{\infty}[\psi(k+1)+\psi((n/2)+k)] \frac{(-4)^{-k}z^{2k}|x-y|^{2k}}{k!\, ((n/2) - 1 + k)!}I_N\no\\
&\quad\quad - 2^{-1-n}\pi^{1-(n/2)}z^n|x-y|\sum_{k=0}^{\infty}\frac{(-4)^{-k}z^{2k}|x-y|^{2k}}{k!\, ((n/2) + k)!}\alpha\cdot \frac{(x-y)}{|x-y|}\no\\
&\quad\quad +i2^{-1}\pi^{-n/2}|x-y|^{1-n}\sum_{k=0}^{(n/2)-1}\frac{((n/2) - k - 1)!\, 4^{-k}z^{2k}|x-y|^{2k}}{k!}\alpha\cdot \frac{(x-y)}{|x-y|}\no\\
&\quad\quad -i2^{-n}\pi^{-n/2}z^n|x-y|\ln(z|x-y|/2)\sum_{k=0}^{\infty}\frac{(-4)^{-k}z^{2k}|x-y|^{2k}}{k!\, ((n/2) + k)!}\alpha\cdot\frac{(x-y)}{|x-y|}\no\\
&\quad\quad +i2^{-1-n}\pi^{-n/2}z^n|x-y|\sum_{k=0}^{\infty}[\psi(k+1)
+\psi((n/2)+k+1)]\frac{(-4)^{-k}z^{2k}|x-y|^{2k}}{k!\, ((n/2) + k)!}\no\\
&\quad\quad\quad \times \alpha\cdot\frac{(x-y)}{|x-y|}\no\\
&\quad = i2^{-n}\pi^{1-(n/2)} \sum_{k=0}^{\infty} \frac{(-4)^{-k}z^{2k+n-1}|x-y|^{2k}}{k!\, ((n/2) + k -1)!}I_N \no \\
&\quad\quad + 4^{-1}\pi^{-n/2}|x-y|^{2-n} \sum_{k=0}^{(n/2) - 2}\frac{((n/2)-k-2)!\, 4^{-k}z^{2k+1}|x-y|^{2k}}{k!}I_N\no\\
&\quad\quad -2^{1-n}\pi^{-n/2}\ln(z|x-y|/2)\sum_{k=0}^{\infty} \frac{(-4)^{-k}z^{2k+n-1}|x-y|^{2k}}{k!\, ((n/2) + k -1)!}I_N\no\\
&\quad\quad +2^{-n}\pi^{-n/2}\sum_{k=0}^{\infty}[\psi(k+1)
+\psi((n/2)+k)] \frac{(-4)^{-k}z^{2k+n-1}|x-y|^{2k}}{k!\, ((n/2) - 1 + k)!}I_N\no\\
&\quad\quad - 2^{-1-n}\pi^{1-(n/2)}|x-y|\sum_{k=0}^{\infty}\frac{(-4)^{-k}z^{2k+n}|x-y|^{2k}}{k!\, ((n/2) + k)!}\alpha\cdot \frac{(x-y)}{|x-y|}\no\\
&\quad\quad +i2^{-1}\pi^{-n/2}|x-y|^{1-n}\sum_{k=0}^{(n/2)-1}\frac{((n/2) - k - 1)!\, 4^{-k}z^{2k}|x-y|^{2k}}{k!}\alpha\cdot \frac{(x-y)}{|x-y|}\no\\
&\quad\quad -i2^{-n}\pi^{-n/2}|x-y|\ln(z|x-y|/2)\sum_{k=0}^{\infty}\frac{(-4)^{-k}z^{2k+n}|x-y|^{2k}}{k!\, ((n/2) + k)!}\alpha\cdot\frac{(x-y)}{|x-y|}\no\\
&\quad\quad +i2^{-1-n}\pi^{-n/2}|x-y|\sum_{k=0}^{\infty}[\psi(k+1)
+\psi((n/2)+k+1)]\frac{(-4)^{-k}z^{2k+n}|x-y|^{2k}}{k!\, ((n/2) + k)!}\no\\
&\quad\quad\quad \times \alpha\cdot\frac{(x-y)}{|x-y|}.     \lb{B.8}
\end{align}
The identity in \eqref{B.8} implies
\begin{align}
&G_0(z;x,y) 
= \frac{i}{2^n\pi^{(n/2)-1}((n/2)-1)!}z^{n-1}\big[1+\Oh\big(z^2|x-y|^2\big)\big] I_N   \no \\
&\quad \quad + \frac{((n/2)-2)!}{4\pi^{n/2}}|x-y|^{2-n}z\big[1+\Oh\big(z^2|x-y|^2\big)\big] I_N\no\\
&\quad\quad - \frac{1}{2^n\pi^{n/2}((n/2)-1)!}\ln(z|x-y|/2)z^{n-1}\big[1+\Oh\big(z^2|x-y|^2\big)\big] I_N\no\\
&\quad\quad + \frac{\psi(1)+\psi(n/2)}{2^n\pi^{n/2}}z^{n-1}\big[1+\Oh\big(z^2|x-y|^2\big)\big] I_N\no\\
&\quad\quad - \frac{1}{2^{1+n}\pi^{(n/2)-1}(n/2)!}|x-y|z^n\big[1+\Oh\big(z^2|x-y|^2\big)\big] \alpha\cdot \frac{(x-y)}{|x-y|}\no\\
&\quad\quad + \frac{i((n/2)-1)!}{2\pi^{n/2}}|x-y|^{1-n}\big[1+\Oh\big(z^2|x-y|^2\big)\big] \alpha\cdot\frac{(x-y)}{|x-y|}\no\\
&\quad\quad -\frac{i}{2^n\pi^{n/2}(n/2)!}|x-y|z^n\ln(z|x-y|/2)\big[1+\Oh\big(z^2|x-y|^2\big)\big]\alpha\cdot \frac{(x-y)}{|x-y|}\no\\
&\quad\quad + \frac{i[\psi(1)+\psi((n/2)+1)]}{2^{1+n}\pi^{n/2}(n/2)!}|x-y|z^n\big[1+\Oh\big(z^2|x-y|^2\big)\big] \alpha\cdot\frac{(x-y)}{|x-y|}     \lb{B.15z} \\
&\hspace*{1.2cm} \text{as $z\to 0$, $z\in \ol{\bbC_+}\big\backslash\{0\}$, $x,y\in \bbR^n$, $x\neq y$, and $n\in \bbN\backslash\{2\}$ even.}\no
\end{align}
One notes that \eqref{B.15z} implies, together with $(n/2-1)! = \Gamma(n/2)$, that
\begin{align}
\begin{split} 
 \lim_{\substack{z \to 0, \\ z \in \ol{\bbC_+} \backslash\{0\}}} G_0(z;x,y) 
&= i2^{-1}\pi^{-n/2}\Gamma(n/2)\alpha\cdot\frac{(x-y)}{|x-y|^n},     \\ 
& \hspace*{-3mm} x,y\in \bbR^n,\; x\neq y, \; n\in \bbN\backslash\{2\} \text{ even,}    \lb{B.19z} 
\end{split} 
\end{align}
consistent with \eqref{5.18}.

For $n\in \bbN\backslash\{2\}$ even and $r\in \bbN$ with $1\leq r\leq n$, term-by-term differentiation of \eqref{B.8} implies
\begin{align}
&\frac{\partial^r}{\partial z^r}G_0(z;x,y) 
= \Bigg[i2^{-n}\pi^{1-(n/2)}\sum_{k=\delta_n(r)}^{\infty}\frac{(-4)^{-k}(2k+n-1)!\, 
z^{2k+n-1-r}|x-y|^{2k}}{k!\, (2k+n-1-r)!\, ((n/2)+k-1)!}\no\\
&\quad\quad + \frac{|x-y|^{2-n}}{4\pi^{n/2}}\chi_{{}_{\leq n-3}}(r)\sum_{k=k_-(r)}^{(n/2)-2}\frac{((n/2)-k-2)!\, 4^{-k}(2k+1)!\, z^{2k+1-r}|x-y|^{2k}}{k!\, (2k+1-r)!}\no\\
&\quad\quad -2^{1-n}\pi^{-n/2}\ln(z|x-y|/2)\sum_{k=\delta_n(r)}^{\infty}\frac{(-4)^{-k}(2k+n-1)!\, z^{2k+n-1-r}|x-y|^{2k}}{k!\, (2k+n-1-r)!\, ((n/2) + k -1)!}\no\\
&\quad\quad -2^{1-n}\pi^{-n/2}\sum_{\ell=0}^{r-1}\sum_{k=0}^{\infty}
\begin{pmatrix}
r\\
\ell
\end{pmatrix}
(-1)^{1+r-\ell}(r-\ell-1)!\no\\
&\quad\quad\quad \times\frac{(-4)^{-k}(2k+n-1)!\, z^{2k+n-1-r}|x-y|^{2k}}{k!\, (2k+n-1-\ell)!\, ((n/2)+k-1)!}\no\\
&\quad\quad+2^{-n}\pi^{-n/2}\sum_{k=\delta_n(r)}^{\infty}[\psi(k+1)+\psi((n/2)+k)]\no\\
&\quad\quad\quad\times\frac{(-4)^{-k}(2k+n-1)!\, z^{2k+n-1-r}|x-y|^{2k}}{k!\, (2k+n-1-r)!\, ((n/2)+k-1)!}\Bigg]I_N\no\\
&\quad\quad+\Bigg[-2^{-1-n}\pi^{1-(n/2)}|x-y|\sum_{k=0}^{\infty}\frac{(-4)^{-k}(2k+n)!\, z^{2k+n-r}|x-y|^{2k}}{k!\, (2k+n-r)!\, ((n/2)+k)!}\no\\
&\quad\quad\quad\quad +i\frac{|x-y|^{1-n}}{2\pi^{n/2}}\chi_{{}_{\leq n-2}}(r) 
\sum_{k=k_+(r)}^{(n/2)-1}\frac{((n/2)-k-1)!\, 4^{-k}(2k)!\, z^{2k-r}|x-y|^{2k}}{k!\, (2k-r)!}\no\\
&\quad\quad\quad\quad-i2^{-n}\pi^{-n/2}|x-y|\ln(z|x-y|/2)\sum_{k=0}^{\infty}\frac{(-4)^{-k}(2k+n)!\, z^{2k+n-r}|x-y|^{2k}}{k!\, ((n/2)+k)!\, (2k+n-r)!}\no\\
&\quad\quad\quad\quad-i2^{-n}\pi^{-n/2}|x-y|\sum_{\ell=0}^{r-1}\sum_{k=0}^{\infty}
\begin{pmatrix}
r\\
\ell
\end{pmatrix}
(-1)^{1+r-\ell}(r-\ell-1)!\no\\
&\quad\quad\quad\quad\quad\times \frac{(-4)^{-k}(2k+n)!\, z^{2k+n-r}|x-y|^{2k}}{k!\, ((n/2)+k)!\, (2k+n-\ell)!}\no\\
&\quad \quad \quad \quad +i2^{-1-n}\pi^{-n/2}|x-y|\sum_{k=0}^{\infty}[\psi(k+1)+\psi((n/2)+k+1)]\no\\
&\quad\quad\quad\quad\quad\times \frac{(-4)^{-k}(2k+n)!z^{2k+n-r}|x-y|^{2k}}{k!\, 
((n/2) + k)!}\Bigg]\alpha\cdot\frac{(x-y)}{|x-y|},\lb{B.16}
\end{align}
where $\chi_{{}_{\leq a}}$, $a\in \bbR$, denotes the characteristic (i.e., indicator) function of the interval $(-\infty,a]$.  That is,
\begin{equation}
\chi_{{}_{\leq a}}(x) =
\begin{cases}
1,& x\in (-\infty,a],\\
0,& x\in (a,\infty),
\end{cases}
\quad x\in \bbR.     \lb{B.21z} 
\end{equation}

The expansion in \eqref{B.16} implies the following asymptotics of $\frac{\partial^r}{\partial z^r}G_0(z;x,y)$ as $z\to 0$:\\[1mm] 
\noindent
$(i)$  If $n\in\bbN\backslash\{2\}$ is even and $1\leq r\leq n -1$ is odd, then
\begin{align}
&\frac{\partial^r}{\partial z^r}G_0(z;x,y) 
= \frac{i2^{-n}\pi^{1-(n/2)}(n-1)!}{(n-1-r)!\, ((n/2)-1)!}z^{n-1-r}\big[1+\Oh\big(z^2|x-y|^2\big)\big]I_N\no\\
&\quad\quad+\chi_{{}_{\leq n-3}}(r)\frac{4^{-(1+r)/2}\pi^{-n/2}[(n-r-3)/2]!\, r!}{[(r-1)/2]!}|x-y|^{1+r-n}\no\\
&\quad\quad\quad\times\big[1+\Oh\big(z^2|x-y|^2\big)\big]I_N\no\\
&\quad\quad-\frac{2^{1-n}\pi^{-n/2}(n-1)!}{(n-1-r)!\, ((n/2)-1)!}z^{n-1-r}\ln(z|x-y|/2) \big[1+\Oh\big(z^2|x-y|^2\big)\big]I_N\no\\
&\quad\quad+\frac{2^{1-n}}{\pi^{n/2}}\sum_{\ell=0}^{r-1}
\begin{pmatrix}
r\\
\ell
\end{pmatrix}
(-1)^{r-\ell}\frac{(r-\ell-1)!\, (n-1)!}{(n-1-\ell)!\, ((n/2)-1)!}z^{n-1-r}   \no \\
& \hspace*{2.6cm} \times \big[1+\Oh\big(z^2|x-y|^2\big)\big]I_N\no\\
&\quad\quad+\frac{2^{-n}\pi^{-n/2}[\psi(1)+\psi(n/2)](n-1)!}{(n-1-r)!\, ((n/2)-1)!}z^{n-1-r}\big[1+\Oh\big(z^2|x-y|^2\big)\big]I_N\no\\
&\quad\quad-\frac{2^{-(1+n)}\pi^{1-(n/2)}n!}{(n-r)!\, (n/2)!}|x-y|z^{n-r}\big[1+\Oh\big(z^2|x-y|^2\big)\big]\alpha\cdot\frac{(x-y)}{|x-y|}\no\\
&\quad\quad+\chi_{{}_{\leq n-2}}(r)\frac{i\pi^{-n/2}[(n-r-3)/2]!\, (r+1)!}{4^{1+(r/2)} [(r+1)/2]!}|x-y|^{r+2-n}z\no\\
&\quad\quad\quad\times\big[1+\Oh\big(z^2|x-y|^2\big)\big]\alpha\cdot\frac{(x-y)}{|x-y|}\no\\
&\quad\quad-\frac{i2^{-n}\pi^{-n/2}n!}{(n/2)!(n-r)!}|x-y|z^{n-r}\ln(z|x-y|/2)\big[1+\Oh\big(z^2|x-y|^2\big)\big]\alpha\cdot\frac{(x-y)}{|x-y|}\no\\
&\quad\quad+i2^{-n}\pi^{-n/2}\sum_{\ell=0}^{r-1}
\begin{pmatrix}
r\\
\ell
\end{pmatrix}
\frac{(-1)^{r-\ell}(r-\ell-1)!\, n!}{(n/2)!\, (n-\ell)!}|x-y|z^{n-r}\no\\
&\quad\quad\quad\times\big[1+\Oh\big(z^2|x-y|^2\big)\big]\alpha\cdot\frac{(x-y)}{|x-y|}\no\\
&\quad\quad+\frac{i[\psi(1)+\psi((n/2)+1)]n!}{2^{1+n}\pi^{n/2}(n/2)!}|x-y|z^{n-r}\big[1+\Oh\big(z^2|x-y|^2\big)\big]\alpha\cdot\frac{(x-y)}{|x-y|}   \lb{B.22z} \\
&\hspace*{5.2cm} \text{as $z\to 0$, $z\in \ol{\bbC_+}\big\backslash\{0\}$, $x,y\in \bbR^n$, $x\neq y$}.\no
\end{align}

\noindent
$(ii)$  If $n\in\bbN\backslash\{2\}$ is even and $1\leq r\leq n-2$ is even with $r\neq n$, then
\begin{align}
&\frac{\partial^r}{\partial z^r}G_0(z;x,y) 
= \frac{i2^{-n}\pi^{1-(n/2)}(n-1)!}{(n-1-r)!\, ((n/2)-1)!}z^{n-1-r}\big[1+\Oh\big(z^2|x-y|^2\big)\big]I_N\no\\
&\quad\quad+\chi_{{}_{\leq n-3}}(r)\frac{((n/2)-(r/2)-2)!\, (r+1)!}{4^{1+(r/2)}\pi^{n/2}(r/2)!}|x-y|^{2+r-n}z\big[1+\Oh\big(z^2|x-y|^2\big)\big]I_N\no\\
&\quad\quad-\frac{2^{1-n}\pi^{-n/2}(n-1)!}{(n-1-r)!\, ((n/2)-1)!}z^{n-1-r}\ln(z|x-y|/2)\big[1+\Oh\big(z^2|x-y|^2\big)\big]I_N\no\\
&\quad\quad+\frac{2^{1-n}}{\pi^{n/2}}\sum_{\ell=0}^{r-1}
\begin{pmatrix}
r\\
\ell
\end{pmatrix}
(-1)^{r-\ell}\frac{(r-\ell-1)!\, (n-1)!}{(n-1-\ell)!\, ((n/2)-1)!}z^{n-1-r}  \no \\
& \hspace*{2.6cm}\times \big[1+\Oh\big(z^2|x-y|^2\big)\big]I_N\no\\
&\quad\quad+\frac{2^{-n}\pi^{-n/2}[\psi(1)+\psi(n/2)](n-1)!}{(n-1-r)!\, ((n/2)-1)!}z^{n-1-r}\big[1+\Oh\big(z^2|x-y|^2\big)\big]I_N\no\\
&\quad\quad-\frac{2^{-(1+n)}\pi^{1-(n/2)}n!}{(n-r)!\, (n/2)!}|x-y|z^{n-r}\big[1+\Oh\big(z^2|x-y|^2\big)\big]\alpha\cdot\frac{(x-y)}{|x-y|}\no\\
&\quad\quad+\chi_{{}_{\leq n-2}}(r)\frac{i ((n/2)-(r/2)-1)!\, r!}{4^{(r+1)/2}\pi^{n/2}(r/2)!}|x-y|^{1+r-n}\no\\
&\quad\quad\quad\times\big[1+\Oh\big(z^2|x-y|^2\big)\big]\alpha\cdot\frac{(x-y)}{|x-y|}\no\\
&\quad\quad-\frac{i2^{-n}\pi^{-n/2}n!}{(n/2)!(n-r)!}|x-y|z^{n-r}\ln(z|x-y|/2)\big[1+\Oh\big(z^2|x-y|^2\big)\big]\alpha\cdot\frac{(x-y)}{|x-y|}\no\\
&\quad\quad+i2^{-n}\pi^{-n/2}\sum_{\ell=0}^{r-1}
\begin{pmatrix}
r\\
\ell
\end{pmatrix}
\frac{(-1)^{r-\ell}(r-\ell-1)!\, n!}{(n/2)!\, (r-\ell)!}|x-y|z^{n-r}\no\\
&\quad\quad\quad\times\big[1+\Oh\big(z^2|x-y|^2\big)\big]\alpha\cdot\frac{(x-y)}{|x-y|}\no\\
&\quad\quad+\frac{i[\psi(1)+\psi((n/2)+1)]n!}{2^{1+n}\pi^{n/2}(n/2)!}|x-y|z^{n-r}\big[1+\Oh\big(z^2|x-y|^2\big)\big]\alpha\cdot\frac{(x-y)}{|x-y|}    \lb{B.23zz} \\
&\hspace*{5.15cm} \text{as $z\to 0$, $z\in \ol{\bbC_+}\big\backslash\{0\}$, $x,y\in \bbR^n$, $x\neq y$}.\no
\end{align}

\noindent
$(iii)$  If $n\in\bbN\backslash\{2\}$ is even, then
\begin{align}
&\frac{\partial^n}{\partial z^n}G_0(z;x,y) 
=-\frac{i\pi^{1-(n/2)}(n+1)!}{2^{n+2}(n/2)!}|x-y|^2z\big[1+\Oh\big(z^2|x-y|^2\big)\big]I_N\no\\
&\quad\quad+\frac{2^{1-n}\pi^{-n/2}4^{-1}(n+1)!}{(n/2)!}|x-y|^2z\ln(z|x-y|/2)\big[1+\Oh\big(z^2|x-y|^2\big)\big]I_N\no\\
&\quad\quad+\frac{(n-1)!}{2^{n-1}\pi^{n/2}(\frac{n}{2}-1)!}\Bigg(\sum_{\ell=0}^{n-1}
\begin{pmatrix}
n\\
\ell
\end{pmatrix}
(-1)^{\ell}\Bigg)z^{-1}\big[1+\Oh\big(z^2|x-y|^2\big)\big]I_N\no\\
&\quad\quad -\frac{[\psi(2)+\psi((n/2)+1)](n+1)!}{2^{n+2}\pi^{n/2}(n/2)!}|x-y|^2z\big[1+\Oh\big(z^2|x-y|^2\big)\big]I_N\no\\
&\quad\quad -\frac{\pi^{1-(n/2)}n!}{2^{n+1}(n/2)!}|x-y|\big[1+\Oh\big(z^2|x-y|^2\big)\big]\alpha\cdot\frac{(x-y)}{|x-y|}\no\\
&\quad\quad-\frac{i2^{-n}\pi^{-n/2}n!}{(n/2)!}|x-y|\ln(z|x-y|/2)\big[1+\Oh\big(z^2|x-y|^2\big)\big]\alpha\cdot\frac{(x-y)}{|x-y|}\no\\
&\quad\quad+\frac{i2^{-n}\pi^{-n/2}n!}{(n/2)!}\Bigg(\sum_{\ell=0}^{n-1}
\begin{pmatrix}
n\\
\ell
\end{pmatrix}
\frac{(-1)^{\ell}}{n-\ell}\Bigg)|x-y|\no\\
&\quad\quad\quad\times\big[1+\Oh\big(z^2|x-y|^2\big)\big]\alpha\cdot\frac{(x-y)}{|x-y|}\no\\
&\quad\quad+\frac{i[\psi(1)+\psi((n/2)+1)]n!}{2^{1+n}\pi^{n/2}(n/2)!}|x-y|\big[1+\Oh\big(z^2|x-y|^2\big)\big]\alpha\cdot\frac{(x-y)}{|x-y|}     \lb{B.24z} \\
&\hspace*{4.45cm} \text{as $z\to 0$, $z\in \ol{\bbC_+}\big\backslash\{0\}$, $x,y\in \bbR^n$, $x\neq y$}.\no
\end{align}

$(\textbf{III})$  If $n=2$, then \eqref{5.8}, \eqref{5.8a}, and \eqref{5.8dd} imply
\begin{align}
&H_{(n/2)-1}^{(1)}(\zeta) = H_0^{(1)}(\zeta) = J_0(\zeta) + iY_0(\zeta)\no\\
&\quad= i\frac{2}{\pi}[\ln(\zeta/2) + \gamma_{E-M} - i (\pi/2)] J_0(\zeta) - \frac{2}{\pi}\sum_{k=1}^{\infty} \bigg(\sum_{\ell=1}^k\frac{1}{\ell}\bigg)\frac{(-4)^{-k}\zeta^{2k}}{(k!)^2}\no\\
&\quad= i\frac{2}{\pi} [\ln(\zeta/2) + \gamma_{E-M} - i (\pi/2)] \sum_{k=0}^{\infty}\frac{(-4)^{-k}\zeta^{2k}}{(k!)^2} - \frac{2}{\pi}\sum_{k=1}^{\infty} \bigg(\sum_{\ell=1}^k\frac{1}{\ell}\bigg)\frac{(-4)^{-k}\zeta^{2k}}{(k!)^2}.\lb{B.10a}
\end{align}

Similarly, by combining \eqref{B.1}, \eqref{B.7} (which is valid for $n=2$), and \eqref{B.10a}, one obtains for $n=2$:
\begin{align}
&G_0(z;x,y) 
= i4^{-1}z H_0^{(1)}(z|x-y|)I_N - 4^{-1}zH_1^{(1)}(z|x-y|)\alpha\cdot\frac{(x-y)}{|x-y|}\no\\
&\quad = -\frac{1}{2\pi} [\ln(z|x-y|/2) + \gamma_{E-M} - i (\pi/2)] \sum_{k=0}^{\infty}\frac{(-4)^{-k}z^{2k+1}|x-y|^{2k}}{(k!)^2}I_N   \no \\
&\quad \quad - \frac{i}{2\pi}\sum_{k=1}^{\infty}\bigg(\sum_{\ell=1}^k\frac{1}{\ell}\bigg) \frac{(-4)^{-k}z^{2k+1}|x-y|^{2k}}{(k!)^2}I_N\no\\
&\quad\quad -\frac{1}{8}\sum_{k=0}^{\infty} \frac{(-4)^{-k}z^{2k+2}|x-y|^{2k+1}}{k!(k+1)!}\alpha\cdot \frac{(x-y)}{|x-y|}\no\\
&\quad\quad + \frac{i}{2\pi}|x-y|^{-1}\alpha\cdot \frac{(x-y)}{|x-y|}\no\\
&\quad\quad -\frac{i}{4\pi}\ln(z|x-y|/2)\sum_{k=0}^{\infty}\frac{(-4)^{-k}z^{2k+2}|x-y|^{2k+1}}{k!\, (k+1)!}\alpha\cdot\frac{(x-y)}{|x-y|}\no\\
&\quad\quad + \frac{i}{8\pi}\sum_{k=0}^{\infty}[\psi(k+1)+\psi(k+2)]\frac{(-4)^{-k}z^{2k+2}|x-y|^{2k+1}}{k!\, (k+1)!}\alpha\cdot\frac{(x-y)}{|x-y|}.   \lb{B.13z} 
\end{align}
The identity in \eqref{B.13z} implies
\begin{align}
G_0(z;x,y)&= -\frac{1}{2\pi}z\ln(z|x-y|/2)\big[1+\Oh\big(z^2|x-y|^2\big)\big]I_N      \no \\
&\quad - \frac{1}{2\pi} [\gamma_{E-M}-i (\pi/2)] z \big[1+\Oh\big(z^2|x-y|^2\big)\big] I_N\no\\
&\quad + \frac{i}{8\pi}z^3|x-y|^2\big[1+\Oh\big(z^2|x-y|^2\big)\big] I_N\no\\
&\quad -\frac{1}{8}z^2|x-y|\big[1+\Oh\big(z^2|x-y|^2\big)\big]\alpha\cdot\frac{(x-y)}{|x-y|}\no\\
&\quad + \frac{i}{2\pi}|x-y|^{-1}\alpha\cdot\frac{(x-y)}{|x-y|}\no\\
&\quad - \frac{i}{4\pi} z^2|x-y|\ln(z|x-y|/2)\big[1+\Oh\big(z^2|x-y|^2\big)\big]\alpha\cdot\frac{(x-y)}{|x-y|}\no\\
&\quad +\frac{i}{8\pi}[\psi(1)+\psi(2)]z^2|x-y|\big[1+\Oh\big(z^2|x-y|^2\big)\big]\alpha\cdot\frac{(x-y)}{|x-y|} \lb{B.23z} \\
& \hspace*{1.35cm} \text{as $z\to 0$, $z\in \ol{\bbC_+}\big\backslash\{0\}$, $x,y\in \bbR^n$, $x\neq y$, and $n=2$.}   \no
\end{align}
One notes that \eqref{B.23z} implies
\begin{equation}
 \lim_{\substack{z \to 0, \\ z \in \ol{\bbC_+} \backslash\{0\}}} G_0(z;x,y) 
= \f{i}{2 \pi} \alpha\cdot\frac{(x-y)}{|x-y|^2},  \quad 
 x,y\in \bbR^2,\; x\neq y,    \lb{B.28z} 
\end{equation}
which is consistent with \eqref{5.18}.

Finally, one employs \eqref{B.13z} to compute:
\begin{align}
&\frac{\partial}{\partial z} G_0(z;x,y) 
= -\frac{1}{2\pi}\sum_{k=0}^{\infty} \frac{(-4)^{-k}z^{2k}|x-y|^{2k}}{(k!)^2}I_N   \no \\
&\quad\quad -\frac{1}{2\pi} [\ln(z|x-y|/2) + \gamma_{E-M} - i (\pi/2)] \sum_{k=0}^{\infty} \frac{(-4)^{-k}(2k+1)z^{2k}|x-y|^{2k}}{(k!)^2}I_N\no\\
&\quad\quad -\frac{i}{2\pi} \sum_{k=1}^{\infty}\bigg(\sum_{\ell=1}^k \frac{1}{\ell}\bigg)\frac{(-4)^{-k}(2k+1)z^{2k}|x-y|^{2k}}{(k!)^2}I_N\no\\
&\quad\quad -\frac{1}{8}\sum_{k=0}^{\infty} \frac{(-4)^{-k}(2k+2)z^{2k+1}|x-y|^{2k+1}}{k!\, (k+1)!}\alpha\cdot\frac{(x-y)}{|x-y|}\no\\
&\quad\quad -\frac{i}{4\pi}\sum_{k=0}^{\infty}\frac{(-4)^{-k}z^{2k+1}|x-y|^{2k+1}}{k!\, (k+1)!}\alpha\cdot\frac{(x-y)}{|x-y|}\no\\
&\quad\quad-\frac{i}{4\pi}\ln(z|x-y|/2)\sum_{k=0}^{\infty}\frac{(-4)^{-k}(2k+2)z^{2k+1}|x-y|^{2k+1}}{k!\, (k+1)!}\alpha\cdot\frac{(x-y)}{|x-y|}   \lb{B.24} \\
&\quad\quad +\frac{i}{8\pi}\sum_{k=0}^{\infty} [\psi(k+1)+\psi(k+2)]\frac{(-4)^{-k}(2k+2)z^{2k+1}|x-y|^{2k+1}}{k!\, (k+1)!}\alpha\cdot\frac{(x-y)}{|x-y|},   \no 
\end{align}
and
\begin{align}
&\frac{\partial^2}{\partial z^2} G_0(z;x,y) 
= -\frac{1}{2\pi}\sum_{k=1}^{\infty}\frac{(-4)^{-k}(2k)z^{2k-1}|x-y|^{2k}}{(k!)^2}I_N   \no \\
&\quad\quad-\frac{1}{2\pi}z^{-1}\sum_{k=0}^{\infty}\frac{(-4)^{-k}(2k+1)z^{2k}|x-y|^{2k}}{(k!)^2}I_N\no\\
&\quad\quad -\frac{1}{2\pi}[\ln(z|x-y|/2)+\gamma_{E-M}-i (\pi/2)] \no\\
&\quad\quad\quad\times \sum_{k=1}^{\infty}\frac{(-4)^{-k}(2k+1)(2k)z^{2k-1}|x-y|^{2k}}{(k!)^2}I_N\no\\
&\quad\quad-\frac{i}{2\pi}\sum_{k=1}^{\infty}\bigg(\sum_{\ell=1}^k\frac{1}{\ell}\bigg) \frac{(-4)^{-k}(2k+1)(2k)z^{2k-1}|x-y|^{2k}}{(k!)^2}I_N\no\\
&\quad\quad -\frac{1}{8}\sum_{k=0}^{\infty} \frac{(-4)^{-k}(2k+2)(2k+1)z^{2k}|x-y|^{2k+1}}{k!\, (k+1)!}\alpha\cdot\frac{(x-y)}{|x-y|}\no\\
&\quad\quad -\frac{i}{4\pi} \sum_{k=0}^{\infty} \frac{(-4)^{-k}(2k+1)z^{2k}|x-y|^{2k+1}}{k!\, (k+1)!}\alpha\cdot\frac{(x-y)}{|x-y|}\no\\
&\quad\quad -\frac{i}{4\pi}\sum_{k=0}^{\infty} \frac{(-4)^{-k}(2k+2)z^{2k}|x-y|^{2k+1}}{k!\, (k+1)!}\alpha\cdot\frac{(x-y)}{|x-y|}\no\\
&\quad\quad -\frac{i}{4\pi}\ln(z|x-y|/2)\sum_{k=0}^{\infty}\frac{(-4)^{-k}(2k+2)(2k+1)z^{2k}|x-y|^{2k+1}}{k!\, (k+1)!}\alpha\cdot\frac{(x-y)}{|x-y|}\no\\
&\quad\quad +\frac{i}{8\pi}\sum_{k=0}^{\infty} [\psi(k+1)+\psi(k+2)]\frac{(-4)^{-k}(2k+2)(2k+1)z^{2k}|x-y|^{2k+1}}{k!\, (k+1)!}\no\\
&\quad\quad\quad \times\alpha\cdot\frac{(x-y)}{|x-y|}.    \lb{B.25}
\end{align}

Finally, the expansions in \eqref{B.24} and \eqref{B.25} imply the following asymptotics of 
$\frac{\partial^r}{\partial z^r} G_0(z;x,y)$, $1\leq r\leq 2$, as $z\to 0$: \\[1mm] 
\noindent
$(i)$ If $n=2$, $r=1$, then
\begin{align}
&\frac{\partial}{\partial z}G_0(z;x,y) 
= -\frac{1}{2\pi} [\ln(z|x-y|/2)+1+\gamma_{E-M}-i (\pi/2)] \big[1+\Oh\big(z^2|x-y|^2\big)\big]I_N  \no \\
&\quad\quad + \frac{3i}{8\pi}z^2|x-y|^2\big[1+\Oh\big(z^2|x-y|^2\big)\big]I_N\no\\
&\quad\quad -\frac{1}{4\pi}\big\{\pi+i-i[\psi(1)+\psi(2)]\big\} z|x-y|\big[1+\Oh\big(z^2|x-y|^2\big)\big]\alpha\cdot\frac{(x-y)}{|x-y|}\no\\
&\quad\quad-\frac{i}{2\pi}z|x-y|\ln(z|x-y|/2)\big[1+\Oh\big(z^2|x-y|^2\big)\big]\alpha\cdot\frac{(x-y)}{|x-y|}   \lb{B.31z} \\
&\hspace*{3.4cm} \text{as $z\to 0$, $z\in \ol{\bbC_+}\big\backslash\{0\}$, $x,y\in \bbR^2$, $x\neq y$.}\no
\end{align}

\noindent
$(ii)$  If $n=2$, $r=2$, then
\begin{align}
&\frac{\partial^2}{\partial z^2} G_0(z;x,y) 
= \frac{1}{4\pi} [1+3(\gamma_{E-M}+i) - (3i\pi/2)] z|x-y|^2 \big[1 + \Oh\big(z^2|x-y|^2\big)\big]I_N \no\\
&\quad\quad -\frac{1}{2\pi}z^{-1}\big[1 + \Oh\big(z^2|x-y|^2\big)\big]I_N \no\\
&\quad\quad +\frac{3}{4\pi}z|x-y|^2\ln(z|x-y|/2) \big[1 + \Oh\big(z^2|x-y|^2\big)\big]I_N \no\\
&\quad\quad -\frac{\pi+3i}{4\pi}|x-y|\big[1 + \Oh\big(z^2|x-y|^2\big)\big]\alpha\cdot\frac{(x-y)}{|x-y|} \no\\
&\quad\quad -\frac{i}{2\pi}\ln(z|x-y|/2)|x-y|\big[1 + \Oh\big(z^2|x-y|^2\big)\big]\alpha\cdot\frac{(x-y)}{|x-y|} \no\\
&\quad\quad + \frac{i}{4\pi} [\psi(1)+\psi(2)] |x-y| \big[1 + \Oh\big(z^2|x-y|^2\big)\big]\alpha\cdot\frac{(x-y)}{|x-y|}  \lb{B.32z}  \\
&\hspace*{3.25cm} \text{as $z\to 0$, $z\in \ol{\bbC_+}\big\backslash\{0\}$, $x,y\in \bbR^2$, $x\neq y$.}\no
\end{align}

Given the results in Appendices \ref{sB} and \ref{sC}, we can summarize the estimates on $G_0(z;\dott,\dott)$ as follows:

\begin{theorem} \lb{tC.1}
Let $r \in \bbN_0$, $0 \leq r \leq n$, $z \in \ol{\bbC_+}$, and $x, y \in \bbR^n$, $x \neq y$. \\[1mm]
$(i)$ For $n \in \bbN$ odd, $n \geq 3$, one has the estimate 
\begin{align}
& \bigg\|\f{\partial^r}{\partial z^r} G_0(z;x,y)\bigg\|_{\cB(\bbC^N)} \no \\ 
& \quad \leq \hatt c_n \begin{cases}
|x-y|^{r+1-n}, & |z||x-y| \leq 1, \\
|z|^{(n-1)/2} |x-y|^{(2r+1-n)/2} e^{- \Im(z) |x-y|}, & |z||x-y| \geq 1
\end{cases}    \no \\[1mm] 
& \quad \leq \wti c_n |x-y|^{r+1-n} \big[1 + |z|^{(n-1)/2} |x-y|^{(n-1)/2}\big]  \no \\[1mm] 
& \quad \leq \wti C_n \big\{|x-y|^{r+1-n} \chi_{[0,1]}(|z| |x-y|)   \no \\[1mm] 
& \hspace*{1.6cm} + |z|^{(n-1)/2} |x-y|^{(2r+1-n)/2} \chi_{[1,\infty)}(|z| |x-y|) \big\}   \no \\
& \quad \leq C_n \big\{|x-y|^{r+1-n} \chi_{[0,1]}(|z| |x-y|)     \no \\[1mm] 
& \hspace*{1.6cm} + |z|^{(n-1)/2} \big[|x|^{(2r+1-n)/2} + |y|^{(2r+1-n)/2}\big] \chi_{[1,\infty)}(|z| |x-y|) \big\}    \no \\
& \quad \leq c_n \big\{|x-y|^{r+1-n} \chi_{[0,1]}(|z| |x-y|)      \lb{C.33}  \\
&  \hspace*{1.5cm} + |z|^{(n-1)/2} [1+|x|]^{(2r+1-n)/2} [1+|y|]^{(2r+1-n)/2} \chi_{[1,\infty)}(|z| |x-y|) \big\},    \no 
\end{align}
where $\hatt c_n, \wti c_n, \wti C_n, C_n, c_n \in (0,\infty)$ are appropriate constants. \\[1mm] 
$(ii)$ For $n \in \bbN$ even, one has the following estimate.  For every $\delta\in (0,1)$,
\begin{align}
&\bigg\|\f{\partial^r}{\partial z^r} G_0(z;x,y)\bigg\|_{\cB(\bbC^N)}  \lb{C.34}  \\[1mm]
& \quad \leq c_n \begin{cases}
|x-y|^{r+1-n}\big[1+|\ln(z|x-y|/2)|\big], & |z||x-y| \leq 1,\, r\neq n, \\
|x-y|\big[1+|\ln(z|x-y|/2)|\big] + |z|^{-1}, & |z||x-y| \leq 1,\, r= n, \\
|z|^{(n-1)/2} |x-y|^{(2r+1-n)/2} e^{- \Im(z) |x-y|}, & |z||x-y| \geq 1\no
\end{cases}\\
& \quad \leq \wti c_{n,\delta} \begin{cases}
|z|^{-\delta}|x-y|^{r+1-\delta-n}, & |z||x-y| \leq 1,\, r\neq n, \\
|z|^{-\delta}|x-y|^{1-\delta}+|z|^{-1}, & |z||x-y| \leq 1,\, r= n, \\
|z|^{(n-1)/2} |x-y|^{(2r+1-n)/2} e^{- \Im(z) |x-y|}, & |z||x-y| \geq 1\no
\end{cases}\\
& \quad = \wti c_{n,\delta} \begin{cases}
|z|^{-\delta}|x-y|^{r+1-\delta-n}, & |z||x-y| \leq 1,\, r\neq n, \\
|z|^{-1}\big[|z|^{1-\delta}|x-y|^{1-\delta}+1\big], & |z||x-y| \leq 1,\, r= n, \\
|z|^{(n-1)/2} |x-y|^{(2r+1-n)/2} e^{- \Im(z) |x-y|}, & |z||x-y| \geq 1\no
\end{cases}\\
& \quad \leq C_{n,\delta} \begin{cases}
|z|^{-\delta}|x-y|^{r+1-\delta-n}, & |z||x-y| \leq 1,\, r\neq n, \\
|z|^{-1}, & |z||x-y| \leq 1,\, r= n, \\
|z|^{(n-1)/2} |x-y|^{(2r+1-n)/2} e^{- \Im(z) |x-y|}, & |z||x-y| \geq 1\no
\end{cases}
\end{align}
where $c_n, \wti c_{n,\delta}, C_{n,\delta} \in (0,\infty)$ are appropriate constants.
\end{theorem}
\begin{proof}
The first estimate in \eqref{C.33} (resp., in \eqref{C.34}) follows immediately in the regime $|z||x-y|\leq 1$ from \eqref{B.12y}, \eqref{B.14y}, and \eqref{B.13y} (resp., \eqref{B.22z}, \eqref{B.23zz}, and \eqref{B.24z} and \eqref{B.23z}, \eqref{B.31z}, and \eqref{B.32z}).  One employs Lemma \ref{lB.6} in conjunction with \eqref{B.1} to obtain the first estimate in \eqref{C.33}, and \eqref{C.34} in the regime $|z||x-y|\geq 1$.  In fact, by Lemma \ref{lB.6} and \eqref{B.1}, $G_0(z;\dott,\dott)$ is of the form
\begin{align}
G_0(z;x,y) 
&= c_1 z^{n/2}|x-y|^{1-(n/2)}H_{(n/2)-1}^{(1)}(z|x-y|)I_N\lb{C.36}\\
&\quad + c_2z^{n/2}|x-y|^{1-(n/2)}H_{n/2}^{(1)}(z|x-y|)\alpha\cdot \frac{(x-y)}{|x-y|}\no\\
& = z^{n/2}|x-y|^{1-(n/2)}e^{iz|x-y|}\no\\
&\quad\times \bigg[ c_1\omega_{(n/2)-1}(z|x-y|)I_N + c_2\omega_{n/2}(z|x-y|)\alpha\cdot \frac{(x-y)}{|x-y|}\bigg],\no\\
&\hspace*{5.35cm} x,y\in \bbR^n,\, x\neq y,\, z\in \bbC_+,\no
\end{align}
for an appropriate pair of constants $c_1,c_2\in \bbC$.  The constants $c_1$ and $c_2$ are independent of $(z,x,y)$, and their precise values are immaterial for the purpose at hand.  Differentiating throughout \eqref{C.36} with respect to $z$, one obtains
\begin{align}
\frac{\partial^r}{\partial z^r}G_0(z;x,y) 
&=\sum_{\substack{j,k,\ell\in \bbN_0 \\ j+k+\ell=r}} c_{j,k,\ell}z^{(n/2)-j}|x-y|^{r-j+1-(n/2)}e^{iz|x-y|}\lb{C.37}\\
&\quad \times \bigg[c_1\omega_{(n/2)-1}^{(\ell)}(z|x-y|)I_N + c_2\omega_{n/2}^{(\ell)}(z|x-y|)\alpha\cdot\frac{(x-y)}{|x-y|}\bigg],\no\\
&\hspace*{5.35cm} x,y\in \bbR^n,\, x\neq y,\, z\in \bbC_+,\no
\end{align}
where the $c_{j,k,\ell}$ are constants which do not depend upon $(z,x,y)$.  By \eqref{C.37},
\begin{align}
&\bigg\|\frac{\partial^r}{\partial z^r}G_0(z;x,y)\bigg\|_{\cB(\bbC^N)}\no\\
&\quad \leq \sum_{\substack{j,k,\ell\in \bbN_0 \\ j+k+\ell=r}} \wti c_{j,k,\ell}|z|^{(n/2)-j}|x-y|^{r-j+1-(n/2)}e^{-\Im(z)|x-y|}|z|^{-(1/2)-\ell}|x-y|^{-(1/2)-\ell}\no\\
&\quad \leq \sum_{\substack{j,k,\ell\in \bbN_0 \\ j+k+\ell=r}} \wti c_{j,k,\ell}|z|^{[(n-1)/2]-(j+\ell)}|x-y|^{r-(j+\ell)+(1/2)-(n/2)}e^{-\Im(z)|x-y|}\no\\
&\quad\leq \wti C |z|^{[(n-1)/2]}\big[|z|^{j+\ell}|x-y|^{j+\ell} \big]^{-1}|x-y|^{(2r+1-n)/2}e^{-\Im(z)|x-y|}\no\\
&\quad \leq \wti C|z|^{[(n-1)/2]}|x-y|^{(2r+1-n)/2}e^{-\Im(z)|x-y|},\no\\
&\hspace*{1cm} x,y\in \bbR^n,\, x\neq y,\, z\in \bbC_+,\,|z||x-y|\geq 1,\no
\end{align}
where the $\wti c_{j,k,\ell}$ are constants which do not depend upon $(z,x,y)$.
\end{proof}

\section{A Product Formula for Modified Fredholm Determinants} \lb{sD}
\renewcommand{\theequation}{D.\arabic{equation}}
\renewcommand{\thetheorem}{D.\arabic{theorem}}
\setcounter{theorem}{0} \setcounter{equation}{0}

The purpose of this appendix is to prove a product formula for regularized (modified) Fredholm determinants extending the well-known Hilbert--Schmidt case.

The result we have in mind is a quantitative version of the following fact:  

\begin{theorem} \lb{tD.1} 
Let $k \in \bbN$, and suppose $A, B \in \cB_k(\cH)$. Then
\begin{align}
{\det}_{\cH,k} ((I_{\cH} - A)(I_{\cH} - B)) = {\det}_{\cH,k} (I_{\cH} - A) {\det}_{\cH,k} (I_{\cH} - B) 
\exp(\tr_{\cH}(X_k(A,B))),     \lb{D.1a} 
\end{align}
where $X_k(\dott,\dott) \in \cB_1(\cH)$ is of the form
\begin{align} 
\begin{split} 
X_1(A,B) &= 0,  \\
X_k(A,B) &= \sum_{j_1,\dots,j_{2k-2} = 0}^{k-1} c_{j_1,\dots,j_{2k-2}} 
C_1^{j_1} \cdots C_{2k-2}^{j_{2k-2}}, \quad k \geq 2,
\end{split}
\end{align}
with 
\begin{align}
\begin{split} 
& c_{j_1,\dots,j_{2k-2}} \in \bbQ,   \\
& C_{\ell} = A \text{ or } B, \quad 1 \leq \ell \leq 2k-2,    \\
& k \leq \sum_{\ell=1}^{2k-2} j_{\ell} \leq 2k - 2, \quad k \geq 2.  
\end{split} 
\end{align}
\end{theorem} 

Explicitly, one obtains:
\begin{align}
X_1(A,B) &= 0,   \no \\
X_2(A,B) &= - AB,    \no \\
X_3(A,B) &= 2^{-1} \big[(AB)^2 - AB(A+B) - (A+B)AB\big],     \\
X_4(A,B) &= 2^{-1} (AB)^2 - 3^{-1} \big[AB(A+B)^2+(A+B)^2AB+(A+B)AB(A+B)\big]  \no \\
& \quad + 3^{-1} \big[(AB)^2(A+B)+(A+B)(AB)^2+AB(A+B)AB\big]     \no \\ 
& \quad - 3^{-1} (AB)^3, \no \\
& \hspace*{-9mm} \text{etc.}    \no 
\end{align}

When taking traces (what is actually needed in \eqref{D.1a}), this simplifies to
\begin{align}
\begin{split}
& \tr_{\cH} (X_1(A,B)) = 0,   \\
& \tr_{\cH}(X_2(A,B)) = - \tr_{\cH}(AB),    \\
& \tr_{\cH}(X_3(A,B)) = - \tr_{\cH}\big(ABA + BAB - 2^{-1} (AB)^2\big),     \\
& \tr_{\cH}(X_4(A,B)) = - \tr_{\cH}\big(A^3 B + A^2 B^2 + A B^3 + 2^{-1} (AB)^2    \\
& \hspace*{3.8cm} - (AB)^2A - B(AB)^2 + 3^{-1} (AB)^3\big),    \\
& \quad \text{etc.}
\end{split} 
\end{align}

In the rest of this appendix we will detail the characterization of $X_k(A,B)$ following the paper \cite{BCGLNSZ20}. We also refer to \cite{Fr17}, \cite{Ha05}, \cite{Ha16} for related, but somewhat 
different product formulas for regularized determinants.  

To prove a quantitative version of Theorem \ref{tD.1} and hence derive a formula for $X_k(A,B)$, we first need to recall some facts on the commutator subspace of an algebra of noncommutative polynomials. 

Let ${\rm Pol_2}$ be the free polynomial algebra in $2$ (noncommuting) variables, $A$ and $B$. Let $W$ be the set of noncommutative monomials (words in the alphabet $\{A,B\}$). (We recall that the set $W$ is a semigroup with respect to concatenation, $1$ is the neutral element of this semigroup, that is, $1$ is an empty word in this alphabet.) Every $x\in{\rm Pol_2}$ can be written as a sum
\begin{equation}
x=\sum_{w\in W}\hatt{x}(w)w.
\end{equation} 
Here the coefficients $\hatt{x}(w)$ vanish for all but finitely many $w\in W$. 

Let $[{\rm Pol_2},{\rm Pol_2}]$ be the commutator subspace of ${\rm Pol_2}$, that is, the linear span of commutators $[x_1,x_2]$, $x_1,x_2\in{\rm Pol_2}$. 

\begin{lemma} \lb{lD.2} 
One has $x\in[{\rm Pol_2},{\rm Pol_2}]$ provided that
\begin{equation} 
\sum_{m=1}^{L(w)}\hatt{x}\bigl(\sigma^m(w)\bigr)=0,\quad w\in W.
\end{equation} 
Here, 
$L(w)$ is the length of each word $w = w_1 w_2\cdots w_{L(w)}$,
$\sigma$ is the cyclic shift given by $\sigma(w)=w_2\cdots w_{L(w)}w_1$.  
\end{lemma}
\begin{proof} One notes that 
\begin{equation} 
x=\sum_{w\in W}\hatt{x}(w)w=\hatt{x}(1)+\sum_{w\neq1} L(w)^{-1}
\sum_{m=1}^{L(w)}\hatt{x}\bigl(\sigma^m(w)\bigr)\sigma^m(w).
\end{equation} 
Obviously, $(\sigma^m(w)-w)\in[{\rm Pol_2},{\rm Pol_2}]$ for each positive integer $m$ and thus,
\begin{equation} 
x\in \bigg(\hatt{x}(1)+\sum_{w\neq1} L(w)^{-1}\sum_{m=1}^{L(w)}\hatt{x}\bigl(\sigma^m(w)\bigr)w+[{\rm Pol_2},{\rm Pol_2}]\bigg).
\end{equation} 
By hypothesis, $\hatt{x}(1)=0$ and
\begin{equation} 
\sum_{m=1}^{L(w)}\hatt{x}\bigl(\sigma^m(w)\bigr)=0,\quad 1\neq w\in W,
\end{equation} 
completing the proof.
\end{proof}

Next, we need some notation. 
Let $k_1,k_2 \in \bbN_0 = \bbN \cup \{0\}$, and set
\begin{equation} 
z_{k_1,k_2}= \begin{cases}
0, & k_1=k_2=0, \\
k_1^{-1} A^{k_1}, & k_1\in \bbN, \, k_2 = 0, \\
k_2^{-1} B^{k_2}, & k_1 = 0, \, k_2 \in \bbN, \\
\sum_{j=1}^{k_1+k_2} j^{-1} 
\sum_{\substack{\pi\in S_j, \, |\pi|=3\\|\pi_1|+|\pi_3|=k_1\\|\pi_2|+|\pi_3|=k_2}}(-1)^{|\pi_3|}z_{\pi}, & k_1,k_2 \in \bbN. 
\end{cases}     \lb{D.z} 
\end{equation} 
Here, $S_j$ is the set of all partitions of the set $\{1,\cdots,j\}$, $1 \leq j \leq k_1 + k_2$. 
(The symbol $|\dott|$ abbreviating the cardinality of a subset of $\bbZ$.) 
The condition $|\pi|=3$ means that $\pi$ breaks the set $\{1,\cdots,j\}$ into exactly $3$ pieces denoted by $\pi_1$, $\pi_2$, and $\pi_3$ (some of them can be empty). The element $z_{\pi}$ denotes the product
\begin{equation}
z_{\pi}=\prod_{m=1}^jz_{m,\pi},\quad z_{m,\pi}=
\begin{cases}
A,& m\in\pi_1,\\
B,& m\in\pi_2,\\
AB,& m\in\pi_3.
\end{cases}
\end{equation} 
Finally, let $W_{k_1,k_2}$ be the collection of all words with $k_1$ letters $A$ and $k_2$ letters $B$. 

Using this notation we now establish a combinatorial fact. 

\begin{lemma} \lb{lD.3} 
Let $k_1,k_2 \in \bbN$. Then  
\begin{equation} 
z_{k_1,k_2}=\sum_{w\in W_{k_1,k_2}}
\Bigg(\sum_{\ell=0}^{n(w)}\frac{(-1)^{\ell}}{k_1+k_2-\ell}\binom{n(w)}{\ell}\Bigg)w, 
\end{equation} 
where 
\begin{equation} 
n(w)=|S(w)|,\quad S(w)= \{1\leq \ell \leq L(w)-1 \, | \,  w_{\ell}=A,\, w_{\ell +1}=B\}.
\end{equation} 
\end{lemma}
\begin{proof} For each $j\in \{1, \ldots, k_1 + k_2\}$, let
\begin{align} 
\Pi_j &= \{ \pi\in   S_j \, | \, |\pi| = 3,\, \, |\pi_1|+|\pi_3| = k_1, \, |\pi_2|+|\pi_3| = k_2\},    \\
\Pi_{j,w} &=\{ \pi\in \Pi_j \, | \, z_\pi = w\}, \quad w\in W_{k_1,k_2}.
\end{align} 
One observes that $|\pi_3|\leq n(w)\leq \min\{k_1,k_2\}$ and that
\begin{equation} 
j = |\pi_1|+|\pi_2|+|\pi_3| = k_1+k_2-|\pi_3|.
\end{equation} 
For any partition $\pi\in\Pi_{j,w}$, let $I\subseteq S(w)$ indicate which subwords $AB$ in $w$ arise from elements in $\pi_3$.  Then $|I|=|\pi_3|=k_1+k_2-j$.  Therefore, each partition in $\pi\in\Pi_{j,w}$ is determined by a unique choice of $I$ and each such choice of $I$ determines the choice of $\pi$ uniquely. This implies that 
\begin{equation} 
|\Pi_{j,w}| = \binom{n(w)}{k_1+k_2-j}.
\end{equation} 
Thus,
\begin{align} 
z_{k_1,k_2} &=\sum_{w\in W_{k_1,k_2}}\sum_{j=1}^{k_1+k_2} j^{-1} \sum_{\pi\in\Pi_{j,w}} (-1)^{|\pi_3|} w 
\no \\
&=\sum_{w\in W_{k_1,k_2}} \sum_{j=1}^{k_1+k_2} (-1)^{k_1+k_2-j} j^{-1} |\Pi_{j,w}|w    \no \\ 
&=\sum_{w\in W_{k_1,k_2}} \sum_{j=1}^{k_1+k_2} (-1)^{k_1+k_2-j} j^{-1} \binom{n(w)}{k_1+k_2-j}w.
\end{align}
Taking into account that
\begin{equation} 
\binom{n(w)}{k_1+k_2-j}=0,\quad k_1+k_2-j\notin\{0,\cdots,n(w)\},
\end{equation} 
it follows that
\begin{align} 
z_{k_1,k_2} &=\sum_{w\in W_{k_1,k_2}} 
\sum_{j=k_1+k_2-n(w)}^{k_1+k_2} (-1)^{k_1+k_2-j} j^{-1} \binom{n(w)}{k_1+k_2-j}w    \no \\
&=\sum_{w\in W_{k_1,k_2}} 
\Bigg(\sum_{\ell=0}^{n(w)}\frac{(-1)^{\ell}}{k_1+k_2-\ell}\binom{n(w)}{\ell}\Bigg)w.
\end{align} 
\end{proof}

We can now prove the main fact about the commutator  subspace of ${\rm Pol_2}$ needed later on.

\begin{lemma} \lb{lD.4} 
For every $k_1,k_2 \in \bbN$, $z_{k_1,k_2}\in[{\rm Pol_2},{\rm Pol_2}]$.  
\end{lemma}
\begin{proof} 
Let $w$ be any element in $W_{k_1,k_2}$ and let $m$ be any positive integer. 
If $\sigma^m(w)$ starts with the subword $AB$, 
then $\sigma^{m+1}(w)$ has the form $B\cdots A$ and therefore has one fewer subwords $AB$ than $\sigma^m(w)$; 
that is, $n\bigl(\sigma^{m+1}(w)\bigr) = n\bigl(\sigma^m(w)\bigr)-1$.
If, however, $\sigma^m(w)$ does not start with the subword $AB$, 
then the $AB$ subwords of $\sigma^{m+1}(w)$ are precisely the $AB$ subwords of $\sigma^m(w)$ each shifted once; 
hence, $n\bigl(\sigma^{m+1}(w)\bigr) = n\bigl(\sigma^m(w)\bigr)$.

Now, to calculate $\sum_{m=1}^{L(w)}\widehat{z_{k_1,k_2}}\bigl(\sigma^m(w)\bigr)$, 
one may assume, by applying cyclic shifts, that $w$ starts with $AB$. 
Then there are $n(w)$ shifted words $\sigma^m(w)$ which start with the subword $AB$, 
and it follows that $n(w)$ of the numbers $\{n\bigl(\sigma^m(w)\bigr)\::\: 1\leq m\leq L(w)\}$ equal $n(w)-1$ 
and that the remaining $L(w)-n(w) = k_1+k_2-n(w)$ numbers equal $n(w)$. 
Lemma~\ref{lD.3} therefore implies that 
\begin{align} 
  \sum_{m=1}^{L(w)}\widehat{z_{k_1,k_2}}\bigl(\sigma^m(w)\bigr)
&=\sum_{m=1}^{L(w)}\Bigg(\sum_{\ell=0}^{n(\sigma^m(w))}\frac{(-1)^{\ell}}{k_1+k_2-\ell}\binom{n\bigl(\sigma^m(w)\bigr)}{\ell}\Bigg)   \no \\
&=n(w) \Bigg(\sum_{\ell=0}^{n(w)-1}\frac{(-1)^{\ell}}{k_1+k_2-\ell}\binom{n(w)-1}{\ell}\Bigg)   \no \\
& \quad +(k_1+k_2-n(w)) \Bigg(\sum_{\ell=0}^{n(w)}\frac{(-1)^{\ell}}{k_1+k_2-\ell}\binom{n(w)}{\ell}\Bigg).
\end{align} 
Since
\begin{equation} 
\binom{n(w)-1}{n(w)}=0,
\end{equation} 
it follows that
\begin{align} 
\sum_{m=1}^{L(w)}\widehat{z_{k_1,k_2}}\bigl(\sigma^m(w)\bigr)
&=n(w) \Bigg(\sum_{\ell=0}^{n(w)}\frac{(-1)^{\ell}}{k_1+k_2-\ell}\binom{n(w)-1}{\ell}\Bigg)    \no \\
& \quad +(k_1+k_2-n(w)) \Bigg(\sum_{\ell=0}^{n(w)}\frac{(-1)^{\ell}}{k_1+k_2-\ell}\binom{n(w)}{\ell}\Bigg)   \no \\
&=\sum_{\ell=0}^{n(w)}\frac{(-1)^{\ell}}{k_1+k_2-\ell} \Bigg(n(w)\binom{n(w)-1}{\ell}    \no \\ 
& \quad +(k_1+k_2-n(w))\binom{n(w)}{\ell}\Bigg).\end{align} 
Clearly,
\begin{equation} 
n(w)\binom{n(w)-1}{\ell}+(k_1+k_2-n(w))\binom{n(w)}{\ell}=(k_1+k_2-\ell)\binom{n(w)}{\ell}, 
\end{equation} 
and thus 
\begin{equation} 
\sum_{m=1}^{L(w)}\widehat{z_{k_1,k_2}}\bigl(\sigma^m(w)\bigr)=\sum_{\ell=0}^{n(w)}(-1)^{\ell}\binom{n(w)}{\ell}=0.  
\end{equation} 
Hence, Lemma~\ref{lD.2} completes the proof. 
\end{proof}

Next, we introduce some further notation. Let $k \in \bbN$ and set
\begin{align} 
\begin{split} 
x_1 &= 0, \\
x_k &=\sum_{j=1}^{k-1} j^{-1} 
\sum_{\substack{\cA \subseteq \{1,\cdots,j\}\\ j+|\cA|\geq k}} 
(-1)^{|\cA|}y_{\cA}, \quad k \geq 2,   \lb{D.33} 
\end{split} \\
\begin{split} 
y_1 &= 0, \\
y_k&=\sum_{j=1}^{k-1} j^{-1} \sum_{\substack{\cA \subseteq \{1,\cdots,j\}\\ j+|\cA|\leq k-1}} 
(-1)^{|\cA|}y_{\cA},   \quad k \geq 2, 
\end{split} \\
y_{\cA}&=\prod_{m=1}^jy_{m,\cA},\quad y_{m,\cA}=
\begin{cases}
A+B,& m \notin \cA, \\
AB,& m \in \cA. 
\end{cases}
\end{align} 
In particular, 
\begin{equation} 
\sum_{j=1}^{k-1} j^{-1} (A+B-AB)^j = x_k + y_k,    \lb{D.binom} 
\end{equation} 
and one notes that the length of the word $y_{\cA}$ subject to $\cA \subseteq \{1,\dots,j\}$, equals 
\begin{equation}
L(y_{\cA}) = \big|\cA^c\big| + 2 |\cA| = j +|\cA|, \quad 1 \leq j \leq k-1, \; k \geq 2    \lb{D.number} 
\end{equation}
(with $A^c = \{1,\dots,j\} \backslash \cA$ the complement of $\cA$ in $\{1,\dots,j\}$). 

Using this notation we can now state the following fact:

\begin{lemma} \lb{lD.5} Let $k \in \bbN$, $k\geq 2$, then 
\begin{equation} 
y_k\in \bigg(\sum_{j=1}^{k-1}\frac1j(A^j+B^j)+[{\rm Pol_2},{\rm Pol_2}]\bigg).
\end{equation} 
\end{lemma}
\begin{proof} Employing 
\begin{equation} 
y_k=\sum_{\substack{k_1,k_2\geq0\\ k_1+k_2 \leq k-1}}z_{k_1,k_2},    \lb{D.a}
\end{equation} 
Lemma \ref{lD.4} yields 
\begin{equation} 
z_{k_1,k_2}\in [{\rm Pol_2},{\rm Pol_2}],\quad k_1,k_2 \in \bbN.   \lb{D.b}
\end{equation} 
Since by \eqref{D.z}, 
\begin{equation} 
z_{0,0}=0, \quad 
z_{k_1,0}= k_1^{-1} A^{k_1},\; k_1\in \bbN,\quad z_{0,k_2}= k_2^{-1} B^{k_2},\; k_2 \in \bbN,    \lb{D.c}
\end{equation} 
combining \eqref{D.a}--\eqref{D.c} completes the proof.
\end{proof}

After these preparations we are ready to return to the product formula for regularized determinants and  specialize the preceding algebraic considerations to the context of
Theorem \ref{tD.1}. 

First we recall that by \eqref{D.33} and \eqref{D.number},  
\begin{equation}
x_k =\sum_{j=1}^{k-1} j^{-1} \sum_{\substack{\cA\subseteq \{1,\cdots,j\}\\ j+|\cA|\geq k}}(-1)^{|\cA|}y_{\cA} := X_k(A,B) \in \cB_1(\cH),  
\quad k \geq 2,
\end{equation} 
since for $1 \leq j \leq k-1$, $L(y_{\cA}) = j+|\cA| \geq k$, and hence one obtains the inequality 
\begin{equation} 
\|x_k\|_{\cB_1(\cH)}\leq c_k\max_{\substack{0\leq k_1,k_2<k\\ k_1+k_2\geq k}} 
\|A\|_{\cB_k(\cH)}^{k_1}\|B\|_{\cB_k(\cH)}^{k_2}, \quad k \in \bbN, \; k \geq 2, 
\end{equation} 
for some $c_k > 0$, $k \geq 2$. We also set (cf.\ \eqref{D.33} $X_1(A,B) = 0$. 

\begin{theorem} Let $k \in \bbN$ and assume that $A,B\in \cB_k(\cH)$. Then 
\begin{equation} 
{\det}_{\cH,k}((I_{\cH}-A)(I_{\cH}-B))={\det}_{\cH,k}(I_{\cH}-A) {\det}_{\cH,k}(I_{\cH}-B) \exp({\tr}_{\cH}(X_k(A,B))).   \lb{D.38A} 
\end{equation} 
\end{theorem}
\begin{proof} First, we suppose that $A,B\in\cB_1(\cH)$. Then it is well-known that
\begin{equation} 
{\det}_{\cH,1}(I_{\cH}-A) {\det}_{\cH,1}(I_{\cH}-B)={\det}_{\cH,1}((I_{\cH}-A)(I_{\cH}-B)),
\end{equation} 
consistent with $X_1(A,B) = 0$. Without loss of generality we may assume that $k \in \bbN$, $k \geq 2$, in 
the following. Employing 
\begin{equation} 
{\det}_{\cH,k}(I_{\cH}-T)={\det}_{\cH}(I_{\cH}-T) \exp\bigg({\tr}_{\cH}\bigg(\sum_{j=1}^{k-1} j^{-1} T^j\bigg)\bigg), 
\quad T \in \cB_1(\cH), 
\end{equation} 
one infers that 
\begin{align} 
&{\det}_{\cH,k}((I_{\cH}-A)(I_{\cH}-B))={\det}_{\cH,k}(I_{\cH}-(A+B-AB))     \no \\
& \quad ={\det}_{\cH}(I_{\cH}-(A+B-AB)) 
\exp\bigg({\tr}_{\cH}\bigg(\sum_{j=1}^{k-1} j^{-1} (A+B-AB)^j\bigg)\bigg)  \no \\
& \quad ={\det}_{\cH}(I_{\cH}-A) {\det}_{\cH}(I_{\cH}-B) 
\exp\bigg({\tr}_{\cH}\bigg(\sum_{j=1}^{k-1} j^{-1} (A+B-AB)^j\bigg)\bigg)    \no \\
& \quad ={\det}_{\cH,k}(I_{\cH}-A) {\det}_{\cH,k}(I_{\cH}-B)   \no \\
& \qquad \times \exp\bigg({\tr}_{\cH}\bigg(\sum_{j=1}^{k-1} j^{-1} \big[(A+B-AB)^j-A^j-B^j\big]\bigg)\bigg).
\end{align} 

By \eqref{D.binom} one concludes that 
\begin{equation} 
{\tr}_{\cH}\bigg(\sum_{j=1}^{k-1} j^{-1} \big[(A+B-AB)^j-A^j-B^j\big]\bigg) 
= {\tr}_{\cH}(x_k) + {\tr}_{\cH}\bigg(y_k-\sum_{j=1}^{k-1} j^{-1} \big(A^j+B^j\big)\bigg).
\end{equation} 
By Lemma \ref{lD.5}, 
\begin{equation} 
y_k-\sum_{j=1}^{k-1} j^{-1} \big(A^j+B^j\big)
\end{equation} 
is a sum of commutators of polynomial expressions in $A$ and $B$. Hence, 
\begin{equation} 
\bigg(y_k-\sum_{j=1}^{k-1} j^{-1} \big(A^j+B^j\big)\bigg) \subset [\cB_1(\cH),\cB_1(\cH)], 
\end{equation} 
and thus, 
\begin{equation} 
{\tr}_{\cH} \bigg(y_k-\sum_{j=1}^{k-1} j^{-1} \big(A^j+B^j\big)\bigg)=0, 
\end{equation} 
proving assertion \eqref{D.38A} for $A,B\in\cB_1(\cH)$.

Since both, the right and left-hand sides in \eqref{D.38A} are continuous with respect to the norm in 
$\cB_k(\cH)$, and $\cB_1(\cH)$ is dense in $\cB_k(\cH)$, \eqref{D.38A} holds for arbitrary 
$A, B \in \cB_k(\cH)$.
\end{proof}

\section{Notational Conventions} \lb{sE}
\renewcommand{\theequation}{E.\arabic{equation}}
\renewcommand{\thetheorem}{E.\arabic{theorem}}
\setcounter{theorem}{0} \setcounter{equation}{0}

For convenience of the reader we now summarize most of our notational conventions used 
throughout this manuscript. \\[2mm] 

\newpage 

{\bf Basic Abbreviations:} \\[1mm]
\indent
We employ the shortcut $\bbN_0 = \bbN\cup\{0\}$.

$\lfloor \, \cdot \, \rfloor$ denotes the floor function on $\bbR$, that is, $\lfloor x \rfloor$ characterizes the largest integer less than or equal to $x \in \bbR$. Similarly, $\lceil \dott \rceil$ denotes ceiling function, that is, 
$\lceil x \rceil$ characterizes the smallest integer larger than or equal to $x \in \bbR$.

We abbreviate $\bbC_{\pm} = \{z \in \bbC \,|\, \pm \Im(z) > 0\}$. \\[2mm] 
{\bf Vectors and Matrices:} \\[1mm] 
\indent 
Vectors in $\bbR^n$ are denoted by $ x = (x_1,\ldots,x_n)^{\top} \in \bbR^n$ (with $\top$ abbreviating the transpose operation) or $ p  = (p_1,\ldots,p_n)^{\top} \in \bbR^n$, $n \in \bbN$. For $ x = (x_1,\ldots,x_n)^{\top} \in \bbR^n$ we abbreviate 
\begin{equation}
\langle x \rangle = \big(1+| x |^2\big)^{1/2},
\end{equation}
where $|x|= \big(x_1^2+\cdots+x_n^2\big)^{1/2}$ denotes the standard Euclidean norm of $x \in \bbR^n$, $n \in \bbN$.

The dot symbol, ``$\, \cdot \,$'',\ is used in three different ways: First, it denotes the standard scalar product in $\bbR^n$, 
\begin{equation}
x \cdot y = \sum_{j=1}^n x_j y_j, \quad x =(x_1,\dots,x_n)^{\top}, \; y = (y_1,\dots,y_n)^{\top} \in \bbR^n.
\end{equation}
Second, we will also use it for $n$-vectors of operators, ${\ul A} = (A_1,\ldots,A_n)^{\top}$ 
and ${\ul B} = (B_1,\ldots, B_n)^{\top}$ acting in the same Hilbert space in the form 
\begin{equation}
{\ul A} \cdot {\ul B} = \sum_{j=1}^n A_j B_j,
\end{equation}
whenever it is obvious how to resolve the domain issues of the possibly unbounded operators involved. 

For $X$ a given space, $T$ a linear operator in $X$, and 
$A = (a_{j,k})_{1 \leq j,k \leq N} \in \bbC^{N \times N}$ an $N \times N$ matrix with constant complex-valued entries acting in $\bbC^N$, $N \in \bbN$, we will avoid tensor product notation as in  
\begin{equation}
T \otimes A \, \text{ in } \, X \otimes \bbC^N,  
\end{equation} 
such that 
\begin{equation}
X \otimes \bbC^N \, \text{ is identified with the symbol } \, 
X^N = (X, \dots, X)^{\top}, 
\end{equation}
and 
\begin{equation}
T \otimes A \, \text{ is identified with } \, T A = (T a_{j,k})_{1 \leq j,k \leq N} = 
(a_{j,k} T)_{1 \leq j,k \leq N} = A T.     \lb{1.TA} 
\end{equation} 
That is, we interpret $T \otimes A$ as entrywise multiplication, resulting in an $N \times N$ block operator matrix $TA = AT$. Thus, if $T = (T_1,\dots,T_n)$, with $T_j$, $1 \leq j \leq n$, operators 
in $\cH$, and $A = (A_1,\dots,A_n)$, with $A_j \in \bbC^{N \times N}$, $1 \leq j \leq n$, $N \times N$ matrices with constant, complex-valued entries acting in $\bbC^N$, we will employ the dot symbol also in the form
\begin{equation}
T \cdot A = \sum_{j=1}^n T_j A_j = \sum_{j=1}^n A_j T_j = A \cdot T,
\end{equation} 
where $T_j A_j = A_j T_j$, $1 \leq j \leq n$, are defined as in \eqref{1.TA}. 

$A \in X^{m \times n}$, $m, n \in \bbN$, represents an $m \times n$ matrix 
$A = (A_{j,k})_{\substack{1 \leq j \leq m \\1 \leq k \leq n}}$, with entries $A_{j,k}$ in $X$, $1 \leq j \leq m, \,1 \leq k \leq n$. In particular, $F = (F_1, \dots, F_n)^{\top} \in X^n$ is a vector with $n$ components and $F_j \in X$ denotes its $j$-th component, $1\leq j\leq n$.

The identity operator in $\bbC^n$ is represented by $I_n$, , $n \in \bbN$. \\[2mm]
{\bf Special Functions and Function Spaces:} \\[1mm] 
\indent 
For special functions such as the Gamma function $\Gamma(\dott)$, the digamma function $\psi(\dott) = \Gamma'(\dott)/\Gamma(\dott)$, Bessel functions of order $\nu$, $J_{\nu}(\dott), Y_{\nu}(\dott)$, Hankel functions of the first kind and order $\nu$, $H^{(1)}_{\nu}(\dott)$, the Euler--Macheroni constant $\gamma_{E-M}$, etc., we refer to \cite{AS72}.  

To simplify notation, we frequently omit Lebesgue measure whenever possible and simply use $L^p(\bbR^n)$ instead of 
$L^p(\bbR^n; d^nx)$, $p \in (0,\infty) \cup \{\infty\}$. 

Weak $L^p$-spaces (i.e., Lorentz spaces 
$L^{p,q}(\bbR^n; d\rho)$ with $q=\infty$), are denoted by 
$L^p_{weak}(\bbR^n; d\rho)$, $p \in (0,\infty)$. Here $(\bbR^n,d\rho)$ represents a separable measure space and the measure $\rho$ is assumed to be $\sigma$-finite. The seminorm on $L^p_{weak}(\bbR^n; d\rho)$ is abbreviated by
\begin{equation}
\|f\|_{L^p_{weak}(\bbR^n; d\rho)} := \sup_{t > 0} \big(t [\mu_f(t)]^{1/p}\big), 
\quad p \in (0, \infty), 
\end{equation}
where
\begin{equation}
\mu_f(t) = \rho(\{x \in \bbR^n \, | \, |f(x)| > t\}).
\end{equation}
In particular,
\begin{align}
\begin{split}
\|f + g\|_{L^p_{weak}(\bbR^n; d\rho)} \leq 
c_p \big[\|f\|_{L^p_{weak}(\bbR^n; d\rho)} 
+ \|g\|_{L^p_{weak}(\bbR^n; d\rho)}\big],& \\ \
c_p = \max \big(2, 2^{1/p} \big), \; p \in (0,\infty).& 
\end{split} 
\end{align}
Again, we omit the measure $\rho$ and just employ the notation 
$L^p_{weak}(\bbR^n)$ in case $\rho$ equals Lebesgue measure on $\bbR^n$.   

If $n\in \bbN$ and $N\in \bbN\backslash\{1\}$, we set
\begin{equation}\lb{1.101}
[L^2(\bbR^n)]^N := L^2(\bbR^n; \bbC^N), \quad [W^{1,2}(\bbR^n)]^N := W^{1,2}(\bbR^n; \bbC^N), \, 
\text{ etc.}
\end{equation}

The symbol $\cF$ is used to denote the Fourier transform, $f^{\wedge} : = \cF f$, similarly $f^{\vee} := \cF^{-1} f$, $f \in \cS'(\bbR^n)$, with $\cS(\bbR^n)$ the Schwartz test function space, and $\cS'(\bbR^n)$ its dual with elements the tempered distributions. In particular,
\begin{align} 
& f^{\wedge}(p) = (2 \pi)^{- n/2} \int_{\bbR^n} d^n x \, e^{- i p \cdot x} f(x), \quad 
 f^{\vee}(x) = (2 \pi)^{- n/2} \int_{\bbR^n} d^n p \, e^{i p \cdot x} f(p),      \no \\
& \hspace*{7.3cm} p,x \in \bbR^n, \; f \in \cS(\bbR^n).
\end{align}

If $-\infty \leq a < b \leq \infty$ and $F$ maps $\bbC\backslash (a,b)$ to a normed linear space, then the normal boundary values of $F$ at $\lambda \in (a,b)$ (when these values exist) are denoted by $F(\lambda \pm i0):= \lim_{\varepsilon \downarrow 0}F(\lambda \pm i \varepsilon)$.  

For $n,k \in \bbN$ and an open set $\Omega\subset \bbR^n$, $C^k(\Omega)$ denotes the set of all $f:\Omega\to \bbC$ that are $k$ times continuously differentiable.
\\[2mm]
{\bf Linear Operators in Hilbert Spaces:} \\[1mm] 
Let $\cH$, $\cK$ be separable, complex Hilbert spaces, $(\, \cdot \,,\, \cdot \,)_{\cH}$ the scalar 
product in $\cH$ (linear in the second argument), $\|\, \cdot \, \|_{\cH}$ the norm on $\cH$, 
and $I_{\cH}$ the identity operator in $\cH$.

If $T$ is a linear operator mapping (a subspace of) a Hilbert space into another, then 
$\dom(T)$ and $\ker(T)$ denote the domain and kernel (i.e., null space) of $T$. 
The closure of a closable operator $A$ is denoted by $\ol A$. 
The set of closed linear operators with domain contained in $\cH$ and range contained in 
$\cK$ is denoted by $\cC(\cH,\cK)$ (or simply by $\cC(\cH)$ if $\cH=\cK$). 

The resolvent set, spectrum, and point spectrum (i.e., the set of eigenvalues) of a closed operator $T$ are denoted by $\rho(T)$, $\sigma(T)$, and $\sigma_p(T)$, respectively. 

If $S$ is self-adjoint in $\cH$, the family of strongly right-continuous spectral projections 
associated with $S$ is denoted by $E_S(\lambda) = E_S((- \infty, \lambda])$, $\lambda \in \bbR$, moreover, the singular, discrete, essential, absolutely continuous, and singularly continuous spectrum 
of $S$ are denoted by $\sigma_s(S)$, $\sigma_d(S)$, $\sigma_{ess}(S)$, $\sigma_{ac}(S)$, 
and $\sigma_{sc}(S)$, respectively. 

For a densely defined closed operator $S$ in $\cH$ we employ the abbreviation $\langle S \rangle := \big(I_{\cH} + |S|^2\big)^{1/2}$, and 
similarly, if $T = (T_1,\dots,T_n)^{\top}$, with $T_j$ densely defined and closed 
in $\cH$, $1 \leq j \leq n$, 
\begin{equation}
\langle T \rangle = \big(I_{\cH} + | T |^2\big)^{1/2}, \quad | T | 
= \big(|T_1|^2+\cdots+|T_n|^2\big)^{1/2},
\end{equation}
whenever it is obvious how to define $|T_1|^2+\cdots+|T_n|^2$ as a self-adjoint operator. 


The Banach spaces of bounded and compact linear operators on a separable complex Hilbert space $\cH$ are denoted by $\cB(\cH)$ and $\cB_\infty(\cH)$, respectively; the corresponding $\ell^p$-based Schatten--von Neumann ideals are denoted by $\cB_p (\cH)$, with associated norm abbreviated by 
$\|\cdot\|_{\cB_p(\cH)}$, $p \geq 1$. 

Following a standard practice in Mathematical Physics, we simplify the notation of operators of multiplication by a scalar or matrix-valued function $V$ and hence use $V$ rather than the more elaborate symbol $M_V$ throughout this manuscript.

\medskip 

\noindent
{\bf Acknowledgments.}~We are indebted to Will Green, Denis Potapov, and Marcus Waurick for very helpful discussions on this subject. The authors are indebted to the Banff International Research Station for Mathematical Innovation and Discovery (BIRS) for their extraordinary hospitality during the focused research group on {\it Extensions of index theory inspired by scattering theory} (17frg668), June 18--25, 2017, where part of this work was initiated.

 
\end{document}